\documentclass{amsart}
\usepackage{amsmath, amsthm,amsfonts,amssymb, mathrsfs,mathtools}
\makeatletter
 \def\@textbottom{\vskip \z@ \@plus 1pt}
 \let\@texttop\relax
\makeatother
\usepackage[OT2,T1]{fontenc}
\usepackage{hyperref}
\hypersetup{
colorlinks=true,
linkcolor=blue,
filecolor=blue,
urlcolor=blue,
citecolor=cyan}

\usepackage[pdftex]{color}
\usepackage[all,arc]{xy}
\usepackage{easyReview}
\usepackage{tikz-cd}
\usepackage{cite}
\usepackage{booktabs}
\usepackage{color} 
\usepackage[margin=2.5cm]{geometry}
\usepackage{graphicx}
\usepackage{grffile}
\usepackage{longtable}
\usepackage{bbm}
\usepackage[noend]{algpseudocode}
\usepackage{algorithmicx,algorithm}
\usepackage{rotating}
\usepackage{mathrsfs}
\usepackage{adjustbox}
\usepackage{xcolor}

\newtheorem{theorem}[subsection]{Theorem}
\newtheorem{lemma}[subsection]{Lemma}
\newtheorem{corollary}[subsection]{Corollary}
\newtheorem{conjecture}[subsection]{Conjecture}

\newtheorem{theoremm}[subsubsection]{Theorem}

\newtheorem{corollaryy}[subsubsection]{Corollary}

\newtheorem{proposition}[subsection]{Proposition}
\theoremstyle{definition}
\newtheorem{construction}[subsection]{Construction}
\newtheorem{definition}[subsection]{Definition}
\newtheorem{hypothesis}[subsection]{Hypothesis}

\newtheorem{remark}[subsection]{Remark}

\newtheorem{notation}[subsection]{Notation}

\newtheorem{propositionn}[subsubsection]{Proposition}

\newtheorem{definitionn}[subsubsection]{Definition}

\newtheorem{remarkk}[subsubsection]{Remark}

\numberwithin{equation}{subsection}

\def\calA{\mathcal{A}}
\def\calB{\mathcal{B}}

\def\calD{\mathcal{D}}
\def\calE{\mathcal{E}}
\def\calH{\mathcal{H}}

\def\calO{\mathcal{O}}

\def\calS{\mathcal{S}}

\def\gothe{\mathfrak{e}}

\def\gothp{\mathfrak{p}}

\def\AAA{\mathbb{A}}

\def\CC{\mathbb{C}}
\def\FF{\mathbb{F}}
\def\GG{\mathbb{G}}
\def\LL{\mathbb{L}}

\def\PP{\mathbb{P}}
\def\QQ{\mathbb{Q}}
\def\RR{\mathbb{R}}

\def\ZZ{\mathbb{Z}}

\def\rmM{\mathrm{M}}

\def\scrC{\mathscr{C}}

\DeclareMathOperator{\End}{End}
\DeclareMathOperator{\Gal}{Gal}
\DeclareMathOperator{\Hom}{Hom}
\DeclareMathOperator{\Aut}{Aut}

\DeclareMathOperator{\Res}{Res}
\DeclareMathOperator{\Spec}{Spec}

\renewcommand{\Im}{\mathrm{Im}}
\newcommand{\cHom}{\calH om}
\newcommand{\cris}{\mathrm{cris}}
\newcommand{\cSh}{\mathcal{S}h}
\newcommand{\dR}{\mathrm{dR}}

\newcommand{\et}{\mathrm{et}}
\newcommand{\Fil}{\mathrm{Fil}}
\newcommand{\Frob}{\mathrm{Frob}}
\newcommand{\GL}{\mathrm{GL}}
\newcommand{\id}{\mathrm{id}}
\newcommand{\Image}{\mathrm{Im}}
\newcommand{\Ker}{\mathrm{Ker}}
\newcommand{\Lie}{\mathrm{Lie}}

\newcommand{\pr}{\mathrm{pr}}

\newcommand{\T}{\rm T}
\newcommand{\A}{\rm A}
\newcommand{{\Sh}}{\mathrm{Sh}}

\newcommand{\Zp}{\ZZ_p}

\newcommand{\coker}{\mathrm{Coker}}

\newcommand{\diag}{\mathrm{diag}}

\newcommand{\JL}{\mathcal{J}\!\mathcal{L}}

\newcommand{\cH}{\mathcal{H}}

\newcommand{\cO}{\mathcal{O}}
\newcommand{\F}{\mathbb{F}}
\newcommand{\Z}{\mathbb{Z}}
\newcommand{\Q}{\mathbb{Q}}

\newcommand{\Tr}{\mathrm{Tr}}

\newcommand{\rp}{{\mathop{p}\limits^{\rightarrow}}_{*}}
\newcommand{\lp}{{\mathop{p}\limits^{\leftarrow}}_{*}}

\newcommand{\ch}{\mathrm{Ch}^{1}({\Sh}_{1,n-1}^{\rm ss},1,L)}

\DeclareMathOperator{\rank}{rank}

\newcommand{\Fpb}{\overline{\mathbb{F}}_p}
\newcommand{\cB}{\mathcal{B}}

\newcommand{\tcD}{\tilde\calD}
\newcommand{\tcE}{\tilde\calE}
\newcommand{\Def}{\mathscr{D}\mathrm{ef}}

\newcommand{\Gr}{\mathbf{Gr}}
\newcommand{\im}{\mathrm{Im}}

\newcommand{\len}{\mathrm{length}}

\newcommand{\hra}{\hookrightarrow}
\newcommand{\ra}{\rightarrow}

\newcommand{\xra}{\xrightarrow}

\usepackage{graphicx}

\title{Arithemetic level raising theorem for some unitary Shimura varieties mod $p$}
\author{Zijie Tao}
\address{Department of Mathematics, The University of Michigan}
\email{\href{mailto:zijietao@umich.edu}{{\tt zijietao@umich.edu}}}

\begin{document}

\begin{abstract}
    Let $F$ be a real quadratic field in which a fixed prime $p$ is inert, and $E_0$ be an imaginary quadratic field in which $p$ splits; put $E=E_0F.$ Let $\Sh_{1,n-1}$ be the special fiber over $\FF_{p^{2}}$ of the Shimura variety for $G(U(1,n-1)\times U(n-1,1))$ with hyperspecial level structure at $p$ for some integer $n\geq 2.$ Let $\Sh_{1,n-1}(K_{\gothp}^{1})$ be the special fiber over $\FF_{p^2}$ of a Shimura variety for $G(U(1,n-1)\times U(n-1,1))$ with parahoric level structure at $p$ for some integer $n\geq 2.$ We exhibit elements in the higher Chow group of the supersingular locus of $\Sh_{1,n-1}$ and investigate the stratification of this locus. Moreover, we study the geometry of $\Sh_{1,n-1}(K_{\gothp}^{1})$ and establish a version of the Ihara lemma. As an application, we prove the arithmetic level raising theorem for $n=3.$
\end{abstract}
\maketitle
\setcounter{tocdepth}{1}\setcounter{tocdepth}{1}
\tableofcontents
\section{Introduction}
The study of the geometry of Shimura varieties lies at the heart of the Langlands program. Their rich arithmetic structure forms a bridge between the world of automorphic representations and the world of Galois representations.

One of the interesting topics in this area is to prove the surjectivity of the arithmetic level raising map for unitary Shimura varieties. Rong Zhou introduced a new method in \cite{Zho23} to prove the surjectivity by calculating the higher Chow group ${\rm Ch}^{1}(X^{\rm ss} ,1, L)$ with $X^{\rm ss}$ to be the supersingular locus of the special fiber of a quaternionic Shimura variety and $L$ a $p$-coprime coefficient ring. More recently, Ruiqi Bai and his collaborator Hao Fu extended this approach by computing ${\rm Ch}^{1}(X^{\rm ss},1, L)$ for $X^{\rm ss}$ the supersingular locus of the special fiber of the unitary Shimura variety for $G(U(2r,1))$ with hyperspecial level structure at an inert prime $p.$ In both cases, the authors esablished a version of the Ihara lemma as a key step in proving the surjectivity of the arithmetic level raising map following the computations of the higher Chow group.

Let $F$ be a real quadratic field in which a fixed prime $p$ is inert, and $E_0$ be an imaginary quadratic field in which $p$ splits; put $E=E_0F.$ Let $\Sh_{1,n-1}$ denote the special fiber of the unitary Shimura varieties for $G(U(1,n-1)\times U(n-1,1))$ with hyperspecial level structure at $p.$ In our work, we calculate ${\rm Ch}^{1}({\Sh^{\rm ss}_{1,n-1}},1, L)$ with $L$ a $p$-coprime coefficient ring and $\Sh^{\rm ss}_{1,n-1}$ the supersingular locus of $\Sh_{1,n-1}$. Our approach is largely inspired by Zhiyuan Ding's work on toyshtukas as developed in \cite{Din19}. By leveraging the correspondences constructed in \cite{HTX17}, we reduce the computation to the problem of identifying principal divisors on certain Deligne-Luszting varieties.

To establish a form of the Ihara lemma, we study the Newton and the Ekedahl--Oort stratifications of $\Sh_{1,n-1}.$ Various cases of unitary Shimura varieties have been extensively investigated prior to our work. Viehmann and Wedhorn \cite{VW13} developed a general theory of Newton and Ekedahl--Oort stratifications for Shimura varieties of PEL-type with good reduction. Furthermore, Wooding studied these stratifications for $GU(m_{1},m_{2})$-Shimura variety  with hyperspecial level at an unramified prime $p$ for $0\leq m_1\leq m_2.$ B$\rm \ddot{u}$ltel and Wedhorn~\cite{BW06} studied $GU(1,n-1)$-Shimura variety with hyperspecial level at an inert prime $p$ and showed that the Newton, Ekedahl--Oort and final stratifications coincide on the non-supersingular locus. In our case, we give an explicit description of the Newton and Ekedahl--Oort stratifications of $\Sh_{1,n-1}$ and analyze their relationship. Moreover, we relate the Ekedahl--Oort stratification to the geometric correspondences ${\rm Y}_j$'s constructed in \cite{HTX17} for $1\leq j\leq n.$  To prove the Ihara lemma, we also introduce a unitary Shimura variety with parahoric level structure at $p$ and study its geometry in detail. Via  the Ihara lemma, we prove the surjectivity of the arithmetic level raising map for $n=3.$ The cases for $n\geq 4$ remain conjectural. The case for $n=2,$ which closely parallels the argument in \cite{Zho23}, is included in the appendix for completeness.

We explain the main results of this paper in more detail. As mentioned above,
let $F$ be a real quadratic field,  $E_0$ be an imaginary quadratic field, and $E=E_0F.$
 Let $p$ be a prime number inert in $F,$ and split in $E_0.$ 
 Let $ \gothp$ and $\bar \gothp$ denote the two places of $E$ above $p$ so that $E_\gothp$ and $ E_{\bar \gothp}$ are both isomorphic to $\QQ_{p^2},$ the unique unramified quadratic extension of $\QQ_{p}.$
For an integer $n\geq 1,$ let $G_{1,n-1}$ and $G_{0,n}$ be two kinds of algebraic groups over $\QQ$ to be defined in Section~\ref{S:Shimura data}. Then $G_{1,n-1}(\QQ_p)\simeq G_{0,n}(\QQ_p)\simeq \QQ_p^\times\times \GL_n(E_\gothp)$ and $G_{1,n-1}(\RR)$ (resp. $G_{0,n}(\RR)$) is the unitary similitude group with signature $(1,n-1)$ and $(n-1,1)$ (resp. $(0,n)$ and $(n,0)$) at the two archimedean places.

Let $\AAA$  denote the ring of finite adeles of $\QQ,$ and $\AAA^{\infty}$ be its finite part. Let $G$ be a similitude unitary group associated to a division algebra over $E$ equipped with an involution of second kind.
 Fix a sufficiently small open compact subgroup  $K\subseteq G(\AAA^{\infty})$  with $K_p=\ZZ_p^{\times}\times \GL_n(\ZZ_{p^{2}}) \subseteq G(\QQ_p),$ where $\ZZ_{p^2}$ is the ring of integers of $\QQ_{p^2}.$
Let $\calS h(G,K)$ be the Shimura variety associated to $G$ of level $K.$
  
According to Kottwitz \cite{Kot92},  when $K^p$ is neat,  $\calS h(G,K)$ admits a proper and smooth integral model over $\ZZ_{p^2}$ which parametrizes certain polarized abelian schemes with $K$-level structure. (See Section~\ref{S:defn of Shimura var}.)  
Let ${\Sh}_{1,n-1}$ (resp. ${\Sh}_{0,n}$) denote the special fiber of $\calS h(G_{1,n-1},K)$ (resp. $\calS h(G_{0,n},K)$) over $\FF_{p^2}$ and let $\overline{{\Sh}}_{1,n-1}$ (resp. $\overline{{\Sh}}_{0,n}$) be its geometric special fiber.  This is a proper smooth variety over $\FF_{p^2}$ of dimension $2n-2.$
Let ${\Sh}_{1,n-1}^{\mathrm{ss}}$ denote the supersingular locus of ${\Sh}_{1,n-1},$ i.e. the reduced closed subvariety of ${\Sh}_{1,n-1}$ that parametrizes supersingular abelian varieties. 
As illustrated in \cite{HTX17}, ${\Sh}_{1,n-1}^{\mathrm{ss}}$ is equidimensional of dimension $n-1.$

Let $\ch$ be the higher Chow group of ${\Sh}_{1,n-1}^{\rm ss}.$ Define $K_{\gothp}^{1}=\diag\{p^{-1},1,\dots,1\}\cdot K_{\gothp}\cdot \diag\{p,1,\dots,1\}\cap K_{\gothp},$ which is a parahoric subgroup of $K_{\gothp}=\GL_n(\ZZ_{p^2}).$
Let ${\Sh}_{0,n}(K_{\gothp}^{1})$ (resp. ${\Sh}_{1,n-1}(K_{\gothp}^{1})$) be the unitary Shimura variety group of level $K_{\gothp}^{1}$ with signature $(0,n)$ and $(n,0)$ (resp. $(1,n-1)$ and $(n-1,1)$) as in Definition~\ref{S:Parahoric level}. By Remark~\ref{D:two-projection}, there are two natural projections from ${\Sh}_{0,n}(K_{\gothp}^{1})$ (resp. ${\Sh}_{1,n-1}(K_{\gothp}^{1})$) to ${\Sh}_{0,n}$ (resp. ${\Sh}_{1,n-1}$), which we denote by ${{\mathop{p}\limits^{\leftarrow}}}$ and ${\mathop{p}\limits^{\rightarrow}}.$
In Definition~\ref{S:Hecke action} and Remark~\ref{S:hecke actionnn}, we define some Hecke actions on ${\Sh}_{0,n}$ and ${\Sh}_{0,n}(K_{\gothp}^{1}).$ Two most important Hecke actions are denoted by $\rm T$ and $\A$ which act on  ${\Sh}_{0,n}$ and ${\Sh}_{0,n}(K_{\gothp}^{1})$ respectively.

Then one of our main theorem with respect to the higher Chow group can be stated as follows:

\begin{theorem}
\label{baby goal}
Let $L$ be a $p$-coprime coefficient ring. With notations as above, we have
\begin{equation*}
    \ch=\Ker({\mathrm{H}}_{\acute{e}t}^{0}({{\Sh}_{0,n}(K_{\gothp}^{1})},L) \xrightarrow{\psi} {{\mathrm{H}}^{0}_{\acute{e}t}}({\Sh}_{0,n},L)^{\oplus n}),
\end{equation*}
where $\psi=(\lp,\rp, \rp {\A}, \cdots, \rp \A^{n-2}).$
\end{theorem}

The theorem gives a relation of unitary Shimura varieties with different signatures. With this theorem, we prove a form of the Ihara lemma, which is a key ingredient in the proof of the arithmetic level raising theorem inspired by \cite{Zho23}. We first need the following notations. Let $\mathtt c$ be the complex conjugation in $\Gal(E/F).$
\begin{definition}
    We say that a (complex) representation $\Pi$ of $\GL_n(\AAA_E)$ is \emph{RACSDC} (that is, regular algebraic conjugate self-dual cuspidal) if
    \begin{enumerate}
        \item $\Pi$ is an irreducible cuspidal automorphic representation;        \item $\Pi\circ \mathtt c\simeq \Pi^{\vee};$
        \item for every archimedean place $w$ of $F,$ $\Pi_{w}$ is regular algebraic.
    \end{enumerate}
\end{definition}

Fix an $RACSDC$ representation $\Pi$ of $\GL_n(\AAA_E).$ Denote by $\{\alpha_1,\cdots,\alpha_{n}\}$ the Satake parameters of $\Pi_{\gothp}.$ Let $\mathfrak{m}$ be the Hecke maximal ideal associate to $\Pi$ which is the kernel of a homomorphism $\phi_\Pi$ as defined in Section~\ref{S:Hecke algebra}.

From now on, we fix a strong coefficient field ${\rm L}$ and a prime $\lambda$ of ${\rm L}$ lying over $\ell$ as in Definition~\ref{S:Strong coeffcient field}. Let $\cO_\lambda$ be the maximal algebraic integral domain and $k_\lambda=\cO_\lambda/\lambda.$ 
\begin{hypothesis}
\label{Main hypo}
	We assume that 
    \begin{enumerate}
        \item The prime $l\nmid p(p^{2n-2}-1);$
        \item The Hecke maximal ideal $\mathfrak{m}$ is non-Eisenstein such that for every $i\neq d(a_{\bullet}),$ ${\rm H}^{i}_{\acute{e}t}({\overline{\Sh}}_{a_{\bullet}},\cO_\lambda)_{\mathfrak{m}}=0$ with ${\overline{\Sh}}_{a_{\bullet}}$ to be the geometric specical fiber of the unitary Shimura variety of signature $a_{\bullet}$ to be defined in Section~\ref{S:defn of Shimura var} and $d(a_{\bullet})$ to be its dimension and ${\rm H}^{d(a_{\bullet})}_{\acute{e}t}({\overline{\Sh}}_{a_{\bullet}},\cO_\lambda)_{\mathfrak{m}}$ is nozero and torsion-free.
        \item The Satake parameters $\alpha_{1},\cdots,\alpha_{n}$ mod $\mathfrak{m}$ at $\gothp$ are distinct.
        \item For any representation $\pi$ that is $\Pi$-congruent in the sense of \cite[Definition 6.1.8]{LTXZZ}, the multiplicities of $\pi$ for $a_{\bullet}=(1,n-1)$ and $a_{\bullet}=(0,n),$ denoted by $m_{1,n-1}(\pi)$ and $m_{0,n}(\pi),$ are equal.
        \item Let $a_\gothp^{(n)}$ be the eigenvalue on $\pi^{\GL_n(\cO_{E_\gothp})}$  of the Hecke operater ${\rm S}_\gothp$ to be defined in Definition~\ref{S:Hecke action}. Then $a_\gothp^{(n)}\equiv 1 \mod \mathfrak{m}.$
    \end{enumerate} 
\end{hypothesis}
\begin{hypothesis}
\label{Main hypo3}
    We assume that the Satake parameters $\alpha_i=p^2\alpha_{i+1} \mod \mathfrak{m}$ for $1\leq i\leq n-1.$
\end{hypothesis}
\begin{hypothesis}
\label{Main hypo2}
	We assume that 
	there exists one pair $(i,j)$ such that $\alpha_i=p^{2}\alpha_j \mod \mathfrak{m}$ and no pair $(i',j')$ such that $\alpha_{i'}=p\alpha_{j'} \mod \mathfrak{m}.$ 
\end{hypothesis}

It was proved in~\cite{HTX17} that $\Sh_{1,n-1}^{\rm ss}$ consists of irreducible components denoted by ${\rm Y}_j$ for $1\leq j\leq n$ which are parametrized by $\Sh_{0,n}.$ Then the closed immersions of these irreducible components into $\Sh_{1,n-1}$ induce the map
\[
\JL_{\mathfrak{m}}:{\rm H}^{0}_\et\big(\overline{\Sh}_{0,n},k_\lambda\big)_{\mathfrak{m}}^{\oplus n}\longrightarrow {\rm H}_\et^{2(n-1)}\big(\overline{\Sh}_{1,n-1},k_\lambda(n-1)\big)_{\mathfrak{m}}^{{\rm Fr}_{p^2}}
\]
In \cite{HTX17}, the injectivity of the map $\JL_{\mathfrak{m}}$ with coefficient $\overline{\QQ}_\ell$ was proved when $\alpha_1,\cdots,\alpha_n$ are all distinct. And under some other conditions, they showed it is an isomorphism. For our use, we proved a similar theorem with coefficient $k_\lambda:$
\begin{theorem}\label{T:main-theoremm}
Under Hypothesis~\ref{Main hypo}, the map 
$\JL_{\mathfrak{m}}$is an isomorphism.
\end{theorem}
The key step in the proof of Theorem~\ref{T:main-theoremm} is to show that ${\rm H}^{0}_\et\big(\overline{\Sh}_{0,n},k_\lambda\big)_{\mathfrak{m}}$ and ${\rm H}_\et^{2(n-1)}\big(\overline{\Sh}_{1,n-1},k_\lambda\big)_{\mathfrak{m}}$ can be expressed as the tensor product of the same $k_\lambda$ module with something depends only on $\bar\rho_\mathfrak{m}$ and the signature. In fact, we have the following proposition:
\begin{proposition}\label{S:mod l cohomologyy}
    For signature $a_\bullet=(a,n-a)$ and $K$ any open compact subgroup of $G(\AAA^\infty)$, we have 
    $$
{\rm H}_{\acute{e}t}^{d(a_\bullet)}(\overline{\Sh}_{a_\bullet}(K),k_\lambda)_\mathfrak{m}=\bar M_{K}\otimes_{k_\lambda}\big((\wedge^a\bar \rho_\mathfrak{m}\otimes_{k_\lambda}\wedge^{n-a}\bar \rho_\mathfrak{m})\otimes_{k_\lambda} {k}_\lambda(\sum\limits_i\frac{a_i(a_i-1)}{2})\big)
$$ for some $k_\lambda$ module $\bar M_{K}$ with trival $\Gamma_E$ action and depending only on $K$ and $\mathfrak{m}.$
\end{proposition}
Then we can prove the Ihara lemma:
\begin{theorem}
\label{Ihara lemma}
Under the Hypothesis~\ref{Main hypo}, we have: 
    \begin{enumerate}
        \item (Definite Ihara) The map 
        \begin{equation*}
        {\mathrm{H}}_{\acute{e}t}^{0}(\overline{\Sh}_{0,n}(K_{\gothp}^{1}),k_\lambda)_{\mathfrak{m}} \xrightarrow{\psi} {\mathrm{H}^{0}_{\acute{e}t}}({\overline{\Sh}}_{0,n},k_\lambda)^{\oplus n}_{\mathfrak{m}}
        \end{equation*} is surjective, where $\psi=(\lp,\rp, \rp {\A}, \cdots, \rp {\A}^{n-2}).$         
        \item (Indefinite Ihara) The map
        \begin{equation*}
        {\rm H}_{\acute{e}t}^{2(n-1)}({\overline{Sh}}_{1,n-1}(K_{\gothp}^{1}),k_\lambda(n))_{\mathfrak{m}} \xrightarrow{\psi} {\rm H^{2(n-1)}_{\acute{e}t}}({\overline{Sh}}_{1,n-1},k_\lambda(n))^{\oplus n}_{\mathfrak{m}}
        \end{equation*} is surjective with $\overline{Sh}_{1,n-1}, \overline{Sh}_{1,n-1}(K_{\gothp}^{1})$ the geometric generic fibers of $\calS h_{1,n-1},\calS h_{1,n-1}(K_{\gothp}^{1})$ and $\psi=(\lp,\rp, \rp {\A}, \cdots, \rp {\rm A}^{n-2})$ defined similarly. 
    \end{enumerate}
\end{theorem}

Before talking about the proof of Theorem~\ref{Ihara lemma}, we state a corollary coming directly from it:
\begin{corollary}
\label{level raising result}
    Fixed an unramified RACSDC representation $\Pi$ satisfies Hypothesis \ref{Main hypo} and \ref{Main hypo3}. Then there exists an irreducible representation $\Pi'$ of $\GL_n(\AAA_E)$ such that the associated Galois representation $\rho_{\Pi'}$ is residually isomorphic to $\rho_\Pi$ and the monodromy operator of $\rho_{\Pi'}$ is conjugate to $\begin{pmatrix}
        1&1\\
        0&1
    \end{pmatrix} \oplus 1_{n-2}.$
\end{corollary}

To prove Theorem~\ref{Ihara lemma}(1), we analyze the Newton and Ekedahl--Oort stratifications of ${{\Sh}}_{1,n-1}.$ We show that there are exactly $n^{2}$ Ekedahl--Oort strata, among which $\frac{n(n-1)}{2}$ of them lie in the nonsupersingular locus. It is proved that the $\mu$-ordinary locus is exactly the same as the Newton stratum with the highest dimension.
Furthermore, we describe how the Ekedahl--Oort strata in the supersingular locus relate to the geometric correspondences ${\rm Y}_{j}$'s. Based on \cite[Theorem 4.7]{Moo01}, we provide an explicit construction of the Dieudonn\'e module corresponding to each Ekedahl--Oort stratum. In addition, by analyzing Hasse invariants, we prove that the $\mu$--ordinary locus of ${{\Sh}}_{1,n-1}$ is affine.

Additionally, we study the geometry of ${{\Sh}}_{1,n-1}(K_{\gothp}^{1}),$ the special fiber of a unitary Shimura variety with parahoric level structure at $p.$ More explicitly, we construct $n$ correspondences ${\rm C}_{j}$ for $1\leq j\leq n$ between ${{\Sh}}_{1,n-1}(K_{\gothp}^{1})$ and ${{\Sh}}_{0,n}$ and show their image collectively cover the supersingular locus ${{\Sh}}_{1,n-1}(K_{\gothp}^{1}).$ Moreover, we prove that under a natrual projection map from ${{\Sh}}_{1,n-1}(K_{\gothp}^{1})$ to ${{\Sh}}_{1,n-1},$ the correspondences ${\rm C}_{j}$'s map onto the ${\rm Y}_{j}$'s for $1\leq j\leq n.$ In this way, the ${\rm C}_{j}$'s play an analogous role in ${{\Sh}}_{1,n-1}(K_{\gothp}^{1})$ to that of the ${\rm Y}_{j}$'s in ${{\Sh}}_{1,n-1}.$ We establish the definite Ihara lemma by analyzing the cohomology groups of ${\rm C}_j$'s and invoking Theorem~\ref{T:main-theoremm}.

Recall that we have fixed an open compact subgroup $K$ which is hyperspecial at $p.$ The key step for the proof of Theorem~\ref{Ihara lemma}(2) is to build a relation between Shimura varieties of different signatures. One relation comes from Theorem~\ref{T:main-theoremm}. To use Theorem~\ref{T:main-theoremm}, we find another rational prime $p'$ inert in $F$ and split in $E_0$ satisfying Hypothesis~\ref{Main hypo} and that $K_{p'}$ is hyperspecial which exists from the Chebotarev density theorem and that $K_v$ is hyperspecial for all but finitely many place $v$ of $\QQ.$ Let $K_p^1=\QQ_p^\times\times K_\gothp^1.$ We take $K'=K^pK_p^1.$ We add subscript $p'$ to denote the Shimura varieties defined at $p'.$ Here we use notation `$(K)$' and `$(K')$' to strengthen the level structure of Shimura varieties since we need to consider Shimura varieties for different primes.\footnote{In particular, $\Sh_{1,n-1}=\Sh_{1,n-1}(K)$ and $\Sh_{0,n}=\Sh_{0,n}(K).$} 
Then we can use proper base change theorem and Theorem~\ref{Ihara lemma}(1) to get a surjection on cohomology groups of generic fiber and these cohomology groups for different primes coincide. Thus by Theorem~\ref{T:main-theoremm}, there is a surjection on cohomology groups of generic fiber of Shimura varieties with signature $(1,n-1)$ and $(n-1,1),$ which coincide for different primes. Then we get Theorem~\ref{Ihara lemma}(2).

Corollary~\ref{level raising result} comes from the non-injectivity of $\psi$ in Theorem~\ref{Ihara lemma}. Denote by $\psi': {\mathrm{H}^{0}_{\acute{e}t}}({\overline{\Sh}}_{0,n},k_\lambda)^{\oplus n}_{\mathfrak{m}}\rightarrow {\mathrm{H}}_{\acute{e}t}^{0}(\overline{\Sh}_{0,n}(K_{\gothp}^{1}),k_\lambda)_{\mathfrak{m}}$ the dual of $\psi.$ It is also equivalent to $\psi'$ is not surjective. Considering $\psi\circ \psi'$ as an $n\times n$ matrix with each element a Hecke operator, it can be shown the determinant is zero under Hypothesis~\ref{Main hypo3}. And thus we get Corollary~\ref{level raising result}.

With the Ihara lemma, we cam prove the surjectivity of the arithmetic level raising map.
Recall the cycle class map:
$\beta:{\rm{Ch}}^1({{\Sh}_{1,n-1}^{{\rm{ss}}}},1,k_\lambda)\to{\rm{Ch}}^{n}({{\Sh}_{1,n-1}},1,k_\lambda)\to {\rm{H}}_{\acute{e}t}^{2n-1}({{\Sh}_{1,n-1}},k_\lambda(n)).$
By combining the Hochschild--Serre spectral sequence and localizing at a maximal ideal 
$\mathfrak{m}$
 of the Hecke algebra in Hypothesis~\ref{Main hypo}, we obtain the following diagram:

\[\begin{tikzcd}
	&& {0={\rm{H}}^0(\mathbb{F}_{p^2},{\rm{H}}_{\acute{e}t}^{2n-1}(\overline{{\Sh}}_{1,n-1},k_\lambda(n))_\mathfrak{m})} & {} & {} \\
	{{\rm{Ch}}^{n}({{\Sh}_{1,n-1}},1,k_\lambda)_{\mathfrak{m}}} & {{\rm{H}}^{2n-1}_{\mathcal{M}} ({{\Sh}_{1,n-1}},k_\lambda(n))_{\mathfrak{m}}} & {{\rm{H}}_{\acute{e}t}^{2n-1}({{\Sh}_{1,n-1}},k_\lambda(n))_{\mathfrak{m}}} & {} & {} \\
	{{\rm{Ch}}^1({{\Sh}_{1,n-1}^{{\rm{ss}}}},1,k_\lambda)_{\mathfrak{m}}} & {} & {{\rm{H}}^1(\mathbb{F}_{p^2},{\rm{H}}_{\acute{e}t}^{2n-2}(\overline{{\Sh}}_{1,n-1},k_\lambda(n))_\mathfrak{m})} & {} & {}
	\arrow[Rightarrow, no head, from=2-2, to=2-1]
	\arrow[from=2-2, to=2-3]
	\arrow["{{{Abel-Jacobi~~map}}}"'{pos=0.05}, from=2-2, to=3-3]
	\arrow[two heads, from=2-3, to=1-3]
	\arrow[from=3-1, to=2-1]
	\arrow["{{{level -raising ~~map}}}", from=3-1, to=3-3]
	\arrow[hook, from=3-3, to=2-3]
\end{tikzcd}\]

Under the hypothesis~\ref{Main hypo}, ${\rm{H}}_{\acute{e}t}^{i}(\overline{\Sh}_{1,n-1},k_\lambda(r))_\mathfrak{m}=0$ whenever $i\neq d.$ So the injection in the diagram is an isomorphism, and we get the level-raising map. For the level raising map we have the following conjecture which is well known as arithmetic level raising theorem:
\begin{conjecture}
\label{Arithmetic level raising map}
    Under the Hypothesis~\ref{Main hypo} and ~\ref{Main hypo2}, the level raising map
    \begin{equation*}
        {\rm{Ch}}^1({{\Sh}_{1,n-1}^{{\rm{ss}}}},1,k_\lambda)_{\mathfrak{m}}\rightarrow
        {{\rm{H}}^1(\mathbb{F}_{p^2},{\rm{H}}_{\acute{e}t}^{2n-2}(\overline{{\Sh}}_{1,n-1},k_\lambda(n))_\mathfrak{m})}
    \end{equation*}
    is surjective.
\end{conjecture}

The idea for a possible proof of Conjecture~\ref{Arithmetic level raising map} is to use Weight spectral sequence for $\cSh_{1,n-1}(K_\gothp^1).$ Since this variety is not strictly semi-stable, we need to blow up $\cSh_{1,n-1}(K_\gothp^1)$ at a family of irreducible components of its special fiber. Then we get a spectral sequence $E_1^{p,q}$ converges to $H^{p+q}_{\acute{e}t}(\overline{Sh}_{1,n-1}(K_\gothp^1),k_\lambda)$ of five columns with indices $p=-2,-1,0,1,2.$ After localizing at $\mathfrak{m},$ It can be shown that $Gr_2H^{2n-2}_{\acute{e}t}(\overline{Sh}_{1,n-1}(K_\gothp^1),k_\lambda(n))_\mathfrak{m}=E_{\infty,\mathfrak{m}}^{-2,2}$ is a subgroup of ${\rm{Ch}}^1({{\Sh}_{1,n-1}^{{\rm{ss}}}},1,k_\lambda)_{\mathfrak{m}}.$ If we can show $Gr_1H^{2n-2}_{\acute{e}t}(\overline{Sh}_{1,n-1}(K_\gothp^1),k_\lambda(n))_\mathfrak{m}=0$ and the map from $Gr_0H^{2n-2}_{\acute{e}t}(\overline{Sh}_{1,n-1}(K_\gothp^1),k_\lambda(n))_\mathfrak{m}$ to $H^{2n-2}_{\acute{e}t}(\overline{Sh}_{1,n-1},k_\lambda(n))_\mathfrak{m}$ induced by $\psi$ in Theorem~\ref{Ihara lemma}(2) has cokernel ${{\rm{H}}^1(\mathbb{F}_{p^2},{\rm{H}}_{\acute{e}t}^{2n-2}(\overline{{\Sh}}_{1,n-1},k_\lambda(n))_\mathfrak{m})},$ then by a linear algebra argument as in Lemma~\ref{fil} Conjecture~\ref{Arithmetic level raising map} can be proved. However, the final step is too difficult to deal with for $n\geq 4.$

We have proved the arithmetic level raising theorem for $n=2,3.$
\begin{theorem}
    \label{Arithmetic level raising maps}
   Under the Hypothesis~\ref{Main hypo} and ~\ref{Main hypo2}, the level raising map
    \begin{equation*}
        {\rm{Ch}}^1({{\Sh}_{1,n-1}^{{\rm{ss}}}},1,k_\lambda)_{\mathfrak{m}}\rightarrow
        {{\rm{H}}^1(\mathbb{F}_{p^2},{\rm{H}}_{\acute{e}t}^{2n-2}(\overline{{\Sh}}_{1,n-1},k_\lambda(n))_\mathfrak{m})}
    \end{equation*}
    is surjective for $n=2,3.$
\end{theorem}

The proof of Theorem ~\ref{Arithmetic level raising maps} for $n=2$ is quite similar to that in \cite{Zho23} but written without the language of weight spectral sequence, which we put in Appendix~\ref{I} and~\ref{A}.

To prove Theorem~\ref{Arithmetic level raising maps} for $n=3,$ we first analyze the geometry of the unitary Shimura varieties with Iwahori level at $p,$ denoted by $\Sh_{1,2}(Iw_\gothp)$ for its special fiber. Here $Iw_\gothp$ denotes the standard Iwahori subgroup of $\GL_n(\ZZ_{p^2}).$ The Hecke action $\A$ can be described by the relation of $\Sh_{1,2}(Iw_\gothp)$ and $\Sh_{1,2}(K_\gothp^1).$ It turns out that $\Sh_{1,2}(Iw_\gothp)$ has $9$ families of irreducible components denoted by ${\rm Z}_{ij}$ for $0\leq i,j\leq 2.$ We explore the relation of ${\rm Z}_{ij}$ and ${\rm Y}_k$ in detail and show that there are purely inseperable morphisms similar to those called `essential Frobenius' in \cite{Zho23}. These are the main parts in the calculation of the cokernel of the map from $Gr_0{\rm H}^{4}_{\acute{e}t}(\overline{Sh}_{1,2}(K_\gothp^1),k_\lambda(3))_\mathfrak{m}$ to ${\rm H}^{4}_{\acute{e}t}(\overline{Sh}_{1,2},k_\lambda(3))_\mathfrak{m}$ induced by $\psi.$ Here blowing up on $\cSh_{1,2}(Iw_\gothp)$ is also needed to use weight spectral sequence.

For the vanishing of $Gr_1{\rm H}^{4}_{\acute{e}t}(\overline{Sh}_{1,2}(K_\gothp^1),k_\lambda(n))_\mathfrak{m},$ it suffices to show the cohomology groups appearing in $E_{1,\mathfrak{m}}^{1,1}$ with integral coefficients are all torsion-free. Then by Weight monodromy conjecture and Hypothesis~\ref{Main hypo2}, we can get the vanishing result. The torsion-freeness comes from the symmetry of $\Sh_{1,2}(K_\gothp^1)$ and a rank counting on cohomology groups.

We briefly describe the structure of the paper. In Section~\ref{S:Shimura Varieties}, we consider a more general setup of unitary Shimura varieties of PEL type. We also prove a$\mod\ell$-version of the Tate conjectures proved in \cite{HTX17}. 
In Section~\ref{DG}, We recall basic notions of Dieudonn\'e modules and Grothendieck-Messing deformation theory. In Section~\ref{D}, we recall basic properties of higher Chow groups and calculate $\ch.$ In section \ref{S}, we study the Newton and Ekedahl--Oort stratification of ${{\Sh}}_{1,n-1}.$ In Section~\ref{G}, we study the geometry of ${{\Sh}}_{1,n-1}({K_{\gothp}^{1}})$ and in Section~\ref{I3} we give the proof of Theorem~\ref{Ihara lemma} for $n\geq 3.$ In section~\ref{GI}, we study the geometry of $\Sh_{1,2}(Iw_\gothp)$ with $Iw_\gothp$ the standard Iwahori subgroup. In Section~\ref{A3},
we prove Theorem~\ref{Arithmetic level raising maps} for $n=3.$ In Appendix~\ref{I}, we prove the Theorem~\ref{Ihara lemma} for $n=2.$
In Appendix~\ref{A}, we prove Theorem~\ref{Arithmetic level raising maps} for $n=2.$

\subsection*{Notation and Conventions}
\begin{itemize}
    \item All rings are commutative and unital; and ring homomorphisms preserve units.
    \item Throughout the article, we fix a prime $p.$ We fix an isomorphism $\iota_{p}: \CC \xra{\sim} \overline \QQ_p.$
    \item Let $F$ be a totally real field of degree $2$ in which $p$ is inert and $E_0$ be an imaginary quadratic extension of $\Q,$ in which $p$ splits. Put $E=E_0F.$ Denote by $\gothp$ and $\bar \gothp$ the two primes $p$ split into in $E.$ Let $q_i$ denote the $p$-adic embeddings of $E$ into $\overline{\QQ}_p$ corresponding to $\gothp$ and $\bar{q}_i$ the analogous embeddings corresponding to $\bar\gothp$ for $i=1,2.$ Put $\Sigma_{\infty,E}=\{q_1, q_2, \bar q_1, \bar q_2\}.$
    \item Let $\mathcal{X}$ be a scheme defined over $\ZZ_{p^2}.$ We use $X$ to denote its generic fiber and $\mathrm{X}$ to denote its special fiber. We use an overline to denote the geometric fiber. (e.g $\overline{X},\overline{\mathcal{X}}$ and $\overline{\mathrm{X}}.$)
    \item For a scheme $\overline{{\rm X}}$ defined over $\overline{\FF}_p,$ denote ${\rm Frob}_{p^2}$ an geometric Frobenius element in $\Gal(\overline{\FF}_p/{\FF}_{p^2}).$ We also denote by ${\rm Fr}_{p^2}$ the Frobenius action on ${\rm H}^i_{\acute{e}t}(\overline{{\rm X}},L(j))$ for $i.j$ integers and $i\geq 0.$ Here $L$ is a $p$-coprime coefficient ring.\footnote{In fact, we have ${\rm Fr}_{p^2}=p^{-2i}{\rm Frob}_{p^2}^{*}.$} 
    \item For a prime $p'$ different from $p,$ we use subscript $p'$ to strengthen a scheme defined over $p'$-adic field. (e.g $X_{p'},\mathrm{X}_{p'}.$)
\end{itemize}

\section*{Acknowledgement}
This work originated as a project in the summer school “Algebra and Number Theory”, jointly organized by Peking University and the Academy of Mathematics and Systems Science, CAS. The first part of the project was completed under the guidance of Ruiqi Bai, while the latter part formed the author’s undergraduate thesis under the mentorship of Professor Liang Xiao. The author is deeply grateful to Professor Shou-wu Zhang and Professor Liang Xiao, whose efforts were instrumental in the successful organization of the summer school. Special thanks also go to Ruiqi Bai and Professor Liang Xiao for their invaluable guidance and insightful advice throughout the project.
\section{Unitary Shimura Varieties}\label{S:Shimura Varieties}

\subsection{Shimura data}
\label{S:Shimura data}
Let $D$ be a division algebra of dimension $n^2$ over its center $E,$ equipped with a positive involution $*$ which restricts to the complex conjugation $c$ on $E.$ In particular, $D^{\mathrm{opp}}\simeq D\otimes_{E, c}E.$ 
We assume that $D$ splits at $\gothp$ and $\bar \gothp,$ and we fix an isomorphism
$
D \otimes_\QQ \QQ_p \simeq \rmM_n(E_{\gothp}) \times\rmM_n(E_{\bar\gothp})\simeq \rmM_{n}(\QQ_{p^2})\times \rmM_n(\QQ_{p^{2}}),
$
where  $*$ switches the two direct factors.  We use $\gothe$ to denote the element of $D \otimes_\QQ \QQ_p$ corresponding to the $(1,1)$-elementary matrix by which we mean an $n\times n$-matrix whose $(1,1)$-entry is $1$ and whose other entries are zero in the first factor.
Let  $a_{\bullet}=(a_1,a_2)$ be a tuple of $2$ numbers with  $a_i\in \{0,\dots, n\}$ for $1\leq i\leq 2.$ Assume that there exists an element $\beta_{a_{\bullet}}\in (D^\times)^{*=-1}$ such that the following condition is  satisfied:\footnote{As explained in the proof of \cite[Lemma~I.7.1]{HT01}, when $n$ is odd, such $\beta_{a_\bullet}$ always exists, and when $n$ is even, existence of $\beta_{a_\bullet}$ depends on the parity of $a_1+a_2.$}

Let $G_{a_{\bullet}}$ be  the algebraic group over $\Q$ such that $G_{a_{\bullet}}(R)$ for a $\QQ$-algebra $R$ consists of  elements $g\in (D^{\mathrm{opp}} \otimes_\QQ R)^{\times}$ with  $g\beta_{a_{\bullet}} g^*=c(g) \beta_{a_{\bullet}}$ for some $c(g)\in R^{\times}.$ If $G_{a_{\bullet}}^1$ denotes the kernel of the similitude character  $c: G_{a_{\bullet}}\ra \GG_{m,\Q},$ then there exists an isomorphism
$
G^1_{a_{\bullet}}(\RR)\simeq U(a_1,n-a_1)\times U(a_2,n-a_2),
$
where the $i$-th factor corresponds to the real  embedding $\tau_i:F\hra \RR.$


Note that the assumption on $D$ at $p$ implies that
$
G_{a_{\bullet}}(\Q_p)\simeq \Q_p^{\times} \times \GL_n(E_{\gothp})\simeq \Q_p^{\times}\times \GL_n(\Q_{p^2}).
$
We put  $V_{a_{\bullet}}=D$ and view it as a left $D$-module.
Let $\langle-,-\rangle_{a_{\bullet}}: V_{a_{\bullet}}\times V_{a_{\bullet}}\ra \Q$ be the perfect alternating pairing  given by
$
\langle x,y\rangle_{a_{\bullet}}=\Tr_{D/\Q}(x\beta_{a_{\bullet}} y^*),\quad \text{for  }x,y\in V_{a_{\bullet}}.
$
Then $G_{a_{\bullet}}$ is identified with the similitude group associated to $(V_{a_{\bullet}},\langle-,-\rangle_{a_{\bullet}}),$ i.e. for all $\Q$-algebra $R,$ we have
\[
G_{a_{\bullet}}(R)=\big\{g\in \End_{D\otimes_{\Q} R}(V_{a_{\bullet}}\otimes_{\Q} R)\;\big|\; \langle gx,gy\rangle_{a_{\bullet}}=c(g)\langle x,y\rangle_{a_{\bullet}} \text{ for some }c(g)\in R^{\times}\big\}.
\]

Consider the homomorphism of $\RR$-algebraic groups $h: \Res_{\CC/\RR}(\GG_m)\ra G_{a_{\bullet},\RR}$ given by
\begin{equation*}\label{E:deligne-homomorphism}
h(z)=\diag(\underbrace{z,\dots,z}_{a_{1}}, \underbrace{\bar z,\dots, \bar z}_{n-a_{1}}) \times \diag(\underbrace{z,\dots,z}_{a_{2}}, \underbrace{\bar z,\dots, \bar z}_{n-a_{2}}),\quad \text{for }z=x+\sqrt{-1}y.
\end{equation*}
Let $\mu_{h}: \GG_{m,\CC}\ra G_{a_{\bullet},\CC}$ be the composite of $h_{\CC}$ with the map $\GG_{m,\CC}\ra \Res_{\CC/\RR}(\GG_{m})_{\CC}\simeq \CC^{\times }\times \CC^{\times}$ given by $z\mapsto (z,1).$ Here, the first copy of $\CC^{\times}$ in $\Res_{\CC/\RR}(\GG_{m})_{\CC}$ is the one indexed by the identity element in $\mathrm{Aut}_{\RR}(\CC),$ and the other copy of $\CC^\times$ is indexed by the complex conjugation.

Let $E_{h}$ be the reflex field of $\mu_{h},$ i.e. the minimal subfield of $\CC$ where the conjugacy class of $\mu_{h}$ is defined.
 It has the following explicit description.  The group $\Aut_{\QQ}(\CC)$ acts naturally on $\Sigma_{\infty,E},$ and hence on the functions on $\Sigma_{\infty,E}.$
Then $E_{h}$ is the subfield of $\CC$ fixed by  the stabilizer of the  $\ZZ$-valued function $a$ on $\Sigma_{\infty,E}$ defined by $a(q_i)=a_i$ and $a(\bar q_i)=n-a_i.$
The isomorphism $\iota_p: \CC \xrightarrow{\sim} \overline \QQ_p$ defines a $p$-adic place $\wp$ of $E_h.$ By our hypothesis on $E,$ the local field $E_{h,\wp}$ is $\Q_{p^2},$  the unique unramified extension over $\Q_p$ of degree $2.$

\subsection{Unitary Shimura varieties of PEL-type}
\label{S:defn of Shimura var}
Let $\calO_D$ be a $*$-stable order of $D$ and $\Lambda_{a_{\bullet}}$ an $\calO_D$-lattice of $V_{a_{\bullet}}$ such that $\langle \Lambda_{a_{\bullet}}, \Lambda_{a_{\bullet}} \rangle_{a_{\bullet}} \subseteq \ZZ$ and  $\Lambda_{a_{\bullet}} \otimes_{\Z} \Z_{p}$ is self-dual under the alternating pairing induced by $\langle-,-\rangle_{a_{\bullet}}.$
We put  $K_p=\Z_p^{\times}\times \GL_{n}(\cO_{E_{\gothp}})\subseteq G_{a_{\bullet}}(\Q_p),$
and fix an open compact subgroup $K^p\subseteq G_{a_{\bullet}}(\AAA^{\infty,p})$ such that $K = K^pK_p$ is \emph{neat}, i.e. $G_{a_{\bullet}}(\Q)\cap g K g^{-1}$ is torsion free  for any $g\in G_{a_{\bullet}}(\AAA^{\infty}).$

\begin{definitionn}\label{S: Hyperspecial level}
Following \cite{Kot92}, we have a unitary Shimura variety $\cSh_{a_\bullet}$ defined over $\ZZ_{p^2}$; it represents the functor that takes a locally Noetherian $\ZZ_{p^2}$-scheme $S$ to the set of isomorphism classes of tuples $(A, \lambda, \eta),$ where
\begin{enumerate}
\item
$A$ is an $2n^2$-dimensional abelian variety over $S$ equipped with an action of $\calO_D$ such that the induced action on $\Lie(A/S)$ satisfies the \emph{Kottwitz determinant condition}, that is, if we view the \emph{reduced} relative de Rham homology ${\rm H}_1^\dR(A/S)^\circ : = \gothe {\rm H}_1^\dR(A/S)$ and its quotient $\Lie^{\circ}_{A/S} : = \gothe\cdot \Lie_{A/S}$  as a module over $F_p \otimes_{\Zp} \calO_S \simeq \bigoplus_{i=1}^2 \calO_S,$ they, respectively, decompose into the direct sums of locally free $\calO_S$-modules ${\rm H}_1^\dR(A/S)^\circ_i$ of rank $n$ and, their quotients, locally free $\calO_S$-modules $\Lie^\circ_{A/S,i}$ of rank $n-a_i$;
\item
$\lambda: A \to A^\vee$ is a prime-to-$p$ $\calO_D$-equivariant polarization such that the Rosati involution induces the involution $*$ on $\calO_D$;
\item
$\eta$ is a collection of,
for each connected component $S_j$ of $S$ with a geometric point $\bar s_j,$ a $\pi_1(S_j, \bar s_j)$-invariant $K^p$-orbit of isomorphisms $\eta_j: \Lambda_{a_{\bullet}} \otimes_\ZZ \widehat \ZZ^{(p)} \simeq T^{(p)}(A_{\bar s_j})$ such that the following diagram commutes for an isomorphism $\nu(\eta_j) \in \Hom( \widehat{\ZZ}^{(p)}, \widehat{\ZZ}^{(p)}(1))$:
\[
\xymatrix@C=60pt{
\Lambda_{a_{\bullet}} \otimes_{\ZZ}\widehat{\ZZ}^{(p)}\ \times\ \Lambda_{a_{\bullet}} \otimes_{\ZZ}\widehat{\ZZ}^{(p)} \ar[d]^{\eta_j \times \eta_j}
\ar[r]^-{\langle -, - \rangle}
& \widehat{\ZZ}^{(p)}
\ar[d]^{\nu(\eta_j)}
\\
T^{(p)}A_{\bar s_j} \ \times \ T^{(p)}A_{\bar s_j}
\ar[r]^-{\textrm{Weil pairing}}
&
\widehat{\ZZ}^{(p)}(1),
}
\]
where $\widehat \ZZ^{(p)} = \prod_{\ell \neq p}\ZZ_\ell$ and $T^{(p)}(A_{\bar s_j})$ denotes the product of the $\ell$-adic Tate modules of $A_{\bar s_j}$ for all $\ell\neq p.$
\end{enumerate}
\end{definitionn}

The Shimura variety $\cSh_{a_\bullet}$ is smooth and projective over $\ZZ_{p^2}$ of relative dimension $d(a_\bullet) := \sum_{i=1}^2 a_i(n-a_i).$
 Note that if $a_i\in \{0,n\}$ for all $i,$ then  $\cSh_{a_{\bullet}}$ is of relative dimension zero; we call it a \emph{discrete Shimura variety}.

We denote by $\cSh_{a_{\bullet}}(\CC)$  the complex points of $\cSh_{a_{\bullet}}$ via the embedding $\Z_{p^2}\hra\overline{ \Q}_p\xra{\iota^{-1}_p} \CC.$
Let $K_{\infty}\subseteq G_{a_{\bullet}}(\RR)$ be the stabilizer of $h$ under the conjugation action, and let $X_{\infty}$ denote the $G_{a_{\bullet}}(\RR)$-conjugacy class of $h.$
 Then $K_{\infty}$ is a maximal compact-modulo-center subgroup of $G_{a_{\bullet}}(\RR).$
 According to \cite[page 400]{Kot92}, the complex manifold $\cSh_{a_{\bullet}}(\CC)$ is the disjoint union of
  $\#\mathrm{ker}^1(\Q,G_{a_{\bullet}})$ copies of
   \begin{equation*}\label{E:shimura-complex}
   G_{a_{\bullet}}(\Q)\backslash \big (G_{a_{\bullet}}(\AAA^{\infty})\times X_{\infty}\big)/K\simeq G_{a_{\bullet}}(\Q)\backslash G_{a_{\bullet}}(\AAA)/K\times K_{\infty}.
   \end{equation*}
Here, if $n$ is even,  then  $\mathrm{ker}^1(\Q,G_{a_{\bullet}})=(0),$ while if $n$ is odd then
 \[
 \mathrm{ker}^1(\Q,G_{a_{\bullet}})=\mathrm{Ker}\Big(F^{\times}/\Q^{\times}N_{E/F}(E^{\times})\ra \AAA^{\times}_{F}/\AAA^{\times}N_{E/F}(\AAA^{\times}_E)\Big).
 \]
 In either case,  $\mathrm{ker}^1(\Q,G_{a_{\bullet}})$ depends only on the CM extension $E/F$ and the parity of  $n$ but not on the tuple $a_\bullet.$ 

Let ${\Sh}_{a_\bullet} : = \cSh_{a_\bullet} \otimes_{\ZZ_{p^2}} \FF_{p^2}$ denote the special fiber of $\cSh_{a_{\bullet}},$ and let $\overline {\Sh}_{a_\bullet}: = {\Sh}_{a_\bullet} \otimes_{\FF_{p^2}} \overline \FF_p$ denote the geometric special fiber. We let $Sh_{a_{\bullet}}:=\cSh_{a_\bullet} \otimes_{\ZZ_{p^2}} \mathbb{Q}_{p^2}$ denote the generic fiber of $\cSh_{a_\bullet}$ and let $\overline{Sh}_{a_\bullet}$ denote the geometric generic fiber.

The unitary Shimura variety $\cSh_{a_{\bullet}}$ is defined with hypersepcial level structure, since $K_{\gothp}=\GL_n(\ZZ_{p^2})$ is a hyperspecial subgroup of $\GL_n(\QQ_{p^2}).$ We now introduce a variant of the unitary Shimura variety with parahoric level structure. Define  $K_{\gothp}^{1}=\diag\{p^{-1},1,\dots,1\}\cdot K_{\gothp}\cdot \diag\{p,1,\dots,1\}\cap K_{\gothp}.$ Then $K_{\gothp}^{1}$ is a parahoric subgroup of $K_{\gothp}=\GL_n(\ZZ_{p^2}).$\footnote{When $n=2,$ $K_{\gothp}^{1}$ coincides the Iwahoric subgroup.} Let $K_{p}^{1}=\ZZ_{p}^{\times}\times K_{\gothp}^{1}$ and $K'=K^{p}K_{p}^{1}.$

\begin{definitionn}\label{S:Parahoric level}
    Let $\cSh_{a_{\bullet}}(K_{\gothp}^{1})$ to be the unitary Shimura variety defined over $\ZZ_{p^{2}}$ which represents the functor the takes a locally Noetherian $\ZZ_{p^{2}}$-scheme $S$ to the set of isomorphism classes of  $(A_1,\lambda_1,\eta_1,A_2,\lambda_2,\eta_2,\phi)$, where
\begin{itemize}
            \item
            $(A_1,\lambda_1,\eta_1)$ is an $S$-point of ${\cSh}_{a_{\bullet}},$  
            \item
            $(A_2,\lambda_2,\eta_2)$ is an $S$-point of ${\cSh}_{a_{\bullet}},$  and 
            \item
            $\phi: A_1\ra A_2$ is an $\calO_D$-equivariant $p$-quasi-isogeny (i.e. $p^{m}\phi$ is an isogeny of $p$-power order for some integer $m$),
        \end{itemize}
        such that 
        \begin{itemize}
            \item
            $\phi^\vee \circ\lambda_2\circ \phi=\lambda_1,$ 
            \item
            $\phi\circ\eta_1=\eta_2,$ and  \item
            the cokernels of the maps
             $
             \phi_{*,1}: {\rm H}^\dR_1(A_1/S)^\circ_1 \to {\rm H}^\dR_1(A_2/S)^\circ_1 \quad\textrm{and}\quad
            \phi_{*,2}: {\rm H}^\dR_1(A_1/S)^\circ_2 \to {\rm H}^\dR_1(A_2/S)^\circ_2
            $
            are both locally free $\cO_S$-modules of rank $1.$\footnote{By \cite[Remark 3.7]{HTX17}, if $a_{\bullet}=(0,n),$ the condition is equivalent to the existence of two $\ZZ_{p^{2}}$-lattice $\LL_1,\LL_2$ such that $\LL_2$ has a $\ZZ_{p^{2}}$-basis $\{e_1,\dots,e_n\}$ and $\LL_1$ has a $\ZZ_{p^{2}}$-basis $\{pe_1,e_2,\dots,e_n\}.$}
        \end{itemize}
\end{definitionn}
Similarly as above, we use $\Sh_{a_{\bullet}}(K_{\gothp}^{1})$ to express the special fiber of $\cSh_{a_{\bullet}}(K_\gothp^1)$ and $Sh_{a_{\bullet}}(K_{\gothp}^{1})$ to express the generic fiber.

\begin{remarkk}\label{D:two-projection}
We have two morphisms ${\mathop{p}\limits^{\rightarrow}}$ and ${\mathop{p}\limits^{\leftarrow}}$ for any point $(z,z',\phi)\in {\rm Sh}_{a_\bullet}(K_\gothp^1)$ with $z,z'$ points in $\Sh_{a_\bullet},$ ${\mathop{p}\limits^{\rightarrow}}((z,z'))=z'$ and ${\mathop{p}\limits^{\leftarrow}}((z,z'))=z.$  
\end{remarkk}

This paper mainly focuses on the cases $a_\bullet=(0,n)$ and $(1,n-1).$ From now on, we fix an isomorphism $G_{1,n-1}(\AAA^{\infty})\simeq G_{0,n}(\AAA^{\infty})$ by \cite[Lemma 2.9]{HTX17} and denote them by $G(\AAA^{\infty}).$
\subsection{Cycles on \texorpdfstring{$\Sh_{1,n-1}$}{Sh1,n-1} and \texorpdfstring{$\Sh_{0,n}$}{Sh0,n}}
 We first construct cycles on ${\Sh}_{1,n-1}$ and ${\Sh}_{0,n}.$
\begin{definitionn}\label{S:cycles on hyperspecial}
    For each integer  $j$ with $1\leq j\leq n ,$ let ${\rm Y}_{j}$ be the moduli space over $\FF_{p^2}$ that associates to  each  locally Noetherian $\FF_{p^2}$-scheme $S,$ the set of isomorphism classes of tuples  $(A,\lambda,\eta, B,\lambda',\eta',\phi),$ where 
\begin{itemize}
\item
$(A,\lambda,\eta)$ is an $S$-point of ${\Sh}_{1,n-1},$  
\item
$(B,\lambda',\eta')$ is an $S$-point of ${\Sh}_{0,n},$  and 
\item
$\phi: B\ra A$ is an $\calO_D$-equivariant isogeny whose kernel is contained in $B[p],$
\end{itemize}
such that 
\begin{itemize}
\item
$p\lambda'=\phi^\vee \circ\lambda\circ \phi,$ 
\item
$\phi\circ\eta'=\eta,$ and  \item
the cokernels of the maps
 $
 \phi_{*,1}: {\rm H}^\dR_1(B/S)^\circ_1 \to {\rm H}^\dR_1(A/S)^\circ_1 \quad\textrm{and}\quad
\phi_{*,2}: {\rm H}^\dR_1(B/S)^\circ_2 \to {\rm H}^\dR_1(A/S)^\circ_2
$
are locally free $\cO_S$-modules of rank $j-1$ and $j,$ respectively.
\end{itemize}

There is a unique isogeny $\psi: A\ra B$ such that $\psi\circ\phi=p \cdot \id_{B}$ and $\phi\circ\psi=p \cdot \id_{A}.$
 We have
 $
 \Ker(\phi_{*,i})=\im(\psi_{*,i}) \quad \textrm{and} \quad \Ker(\phi_{*,i})=\im(\psi_{*,i}),
 $
 where $\psi_{*,i}$ for $i=1,2$ is the induced homomorphism on the reduced de Rham homology in the evident sense.
This moduli space ${\rm Y}_j$ is represented by a scheme of finite type over $\FF_{p^2}.$
We have two morphisms from ${\rm Y}_j$ to ${\rm Sh}_{1,n-1}$ and ${\rm Sh}_{0,n}$ denoted by
${\mathop{p}\limits^{\leftarrow}}_{j}$ and ${\mathop{p}\limits^{\rightarrow}}_{j}$ sending a tuple $(A,\lambda,\eta, B,\lambda',\eta',\phi)$ to $(A, \lambda, \eta)$ and to $(B, \lambda', \eta')$ respectively.
Letting $K^p$ vary, we see easily that both ${\mathop{p}\limits^{\leftarrow}}_{j}$ and ${\mathop{p}\limits^{\rightarrow}}_{j}$ are equivariant under prime-to-$p$ Hecke actions given by the double cosets $K^p\backslash G(\AAA^{\infty,p})/K^p.$ ${\rm Y}_{j}$ gives a correspondence between ${\Sh}_{1,n-1}$ and ${\Sh}_{0,n}.$
\end{definitionn}

Let $(\mathcal{B},\lambda',\eta')$ be the universal object of $\Sh_{0,n}.$ For $1\leq j\leq n,$ ${\rm Y}_j$ can be realized as a closed subfunctor of the product of Grassmannian schemes
$
\Gr({\rm H}_1^{\rm dR}(\mathcal{B}/\Sh_{0,n})_1^\circ,j)\times \Gr({\rm H}_1^{\rm dR}(\mathcal{B}/\Sh_{0,n})_2^\circ,j-1)
$ defined over $\Sh_{0,n}$ by mapping the universal object of ${\rm Y}_j:$ 
$(\mathcal{A},\lambda,\eta,\mathcal{B},\lambda',\eta',\phi)$ to $(\mathcal{B},\lambda',\eta',\phi_{*,1}^{-1}(\omega^{\circ}_{\mathcal{A}^\vee/S,1}),\psi_{*,2}(\omega^{\circ}_{\mathcal{A}^\vee/S,2}))$ as in \cite[Subsection 4.3]{HTX17}. Here $\psi:\mathcal{B}\ra \mathcal{A}$ is the isogeny such that $\psi\circ\phi=p$ and $\phi\circ\psi=p.$

To illustrate the relationship of ${\rm Y}_i$ with $\Sh_{1,n-1}$ and the intersection of different ${\rm Y}_i, {\rm Y}_j,$ we need introduce a special kind Deligne-Lusztig variety.
\begin{definitionn}\label{S:Deligne-Lusztig}
    We define $Z_{i}^{\langle n \rangle }$(resp. ${\tilde Z}_{i}^{\langle n \rangle},$ ${\widehat Z}_{i}^{\langle n \rangle}$) as a closed subscheme of $\Gr(n,i)\times \Gr(n,i-1)$ over $\mathbb{F}_{p^{2}}$ whose $S$-valued points are the isomorphism classes of pairs $(H_{1}, H_{2})$  where $H_{1}$ and $H_{2}$ are respectively subbundles of $\mathcal{O}_{S}^{\oplus n}$ of rank $i$ and $i-1$ satifying $H_{2}\subseteq H_{1}^{(p)}$ and $H_{2}^{(p)}\subseteq H_{1}$(resp. $H_{2}\subseteq H_{1}$ and $H_{2}\subseteq H_{1}^{(p^{2})},$ $H_{2}\subseteq H_{1}$ and $H_{2}^{(p^{2})}\subseteq H_{1}$).
There are morphisms called ``relative Frobenius'':
\begin{enumerate}
    \item $\widehat\phi:{\widehat{Z}_{i}^{\langle n \rangle }}\ra {Z_{i}^{\langle n\rangle }}$ mapping any $S$-point ${(H_{1},H_{2})}$ to ${(H_{1},H_{2}^{(p)})};$
    \item $\widehat\psi:{{Z}_{i}^{\langle n \rangle }}\ra {\widehat Z_{i}^{\langle n\rangle }}$ mapping any $S$-point ${(H_{1},H_{2})}$ to ${(H_{1}^{(p)},H_{2})};$
    \item $\tilde\phi:{{Z}_{i}^{\langle n \rangle }}\ra {\tilde{Z}_{i}^{\langle n\rangle }}$ mapping any $S$-point ${(H_{1},H_{2})}$ to ${(H_{1},H_{2}^{(p)})};$
    \item $\tilde\psi:{\tilde{Z}_{i}^{\langle n \rangle }}\ra {Z_{i}^{\langle n\rangle }}$ mapping any $S$-point ${(H_{1},H_{2})}$ to ${(H_{1}^{(p)},H_{2})};$
\end{enumerate}
Since these morphisms are purely inseparable, we can study the divisors of ${Z_{i}^{\langle n \rangle }}$ by studying divisors of $\tilde{Z}_{i}^{\langle n \rangle }$ and $\widehat{Z}_{i}^{\langle n \rangle }.$
\end{definitionn}
\begin{remark}
It should be noted that the Frobenius map here is defined over $\mathbb{F}_{p},$ which implies the total vector space should have a $\mathbb{F}_{p}$-structure. It does hold in our cases since we take ${\rm H}_{1}^{\rm dR}(A/S)^{\circ}_{1}$ and ${\rm H}_{1}^{\rm dR}(A/S)^{\circ}_{2}$ to be whole space and we can take a non-canonical basis of ${\rm H}_{1}^{\rm dR}(A/S)^{\circ}_{i}$ for $i=1,2.$ such that $FV^{-1}$ acts as identity on ${\rm H}_{1}^{\rm dR}(A/S)_{1}^{\circ}$ and ${\rm H}_{1}^{\rm dR}(A/S)_{2}^{\circ}.$ 
\end{remark}

The spaces ${Z}_{i}^{\langle n \rangle },$ $\tilde{Z}_{i}^{\langle n \rangle }$ and $\widehat{Z}_{i}^{\langle n \rangle }$ can be realized as the disjoint union of Deligne--Lusztig varieties. Furthermore, it can be shown the three schemes are irreducible and smooth of dimension $n-1$ over $\mathbb{F}_{p^{2}}.$ 
\begin{definitionn}\label{S:SD}
We introduce two kinds of subbundles of $\mathcal{O}_{\mathbb{F}_{p^{2}}}^{\oplus n}:$
    \begin{enumerate}
        \item When $i<n,$ for every subbundles $H$ in $\mathcal{O}_{\mathbb{F}_{p^{2}}}^{\oplus n}$ of rank $n-1$ , denote by $[H]$ the locus on $Z_i^{\langle n\rangle}$ where $H_1\subseteq H\otimes_{\mathcal{O}_{\mathbb{F}_{p^{2}}}}\mathcal{O}_{S}$ for any $\mathbb{F}_{p^{2}}$ scheme $S.$ Then $[H]\subset Z_i^{\langle n\rangle}$ is a closed subvariety of codimension 1, and we have $[H]\simeq Z_{i}^{\langle n-1\rangle}.$Let $\mathtt{SD}_+=\{[H] : H\text{ is a subbundle of rank $n-1$ in } \mathcal{O}_{\mathbb{F}_{p^{2}}^{\oplus n}} \}.$
    \item When $i>1,$ for every line bundle $L$ in $\mathcal{O}_{\mathbb{F}_{p^{2}}}^{\oplus n},$ denote by $[L]$ the locus on $Z_i^{\langle n\rangle}$ where $L\otimes_{\mathcal{O}_{\mathbb{F}_{p^{2}}}}\mathcal{O}_{S} \subset H_2$ for any $\FF_{p^2}$ scheme $S.$ Then $[L]\subset Z_i^{\langle n\rangle}$ is a closed subvariety of codimension 1, and we have $[L]\simeq Z_{i-1}^{\langle n-1\rangle}.$ Let $\mathtt{SD}_-=\{[L] : L\text{ is a line bundle of }{\mathcal{O}_{\mathbb{F}_{p^{2}}}^{\oplus n}}\}.$
    \end{enumerate}
\end{definitionn}
We have the following proposition:
\begin{propositionn}\label{S:corre(1,n-1)(0,n)}
The cycles ${\rm Y}_{j}$ for $1\leq j\leq n$ satisfies:
\begin{enumerate}
    \item For each fixed closed point $z\in {\Sh}_{0,n}$ and $1\leq j \leq n,$ ${\rm Y}_{j,z}$ is isomorphic to $Z_{j}^{\langle n \rangle }.$
    \item For each $1\leq j\leq n$ and $z\in\Sh_{0,n},$ the map ${{\mathop{p}\limits^{\leftarrow}}_j}|_{{\rm Y}_{j,z}}:{\rm Y}_{j,z}\ra {\Sh}_{1,n-1}$ is a closed immersion.
    \item The union of the images of ${{\mathop{p}\limits^{\leftarrow}}_j}$ for all $1\leq j\leq n$ is the supersingular locus ${\Sh}_{1, n-1}^{\rm ss}.$
\end{enumerate}
\end{propositionn}
\begin{proof}
    We refer to \cite[Proposition 4.14]{HTX17} for the proof of $(2)$ and $(3).$ The proof of $(1)$ follows from the discussion above Definition~\ref{S:Deligne-Lusztig}.
\end{proof}
\begin{propositionn}\label{S:intersection of Yi}
    Let $i,j$ be integers with  $1\leq i\leq j\leq n$ and $z,z'\in {\Sh}_{0,n}(\Fpb).$
The subvarieties $\overline{\rm Y}_{i,z}$ and $\overline{\rm Y}_{j,z'}$ of $\overline{\Sh}_{1,n-1}$ have non-empty intersection of dimension $n-2$ if and only if $j=i+1$ and there is a $p$-quasi-isogeny $\phi$ from $z'$ to $z$ such that $(z', z,\phi) \in \overline\Sh_{0,n}(K_{\gothp}^1).$
Furthermore, if we identify $\overline{\rm Y}_{i,z}$ with $\overline Z_i^{\langle n\rangle},$ then the intersection $\overline{\rm Y}_{i,z}\times_{\overline{{\Sh}}_{1,n-1}}\overline{ \rm Y}_{i+1,z'}$ (resp. $\overline{\rm Y}_{i,z}\times_{\overline{{\Sh}}_{1,n-1}}\overline{ \rm Y}_{i-1,z'}$) belongs to the special divisor class $\mathtt{SD}_+$ (resp. $\mathtt{SD_-}$) on $\overline Z_i^{\langle n\rangle}$. 
\end{propositionn}
\begin{proof}
    As in \cite[Proposition 6.4]{HTX17}, where we take $j\geq i$ and $\delta$ satisfies $0\leq \delta \leq \min\{n-j,i-1\},$ we have $\overline{\rm Y}_{i,z}\times_{\overline{{\Sh}}_{1,n-1}}\overline Y_{j,z'}$ is isomorphic to the variety $\overline Z_{i-\delta}^{\langle n+i-j-2\delta\rangle}.$ Now we require the dimension of $\overline Z_{i-\delta}^{\langle n+i-j-2\delta\rangle}$ to be $n+i-j-2\delta-1=n-2,$ i.e., $i-j-2\delta=-2.$ We get $0\leq 2\delta=i-j+1\leq 2$ and $j=i+1$ since $j\geq i.$

    If we identify $\overline{\rm Y}_{i,z}$ with $\overline Z_i^{\langle n\rangle}$ as in Proposition~\ref{S:corre(1,n-1)(0,n)}, we want to show $\overline{\rm Y}_{i,z}\times_{\overline{{\Sh}}_{1,n-1}}\overline{ \rm Y}_{i+1,z'}$ is ismorphic to $\mathtt{SD}_+.$ 

Let $(\cB_{z},\lambda_{z},\eta_{z})$ and  $(\cB_{z'},\lambda_{z'},\eta_{z'})$ be  the universal polarized abelian varieties on $\overline {\Sh}_{0,n}$ at $z$ and $z',$ respectively.
Then $\overline{\rm Y}_{i,z}\times_{\overline {\Sh}_{1,n-1}}\overline{\rm Y}_{j,z'}$ is the moduli space of  tuples $(A,\lambda, \eta, \phi, \phi')$  where $\phi:\cB_{z}\ra A$ and $\phi': \cB_{z'}\ra A$ are isogenies
 such that $(A,\lambda, \eta, \cB_{z},\lambda_{z},\eta_{z},\phi)$ and $(A,\lambda,\eta,\cB_{z'},\eta_{z'},\phi')$ are  points of $\overline{\rm Y}_{i,z}$ and $\overline{\rm Y}_{j,z'}$ respectively. We take   $
  M_k= \big( \tcD(\cB_z)^{\circ}_k\cap \tcD(\cB_{z'})^{\circ}_k\big) \big/p\big(\tcD(\cB_z)^{\circ}_k+\tcD(\cB_{z'})^{\circ}_k\big)
  $
  for $k=1,2.$ Then one has
$\dim_{\Fpb}(M_k)=n-1,$
 since we require $\overline{\rm Y}_{i,z}\times_{\overline{{\Sh}}_{1,n-1}}\overline{ \rm Y}_{i+1,z'}$ is isomorphic to $Z_{i}^{\langle n-1\rangle}.$
The Frobenius and Verschiebung  on $\tcD(\cB_z)$ induce two bijective Frobenius semi-linear maps $F: M_1\ra M_2$ and $V^{-1}: M_2\ra M_1.$ We denote their linearizations by the same notation if no confusions arise.
Let $Z(M_{\bullet})$ be the moduli space which attaches to each locally Noetherian $\Fpb$-scheme $S$ the set of isomorphism classes of pairs  $(L_1,L_2),$ where
 $L_1\subseteq M_1\otimes_{\Fpb}\cO_S$ and $L_2\subseteq M_2\otimes_{\Fpb}\cO_S$  are subbundles of rank $i$ and $i-1$ respectively such that
 $
L_2\subseteq  F(L_1^{(p)}),\quad V^{-1}(L_2^{(p)})\subseteq L_1.$
Pick a basis $(\varepsilon_{k,1}, \dots,\varepsilon_{k, n-1})$ of $M_k$ for $k=1,2$ under which the matrices of
$F$ and $V^{-1}$ are both identity.
We put $\varepsilon_{2,\ell}=F(\varepsilon_{1,\ell}).$ Using these bases to identify both $M_1$ and $M_2$ with $\Fpb^{n-1},$ it is clear that $Z(M_{\bullet})$ is isomorphic to the variety $\overline Z_{i}^{\langle n-1\rangle}.$

Now since we identify $\overline{\rm Y}_{i,z}$ with $Z_i^{\langle n\rangle},$ we get the closed immersion of ${\rm Y}_{i,z}\times_{\overline{{\Sh}}_{1,n-1}}{ \rm Y}_{i+1,z'}$ into ${\rm Y}_{i,z}$ can be identified with a closed immersion of $\overline Z_{i}^{\langle n-1\rangle}$ into $\overline Z_{i}^{\langle n\rangle}$ which is induced by $\tcD(\cB_z)^{\circ}_k\cap \tcD(\cB_{z'})^{\circ}_k \subseteq \tcD(\cB_z)^{\circ}_k$ for $k=1,2.$ 

The proof of $\overline{\rm Y}_{i,z}\times_{\overline{{\Sh}}_{1,n-1}}\overline{ \rm Y}_{i-1,z'}$ is ismorphic to $\mathtt{SD}_-$ is similar with a closed immersion $\overline Z_{i-1}^{\langle n-1\rangle}$ into $\overline Z_{i}^{\langle n\rangle}$ induced by $p\tcD(\cB_z)^{\circ}_k\subseteq p\big(\tcD(\cB_z)^{\circ}_k+ \tcD(\cB_{z'})^{\circ}_k\big)$ for $k=1,2$ and we omit here.
\end{proof}
\subsection{Automorphic representations}\label{S:atuomorphic representations}
In this subsection, we collect some facts concerning automorphic representation.
\subsubsection{Notation} We denote by $\ZZ_{\leq}^{n}$ and the subset of $\ZZ^{n}$ consisting of nondecreasing sequences For a finite set $T$ and an element $\xi=(\xi_{\tau})_{\tau\in T}\in (\ZZ_{\leq}^{n})^{T}$ put $a_{\xi}:=\min\limits_{\tau\in T}\{\xi_{\tau,1}\},~b_{\xi}:=\max\limits_{\tau\in T}\{\xi_{\tau,n}\}+n-1.$

Let $w$ be a nonarchimedean place of $E.$ For every irreducible admissible (complex) representation $\Pi$ of $\GL_n(E_w),$ every rational prime $\ell,$ and every isomorphism $\iota_{\ell}:\CC\xrightarrow{\simeq}\overline{\QQ}_{\ell},$ we denote by WD($\iota_{\ell}\Pi$) the (Frobenius semisimple) Weil-Deligne representation associated to $\iota_{\ell}\Pi$ via the local Langlands correspondence \cite{HT01}. We denote by $\Gamma_E$ the absolute Galois group of $E.$ Let $\mathtt c$ be the complex conjugation in $\Gal(\overline{E}/F).$
\begin{definitionn}
    Let ${\rm c}$ denote the complex conjugation on E. We say that a (complex) representation $\Pi$ of $\GL_n(\AAA_E)$ is $RACSDC$ (that is, regular algebraic conjugate self-dual cuspidal) if
    \begin{enumerate}
        \item $\Pi$ is an irreducible cuspidal automorphic representation;
        \item $\Pi\circ \mathtt{c}\simeq \Pi^{\vee}$;
        \item for every archimedean place $w$ of $E,$ $\Pi_w$ is regular algebraic (in the sense of \cite[Definition 3.12]{Clo90})
    \end{enumerate}
\end{definitionn}
If $\Pi$ is RACSDC, then there exists a unique element $\xi_{\Pi}=(\xi_{\tau,1},\dots,\xi_{\tau,n})_\tau\in(\ZZ_{\leq}^{n})^{\Sigma_{\infty,E}},$ which we call the \emph{archimedean weights of} $\Pi,$ satisfying $\xi_{\tau,i}=-\xi_{\tau^{\rm c},n+1-i}$ for every $\tau$ and $i,$ such that $\Pi_\tau$ (as a representation of $\GL_n(\CC)$) is isomorphic to the (irreducible) principal series representation induced by the characters $$(\arg^{1-n+2\xi_{\tau,1}},\arg^{3-n+2\xi_{\tau,2}},\dots,\arg^{n-1+2\xi_{\tau,n}}),$$ where $\arg:\CC^{\times}\ra\CC^{\times}$ is the \emph{argument character} defined by the formula $\arg(z):=z/\sqrt{z\overline{z}}.$
\begin{propositionn}\label{S:galois}
Let $\Pi$ be an \emph{RACSDC} representation of $\GL_n(\AAA_E)$ with the archimedean weights $\xi=\xi_{\Pi}.$ 
\begin{enumerate}
  \item For every nonarchimedean place $w$ of $E$, $\Pi_w$ is tempered.

  \item For every rational prime $\ell$ and every isomorphism $\iota_\ell\colon\CC\xrightarrow{\sim}\overline{\QQ}_\ell$, there is a semisimple continuous homomorphism
      \[
      \rho_{\Pi,\iota_\ell}\colon\Gamma_E\to\GL_n(\overline{\QQ}_\ell),
      \]
      unique up to conjugation, satisfying that for every nonarchimedean place $w$ of $E$, the Frobenius semisimplification of the associated Weil--Deligne representation of $\rho_{\Pi,\iota_\ell}|_{\Gamma_{E_w}}$ corresponds to the irreducible admissible representation $\iota_\ell\Pi_w|\det|_w^{\frac{1-n}{2}}$ of $\GL_n(E_w)$ under the local Langlands correspondence. Moreover, $\rho_{\Pi,\iota_\ell}^{\rm c}$ and $\rho_{\Pi,\iota_\ell}^\vee(1-n)$ are conjugate.
\end{enumerate}
\end{propositionn}
\begin{proof}
    We refer to \cite[Proposition 3.2.4]{LTXZZ} as a proof.
\end{proof}
\begin{propositionn}
\label{S:level raising}
    Let $\Pi$ be an \emph{RACSDC} representation of $\GL_n(\AAA_E)$ with the archimedean weights $\xi=\xi_{\Pi}.$ Denote the Satake parameter of $\Pi_\gothp$ by $\{\alpha_1,\dots,\alpha_n\}.$ Suppose that $\Pi_\gothp^{K_\gothp^1}\neq 0$ and there exists one pair $(i,j)$ such that $\alpha_i=p^2\alpha_j,$ then $\Pi$ is unramified or  the monodromy operator of WD$(\iota_\ell\Pi_\gothp)$ is conjugate to $\begin{pmatrix}
        1&1\\0&1
    \end{pmatrix}\oplus1_{n-2}$ for every rational prime $\ell$ and every isomorphism $\iota_\ell:\CC\xra{\sim}\overline{\QQ}_{\ell}.$
\end{propositionn}
\begin{proof}
    It follows from \cite[Lemma 3.1.5]{CHT08}.
\end{proof}
\begin{definitionn}\label{S:Strong coeffcient field}
    Let $\Pi$ be an RACSDC representation of $\GL_n(\AAA_E).$ We say that a subfield ${\rm L}\subseteq \CC$ is a strong coefficient field of $\Pi$ if ${\rm L}$ is a number field containing $\QQ(\Pi)$ which is the coefficient field of $\Pi$ as defined in \cite[Definition 3.1.1]{LTXZZ} and for every prime $\lambda$ of ${\rm L},$ there exists a continuous homomorphism
    $$
    \rho_{\Pi,\lambda}:\Gamma_E\ra\GL_n({\rm L}_\lambda),
    $$
    necessarily unique up to conjugation, such that for every isomorphism $\iota_\ell:\CC\xra{\sim}\overline{\QQ}_\ell$ inducing the prime $\lambda,$ $\rho_{\Pi,\lambda}\otimes_{{\rm L}_\lambda}\overline{\QQ}_{\ell}$ and $\rho_{\Pi,\iota_\ell}$ are conjugate.
\end{definitionn}

From now on, we fix a prime $\lambda$ of ${\rm L}$ with underlying rational prime $\ell$ and denote by $\cO_\lambda$ the maximal intergral domain of ${\rm L}_\lambda.$
Up to conjugation, we may suppose the image of $\rho_{\Pi,\lambda}$ lies in $\GL_n(\cO_\lambda).$

\subsection{Unitary Hecke algebra}\label{S:Hecke algebra}
We refer to \cite[Subsection 3.1]{LTXZZ} for the general theory of unitary Satake parameters and unitary Hecke algebras. Let $\Sigma_{\rm bad}^+$ be the union of all primes of $F$ that ramifies in $E.$ For every nonarchimedean place $v$ of $F$ not in $\Sigma_{\rm bad}^+,$ denote by $\mathbbm{T}_{n,v}$ the local spherical Hecke algebra. 

We first pick another prime $p'$ inert in $F$ and split in $E_0$ satisfying Hypothesis~\ref{Main hypo} and that $K_{p'}$ is of hyperspecial level for the proof of Theorem~\ref{Ihara lemma}. The existence of such prime follows from the Chebotarev density theorem and $K_v$ is hyperspecial for all but finitely many place $v$ of $\QQ.$

We take two finite sets $\Sigma_{\rm min}^+,\Sigma_\ell^+$ of nonarchimedean places of $F$ such that
\begin{itemize}
    \item $\Sigma_\ell^+$ consists of $\ell$-adic places of $F^+;$
   \item $\Sigma_{\rm min}^+$ and $\Sigma_\ell^+$ are mutually disjoint;
    \item $\Sigma_{\rm min}^+$ contains all primes that ramify in $E$ and all primes such that $K$ is not of hyperspecial level.
\end{itemize}
Let $\Sigma^+$ be a finite set of nonarchimedean places of $F^+$ containing $\Sigma_{\infty}^+\cup\Sigma_{\rm min}^+\cup\{p,p'\}.$ Denote the abstract unitary Hecke algebra away from $\Sigma^+$ by $\mathbbm{T}^{\Sigma^+}_{n}$ defined in \cite[Definition 3.1.9]{LTXZZ}. Let $\phi_\Pi:\mathbbm{T}^{\Sigma^+}_{n}\ra \cO_\lambda$ be the homomorphism defined as in \cite[Subsection 3.6]{LTXZZ2}.
We write $\mathfrak{m}$ for the preimage of $(\lambda)$ under $\phi_\Pi.$ 


\subsection{Hecke action on Shimura varieties}\label{Hecke action}
In this subsection, we consider Hecke action on the Shimura sets $\Sh_{0,n}$ and $\Sh_{0,n}(K_{\gothp}^{1}).$
\begin{definitionn}\label{S:Hecke action}
We denote
\begin{itemize}
    \item ${\mathbbm{T}_{0,0}}$ the Hecke algebra $\ZZ[K_{\gothp}\backslash GL_n(\QQ_{p^{2}})/K_{\gothp}];$ 
    \item ${\mathbbm{T}_{0,1}}$ the Hecke algebra $\ZZ[K_{\gothp}\backslash GL_n(\QQ_{p^{2}})/K^{1}_{\gothp}];$ 
    \item ${\mathbbm{T}_{1,0}}$ the Hecke algebra $\ZZ[K^{1}_{\gothp}\backslash GL_n(\QQ_{p^{2}})/K_{\gothp}];$ 
    \item ${\mathbbm{T}_{1,1}}$ the Hecke algebra $\ZZ[K^{1}_{\gothp}\backslash GL_n(\QQ_{p^{2}})/K^{1}_{\gothp}];$ 
    \item ${\rm T}^{(i)}_{0,0}\in \mathbbm{T}_{0,0}$ the function $K_{\gothp}\diag\{\underbrace{p,\dots,p}_{i},1,\dots,1\}K_{\gothp},$ for $1\leq i\leq n;$
    \item ${\rm R}^{(a,b)}_{0,0}\in \mathbbm{T}_{0,0}$ the function $K_{\gothp}\diag\{\underbrace{p^{2},\dots,p^{2}}_{a},\underbrace{p,\dots,p}_{b},1,\dots,1\}K_{\gothp},$ for $1\leq i\leq n;$
    \item ${\rm T}_{0,1}^{(0)}\in \mathbbm{T}_{0,1}$ the characteristic function $K_{\gothp}K_{\gothp}^{1};$
    \item ${\rm T}_{0,1}^{(i)}\in \mathbbm{T}_{0,1}$ the characteristic function $K_{\gothp}\diag\{\underbrace{p,\dots,p}_{i},1,\dots,1\}K_{\gothp}^{1},$ for $1\leq i\leq n;$
    \item ${\rm T}_{1,0}^{(0)}\in \mathbbm{T}_{1,0}$ the characteristic function $K_{\gothp}^{1}K_{\gothp};$
    \item ${\rm T}_{1,0}^{(i)}\in \mathbbm{T}_{1,0}$ the characteristic function $K_{\gothp}^{1}\diag\{\underbrace{p^{-1},\dots,p^{-1}}_{i},1,\dots,1\}K_{\gothp},$ for $1\leq i\leq n;$
      \item ${\rm T}_{1,1}\in \mathbbm{T}_{1,1}$ the characteristic function $K_{\gothp}^{1}\begin{pmatrix}
0      & 0      & \cdots & 0      & p^{-1} \\
1      & 0      & \cdots & 0      & 0      \\
0      & 1      & \cdots & 0      & 0      \\
\vdots & \ddots & \ddots & \vdots & \vdots \\
0      & \cdots & 0      & 1      & 0      
\end{pmatrix}K_{\gothp}^{1};$  
\end{itemize}
\end{definitionn}
\begin{remarkk}~\label{S:hecke actionnn}
    The definition of ${\rm T}_{0,0}^{(i)},{\rm R}^{(a,b)}_{0,0}$ coincide with ${\rm T}_\gothp^{(i)},{\rm R}_\gothp^{(a,b)}$ and ${\rm S}_{\gothp}$ in \cite[Section 6.2]{HTX17} and we will use these notations interchangeably. For simplicity, we denote ${\rm T}_{\gothp}^{(1)}$ by ${\rm T}$ and ${\rm T}_{1,1}$ by ${\rm A}.$
\end{remarkk}
\begin{propositionn}
    \begin{enumerate}
        \item Let $K_{\gothp}\gamma K_{\gothp}\in {\mathbbm T}_{0,0}.$ Let $z_1=(B_1,\lambda_1,\eta_1)$ and $z_2=(B_2,\lambda_2,\eta_2)$ be two points in $\Sh_{0,n}.$ Then $z_1\in K_{\gothp}\gamma K_{\gothp}(z_2)$ if and only if there exists an $\cO_D$-equivariant $p$-quasi-isogeny $\phi:B_1\ra B_2$ such that 
        \begin{itemize}
            \item $\phi^{\vee}\circ \lambda_2\circ \phi=\lambda_1;$
            \item $\phi\circ \eta_1=\eta_2;$
            \item There exist two lattices $\LL_1,\LL_2$ of the derham homology groups of $B_1,B_2$ defined over $\ZZ_{p^{2}}$ such that there exists a $\ZZ_{p^{2}}$-basis $(e_1,\dots,e_n)$ for $\LL_2$ such that  $(e_1,\dots,e_n)\gamma$ is a $\ZZ_{p^{2}}$-basis for $\phi_{*}(\LL_1).$
        \end{itemize}
        \item Let $z_1=(B_1,\lambda_1,\eta_1)$ and $z_2=(B_2,\lambda_2,\eta_2,B_3,\lambda_3,\eta_3,\psi)$ be two points in $\Sh_{0,n}$ and $\Sh_{0,n}(K_{\gothp}^{1}).$ Then 
        $z_2\in {\rm T}_{0,1}^{(0)}(z_1)$ if and only if $B_3=B_1.$
        \item Let $z_1=(B_1,\lambda_1,\eta_1)$ and $z_2=(B_2,\lambda_2,\eta_2,B_3,\lambda_3,\eta_3,\psi)$ be two points in $\Sh_{0,n}$ and $\Sh_{0,n}(K_{\gothp}^{1}).$ Then for $1\leq i\leq n,$
        $z_2\in {\rm T}_{0,1}^{(i)}(z_1)$ if and only if there exists an $\cO_D$-equivariant $p$-quasi-isogeny $\phi:B_3\ra B_1$ such that
            \begin{itemize}
                \item $\phi^{\vee}\circ \lambda_1\circ \phi=\lambda_3;$
                \item $\phi\circ \eta_3=\eta_1;$
                \item There exist three lattices $\LL_1,\LL_2,\LL_3$ of the derham homology groups of $B_1,B_2,B_3$ over $\ZZ_{p^{2}}$ such that there exists a basis $(e_1,\dots,e_n)$ for $\LL_1$ such that  $(e_1,pe_2,\dots,pe_{i},e_{i},\dots,e_n)$ is a $\ZZ_{p^{2}}$-basis for $\phi_{*}(\LL_3)$ and $(pe_1,\dots,pe_{i},e_{i+1},\dots,e_n)$ is a $\ZZ_{p^{2}}$-basis for $\phi_{*}\circ\psi_{*}(\LL_2).$
            \end{itemize}    
        \item Let $z_1=(B_1,\lambda_1,\eta_1)$ and $z_2=(B_2,\lambda_2,\eta_2,B_3,\lambda_3,\eta_3,\psi)$ be two points in $\Sh_{0,n}$ and $\Sh_{0,n}(K_{\gothp}^{1}).$ Then 
        $z_1\in {\rm T}_{1,0}^{(0)}(z_2)$ if and only if $B_3=B_1.$
        \item Let $z_1=(B_1,\lambda_1,\eta_1)$ and $z_2=(B_2,\lambda_2,\eta_2,B_3,\lambda_3,\eta_3,\psi)$ be two points in $\Sh_{0,n}$ and $\Sh_{0,n}(K_{\gothp}^{1}).$ Then for $1\leq i\leq n,$
        $z_1\in {\rm T}_{1,0}^{(i)}(z_2)$ if and only if there exists an $\cO_D$-equivariant $p$-quasi-isogeny $\phi:B_3\ra B_1$ such that
            \begin{itemize}
                \item $\phi^{\vee}\circ \lambda_1\circ \phi=\lambda_3;$
                \item $\phi\circ \eta_3=\eta_1;$
                \item There exist three lattices $\LL_1,\LL_2,\LL_3$ of the derham homology groups of $B_1,B_2,B_3$ over $\ZZ_{p^{2}}$ such that there exists a basis $(e_1,\dots,e_n)$ for $\LL_1$ such that  $(e_1,pe_2,\dots,pe_{i},e_{i},\dots,e_n)$ is a $\ZZ_{p^{2}}$-basis for $\phi_{*}(\LL_3)$ and $(pe_1,\dots,pe_{i},e_{i+1},\dots,e_n)$ is a $\ZZ_{p^{2}}$-basis for $\phi_{*}\circ\psi_{*}(\LL_2).$
            \end{itemize}
        \item Let $z_1=(B_1,\lambda_1,\eta_1,B_2,\lambda_2,\eta_2,\psi)$ and $z_2=(B_3,\lambda_3,\eta_3,B_4,\lambda_4,\eta_4,\phi)$ be two points in $\Sh_{0,n}(K_{\gothp}^{1}).$ Then $z_2\in {\rm T}_{1,1}(z_1)$ if and only if $B_2=B_3$ and there exist three lattices $\LL_1,\LL_2,\LL_4$ of the derham homology of $B_1,B_2=B_3,B_4$ defined over $\ZZ_{p^{2}}$ such that there exists a $\ZZ_{p^{2}}$-basis $(e_1,\dots,e_n)$ for $\LL_4$ such that  $(pe_1,\dots,e_n)$ is a $\ZZ_{p^{2}}$-basis for $\phi_{*}(\LL_2)$ and $(pe_1,pe_2,\dots,e_n)$ is a $\ZZ_{p^{2}}$-basis for $\phi_{*}\circ\psi_{*}(\LL_1).$
        \end{enumerate}
\end{propositionn}
\begin{proof}
    Note that $\Sh_{0,n}(\overline{\FF}_{p})$ is a union of $\#\ker^{1}(\QQ,G_{0,n})$ copies of $ G_{a_{\bullet}}(\QQ)\backslash G_{a_{\bullet}}(\AAA)/K^{p}\times(\ZZ_{p}^{\times}\times K_{\gothp})\times K_{\infty},$ and $\Sh_{0,n}(\overline{\FF}_{p})(K_{\gothp}^{1})$ is a union of $\#\ker^{1}(\QQ,G_{0,n})$ copies of $ G_{a_{\bullet}}(\QQ)\backslash G_{a_{\bullet}}(\AAA)/K^{p}\times(\ZZ_{p}^{\times}\times K_{\gothp}^{1})\times K_{\infty}.$ Then all the things can be checked by the defintion of Hecke action. 
\end{proof}
\begin{remarkk}\label{relationship T-Sh}
    It can be seen easily that two points $z,z'\in \Sh_{0,n}$ satisfies $z\in {\rm T}(z')$ if and only if there exists a $p$-quasi-isogeny $\phi$ from $z$ to $z'$ such that $(z,z',\phi)\in \Sh_{0,n}(K_{\gothp}^{1}).$ Hence we tend to use ${\rm T}$ and $\Sh_{0,n}(K_{\gothp}^{1})$ interchangably.
\end{remarkk}

\subsection{\texorpdfstring{$\ell$}{l}-adic cohomology}\label{S:l-adic-cohomology}
We introduce the $\ell$-adic cohomology of unitary Shimura varieties. In this subsection we do not assume $K_p$ is hyperspecial and use additional notation $(K)$ to express Shimura varieties of level $K.$ Different $K$'s only differ by the $p,p'$-part. Denote by $a_\bullet$ the signature $(a_1,a_2)=(a,n-a).$ Then by \cite[Lemma 2.9]{HTX17}, $G_{a_\bullet}(\AAA^\infty)\simeq G(\AAA^\infty).$

It is known that the \'etale cohomology group ${\rm H}^{d(a_{\bullet})}_{\et}(\overline {\Sh}_{a_{\bullet}}(K), {\rm L}_\lambda)$ is equipped with a natural action of $\mathbbm{T}_n^{\Sigma^+}\times \Gal(\Fpb/\FF_{p^2}).$ 
We write ${\rm H}_\et^{d(a_\bullet)}(\overline{\Sh}_{a_\bullet}(K), {\rm L}_\lambda)_\mathfrak{m}$ for the localization of the cohomology group at $\mathfrak{m}\subseteq \mathbbm{T}^{\Sigma^+}_n.$ The maximal ideal $\mathfrak{m}$ is called non-Eisenstein if $\bar\rho_{\mathfrak{m}}$ is absolutely irreducible. 
By \cite{Car12}, the cohomology group ${\rm H}^{i}_{\acute{e}t}({\overline{\Sh}}_{a_{\bullet}},\cO_{\lambda})_{\mathfrak{m}}$
is zero if $i\neq d(a_{\bullet})$ and is non-torsion otherwise. 

If ${\rm H}_\et^{d(a_\bullet)}(\overline{\Sh}_{a_\bullet}(K), {\rm L}_\lambda)_\mathfrak{m}$ is not zero, then by \cite[Proposition 3.4.2]{CHT08} there exists a (unique) semisimple Galois representation $\bar\rho_\mathfrak{m}:\Gamma_E\ra \GL_n(k_\lambda)$ which can be assumed to be isomorphic to $\bar \rho_{\Pi,\lambda}$ such that $\bar\rho_\mathfrak{m}$ is unramified outside $\Sigma_{\min}^+\cup\Sigma_{\ell}^+$ and for a prime number $r\notin \Sigma_{\min}^+\cup\Sigma_{\ell}^+$ that splits in $E$ and a place $w$ of $E$ above $r,$ the characteristic polynomial of $\bar\rho_\mathfrak{m}(\Frob_w)$ is given by the image of 
$$
X^n-T_{1,w}X^{n-1}+\dots+(-1)^iq^{\frac{i(i-1)}{2}}_wX^{n-i}+\dots+(-1)^nq_w^{\frac{n(n-1)}{2}}T_{n,w}
$$ in $k_\lambda[X].$ Here $\Frob_w$ is the geometric Frobenius Frobenius at $w,$ $q_w$ is the cardinality of the residue field of $w,$ and $T_{i,w}\in \mathbbm{T}_{n,r}$ are Hecke operators corresponding to the characteristic functions
$$
\ZZ_r^\times\times\GL_n(\cO_{E_w})\diag\{\underbrace{r,\dots,r}_{i},1,\dots,1\}\GL_n(\cO_{E_w}).
$$

We have a canonical decomposition of $\mathbbm{T}_n^{\Sigma^+}\times \Gal(\Fpb/\FF_{p^2})$-modules
\begin{equation*}\label{E:decomposition-cohomology}
{\rm H}^{d(a_{\bullet})}_{\et}(\overline{\Sh}_{a_{\bullet}}(K),{\rm L}_\lambda)_{\mathfrak{m}}=\bigoplus_{\pi\in S} \iota_{\ell}(\pi^K) \otimes R_{a_{\bullet},\ell}(\pi),
\end{equation*}
where $S$ denotes the set of irreducible admissible representations of $G(\AAA^{\infty})$ which are $\Pi$-congruent in the sense of \cite[Definition 6.1.8]{LTXZZ} with the coefficient ring replaced by $k_\lambda$ and $L_\lambda$.\footnote{The decomposition for $\overline{\QQ}_{\ell}$-coefficient comes from \cite[Proposition III.2.1]{HT01}. The decompoisition for ${\rm L}_\lambda$-coefficient comes from \cite[Remark 3.2.6; Proposition 3.2.8,3.2.11; Hypothesis 3.2.10; Lemma 6.1.10]{LTXZZ}. In fact, we can also take $L_\lambda$ to be $\QQ(\Pi)$ by \cite[Remark 3.2.6]{LTXZZ} for $a_\bullet=(1,n-1)$.} For any $\pi\in S,$ we denote by ${\rm BC}(\pi)$ its base change.
By \cite[Theorem 2.6]{HTX17} we have $$R_{a_{\bullet},\ell}(\pi)=\#\Ker^{1}(\QQ,G_{a_{\bullet}})m_{a_{\bullet}(\pi)}[(r_{a_{\bullet}}\circ\rho_{{{\rm BC}(\pi_{\gothp})}})\otimes \chi_{\pi_{p,0}}^{-1}\otimes \cO_{\lambda}(\sum_i\frac{a_i(a_i-1)}{2})].$$
For simplicity, we assume $\chi_{\pi_p,0}^{-1}$ to be trivial.

Let $\mathbbm{T}_\mathfrak{m}$ denote the image of $\mathbbm{T}_n^{\Sigma^+}$ in $\End(H^{d(a_{\bullet})}_{\et}(\overline{\Sh}_{a_{\bullet}}(K),{\rm L}_\lambda)_{\mathfrak{m}}).$
Then if we denote by $\rho^{univ}_\mathfrak{m}:\Gamma_E\ra \GL_n(\mathbbm{T}_\mathfrak{m}[1/\ell])$ the universal Galois representation, we have 
$$
{\rm H}^{d(a_{\bullet})}_{\et}(\overline{\Sh}_{a_{\bullet}}(K),{\rm L}_\lambda)_{\mathfrak{m}}=\tilde M_K \otimes_{\mathbbm{T}_\mathfrak{m}[1/\ell]} \big((\wedge^a\rho^{\rm univ}_{\mathfrak{m}}\otimes_{\mathbbm{T}_\mathfrak{m}[1/\ell]}\wedge^{n-a}\rho^{\rm univ}_{\mathfrak{m}})\otimes_{\mathbbm{T}_\mathfrak{m}[1/\ell]} {\rm L}_\lambda(\sum\limits_i\frac{a_i(a_i-1)}{2})\big).
$$ Here $\tilde M_K$ is a $\mathbbm{T}_\mathfrak{m}[1/\ell]$ module depending only on $K$ and $\mathfrak{m}.$
To describe the cohomology group with coefficient $k_\lambda$, we need 
to recall the following definition as in \cite[Definition 5.2]{Sch18}:
\begin{definitionn}
\label{sigma-typic}
    Let $(R,\mathfrak{m}_{R})$ be a Noetherian local ring, $G$ some group, and $\sigma_{R}:G\ra\GL_{n}(R)$
    an $n$-dimensional representation such that $\overline{\sigma}_{R}=\sigma_{R}\mod\mathfrak{m}_{R}$ is absolutely irreducible. Let $M$ be an $R[G]$-module. Then $M$ is said to be $\sigma_{R}$-typic if one can write $M$ as a tensor product $M=\sigma_{R}\otimes_{R}M_{0},$
    where $M_{0}$ is an $R$-module, and $G$ acts only through its action on $\sigma_{R}.$
\end{definitionn}

By \cite[Proposition 5.3]{Sch18}, we have the following proposition:
\begin{propositionn}
\label{Scho}
    In the situation of Definition~\ref{sigma-typic}, if $M$ is $\sigma_R$-typic, then $M_{0}=Hom_{R[G]}(\sigma_{R},M).$

    The functor $M_{0}\mapsto \sigma_{R}\otimes_{R}M_{0},~M\mapsto Hom_{R[G]}(\sigma_{R},M)$ induce an equivalence of categories between the category of $\sigma_{R}$-typic $R[G]$-modules and the category of $R$-modules.
\end{propositionn}
\begin{propositionn}\label{sigma_typic}
    In the situation of Definition~\ref{sigma-typic}, if $M$ is $\sigma_R$-typic and $N\subseteq M$ is an $R[G]$-submodule. Then $N$ is $\sigma_R$-typic. Let $k=R/\mathfrak{m}_R$ and $\bar \sigma_R$ be the residue representation. Then $M\otimes_R k$ is $\bar \sigma_R$-typic.
\end{propositionn}
\begin{proof}
    The first half of the proposition follows from \cite[Proposition 5.4]{Sch18}. The second can be checked directly.
\end{proof}

By Proposition~\ref{sigma_typic}, we get
$$
{\rm H}_{\acute{e}t}^{d(a_\bullet)}(\overline{\Sh}_{a_\bullet}(K),\cO_\lambda)_\mathfrak{m}=M_{K}\otimes_{\mathbbm{T}_\mathfrak{m}}\big((\wedge^a\rho^{\rm univ}_{\mathfrak{m}}\otimes_{\mathbbm{T}_\mathfrak{m}}\wedge^{n-a}\rho^{\rm univ}_{\mathfrak{m}})\otimes_{\mathbbm{T}_\mathfrak{m}} {\cO}_\lambda(\sum\limits_i\frac{a_i(a_i-1)}{2})\big)$$
with $M_{K}$ a $\mathbbm{T}_\mathfrak{m}$-module satisfying that $M_K[1/\ell]=\tilde M_K$ have no $\Gamma_E$-action and only depends on $\mathfrak{m}$ and $K;$ and
$$
{\rm H}_{\acute{e}t}^{d(a_\bullet)}(\overline{\Sh}_{a_\bullet}(K),k_\lambda)_\mathfrak{m}=\bar M_{K}\otimes_{k_\lambda}\big((\wedge^a\bar\rho^{\rm univ}_{\mathfrak{m}}\otimes_{k_\lambda}\wedge^{n-a}\bar\rho^{\rm univ}_{\mathfrak{m}})\otimes_{k_\lambda} {k}_\lambda(\sum\limits_i\frac{a_i(a_i-1)}{2})\big).
$$ with $\bar M_K=M_K\otimes_{\mathbbm{T}_\mathfrak{m}}k_\lambda$ and $\bar\rho^{\rm univ}_{\mathfrak{m}}$ the residue representation, which is $\bar \rho_\mathfrak{m}$ exactly. Hence we get the following proposition:
\begin{propositionn}\label{S:mod l cohomology}
    For signature $a_\bullet=(a,n-a)$ and $K$ an open compact subgroup defined above, we have 
    $$
{\rm H}_{\acute{e}t}^{d(a_\bullet)}(\overline{\Sh}_{a_\bullet}(K),k_\lambda)_\mathfrak{m}=\bar M_{K}\otimes_{k_\lambda}\big((\wedge^a\bar \rho_\mathfrak{m}\otimes_{k_\lambda}\wedge^{n-a}\bar \rho_\mathfrak{m})\otimes_{k_\lambda} {k}_\lambda(\sum\limits_i\frac{a_i(a_i-1)}{2})\big)
$$ for some $k_\lambda$ module $\bar M_{K}$ with trival $\Gamma_E$ action and depending only on $K$ and $\mathfrak{m}.$
\end{propositionn}

\subsection{Tate conjecture for \texorpdfstring{${{\Sh}}_{1,n-1}$}{Sh1,n-1} and \texorpdfstring{${{\Sh}}_{0,n}$}{Sh0,n}}
\label{Tatettttt}

In this subsection, we state a mod $ \ell$-version of the Tate conjecture proved in \cite{HTX17}.

Let $\ell\neq p$ be a prime number.
 For $1\leq j\leq n,$ there is a natural morphism
\begin{equation*}\label{E:jacquet-langlands}
\JL_j\colon {\rm H}^0_\et\big(\overline {\Sh}_{0,n},k_\lambda\big)\xra{{{\mathop{p}\limits^{\rightarrow}}^{*}_{j}}} {\rm H}^0_\et\big(\overline {\rm Y}_{j},k_\lambda\big)\xra{{{\mathop{p}\limits^{\leftarrow}}_{j,!}}} {\rm H}^{2(n-1)}_\et\big(\overline{\Sh}_{1,n-1},k_\lambda(n-1)\big),
\end{equation*}
where $\pr_{j,!}$ is the map whose restriction to each ${\rm H}^0_\et({\rm Y}_{j,z},k_\lambda)$ for $z\in {\Sh}_{0,n}(\Fpb)$ is the Gysin map associated to the closed immersion ${\rm Y}_{j,z}\hra \overline{\Sh}_{1,n-1}.$
It is clear that the image of $\JL_j$ is the subspace generated by the cycle classes of $[{\rm Y}_{j,z}]\in A^{n-1}(\overline{\Sh}_{1,n-1})$ with $z\in{\Sh}_{0,n}(\Fpb) .$
Putting all $\JL_j$'s together, we get  a morphism
\begin{equation*}
\JL_\mathfrak{m}=\sum_{j}\JL_j\colon \bigoplus_{j=1}^n {\rm H}^{0}_\et\big(\overline{\Sh}_{0,n},k_\lambda\big) \longrightarrow {\rm H}^{2(n-1)}_\et\big(\overline{\Sh}_{1,n-1},k_\lambda(n-1)\big).
\end{equation*}

\begin{theoremm}\label{T:main-theorem}
Let $\mathfrak{m}$ be the non-Eisenstein Hecke ideal considered in Subsection~\ref{S:Hecke algebra} and $\bar\rho_\mathfrak{m}$ the assoicated Galois representation.
Let $\bar \alpha_{1}, \dots, \bar\alpha_{n}$ denote the eigenvalues of $\bar\rho_{\mathfrak{m}}(\Frob_{p^2}).$ 
Let $m_{1,n-1}(\pi)$ $($resp. $m_{0,n}(\pi))$ denote the multiplicity for any irreducible representation $\pi$ of $G_{a_\bullet}(\AAA^\infty)$ for $a_{\bullet}=(1,n-1)$ (resp. for $a_{\bullet}=(0,n)$). Assume that $m_{1,n-1}(\pi)=m_{0,n}(\pi)$ for any $\pi$ appearing in ${\rm H}^{0}_\et\big(\overline{\Sh}_{0,n},k_\lambda\big)_{\mathfrak{m}}$ and ${\rm H}^{2n-2}_\et\big(\overline{\Sh}_{1,n-1},k_\lambda\big)_{\mathfrak{m}};$ and that $\bar \alpha_{1}, \dots, \bar\alpha_{n}$ are distinct, then the map 
\[
\JL_{\mathfrak{m}}:\bigoplus_{j=1}^n {\rm H}^{0}_\et\big(\overline{\Sh}_{0,n},k_\lambda\big)_{\mathfrak{m}}\longrightarrow {\rm H}_\et^{2(n-1)}\big(\overline{\Sh}_{1,n-1},k_\lambda(n-1)\big)_{\mathfrak{m}}^{{\rm Fr}_{p^2}}
\]is an isomorphism.
\end{theoremm}
\begin{proof}
    The injectivity follows from the determinant calculated in the proof of \cite[Theorem 4.18]{HTX17}. Also note that no eigenspace of the repsentation $\bar\rho_{\mathfrak{m}}\otimes\wedge^{n-1}\bar\rho_{\mathfrak{m}}$ with eigenvalues $\bar\alpha_i\bar\alpha_{j_1}\dots\bar\alpha_{j_{n-1}}$ contributes to the ${\rm Fr}_{p^2}$-invariant part for $i\in \{j_1,\dots,j_{n-1}\}$. Then by the same dimension comparison argument as in the proof of \cite[Theorem 4.18]{HTX17}, we get the surjectivity. Hence we finish the proof.
\end{proof}
\section{Dieudonn\'e modules and Grothendieck-Messing deformation theory}\label{DG}

In this section we focus on the connection between Dieudonn\'e modules and abelian varieties. We refer to \cite{Dem72} for general construction of Dieudonn\'e modules with respect to $p$-divisible groups, which we omit here.\footnote{The definition of the Dieudonne module we use in this paper is in fact the Serre dual of that in \cite{Dem72}.}
\subsection{Dieudonn\'e modules}

As in  Section~\ref{S:Shimura data}, we have the following isomorphism
\[
\calO_D \otimes_\ZZ \ZZ_{p^2}
\simeq \bigoplus_{i=1}^2
\big(\calO_D \otimes_{\calO_E, q_i}\ZZ_{p^2} \oplus \calO_D \otimes_{\calO_E,\bar q_i}\ZZ_{p^2} \big) \simeq \bigoplus_{i=1}^2 \big(\rmM_n(\ZZ_{p^2}) \oplus \rmM_n(\ZZ_{p^2})\big).
\]

Let $S$ be a locally Noetherian $\ZZ_{p^2}$-scheme. An $\cO_{D}\otimes_{\ZZ}\cO_{S}$-module $M$ admits a canonical decomposition
$
M=\bigoplus_{i=1}^2\big( M_{q_{i}}\oplus M_{\bar{q}_i}\big),
$
where $M_{q_i}$ (resp. $M_{\bar{q}_{i}}$) is  the direct summand of $M$ on which $\cO_E$ acts via $q_i$ (resp. via $\bar{q}_i$).
 Then each $M_{q_i}$ has a natural action by $\rmM_{n}(\cO_S).$
Let $\gothe$ denote the element of $\rmM_{n}(\calO_S)$  whose $(1,1)$-entry is $1$ and whose other entries are $0.$ We put $M^{\circ}_{i}:=\gothe M_{q_i},$ which is called the \emph{reduced part} of $M_{q_i}.$\footnote{The idea of taking the reduced part comes from Morita Equivalence.}.

Let $A$ be an $2n^2$-dimensional abelian variety over an $\FF_{p^2}$-scheme $S,$ equipped with an $\calO_D$-action.
The de Rham homology ${\rm H}^\dR_1(A/S)$ 
has a Hodge filtration\footnote{The exact sequence splits as direct sum in fact.}
$
0\ra \omega_{A^\vee/S}\ra {\rm H}^\dR_1(A/S)\ra \Lie_{A/S}\ra 0,
$
compatible with the natural  action of $\calO_D \otimes_\ZZ \calO_S$ on ${\rm H}^{\dR}_1(A/S).$
When $A \to S$ satisfies the moduli problem in Section~\ref{S:defn of Shimura var},  ${\rm H}^{\dR}_1(A/S)_{i}^\circ$ is locally free of rank $n$ and $\omega_{A^\vee/S,i}^\circ$ is subbundle of rank $a_i.$

When $S =  \Spec(k)$ with $k$ a perfect field containing $\F_{p^2},$ let $W(k)$ denote the ring of Witt vectors in $k.$ Let  $\tcD(A)$ denote  the (covariant) Dieudonn\'e module associated to the $p$-divisible group of  $A.$
 This is a free $W(k)$-module of rank $4n^2$ equipped with a $\Frob$-linear action of  $F$ and a $\Frob^{-1}$-linear action of $V$ such that $FV=VF = p.$
  The $\cO_D$-action on $A$ induces a natural action of $\cO_D$ on $\tcD(A)$ that commutes with $F$ and $V.$
For each $i\in \Z/2\Z,$ the Verschiebung and the Frobenius morphism on $A$ induce natural maps
$$
V: \tilde \calD(A)^\circ_i \longrightarrow \tilde \calD(A)^\circ_{i-1},\quad F: \tilde \calD(A)^\circ_i \longrightarrow \tilde \calD(A)^\circ_{i+1}.
$$
 Moreover, there is a canonical isomorphism  $\tilde \calD(A)_{i}/p\tilde \calD(A)_{i}\simeq {\rm H}^\dR_1(A/k)$ compatible with all structures on both sides.
 The action of $F$ and $V$ on $\calD(A)_{i}$ induces the Frobenius and Vershibung morphism on the de Rham homology:
 $
F: {\rm H}^\dR_1(A/S)^{\circ, (p)}_{i-1} \longrightarrow {\rm H}^\dR_1(A/S)^\circ_i,~
V: {\rm H}^\dR_1(A/S)^\circ_i \longrightarrow {\rm H}^\dR_1(A/S)^{\circ, (p)}_{i-1},
$
where by `$(p)$' we mean the pullback via the absolute Frobenius $\sigma$ of $S$ and there is an isomorphism ${\rm H}^\dR_1(A/S)^{\circ, (p)}_{i}\simeq {\rm H}^\dR_1(A/S)^{\circ}_{i}\otimes_{S,\sigma}S.$

It can be shown $\Ker(F) = \Image(V)=\omega^{\circ,(p)}_{A^\vee/S, i-1}$ and $\Ker(V) = \Image(F).$ This implies isomorphisms:
 \[
V\tilde \calD(A)_{i-1}/p\tilde \calD(A)_{i}\simeq \omega^{\circ}_{A^\vee/S, i},
~
\tilde \calD(A)_{i}/V\tilde \calD(A)_{i-1}\simeq \Lie^\circ_{A/S,i}
\]
For any $2n^2$-dimensional abelian variety $A'$ over $k$ equipped with an $\calO_D$-action, an $\calO_D$-equivariant isogeny $A' \to A$ induces a morphism $\tilde \calD(A')^\circ_i \to \tilde \calD(A)^\circ_i$ compatible with the actions of $F$ and $V.$ Conversely, \cite{HTX17} provides a proposition to obtain a new abelian variety corresponds to submodules of $\tcD(A).$ Here we give a similar proposition.

For any $2n^2$-dimensional abelian variety $A'$ over $k$ equipped with an $\calO_D$-action, an $\calO_D$-equivariant isogeny $A' \to A$ induces a morphism $\tilde \calD(A')^\circ_i \to \tilde \calD(A)^\circ_i$ compatible with the actions of $F$ and $V.$
Conversely, we have the following.

\begin{propositionn}
\label{P:abelian-Dieud}
Let $A$ be an abelian variety of dimension $2n^2$ over prefect field $k$ which contains $\FF_{p^2},$ equipped with an $\cO_{D}$-action and  an $\cO_D$-compatible prime-to-$p$ polarization $\lambda.$
Suppose given an integer $m\geq 1$ and a $W(k)$-submodule $\tcE_i\subseteq \tcD(A)^{\circ}_i$ for each  $i\in \Z/2\Z$ such that
\begin{equation}
\label{E:condition for getting another abelian varieties}
p^m\tcD(A)^{\circ}_{i}\subseteq \tcE_i, \quad F(\tcE_i)\subseteq \tcE_{i+1}, \quad \textrm{and} \quad  V(\tcE_i)\subseteq \tcE_{i-1}.
\end{equation}
Then there exists a unique abelian variety $A'$ over $k$ (depending on $m$) equipped with  an $\cO_{D}$-action, a prime-to-$p$ polarization $\lambda',$  and an $\cO_{D}$-equivariant $p$-isogeny $\phi: A'\ra A$ such that the natural inclusion $\tcE_i\subseteq \tcD(A)_i^\circ$ is naturally identified with the map
$\phi_{*,i}\colon\tcD(A')^{\circ}_i\ra \tcD(A)_i^{\circ}$
induced by $\phi$ and such that $\phi^\vee\circ\lambda\circ\phi=p^m\lambda'.$
Moreover, we have
\begin{enumerate}
\item

If $\dim \omega^\circ_{A^\vee/k, i} = a_i$ and $\len_{W(k)} \big( \tilde \calD(A)^\circ_i / \tilde \calE_i\big) = \ell_i$ for $i \in \ZZ/2\ZZ,$ then
$
\dim \omega^\circ_{A'^\vee/k, i} =a_i + \ell_i -\ell_{i+1}.
$

\item
If $A$ is equipped with a prime-to-$p$ level structure $\eta$ (in the sense of Subsection~\ref{S:defn of Shimura var}(3)), then there exists a unique prime-to-$p$ level structure $\eta'$ on $A'$ such that $\eta=\phi\circ\eta'.$
\end{enumerate}
\end{propositionn}
\begin{proof}
    The proof can be found in \cite[Proposition 3.2]{HTX17}.
\end{proof}

Let $\mathfrak{D}(A)$ be the Dieudonn\'e module corresponds to $A[p]$ and $\mathfrak{D}(A)^\circ$ denote its reduced part. Proposition~\ref{P:abelian-Dieud} also holds for $\mathfrak{D}(A)^{\circ}.$
By Proposition~\ref{P:abelian-Dieud}, we have a corollary, which is also useful:
\begin{corollaryy}
\label{CORP}
Let $A$ be an abelian variety of dimension $2n^2$ over prefect field $k$ which contains $\FF_{p^2},$ equipped with an $\cO_{D}$-action and  an $\cO_D$-compatible prime-to-$p$ polarization $\lambda.$
Suppose given a $W(k)$-submodule $\tcD(A)^{\circ}_i\subseteq \tcE_i\subseteq p^{-1}\tcD(A)^{\circ}_i$ (resp. $p\tcD(A)^{\circ}_i\subseteq \tcE_i\subseteq p^{-1}\tcD(A)^{\circ}_i$) for each  $i\in \Z/2\Z$ such that
\begin{equation*}
 F(\tcE_i)\subseteq \tcE_{i+1}, \quad \textrm{and} \quad  V(\tcE_i)\subseteq \tcE_{i-1}.
\end{equation*}
Then there exists an abelian variety $A'$ over $k$ equipped with an $\cO_{D}$-action, a prime-to-$p$ polarization $\lambda',$  and an $\cO_{D}$-equivariant $p$-isogeny $\phi: A\ra A'$ (resp. an $\cO_{D}$-equivariant $p$-quasi-isogeny $\phi: A\ra A'$) such that $\phi^\vee\circ\lambda'\circ\phi=p\lambda$ (resp. $\phi^\vee\circ\lambda'\circ\phi=\lambda$).
Moreover, we have
\begin{enumerate}
\item
If $\dim \omega^\circ_{A^\vee/k, i} = a_i$ and $\len_{W(k)} \big( \tilde \calD(A)^\circ_i / p\tilde \calE_i\big) = \ell_i$ for $i \in \ZZ/2\ZZ,$ then
$
\dim \omega^\circ_{A'^\vee/k, i} =a_i + \ell_i -\ell_{i+1}.
$

\item
If $A$ is equipped with a prime-to-$p$ level structure $\eta$ (in the sense of Subsection~\ref{S:defn of Shimura var}(3)), then there exists a unique prime-to-$p$ level structure $\eta'$ on $A'$ such that $\eta'=\phi\circ\eta.$
\end{enumerate}
\end{corollaryy}
\begin{proof}
    The two modules $p\tcE_{1},p\tcE_{2}$ satisfies (\ref{E:condition for getting another abelian varieties}). Applying Proposition~\ref{P:abelian-Dieud} with $m=1,$ we get there is an abelian variety $A'$ over $k$ equipped with  an $\cO_{D}$-action, a prime-to-$p$ polarization $\lambda',$ a unique prime-to-$p$ level structure $\eta',$ and an $\cO_{D}$-equivariant $p$-isogeny $\psi: A'\ra A$ such that the natural inclusion $p\tcE_i\subseteq \tcD(A)_i^\circ$ is naturally identified with the map
$\psi_{*,i}\colon\tcD(A')^{\circ}_i\ra \tcD(A)_i^{\circ}$
induced by $\psi$ and such that $\psi^\vee\circ\lambda\circ\psi=p\lambda',p\eta=\psi\circ\eta'.$ (The last equation holds since we can simply multiply $p$ on the level structure $\eta'$ we get by Proposition~\ref{P:abelian-Dieud}.) There is an isogeny $\phi:A\ra A'$ such that $\phi\circ \psi=p\circ \id_{A'}$ and $\psi\circ\phi=p\circ \id_{A}.$ Therefore $\phi^\vee\circ\lambda'\circ\phi=p\circ \lambda.$ Moreover, $\eta'=\phi\circ\eta.$

Moreover, if we apply Proposition~\ref{P:abelian-Dieud} with $m=2,$ we get there is an abelian variety $A'$ over $k$ equipped with  an $\cO_{D}$-action, a prime-to-$p$ polarization $\lambda',$ a unique prime-to-$p$ level structure $\eta',$ and an $\cO_{D}$-equivariant $p$-isogeny $\psi': A'\ra A$ such that the natural inclusion $p\tcE_i\subseteq \tcD(A)_i^\circ$ is naturally identified with the map
$\psi'_{*,i}\colon\tcD(A')^{\circ}_i\ra \tcD(A)_i^{\circ}$
induced by $\psi'$ and such that $\psi'^\vee\circ\lambda\circ\psi'=p^{2}\lambda',p\eta=\psi'\circ\eta'$(The last equation holds since we can simply multiply $p$ on the level structure $\eta'$ we get by Proposition~\ref{P:abelian-Dieud}.). There is an isogeny $\phi':A\ra A'$ such that $\phi'\circ \psi'=p^{2}\circ\id_{A'}$ and $\psi'\circ\phi'=p^{2}\circ \id_{A}.$ Therefore $\phi'^\vee\circ\lambda'\circ\phi'=p^{2}\circ \lambda.$ Moreover, $p\eta'=\phi'\circ\eta.$
Take a $p$-quasi-isogeny $\phi: A\ra A'$ such that $p\circ \phi=\phi'.$ Then $\phi^\vee\circ\lambda'\circ\phi=\lambda.$ Moreover, $\eta'=\phi\circ\eta.$ Hence we finish the proof.
\end{proof}
\begin{remarkk}\label{dual_isogeny}
    In the case $p\tcD(A)^{\circ}_i\subseteq \tcE_i\subseteq p^{-1}\tcD(A)^{\circ}_i,$ take a $p$-quasi-isogeny $\psi:A'\ra A$ such that $p\circ\psi=\psi'$ constructed in the proof. Then $\psi\circ\phi=1$ and $\phi\circ\psi=1.$ Moreover, for $i=1,2,$ $\psi_{*,i}:\tcD(A')_i^\circ\ra\tcD(A)^\circ_i$  induces inclusion $p\tcD(A)^\circ_i\subseteq\tcE_i\subseteq p^{-1}\tcD(A)^\circ_i.$
\end{remarkk}
\subsection{Grothendieck-Messing deformation theory}
\label{Deformaion}
Grothendieck-Messing deformation theory is important to compare the tangent spaces of moduli spaces. We state the theory following \cite{HTX17}.
We shall frequently use Grothendieck--Messing deformation theory to compare the tangent spaces of moduli spaces.  We make this explicit in our setup.

Let $ \hat R$ be a Noetherian $\FF_{p^2}$-algebra and $ \hat I \subset \hat R$ an ideal such that $\hat I^2=0.$  Put $R=  \hat R/ \hat I.$
 Let $\scrC_{\hat R}$ denote the category of tuples $(\hat A, \hat \lambda,\hat  \eta),$ where $\hat  A$ is an $2n^2$-dimensional abelian variety over $ \hat R$ equipped with an $\calO_D$-action, $ \hat  \lambda$ is a polarization on $\hat A$ such that the Rosati involution induces the $*$-involution on $\calO_D,$ and $ \hat \eta$ is a level structure as in Subsection~\ref{S:defn of Shimura var}(3). We define $\scrC_{R}$ in the same way.
For an object $(A,\lambda,\eta)$ in the category $\scrC_{R},$ let ${\rm H}^{\cris}_1(A/\hat R)$ be the evaluation of the first relative crystalline homology (i.e. dual crystal of the first crystalline cohomology) 
of $A/R$ at the divided power thickening $\hat R\ra R,$ and ${\rm H}^{\cris}_1(A/\hat R)^{\circ}_{i}:=\gothe {\rm H}^{\cris}_1(A/\hat R)_{q_i}$ be the $i$-th reduced part.
  We denote by  $\Def( R,\hat R)$ the category of tuples $(A, \lambda,\eta, (\hat \omega^{\circ}_i)_{i=1, 2}),$ where $(A,\lambda,\eta)$ is an object in $\scrC_{R},$ and $\hat \omega^{\circ}_{i}\subseteq {\rm H}^{\cris}_1(A/\hat R)_{i}^{\circ}$ for each $i\in \Z/2\Z$ is a subbundle that lifts $\omega^{\circ}_{A^\vee/R,i}\subseteq {\rm H}^{\dR}_1(A/R)^{\circ}_i.$
   The following is a combination of Serre--Tate  and Grothendieck--Messing deformation theory.

\begin{theoremm}[Serre--Tate, Grothendieck--Messing]
\label{T:STGM}
The functor $(\hat A, \hat \lambda, \hat \eta)\mapsto (\hat A \otimes_{\hat R}R , \lambda,\eta,\omega_{\hat A^\vee/\hat R,i}^{\circ}),$  where   $\lambda$ and $\eta$ are the natural induced polarization and level structure on $\hat A\otimes_{\hat R}R,$ is an equivalence of categories between $\scrC_{\hat R}$ and $\Def(R,\hat R).$
\end{theoremm}
\begin{proof}
 The proof can be found in \cite[Theorem 3.4]{HTX17}.
\end{proof}

\begin{corollaryy}
\label{C:tangent space of Sh}
If $\calA_{a_\bullet}$ denotes the universal abelian variety over ${\Sh}_{a_\bullet},$ then the tangent space $\mathcal{T}_{{\Sh}_{a_\bullet}}$ of ${\Sh}_{a_\bullet}$ is
$
\bigoplus_{i=1}^f \Lie_{\calA^\vee_{a_\bullet}/{\Sh}_{a_\bullet},i}^\circ \otimes \Lie_{\calA_{a_\bullet}/{\Sh}_{a_\bullet},i}^\circ.
$
\end{corollaryy}
\begin{proof}
 The proof can be found in \cite[Corollary 3.5]{HTX17}.
\end{proof}

\section{Description for the Higher Chow group}\label{D}
In this section, we give the proof of Theorem \ref{baby goal}. First we give the definition of higher Chow groups.
\begin{definition}
    Let $X$ be a smooth variety over a field $k$ and $\Delta^n=\Spec k[x_0,\dots,x_n]/(\sum\limits_{i=0}^{n}x_i-1)$ the standard $n$-simplex. For integers $n,r,$ we define $Z^r(X,n)$ to be the free abelian group generated by the integral closed subvarieties $Z$ contained in $X\time\Delta^n$ such that for any face $F\subseteq \Delta^n,$ we have ${\rm codim}_{X\times F}(Z\cap (X\times F))\geq r.$
    then get a chain complex:
    \[ \dots \rightarrow Z^r(X,n) \rightarrow Z^r(X,n-1) \rightarrow \dots \rightarrow Z^r(X,0) \rightarrow 0  \]
    where the differential is given by taking the alternating sum of the induced face maps.
    The higher Chow group ${\rm Ch}^r(X,n) $ is defined to be the $n^{th}$ homology of the above complex.
    Moreover, for any ring $R,$ we can define the $R$-coefficient higher Chow groups  by tensoring above chain complex with $R.$ We denote it by ${\rm Ch}^{r}(X,n,R).$
\end{definition}
\begin{proposition}
\label{property of higher Chow groups}
    Suppose $X$ is a smooth variety over a field $k$ and $R$ is a ring, then we have
    \begin{enumerate}
    \item
    ${{\rm{Ch}}}^i ({{X}},0, R) = {{\rm{Ch}}}^i(X,R),$
    where $\mathrm{Ch}^i({\rm X},R)$ is the Chow group with coefficients in $R$ as usual.
    \item The motivic cohomology  ${{\rm{H}}}^{i}_{\mathcal{M}} ({{{X}}},R(j)) = {{{\rm Ch}}}^j({{X}},2j-i,R) .$
    \item If ${{Y}}\subseteq {{X}}$ is a closed subscheme smooth of codimension $c,$ then the pushforward of cycles along $Z\times \Delta^{n}\rightarrow X\times \Delta^{n}$ induces the Gysin map $\mathrm{Ch}^{i}({{Y}},j,R)\to \mathrm{Ch}^{i+c}({{X}},j,R).$
\end{enumerate}
\end{proposition}
\begin{proof}
    $(1), (3)$ can be checked by definition. $(2)$ comes from \cite[Corollary 2]{Voe02}.
\end{proof}

Let $L$ be a $p$-comprime coefficient ring. We mainly concern about ${\rm Ch}^{1}({\Sh}_{1,n-1}^{{\rm ss}},1,L).$ By Proposition~\ref{S:corre(1,n-1)(0,n)}, ${\Sh}_{1,n-1}^{{\rm ss}}$ is equi-dimensional and its irreducible components can be expressed as the images of closed immersions ${\mathop{p}\limits^{\leftarrow}}_j|_{{\rm Y}_{j,z}}: {\rm Y}_{j,z}\rightarrow {\Sh}_{1,n-1}~(1 \leqslant j \leqslant n, z \in {\Sh}_{0,n}).$ Taking the Hecke correspondence $\T,$ we can define \[\mathcal{D} := \{D \in {\rm{Div}}({\Sh}_{1,n-1}^{{\rm{ss}}}) ~|~ D \subseteq {\mathop{p}\limits^{\leftarrow}}_j({\rm Y}_{j,z}) \cap {\mathop{p}\limits^{\leftarrow}}_j({\rm Y}_{i,z'}) \text{ for some} ~~ (j,z) \neq (i,z')\}, \]
\[D_{i,i+1} : = ({\rm Y}_i \times_{{\Sh}_{1,n-1}} {\rm Y}_{i+1}) \times_{{\Sh}_{0,n} \times {\Sh}_{0,n}} \Sh_{0,n}(K_\gothp^1) =  \coprod \limits_{z' \in \T(z)} ({\rm Y}_{i,z} \times_{{\Sh}_{1,n-1}} Y_{i+1,z'}).\] 
Then ${\rm{H}}^0(D_{i,i+1},L) \simeq {\rm{H}}^0(\Sh_{0,n}(K_\gothp^1),L)$ and by Proposition~\ref{S:intersection of Yi}, $\coprod_{D\in \mathcal{D}} D = \coprod_{i=1}^{n-1} D_{i,i+1}.$ By Nart \cite{Nar89}, we can get the following expression of ${\rm Ch}^{1}({\Sh}_{1,n-1}^{{\rm ss}},1,L)$:
\begin{proposition}
    \label{Chow}
    Let ${\rm Y}_i^{\circ} = {\mathop{p}\limits^{\leftarrow}}_i({\rm Y}_i) \backslash {\mathop{p}\limits^{\leftarrow}}_j(\cup_{j\neq i} {\rm Y}_j) ~(1 \leqslant i \leqslant n),$ then 
    \begin{equation*}
        {\rm{Ch}}^1({\Sh}_{1,n-1}^{{\rm ss}},1,L) 
        = \Ker \left(\bigoplus_i R({\rm Y}_{i}^{\circ})^{*} \stackrel{div}{\longrightarrow} \bigoplus_{i=1}^{n-1}  {\rm H}^0(\Sh_{0,n}(K_\gothp^1),L)\right),
    \end{equation*}
    where $R({\rm Y}_{i}^{\circ})^{*}:=(\cO({\mathop{p}\limits^{\leftarrow}}_i({\rm Y}_i^\circ))\otimes L)^*$ and $div: \bigoplus\limits_i R({\rm Y}_{i}^{\circ})^{*} \rightarrow \bigoplus_{i=1}^{n-1}  {\rm H}^0(\Sh_{0,n},L),$ is induced by the $div$ map on each ${\mathop{p}\limits^{\leftarrow}}_i({\rm Y}_i^{\circ})$ as usual.
\end{proposition}
\begin{proof}
    By~\cite[Corollary 1.2]{Nar89}, the higher Chow group ${\rm Ch}^{1}({\Sh}_{1,n-1}^{\rm ss},1,L)=\Ker(R({\Sh}_{1,n-1}^{\rm ss})^{*}\xra{div} Z^{1}({\Sh}_{1,n-1}^{\rm ss})).$ Here $R({\Sh}_{1,n-1}^{\rm ss})$ stands for the ring of rational functions of ${\Sh}_{1,n-1}^{\rm ss}.$ $Z^{1}({\Sh}_{1,n-1}^{\rm ss})$ is the group of $1$-codimensional cycles. Since $\{{\mathop{p}\limits^{\leftarrow}}_i({\rm Y}_{i,z})~|~ 1\leq i\leq n, z\in {\Sh_{0,n}(\overline{\FF}_{p}})\}$ are irreducible components of ${\rm Sh}_{1,n-1}^{\rm ss},$ we get further \[{\rm Ch}^{1}({\Sh}_{1,n-1}^{\rm ss},1,L)=\Ker(\bigoplus\limits_{i=1}^{n}R({\mathop{p}\limits^{\leftarrow}}_i({\rm Y}_{i}))^{*}\xra{div} Z^{1}({\Sh}_{1,n-1}^{\rm ss})).\]
    Since ${\rm Y}_{i}^{\circ}$ is an dense open subset of ${\mathop{p}\limits^{\leftarrow}}_i({\rm Y}_{i}),$ the ring of rational functions $R({\rm Y}_{i}^{\circ})=R({\mathop{p}\limits^{\leftarrow}}_i({\rm Y}_{i})).$ For any $1\leq i\leq n$ and any $f\in R({\rm Y}_{i}^{\circ})^{*}\bigcap {\rm Ch}^{1}({\Sh}_{1,n-1}^{\rm ss},1,L),$ the principal divisor $div(f)$ can be expressed sums of divisors contained in ${\mathop{p}\limits^{\leftarrow}}_i({\rm Y}_{i,z})\bigcap {\mathop{p}\limits^{\leftarrow}}_j({\rm Y}_{j,z'})$ for some $1\leq i,j\leq n$ and $z,z'\in {\rm Sh}_{1,n-1}$ such that $(i,z)\neq (j,z').$ Since each summation in $div(f)$ has codimension $1,$ we get $z'\in \T(z)$ with $j=i+1$ or $z\in \T(z')$ with $j=i-1$ from Proposition~\ref{S:intersection of Yi}. Furthermore, we have
    \begin{equation*}
        \begin{aligned}
            {\rm{Ch}}^1({\Sh}_{1,n-1}^{{\rm ss}},1,L) &= \Ker \left(\bigoplus_i R({\rm Y}_{i}^{\circ})^{*} \stackrel{div}{\longrightarrow} \bigoplus_{D \in \mathcal{D}}  L\right)\\
            &= \Ker \left(\bigoplus_i R({\rm Y}_{i}^{\circ})^{*} \stackrel{div}{\longrightarrow} \bigoplus_{i=1}^{n-1}  {\rm H}^0(D_{i,i+1},L)\right)\\
            &= \Ker \left(\bigoplus_i R({\rm Y}_{i}^{\circ})^{*} \stackrel{div}{\longrightarrow} \bigoplus_{i=1}^{n-1}  {\rm H}^0(\Sh_{0,n}(K_\gothp^1),L)\right).
        \end{aligned}
    \end{equation*}
\end{proof}
Hence we need to consider the principal divisors on ${{\Sh}}_{1,n-1}.$ More explicitly, if we identify ${\rm Y}_{i,z}$ with $Z_{i}^{\langle n \rangle},$ we need to consider when a divisor in $\mathbb{Q}[\mathtt{SD}_-]\oplus \mathbb{Q}[\mathtt{SD}_+]$ can be expressed as a principal divisor of ${\rm Y}_{i}^{\circ}.$

In fact, we have the following proposition:
\begin{proposition}
\label{ding}
      For all $1 \leqslant i \leqslant n-1,$ the divisor $\sum \limits_{[L]\in\mathtt{SD}_-}a_{L}[L]+\sum \limits_{[H]\in\mathtt{SD}_+}b_H[H]\in \mathbb{Q}[\mathtt{SD}_-]\oplus \mathbb{Q}[\mathtt{SD}_+]$ in  $Z_i^{\langle n \rangle}$ is principal if and only if:
      $ \sum \limits_{[L]\in\mathtt{SD}_-} a_L = 0, $
      and
      $ b_H = p^{i+1-n} \cdot\sum \limits_{\substack{L\subseteq H \\ [L] \in \mathtt{SD}_-}}a_L~,~ \forall [H] \in \mathtt{SD}_+ ;$
      or equvialently:
      $ \sum \limits_{[H]\in\mathtt{SD}_+} b_H = 0, $
      and
      $ a_L = p^{1-i} \cdot\sum \limits_{\substack{L\subseteq H \\ [H] \in \mathtt{SD}_+}}b_H~,~ \forall [L] \in \mathtt{SD}_- .$
\end{proposition}
\begin{proof}
    The proof can be proved exactly the same as \cite[Theorem 5.3.4]{Din19} with $\hat Z_i^{\langle n\rangle}={\rm RToySht}_V^{i-1}$ and $\widetilde Z_i^{\langle n\rangle}={\rm LToySht}_V^{i}$ except that we have a `middle object' $Z_i^{\langle n\rangle}$ here and it is why here the constant is $p^{i+1-n}$ not $p^{2(i+1-n)}.$
\end{proof}

For $1< i< n,$ we have the map $div: R({\rm Y}_{i}^{\circ})^{*} \stackrel{div}{\longrightarrow} {\rm H}^0(\Sh_{0,n}(K_\gothp^1),L)^{\oplus 2}$ as in \ref{Chow}. We can express any $(x,y)\in {\rm H}^{0}(\Sh_{0,n}(K_\gothp^1),L)^{\oplus 2}$ as:$ x= \sum \limits_{z'\in {\rm T}(z)} a_{z,z'} \cdot [{\rm Y}_{i-1,z}\bigcap {\rm Y}_{i,z'}],~ y = \sum \limits_{z'' \in {\rm T}(z')}b_{z',z''}\cdot [{\rm Y}_{i,z'}\bigcap {\rm Y}_{i+1,z^{''}}].$  With Proposition~\ref{ding}, we get the following corollary about the coeffcients $a_{z,z'}$ and $b_{z',z''}$:

\begin{corollary}
\label{connectionn}
    For $1<i<n,$ an element $(x,y)\in {\rm H}^{0}({\rm Sh}_{0,n}(K_\gothp^1),L)^{\oplus 2}$ is in the image of the map $div : R({\rm Y}_i^\circ)^\times \rightarrow {\rm H}^0(\Sh_{0,n}(K_\gothp^1),L)^{\oplus 2},$ i.e the divisors corresponds to $x$ and $y$ are all principal divisors defined by rational functions on $Y_{i}^{\circ}$ if and only if 
    $ \sum \limits_{z'\in {\rm T}(z)} a_{z,z'} = 0, $
      and
      $ b_{z',z''} = p^{i+1-n} \cdot\sum \limits_{\substack{(z',z)\in {\rm A}(z'',z') \\ z'\in {\rm T}(z)}} a_{z,z'}~,~ \forall z'' \in {\rm T}(z') ,$ where ${\rm A}$ is the Hecke action on ${{\Sh}}_{1,n-1}(K_{\gothp}^{1})$ defined in Definition~\ref{S:Hecke action}.
\end{corollary}

\begin{theorem}
\label{Choww}
    With notations as above, we have
\begin{equation*}
    \ch=\Ker({\mathrm{H}}^{0}({\Sh}_{0,n}(K_\gothp^1),L) \xrightarrow{\alpha} {\mathrm{H}^{0}}({{\Sh}}_{0,n},L)^{\oplus n}),
\end{equation*}
where $\alpha=(\lp,\rp, \rp {\rm A}, \cdots, \rp{\rm A}^{n-2} )$ with $(\lp,\rp)$ induced by the Hecke action $\T$ on $\Sh_{0,n}$ and $\A$ the Hecke action defined as above.
\end{theorem}
\begin{proof}
    For $z\in {\rm Sh}_{0,n}$ and every $f\in R({\rm Y}_{1,z}^\circ)^\times,$ $divf=\sum\limits_{z'\in {\rm T}(z)}a_{z,z'}[{\rm Y}_{1,z}\bigcap {\rm Y}_{2,z'}]\in {\rm H}^{0}(D_{1,2},L)={\rm H}^{0}({\Sh_{0,n}(K_\gothp^1)},L)$ satisfies $\sum\limits_{z'\in {\rm T}(z)}a_{z,z'}=0.$ The converse is also true. Therefore, $R({\rm Y}_1^\circ)^\times\simeq \Ker({\mathrm{H}}^{0}({\Sh_{0,n}(K_\gothp^1)},L) \xrightarrow{\lp} {\mathrm{H}^{0}}({{\Sh}}_{0,n},L)).$
    
    Now by induction on $1<i<n,$ for $z\in {\rm Sh}_{0,n}$ and any function $f\in R({\rm Y}_{i,z}^\circ)^\times,$ the map $div: R({\rm Y}_{i}^{\circ})^{*} \rightarrow {\rm H}^0(\Sh_{0,n}(K_\gothp^1),L)^{\oplus 2}$ maps $f$ to $(x,y)$ with $ x= \sum \limits_{z\in {\rm T}(z')} a_{z,z'} \cdot [{\rm Y}_{i-1,z'}\bigcap {\rm Y}_{i,z}],~ y = \sum \limits_{z'' \in \T(z)}b_{z,z''}\cdot [{\rm Y}_{i+1,z''}\bigcap {\rm Y}_{i,z}],$ if $divf=\sum \limits_{z\in {\rm T}(z')} a_{z,z'} \cdot [{\rm Y}_{i-1,z'}\bigcap {\rm Y}_{i,z}]+\sum \limits_{z''\in \T(z)}b_{z,z''}\cdot [{\rm Y}_{i+1,z''}\bigcap {\rm Y}_{i,z}].$ Then by \ref{connection}, we have $y=Ax$ and $\rp x=0.$
    Thus by induction we get $
    \ch=\Ker({\mathrm{H}}^{0}({\Sh}_{0,n}(K_\gothp^1),L) \xrightarrow{\alpha} {\mathrm{H}^{0}}({{\Sh}}_{0,n},L)^{\oplus n}),$
where $\alpha=(\lp,\rp, \rp {\rm A}, \cdots, \rp{\rm A}^{n-2} ).$
\end{proof}

\section{Stratification of \texorpdfstring{${{\Sh}}_{1,n-1}$}{Sh1,n-1}
}\label{S}
In this section, we analyze the Ekedahl--Oort stratification and the Newton stratification of ${{\Sh}}_{1,n-1}.$ For general theory of these two stratifications of unitary Shimura varieties, we refer to \cite{VW13} as a reference.

Firstly, we analyze the Ekedahl--Oort stratification of ${{\Sh}}_{1,n-1}.$

We take $G=\Res_{\QQ_{p^{2}/\QQ}}\GL_{n}\times \GG_{m}.$ Then $G$ has Weyl group $W=S_{n}\times S_{n}$ and the $p$-Frobenius morphism of $G$ induces a map $\Psi:W\rightarrow W$ by switching the two components of the Weyl group.
Let $\underline{A}=(A,\lambda,\eta)\in {{\Sh}}_{1,n-1}(\overline{\FF}_{p})$ and let $\mathfrak{D}^{\circ}=\mathfrak{D}_{1}(A)^{\circ}\oplus \mathfrak{D}_{2}(A)^{\circ}$ be the summation of the Dieudonn\'e module of $A[p]$ of rank $2n.$ Then there is a canonical action of $G$ on $\mathfrak{D}^{\circ}.$ By applying $F,V^{-1}$ to $(0)\subseteq \mathfrak{D}^{\circ}$ until it stabilizes, we obtain an $F,V^{-1}$-stable flag of $\mathfrak{D}^{\circ},$
\[\mathcal{C}_{\bullet}:\mathcal{C}_{0}=(0)\subseteq \dots \subseteq \mathcal{C}_{n}=\mathfrak{D}^{\circ}[V]=F(\mathfrak{D}^{\circ})\subseteq \dots \subseteq \mathcal{C}_{2n}=\mathfrak{D}^{\circ},\] where $\dim \mathcal{C}_{i}=i,$ called the canonical flag of $\underline{A}.$ More details on the canonical filtration can be found in \cite[Section 2.5,4.4,6.3]{Moo01}.

Let any extension of $\mathcal{C}_{\bullet}$ to a complete $\mathcal{O}_{K}$-invariant symplectic flag of $\mathfrak{D}^{\circ}$ be called a conjugate flag of $\underline{A}.$ Let $\mathcal{C}_{\bullet}$ denote a conjugate filtration of $\underline{A}$ and let $Q=Stab(\mathcal{C}_{\bullet}).$ It is easy to see that $Q$ is a Borel group as $\mathcal{C}_{\bullet}$ is a complete filtration. 

Let $J$ be the type of $P=Stab(\mathfrak{D}^{\circ}[F]=V\mathfrak{D}^{\circ}\subseteq \mathfrak{D}^{\circ})$ which is independent on the choice of $\underline{A}$ but only dependent on the moduli problem of ${{\Sh}}_{1,n-1}.$ Now for each $\underline{A},$ we get an element $w(\underline{A})$ in $^{J}W=W_{J}\ W$ which represents the relative position of $P$ and $Q$. Since $\mathfrak{D}(A)_{1}^{\circ}/V\mathfrak{D}(A)_{1}^{\circ}$ has rank $n-1$ over $\overline{\FF}_{p}$ and $\mathfrak{D}(A)_{2}^{\circ}/V\mathfrak{D}(A)_{2}^{\circ}$ has rank $1$ over $\overline{\FF}_{p},$ we get $W_{J}=S_{n}\backslash \{s_{1}\}\times S_{n}\backslash \{s_{n-1}\}$ where $s_{1}=(1,2)$ and $s_{n-1}=(n-1,n).$ Any $(w_{1},w_{2})\in {}^{J}W$ satisfies $w^{-1}_{1}(2)< w^{-1}_{1}(3)\dots <w^{-1}_{1}(n)$ and $w^{-1}_{2}(1)<w^{-1}_{2}(2)<\dots <w^{-1}_{2}(n-1).$

There is a partial order on $^{J}W,$ denoted by $\leqslant_{\Psi}$: For any $(w_{1},w_{2}), (w_{1}',w_{2}')\in \prescript{J}{}{W},$ we say $(w_{1}',w_{2}')\leqslant_{\Psi}(w_{1},w_{2})$ if and only if there exists $y\in W_{J}$ such that $y(w_{1}',w_{2}')x\Psi(y^{-1})x^{-1}\leqslant (w_{1},w_{2}),$ where $\leqslant$ is the Bruhat order and  $x=w_{0}w_{0,\Psi(J)}$ with $w_{0}$ and $w_{0,\Psi(J)}$ to be the element of maximal length in $W$ and $W_{J}.$

\begin{definition}
    In ${{\Sh}}_{1,n-1},$ the Ekedahl--Oort stratum associated to $w\in \prescript{J}{}{W}$ is a locally closed reduced subscheme $V^{w}$ with geometric points given by $V^{w}:=\{\underline{A}\in {{\Sh}}_{1,n-1}|w(\underline{A})=w\}.$
\end{definition}

\begin{theorem}\label{VW}
By \cite[Theorem 2,3]{VW13}, we get:
    \begin{enumerate}
        \item For all $w\in \prescript{J}{}{W},$ the Ekedahl--Oort stratum $V^{w}$ is non-empty and equidimensional of dimension $\ell(w),$ which is the length of $w\in W$ and is equal to $\ell(w_{1})+\ell(w_{2})$ if $w=(w_{1},w_{2}).$
        \item The Ekedahl--Oort strata are non-singular and quasi-affine.
        \item The closure of an Ekedahl--Oort stratum is a union of Ekedahl--Oort strata with respect to the partial order $\leqslant_{\Psi}$ on $^{J}W.$ That is, \[\overline{V}^{w}=\bigsqcup\limits_{w'\leqslant_{\Psi} w}V^{w'}.\]
    \end{enumerate} 
\end{theorem}

By convention, we will call the minimal Ekedahl--Oort stratum the core locus and the maximal Ekedahl--Oort stratum the $\mu$-ordinary locus.

\begin{proposition}
\label{order}
Based on Theorem~\ref{VW}, we have:
    \begin{enumerate}
        \item There are $n^{2}$ Ekedahl--Oort strata in ${{\Sh}}_{1,n-1}$ which is characterized by $w_{1}^{-1}(1)$ and $w_{2}^{-1}(n)$ for $(w_{1},w_{2})\in {}^{J}W.$ From now on, we use the notation $(w_{1},w_{2})=(a,b)$ to mean $w_{1}^{-1}(1)=a$ and $w_{2}^{-1}(n)=b.$
        \item The Ekedahl--Oort stratum corresponds to $(w_{1},w_{2})\in {}^{J}W$ has dimension $d=w^{-1}_{1}(1)-w^{-1}_{2}(n)+n-1.$
    \end{enumerate}
\end{proposition}
\begin{proof}
    Since any $(w_{1},w_{2})\in {}^{J}W$ satisfies $w^{-1}_{1}(2)< w^{-1}_{1}(3)\dots <w^{-1}_{1}(n)$ and $w^{-1}_{2}(1)<w^{-1}_{2}(2)<\dots <w^{-1}_{2}(n-1),$ we get $(1)$ easily.

    For $(2),$ we notice that $\ell(w_{1})=\sum\limits_{i=1}^{n-1}w^{-1}_{1}(i)-i$ and $\ell(w_{2})=w^{-1}_{2}(1)-1.$ Thus we can get $(2)$ directly.
\end{proof}

By \cite[Theorem 4.7]{Moo01}, we get the following explicit description of the action of $F,V$ on $\mathfrak{D}^{\circ}$:
\begin{proposition}\label{Formula}
    For every $(w_{1},w_{2})\in {}^{J}W$ and each $\underline{A}\in V^{(w_{1},w_{2})},$ we get there is a model for $\mathfrak{D}^{\circ}$ such that each $\mathfrak{D}^{\circ}_{i}$ has a basis $e_{i,1}\dots e_{i,n}$ for $i=1,2$ and $F,V$ act on $\mathfrak{D}^{\circ}$ as follows:
    \begin{equation*}
        F(e_{i,j})=\left\{
        \begin{array}{ll}
             0&w_{i}(j)\leq f(i)  \\
             e_{i+1,a}& w_{i}(j)=f(i)+a
        \end{array}
        \right.
    \end{equation*} and
    \begin{equation*}
    V(e_{i+1,j})=\left\{
    \begin{array}{ll}
         0&j\leq n-f(i)  \\
         e_{i,a}& j=n-f(i)+w_{i}(a)
    \end{array}
    \right.
    \end{equation*}
    where $i\in \Z/2\Z$ and $f(1)=1,f(2)=n-1.$
\end{proposition}
\begin{proof}
    It can be checked directly by taking $\mathcal{F}=\{1,2\}$ and $f(1)=1,f(2)=n-1$ in \cite[Theorem 4.7]{Moo01}.
\end{proof}

With Proposition~\ref{Formula}, we get the following proposition:
\begin{proposition}\label{supersingular}
    For $(w_{1},w_{2})=(a,b)\in {}^{J}W,$ we have $V^{(w_{1},w_{2})}\in {\mathop{p}\limits^{\leftarrow}}_{n+1-i}({\rm Y}_{n+1-i})$ if and only if $a\leq i\leq b.$ In particular, we have $V^{(w_{1},w_{2})}\in {{\Sh}}_{1,n-1}^{\rm ss}$ if and only if $a\leq b.$
\end{proposition}
\begin{proof}
    Let $k$ be a finite extension of $\FF_{p^{2}}.$ By Proposition~\ref{Formula}, we get for any $\underline{A}=(A,\lambda,\eta)\in V^{(w_{1},w_{2})}(k),$
    there is a basis of $\mathfrak{D}(A)_{1}^{\circ}\oplus \mathfrak{D}(A)_{2}^{\circ},$ denoted by $\{e_{i,j}~|~i=1,2;1\leq j\leq n\}$ such that $F,V$ act on $\mathfrak{D}(A)_{1}^{\circ}\oplus \mathfrak{D}(A)_{2}^{\circ}$ by \[
    F(e_{1,i})=\left\{\begin{array}{ll}
         e_{2,i}&\text{if } 1\leq i\leq a-1; \\
         0&\text{if }i=a;\\
         e_{2,i-1}&\text{if }i\geq a+1.
    \end{array}\right.
    F(e_{2,i})=\left\{\begin{array}{ll}
         0&\text{if } 1\leq i\leq b-1; \\
         e_{1,1}&\text{if }i=b;\\
         0&\text{if }i\geq b+1.\\
    \end{array}\right.
    \]
    \[
    V(e_{1,i})=\left\{\begin{array}{ll}
         0&\text{if } i=1; \\
         e_{2,i-1}&\text{if }2\leq i \leq b;\\
         e_{2,i}&\text{if }b+1\leq i\leq n;\\
    \end{array}\right.
    V(e_{2,i})=\left\{\begin{array}{ll}
         0&\text{if } 1\leq i\leq a-1; \\
         0&\text{if }a\leq i\leq n-1;\\
         e_{1,a}&\text{if }i=n.\\
    \end{array}\right.
    \]
    
    Therefore, for any $a\leq i\leq b,$ we let $\mathfrak{E}_1=k\{e_{1,1}, e_{1,2},\dots,e_{1,i}\}$ and $\mathfrak{E}_2=k\{e_{2,1},\dots,e_{2,i-1}\}.$
Then it can be checked that $F(\mathfrak{E}_i) \subseteq \mathfrak{E}_{3-i}$ and $V(\mathfrak{E}_i) \subseteq \mathfrak{E}_{3-i}$ for $i = 1,2.$
Applying Proposition~\ref{P:abelian-Dieud} and the remark below with $\mathfrak{E}_1,\mathfrak{E}_2,$ we get a point $(B,\lambda',\eta')$ of ${\Sh}_{0,n}$ and an isogeny $\phi: B\ra A,$ such that $(A,\lambda, \eta,B,\lambda',\eta',\phi)\in {\rm Y}_{n+1-i}.$ 

Conversely, by \cite[Proposition 6.4]{HTX17}, we get the intersection of ${\rm Y}_{i}$ and ${\rm Y}_{j}$ has dimension at most of $n+i-j$ for $1\leq i\leq j\leq n.$ This restricts the `only if' part must hold.
\end{proof}

Now we begin to analyze the Newton stratification of ${{\Sh}}_{1,n-1}.$
Via~\cite[Theorem 3.8]{RR96}, the Newton stratification of $\Sh_{1,n-1}$ coincides with the Newton polygon stratification $\Sh_{1,n-1}.$ So we study the Newton polygon stratification below and call it the Newton stratification of $\Sh_{1,n-1}.$

\begin{definition}\label{iisocrystal}
    Let $k$ be a perfect field of characteristic $p$ and $W(k)$ be the Witt vector ring corresponding to $k.$ We say a pair $(P,F)$ is a $Q(k)={\rm Frac}(W(k))$-isocrystal if $P$ is a finite-dimensional $Q(k)$-vector space together with a $\sigma$-linear automorphism $F.$ In particular, for any abelian variety $A$ over $k,$ the p-divisible group $A[p^{\infty}]$ gives an isocrystal $(\tcD^{\circ}\otimes_{W(k)}Q(k)=(\tcD(A)_{1}^{\circ}\oplus\tcD(A)_{2}^{\circ})\otimes_{W(k)}Q(k),F\otimes 1).$
\end{definition}

Following \cite{Dem72}, for each field $k$ of finite extension over $\FF_{p}$ and abelian variety $A$ over $k,$ the isocrystal $(\tcD^{\circ}\otimes_{W(k)}Q(k)=(\tcD(A)_{1}^{\circ}\oplus\tcD(A)_{2}^{\circ})\otimes_{W(k)}Q(k),F\otimes 1)$ has a unique decomposition such that $\tcD^{\circ}\otimes_{W(k)}Q(k)=\bigoplus\limits_{i=1}^{r}N(\lambda_{i}),$ where $0\leq \lambda_{1}< \dots < \lambda_{r}$ are the slopes corresponds to the isocrystal and $N(\lambda_{i})$ has pure slope $\lambda_i.$ The multiplicity of each $\lambda_i$ is equal to the dimension of $N(\lambda_i)$ over $Q(k).$ Considering the multiplicity, we use $\lambda_{\bullet}$ to denote a sequence of slopes $0\leq \lambda_{1}\leq \dots \leq \lambda_{2n}$ and $\lambda_{\bullet}(A)$ to denote the sequence of slopes of the isocrystal corresponds to $A.$ Each slope $\lambda_{\bullet}$ defines a Newton polygon by connecting the points in $x$-$y$-plane with coordinates $(i,\lambda_i)$ by line segments. Now we give the definition of Newton stratification:
\begin{definition}
    For every sequence of slopes $\lambda_{\bullet},$ the Newton stratum associated to it is a locally closed reduced subscheme of ${{\Sh}}_{1,n-1},$ denoted by $N^{\lambda_{\bullet}}$ such that
    $N^{\lambda_{\bullet}}=\{\underline{A}\in {{\Sh}}_{1,n-1}~|~\lambda_{\bullet}(\underline{A})=\lambda_{\bullet}\}.$

    For two different Netwon strata determined by $\lambda_{\bullet}$ and $\lambda_{\bullet}',$ we have $N^{\lambda_{\bullet}'}\subseteq \overline{N}^{\lambda_{\bullet}}$ if and only if the Newton polygon defined by $\lambda_{\bullet}$ is below the one of $\lambda_{\bullet}'.$
\end{definition}

\begin{proposition}\label{connection}
    There are $\frac{n(n-1)}{2}+1$ Newton strata in ${{\Sh}}_{1,n-1}.$ More explicitly, we have ${{\Sh}}_{1,n-1}^{\rm ss}$ corresponds to the Newton stratum with all the slopes to be $\frac{1}{2}.$ The others can be expressed as $N^{a,b},$ the sequence of slopes of which is $((\frac{a-1}{2a})^{2a},(\frac{1}{2})^{2b-2a},(\frac{n-b+1}{2n-2b})^{2n-2b})$ with $1\leq a\leq b\leq n-1.$ The dimension of $N_{a,b}$ is $b-a+n.$ Moreover, we have the $\mu$-ordinary locus of the Ekedahl--Oort strata is exactly $N^{1,n-1}.$
\end{proposition}
\begin{proof}
    Calculate by definition as in~\cite{VW13}, we have each slope must have even multiplicities and the $\mu$-ordinary locus corresponds to the sequence of slopes $(0^{2},(\frac{1}{2})^{2n-4},1^{2}).$ This forces 
    admissible sequences of slopes corresponds to nonsupersingular locus can only be $((\frac{a-1}{2a})^{2a},(\frac{1}{2})^{2b-2a},(\frac{n-b+1}{2n-2b})^{2n-2b})$ with $1\leq a\leq b\leq n-1.$
    Thus we get there can only be $\frac{n(n-1)}{2}+1$ Newton strata in ${{\Sh}}_{1,n-1}.$ It is well known that ${{\Sh}}^{\rm ss}_{1,n-1}$ corresponds to the Newton stratum with all the slopes to be $\frac{1}{2}.$  
    
    Now we want to show that $N^{1,n-1}\simeq V^{(w_1,w_2)}$ with $(w_1,w_2)=(n,1).$ Let $k$ be a finite extension of $\FF_{p^2}$ and $(A,\lambda,\eta)$ be a point of $V^{(w_1,w_2)}(k).$
    Suppose $\{e_{i,j}~|~i=1,2;1\leq j\leq n\}$ is a basis as in Proposition~\ref{supersingular}. Take $\tcD^{'\circ}=W(\overline{\FF}_{p})\{\tilde e_{i,j}~|~i=1,2;1\leq j\leq n\}$ such that 
    $\tilde e_j=e_j\mod p.$   Since after mod $p,$ $F^{2i}(\tilde e_{1,1})=\tilde e_{1,1},$ $F^{2i}(\tilde e_{2,1})=\tilde e_{2,1}$ and $V^{2i}(\tilde e_{1,n})=\tilde e_{1,n},$ $V^{2i}(\tilde e_{2,n})=\tilde e_{2,n}$ for all $i\geq 0,$ $\lambda_{\min}=0$ and $\lambda_{\max}=1.$ This shows that $V^{(w_1,w_2)}\subseteq N^{1,n-1}.$ Conversely, you can check in the same way that for points not in $V^{(w_1,w_2)},$ either the minimal slope is not equal to $0$ or the maximal slope is not equal to $1.$

\end{proof}



As the end of this section, we show the $\mu$-ordinary locus of ${{\Sh}}_{1,n-1}$ is affine:
\begin{proposition}
\label{affineness}
    The $\mu$-ordinary locus of ${{\Sh}}_{1,n-1}$ is affine.
\end{proposition}

By Proposition~\ref{connections}, any $S$-point $(A,\lambda,\eta)$ of ${{\Sh}}_{1,n-1}$ is in the Newton strata of minimal slope not less than $\frac{1}{4}$ if and only if $F^{2}\tcD(A)_{2}^{\circ}\subseteq V\tcD(A)_{1}^{\circ}.$ Moreover, the $S$-point $(A,\lambda,\eta)$ is in the Newton strata of maximal slope not larger than $\frac{3}{4}$ if and only if $V^{2}\tcD(A)_{2}^{\circ}\subseteq F\tcD(A)_{1}^{\circ}.$ 

Take $(\mathcal{A},\lambda,\eta)$ to be the universal object of ${{\Sh}}_{1,n-1}.$ It is in the Newton strata of minimal slope not less than $\frac{1}{4}$ is equivalent to $F^{2}\tcD(\mathcal{A})_{2}^{\circ}\subseteq V\tcD(\mathcal{A})_{1}^{\circ}.$ This is equivalent to the map $F^{2}:\Lie_{\mathcal{A}/{{\Sh}}_{1,n-1},2}^{\circ,(p^{2})}\ra \Lie_{\mathcal{A}/{{\Sh}}_{1,n-1},2}^{\circ}$ is trivial. In other words, $F^{2}$ is a trivial section in $\Gamma({{\Sh}}_{1,n-1},\Lie_{\mathcal{A}/{{\Sh}}_{1,n-1},2}^{\circ,(1-p^{2})})$ if and only if it is in the Newton strata of minimal slope not less than $\frac{1}{4}.$

In the other way, the universal object $(\mathcal{A},\lambda,\eta)$ is in the Newton strata of maximal slope not larger than $\frac{3}{4}$ is equivalent to $V^{3}\tcD(\mathcal{A})_{2}^{\circ}\subseteq p\tcD(\mathcal{A})_{1}^{\circ}.$ This is equivalent to $V^{2}:\omega_{\mathcal{A}^{\vee}/{{\Sh}}_{1,n-1},1}^{\circ}\ra\omega_{\mathcal{A}^{\vee}/{{\Sh}}_{1,n-1},1}^{\circ,(p^{2})}$ is trivial. In other words, $V^{2}$ is a trivial section in $\Gamma({{\Sh}}_{1,n-1},\omega_{\mathcal{A}^{\vee}/{{\Sh}}_{1,n-1},1}^{\circ,(p^{2}-1)})$ if and only if it is in the Newton strata of maximal slope not larger than $\frac{3}{4}.$

Thus the universal object in $\mu$-ordinary locus is nonvanishing locus of $F^{2}\otimes V^{2}\in \Gamma({{\Sh}}_{1,n-1},\Lie_{\mathcal{A}/{{\Sh}}_{1,n-1},2}^{\circ,(1-p^{2})}\otimes\omega_{\mathcal{A}^{\vee}/{{\Sh}}_{1,n-1},1}^{\circ,(p^{2}-1)}).$ Since $\wedge^{n} \mathcal{H}_{1}^{dR}(\mathcal{A}/{{\Sh}}_{1,n-1})_{2}^{\circ}$ is trivial, we have the line bundle $\Lie_{\mathcal{A}/{{\Sh}}_{1,n-1},2}^{\circ,-1}\otimes\omega_{\mathcal{A}^{\vee}/{{\Sh}}_{1,n-1},1}^{\circ}$ is ample if and only if $\wedge^{n}\omega_{\mathcal{A}/{{\Sh}}_{1,n-1}}^{\circ}$ is ample, which is a result of \cite[Theorem 7.2.4.1]{Lau14}. Since the the nonvanishing locus of a section of an ample line bundle on
a projective scheme is affine, we have the $\mu$-ordinary locus of ${{\Sh}}_{1,n-1}$ is affine.

\section{Geometry of \texorpdfstring{${\Sh}_{1,n-1}(K_{\gothp}^{1})$}{Sh_1,n-1}
}\label{G}
In this section, we analyze the geometry of ${{\Sh}}_{1,n-1}(K_{\gothp}^{1})$ for $n\geq 3.$
\begin{definition}\label{G:Condition on closed subschemes}
    Let $(\mathcal{A}_1,\lambda_1,\eta_1,\mathcal{A}_2,\lambda_2,\eta_2,\phi)$ be the universal object of $\cSh_{1,n-1}(K_{\gothp}^{1}).$
    For $0\leq i,j\leq 1,$ let ${\rm Y}_{ij}$ be the locus of $\Sh_{1,n-1}(K_{\gothp}^{1})$ on which the universal object satisfies $(1.i)(2.j)$ in the following:
    \begin{itemize}
        \item 
        $(1.0)$ $\omega^{\circ}_{\mathcal{A}_1^{\vee},1}=\Ker(\phi_{*,1}),$
        $(1.1)$ $\omega^{\circ}_{\mathcal{A}_2^{\vee},1}\subseteq \Im(\phi_{*,1}).$
        \item 
        $(2.0)$ $\Ker(\phi_{*,2})\subseteq\omega^{\circ}_{\mathcal{A}_1^{\vee},2},$
        $(2.1)$ $\Im(\phi_{*,2})=\omega^{\circ}_{\mathcal{A}_2^{\vee},2}.$
    \end{itemize}
\end{definition}

\begin{proposition}\label{G:stalk of points}
We have
    \begin{enumerate}
        \item The scheme $\cSh_{1,n-1}(K_{\gothp}^{1})$ is quasi-projective over $\ZZ_{p^{2}}$ of dimension $2n-2;$ and we have $${\Sh}_{1,n-1}(K_{\gothp}^{1})=\rm {\rm Y_{00}}\bigcup \rm {\rm Y_{01}} \bigcup \rm {\rm Y_{10}} \bigcup \rm {\rm Y_{11}};$$
        \item For $0\leq i,j\leq 1,$ ${\rm Y}_{ij}$ is smooth over $\FF_{p^{2}}$ of dimension $2n-2;$
        \item Let $k$ be a perfect field containing $\FF_{p^{2}}.$ $x$ be a closed point of $\Sh_{1,n-1}(K_{\gothp}^{1})(k).$ Let $S$ be the set of $i$ such that the universal object satisfies condition $(i.0)$ defined in Definition~\ref{G:Condition on closed subschemes} at $x$ and let $S'$ be the set of $j$ such that the universal object satisfies condition $(j.1)$ defined in Definition~\ref{G:Condition on closed subschemes} at $x.$ Then the completed local ring of $\cSh_{1,n-1}(K_{\gothp}^{1})$ at $x$ is $W(k)[\![X_i,i\in S;Y_j,j\in S';Z_1,\dots,Z_{r}]\!]/(X_iY_i-p,i\in S\cap S')$ with $r$ a positive integer to make the dimension of the local ring is $2n-2.$ Here $W(k)$ is the Witt ring of $k.$
    \end{enumerate}
\end{proposition}
\begin{proof}
    $(1),(2)$ follow immediately from $(3).$ We only give the proof of $(3).$ Let $\calO_x$ be the local ring of $\cSh_{1,n-1}(K_{\gothp}^{1})$ at $x.$ Suppose $x=(A_1,\lambda_1,\eta_1,A_2,\lambda_2,\eta_2,\phi).$ Then by Proposition~\ref{T:STGM}, for any Artinian $W(k)$-ring $R$ that is a quotient of $\calO_x,$ that $\Hom_{W(k)}(\calO_x,R)$ is the set of pairs of $R$-subbundles $$M_{1}^{(1)}\subseteq\tilde\calD(A_1)_1^{\circ}\otimes_{W(k)}R,~M_2^{(1)}\subseteq\tilde\calD(A_2)_1^{\circ}\otimes_{W(k)}R,~M_{1}^{(2)}\subseteq\tilde\calD(A_1)_2^{\circ}\otimes_{W(k)}R~M_{2}^{(2)}\subseteq\tilde\calD(A_2)_2^{\circ}\otimes_{W(k)}R$$ of ranks $1,1,n-1,n-1$ lifting $\omega_{A^{\vee}_1/k,1},\omega_{A^{\vee}_2/k,1},\omega_{A^{\vee}_1/k,2}$ and $\omega_{A^{\vee}_2/k,2}$ and a map $\Phi: \tilde\calD(A_1)^{\circ}\otimes_{W(k)}R\ra\tilde\calD(A_2)^{\circ}\otimes_{W(k)}R$ that lifts $\phi_*$ such that
    the cokernel of $\Phi$ is of rank $1$ and 
    $\Phi(M_1^{(i)})\subseteq M_{2}^{(i)}$ for $i=1,2.$  

    Let $\psi:A_2\ra A_1$ be the $p$-quasi-isogeny such that $\phi\circ\psi=p$ and $\psi\circ\phi=p.$ Let $\Psi$ be the lift of $\psi_*$ such that $\Psi\circ\Phi=p$ and $\Phi\circ\Psi=p.$

    By changing of coordinates, we can choose basis $({\hat{e}_{i,k}^{j}})_{k=1,\dots,n}$ of $\tilde\calD(A_i)^{\circ}_j\otimes_{W(k)}R$ for $1\leq i,j\leq 2$ such that $\Phi(\hat{e}_{1,1}^{j})=p\hat{e}_{2,1}^{j}~\Phi(\hat{e}_{1,k}^{j})=\hat{e}_{2,k}^{j}(k\geq 2)~\Psi(\hat{e}_{2,1}^{j})=\hat{e}_{1,1}^{j}~\Psi(\hat{e}_{2,k}^{j})=p\hat{e}_{1,k}^{j}(k\geq 2)$ for $j=1,2.$
    
    Let $({e_{i,k}^{j}})_{k=1,\dots,n}$ denote the reduction of $({\hat{e}_{i,k}^{j}})_{k=1,\dots,n}$ in ${\rm H}_{1}^{\rm cris}(A_i/k)_j^{\circ}$ for $1\leq i,j\leq 2.$
    Suppose $\omega_{A_1^{\vee}/k,1}^{\circ}$ is spanned by $(\alpha_1,\dots,\alpha_n)$ in ${\rm H}_{1}^{\rm cris}(A_1/k)_1^{\circ},$ $\omega_{A_2^{\vee}/k,1}^{\circ}$ is spanned by $(\beta_1,\dots,\beta_n)$ in ${\rm H}_{1}^{\rm cris}(A_2/k)_1^{\circ}$ with some $\alpha_i,\beta_i\in k$ for $1\leq i\leq n.$ Suppose the normal vector of $\omega_{A_1^{\vee}/k,2}^{\circ}$ in ${\rm H}_{1}^{\rm cris}(A_1/k)_2^{\circ}$ is $(\tilde\alpha_1,\dots,\tilde\alpha_n)$ and the nomral vector of $\omega_{A_2^{\vee}/k,2}^{\circ}$ in ${\rm H}_{1}^{\rm cris}(A_2/k)_2^{\circ}$ is $(\tilde\beta_1,\dots,\tilde\beta_n)$ with some $\tilde\alpha_i,\tilde\beta_i\in k$ for $1\leq i\leq n.$
    There are nine possible cases.
    \begin{enumerate}
        \item Suppose $\phi_{*}(\omega_{A_1^{\vee}/k,1}^{\circ})=0,~\psi_{*}(\omega_{A_2^{\vee}/k,1}^{\circ})=0,~\psi(\omega_{A_2^{\vee}/k,2}^{\circ})=0$ and $\phi_{*}(\omega_{A_1^{\vee}/k,2}^{\circ})$ is a vector bundle of corank one in $\omega_{A_2^{\vee}/k,2}^{\circ}.$ Then $\alpha_2,\dots,\alpha_n,\beta_1,\tilde\alpha_1,\tilde\beta_2,\dots,\tilde\beta_n$ are all zero. Thus we can assume that $\alpha_1=1,\tilde\beta_1=1.$ Suppose the four vectors are lifted to $(1+x_1,x_2,\dots,x_n),(y_1,\beta_2+y_2,\beta_3+y_3,\dots,\beta_n+y_n),(\tilde x_1,\tilde\alpha_2+\tilde x_2,\dots,\tilde\alpha_n+\tilde x_n)$ and $(1+\tilde y_1,\tilde y_2,\dots,\tilde y_n)$ with some $x_i,y_i,\tilde x_i,\tilde y_i$ belong to the maximal ideal of $R$ for $1\leq i\leq n.$ Since $\Phi(M_1^{(i)})\subseteq M_{2}^{(i)}$ and $\Psi(M_2^{(i)})\subseteq M_{1}^{(i)},$ there exists $a,b,\tilde a,\tilde b\in R$ such that $$(p(1+x_1),x_2,\dots,x_n)=a(y_1,\beta_2+y_2,\beta_3+y_3,\dots,\beta_n+y_n),$$$$~(y_1,p(\beta_2+y_2),p(\beta_3+y_3),\dots,p(\beta_n+y_n))=b(1+x_1,x_2,\dots,x_n),$$ and $$(\tilde x_1,p(\tilde\alpha_2+\tilde x_2),\dots,p(\tilde\alpha_n+\tilde x_n))=\tilde a(1+\tilde y_1,\tilde y_2,\dots,\tilde y_n),$$$$~ (p(1+\tilde y_1),\tilde y_2,\dots,\tilde y_n)=\tilde b(\tilde x_1,\tilde\alpha_2+\tilde x_2,\dots,\tilde\alpha_n+\tilde x_n).$$ Therefore, $ab=p, \tilde a\tilde b=p,$ $ay_1=p(1+x_1),~a(\beta_i+y_i)=x_i(i\geq 2)$ and $\tilde b\tilde x_1=p(1+\tilde y_1), ~\tilde b(\tilde \alpha_i+\tilde x_i)=\tilde y_i(i\geq 2).$ Thus $x_i,\tilde y_i\neq 0$ for $i\geq 2$ and $y_1,\tilde x_1\neq 0.$ Rescaling the basis, we may suppose $x_2=a, \tilde y_2=\tilde b$ and $y_1=b,\tilde x_1=\tilde a.$ Consider the $W(k)$-homomorphism from $W(k)[\![X_1,X_2,Y_1,Y_2,Z_1,\dots,Z_{2n-4}]\!]/(X_1Y_1-p,X_2Y_2-p)$ to $R$ mapping $X_1, Y_1$ to $a,b;$ $X_2,Y_2$ to $\tilde a,\tilde b;$ $Z_1\dots,Z_{n-2}$ to $x_3,\dots,x_n$ and $Z_{n-1},\dots,Z_{2n-4}$ to $\tilde{y}_3,\dots,\tilde{y}_{n},$ this shows $\calO_x$ is isomorphic to $W(k)[\![X_1,X_2,Y_1,Y_2,Z_1,\dots,Z_{2n-4}]\!]/(X_1Y_1-p,X_2Y_2-p).$
        \item Suppose $\phi_{*}(\omega_{A_1^{\vee}/k,1}^{\circ})=\omega_{A_2^{\vee}/k}^{\circ},~\psi_{*}(\omega_{A_2^{\vee}/k,1}^{\circ})=0,~\psi(\omega_{A_2^{\vee}/k,2}^{\circ})=0$ and $\phi_{*}(\omega_{A_1^{\vee}/k,2}^{\circ})$ is a vector bundle of corank one in $\omega_{A_2^{\vee}/k,2}^{\circ}.$ Then $\beta_1,\tilde\alpha_1,\tilde\beta_2,\dots,\tilde\beta_n$ are all zero. We can also assume that $\alpha_i=\beta_i$ for $i=2,\dots,n$ and $\tilde\beta_1=1.$ Suppose the four vectors are lifted to $(\alpha_1+x_1,\beta_1+x_2,\dots,\beta_n+x_n),(y_1,\beta_2+y_2,\beta_3+y_3,\dots,\beta_n+y_n),(\tilde x_1,\tilde\alpha_2+\tilde x_2,\dots,\tilde\alpha_n+\tilde x_n)$ and $(1+\tilde y_1,\tilde y_2,\dots,\tilde y_n)$ with some $x_i,y_i,\tilde x_i,\tilde y_i$ belong to the maximal ideal of $R$ for $1\leq i\leq n.$ Since $\Phi(M_1^{(i)})\subseteq M_{2}^{(i)}$ and $\Psi(M_2^{(i)})\subseteq M_{1}^{(i)},$ there exists $a,b,\tilde a,\tilde b\in R$ such that $$(p(\alpha_1+x_1),\beta_2+x_2,\dots,\beta_nx_n)=a(y_1,\beta_2+y_2,\beta_3+y_3,\dots,\beta_n+y_n),$$$$~(y_1,p(\beta_2+y_2),p(\beta_3+y_3),\dots,p(\beta_n+y_n))=b(\alpha_1+x_1,\beta_2+x_2,\dots,\beta_n+x_n),$$ and $$(\tilde x_1,p(\tilde\alpha_2+\tilde x_2),\dots,p(\tilde\alpha_n+\tilde x_n))=\tilde a(1+\tilde y_1,\tilde y_2,\dots,\tilde y_n),$$$$~ (p(1+\tilde y_1),\tilde y_2,\dots,\tilde y_n)=\tilde b(\tilde x_1,\tilde\alpha_2+\tilde x_2,\dots,\tilde\alpha_n+\tilde x_n).$$ Therefore, $ab=p, \tilde a\tilde b=p,$ $ay_1=p(\alpha_1+x_1),~a(\beta_i+y_i)=\beta_i+x_i(i\geq 2)$ and $\tilde b\tilde x_1=p(1+\tilde y_1), ~\tilde b(\tilde \alpha_i+\tilde x_i)=\tilde y_i(i\geq 2).$ Thus after proper coordinates changing, $a=1,b=p,y_1=p(\alpha_1+x_1),y_i=x_i(i\geq 2).$ If $\alpha_1\neq 0,$ by rescaling of basis, we may suppose $y_1=1$ and $\tilde x_1=\tilde a,\tilde y_2=\tilde b.$ Consider the $W(k)$-homomorphism from $W(k)[\![X_2,Y_1,Y_2,Z_1,\dots,Z_{2n-4}]\!]/(X_2Y_2-p)$ to $R$ mapping $X_2, Y_2$ to $\tilde a,\tilde b;$ $Y_1$ to $x_2;$$Z_1\dots,Z_{n-2}$ to $x_3,\dots,x_n$ and $Z_{n-1},\dots,Z_{2n-4}$ to $\tilde{y}_3,\dots,\tilde{y}_{n},$ this shows $\calO_x$ is isomorphic to $W(k)[\![X_2,Y_1,Y_2,Z_1,\dots,Z_{2n-3}]\!]/(X_2Y_2-p).$ If $\alpha_1=0,$ suppose $x_i/x_2\in W(k)$ for $2\leq i\leq n.$ By rescaling coordinates, we can suppose $x_2=1.$ And with a similar argument as above, we can show $\calO_x$ is isomorphic to $W(k)[\![X_2,Y_1,Y_2,Z_1,\dots,Z_{2n-4}]\!]/(X_2Y_2-p).$
        
        The following several possible conditions can be dealed with similar as above, so we omit details and just list the result here.
        \item Suppose $\phi_{*}(\omega_{A_1^{\vee}/k,1}^{\circ})=0,~\psi_{*}(\omega_{A_2^{\vee}/k,1}^{\circ})\neq 0,~\psi(\omega_{A_2^{\vee}/k,2}^{\circ})=0$ and $\phi_{*}(\omega_{A_1^{\vee}/k,2}^{\circ})$ is a vector bundle of corank one in $\omega_{A_2^{\vee}/k,2}^{\circ}.$ Then $\calO_x$ is isomorphic to $W(k)[\![X_1,X_2,Y_2,Z_1,\dots,Z_{2n-4}]\!]/(X_2Y_2-p).$
         \item Suppose $\phi_{*}(\omega_{A_1^{\vee}/k,1}^{\circ})=\omega_{A_2^{\vee}/k}^{\circ},~\psi_{*}(\omega_{A_2^{\vee}/k,1}^{\circ})= 0,~\psi(\omega_{A_2^{\vee}/k,2}^{\circ})\neq0$ and $\phi_{*}(\omega_{A_1^{\vee}/k,2}^{\circ})$ is a vector bundle of corank one in $\omega_{A_2^{\vee}/k,2}^{\circ}.$ Then $\calO_x$ is isomorphic to $W(k)[\![Y_1,Y_2,Z_1,\dots,Z_{2n-4}]\!].$
         \item Suppose $\phi_{*}(\omega_{A_1^{\vee}/k,1}^{\circ})=0,~\psi_{*}(\omega_{A_2^{\vee}/k,1}^{\circ})= 0,~\psi(\omega_{A_2^{\vee}/k,2}^{\circ})\neq0$ and $\phi_{*}(\omega_{A_1^{\vee}/k,2}^{\circ})$ is a vector bundle of corank one in $\omega_{A_2^{\vee}/k,2}^{\circ}.$ Then $\calO_x$ is isomorphic to $W(k)[\![X_1,Y_1,Y_2,Z_1,\dots,Z_{2n-4}]\!]/(X_1Y_1-p).$
          \item Suppose $\phi_{*}(\omega_{A_1^{\vee}/k,1}^{\circ})=0,~\psi_{*}(\omega_{A_2^{\vee}/k,1}^{\circ})\neq0,~\psi(\omega_{A_2^{\vee}/k,2}^{\circ})\neq0$ and $\phi_{*}(\omega_{A_1^{\vee}/k,2}^{\circ})$ is a vector bundle of corank one in $\omega_{A_2^{\vee}/k,2}^{\circ}.$ Then $\calO_x$ is isomorphic to $W(k)[\![X_1,Y_2,Z_1,\dots,Z_{2n-4}]\!].$
           \item Suppose $\phi_{*}(\omega_{A_1^{\vee}/k,1}^{\circ})=\omega_{A_2^{\vee}/k}^{\circ},~\psi_{*}(\omega_{A_2^{\vee}/k,1}^{\circ})= 0,~\psi(\omega_{A_2^{\vee}/k,2}^{\circ})=0$ and $\phi_{*}(\omega_{A_1^{\vee}/k,2}^{\circ})=\omega_{A_2^{\vee}/k,2}^{\circ}.$ Then $\calO_x$ is isomorphic to $W(k)[\![Y_1,X_2,Z_1,\dots,Z_{2n-4}]\!].$
         \item Suppose $\phi_{*}(\omega_{A_1^{\vee}/k,1}^{\circ})=0,~\psi_{*}(\omega_{A_2^{\vee}/k,1}^{\circ})= 0,~\psi(\omega_{A_2^{\vee}/k,2}^{\circ})=0$ and $\phi_{*}(\omega_{A_1^{\vee}/k,2}^{\circ})=\omega_{A_2^{\vee}/k,2}^{\circ}.$ Then $\calO_x$ is isomorphic to $W(k)[\![X_1,X_2,Y_1,Z_1,\dots,Z_{2n-4}]\!]/(X_1Y_1-p).$
          \item Suppose $\phi_{*}(\omega_{A_1^{\vee}/k,1}^{\circ})=0,~\psi_{*}(\omega_{A_2^{\vee}/k,1}^{\circ})\neq0,~\psi(\omega_{A_2^{\vee}/k,2}^{\circ})=0$ and $\phi_{*}(\omega_{A_1^{\vee}/k,2}^{\circ})=\omega_{A_2^{\vee}/k,2}^{\circ}.$ Then $\calO_x$ is isomorphic to $W(k)[\![X_1,X_2,Z_1,\dots,Z_{2n-4}]\!].$
    \end{enumerate}
    We denote by 
    \begin{enumerate}
        \item ${\rm Y}_{00}$ the closed subscheme consisting of points satisfying $(1.0)$ and $(2.1);$
        \item ${\rm Y}_{01}$ the closed subscheme consisting of points satisfying $(1.0)$ and $(2.0);$
        \item ${\rm Y}_{10}$ the closed subscheme consisting of points satisfying $(1.1)$ and $(2.1);$
        \item ${\rm Y}_{11}$ the closed subscheme consisting of points satisfying $(1.1)$ and $(2.0).$
    \end{enumerate}
\end{proof}
\begin{proposition}\label{G:blow up}
    Let $\widetilde\cSh_{1,n-1}(K_\gothp^1)$ be the blow up of $\cSh_{1,n-1}(K_\gothp^1)$ at the closed subscheme ${\rm Y}_{00}.$ Then $\widetilde\cSh_{1,n-1}(K_\gothp^1)$ is strictly semistable\text{~\cite[Subsection 1.1]{Sai03}} with special fiber $\widetilde\Sh_{1,n-1}(K_\gothp^1)=\widetilde{\rm Y}_{00}\bigcup \tilde{\rm Y}_{10}\bigcup\widetilde{\rm Y}_{01}\bigcup\widetilde{\rm Y}_{11}$ satisfying 
    \begin{itemize}
        \item $\widetilde{\rm Y}_{00}$ is the blow up of ${\rm Y}_{00}$ at ${\rm Y}_{00}\bigcap{\rm Y}_{11},$ 
        $\widetilde{\rm Y}_{11}$ is the blow up of ${\rm Y}_{11}$ at ${\rm Y}_{00}\bigcap{\rm Y}_{11};$
        \item $\widetilde{\rm Y}_{10}$ is isomorphic to ${\rm Y}_{10},$ 
        $\widetilde{\rm Y}_{01}$ is isomorphic to ${\rm Y}_{01};$
        \item $\widetilde{\rm Y}_{00}\bigcap \widetilde{\rm Y}_{01}$ is isomorphic to ${\rm Y}_{00}\bigcap {\rm Y}_{01},$ 
        $\widetilde{\rm Y}_{00}\bigcap \widetilde{\rm Y}_{10}$ is isomorphic to ${\rm Y}_{00}\bigcap {\rm Y}_{10};$
        \item $\widetilde{\rm Y}_{11}\bigcap \widetilde{\rm Y}_{10}$ is isomorphic to ${\rm Y}_{11}\bigcap {\rm Y}_{10},$ 
        $\widetilde{\rm Y}_{11}\bigcap \widetilde{\rm Y}_{01}$ is isomorphic to ${\rm Y}_{11}\bigcap {\rm Y}_{01};$
        \item $\widetilde{\rm Y}_{11}\bigcap \widetilde{\rm Y}_{00}$ is isomorphic to a $\PP^1$-bundle over ${\rm Y}_{11}\bigcap {\rm Y}_{00};$ 
        \item $\widetilde{\rm Y}_{01}\bigcap \widetilde{\rm Y}_{10}$ is empty;
        \item For the intersection of three irreducible components, only $\widetilde{\rm Y}_{11}\bigcap \widetilde{\rm Y}_{00}\bigcap\widetilde{\rm Y}_{10}$ and $\widetilde{\rm Y}_{11}\bigcap \widetilde{\rm Y}_{00}\bigcap\widetilde{\rm Y}_{01}$ are non empty and all isomorphic to ${\rm Y}_{11}\bigcap {\rm Y}_{00}.$
    \end{itemize}
\end{proposition}
\begin{proof}
    By Proposition~\ref{G:stalk of points}, ${\cSh_{1,n-1}}(K_{\gothp}^1)-{\rm Y}_{00}\bigcap
    {\rm Y}_{11}$ is smooth. Hence ${\rm Y}_{00}-{\rm Y}_{00}\bigcap{\rm Y}_{11}$ is a Cartier divisor of ${\cSh_{1,n-1}}(K_{\gothp}^1)-{\rm Y}_{00}\bigcap
    {\rm Y}_{11}$ and blow up of ${\cSh_{1,n-1}}(K_{\gothp}^1)-{\rm Y}_{00}\bigcap{\rm Y}_{11}$ at ${\rm Y}_{00}-{\rm Y}_{00}\bigcap{\rm Y}_{11}$ changes nothing.

    Since blow up commutes with flat base change, so we can pass to the completed local rings. Then we only need to consider the blow up of $\Spec W(k)[\![X_1,X_2,Y_1,Y_2,Z_1,\dots,Z_{2n-4}]\!]/(X_1Y_1-p,X_2Y_2-p)$ at $(X_1,X_2),$ which is ${\rm Proj}W(k)[\![X_1,X_2,Y_1,Y_2,Z_1,\dots,Z_{2n-4}]\!][T_1,T_2]/(X_1Y_1-p,X_2Y_2-p,X_1T_2-X_2T_1,Y_1T_1-Y_2T_2).$ This shows the fiber of the blow up at any point of ${\rm Y}_{00}\bigcap{\rm Y}_{11}$ is isomorphic to $\PP^1.$

    Let $\phi:\widetilde\cSh_{1,n-1}(K_{\gothp}^1)\ra\cSh_{1,n-1}$ blow up morphism. Since $\phi$ induces isomorphisms on $\cSh_{1,n-1}(K_\gothp^1)-{\rm Y}_{00}\bigcap{\rm Y}_{11}$, $\widetilde\Sh_{1,n-1}(K_\gothp^1)$ has four irreducible components.
    Suppose $\widetilde\Sh_{1,n-1}(K_\gothp^1)=\widetilde{\rm Y}_{00}\bigcup \tilde{\rm Y}_{10}\bigcup\widetilde{\rm Y}_{01}\bigcup\widetilde{\rm Y}_{11},$ 
    with ${\widetilde{\rm Y}_{ij}}$ the irreducible component contains the preimage of the generic point of ${\rm Y}_{ij}$ for $0\leq i,j\leq 1.$ Denote by $\phi_{ij}:\widetilde{\rm Y}_{ij}\ra{\rm Y}_{ij}$ the induced morphism by $\phi.$

    Let $x\in \widetilde{\rm Y}_{00}\bigcap\widetilde{\rm Y}_{11}\bigcap \phi^{-1}({\rm Y}_{10}),$ the stalk of $\phi^{-1}({\rm Y}_{10})$ at $x$ is 
    $\Spec k[\![X_1,Y_2,Z_1,\dots,Z_{2n-4},W]\!]/(X_1W, Y_2W),$ which is reducible. Thus $\widetilde{\rm Y}_{10}$ is not isomorphic to $\phi^{-1}({\rm Y}_{10})$ and after considering the generic point of $\widetilde{\rm Y}_{10},$ we can see the stalk of $\widetilde{\rm Y}_{10}$ at $x$ is isomorphic to $\Spec k[\![X_1,Y_2,Z_1,\dots,Z_{2n-4}]\!],$ which also shows that for any $x\in {\rm Y}_{00}\bigcap{\rm Y}_{11},$ $\widetilde{\rm Y}_{10}\bigcap \phi^{-1}(x)$ contains only a single point.

    Since $\phi_{10}:\widetilde{\rm Y}_{10}\ra {\rm Y}_{10}$ is bijective on points and induces isomorphism on local rings, then by Zariski main theorem, $\phi_{10}$ induces an ismorphism. The case for $\phi_{01}$ is also the same.

    Let $x\in \widetilde{\rm Y}_{00}\bigcap\widetilde{\rm Y}_{11}\bigcap \phi^{-1}({\rm Y}_{00}),$ the stalk of $\phi^{-1}({\rm Y}_{00})$ at $x$ is 
    $\Spec k[\![Y_1,Y_2,Z_1,\dots,Z_{2n-4},W]\!]/(Y_1T_1-Y_2T_2),$ which is irreducible. 
    Thus tha stalk of $\widetilde{\rm Y}_{00}$ at $x$ is $\Spec k[\![Y_1,Y_2,Z_1,\dots,Z_{2n-4},W]\!]/(Y_1T_1-Y_2T_2),$ which shows $\widetilde{\rm Y}_{00}=\phi^{-1}({\rm Y}_{00}).$
    
    Consider $\phi_{00}:\widetilde{\rm Y}_{00}\ra {\rm Y}_{00}.$ Since $\widetilde{\rm Y}_{00}$ is smooth, $\widetilde{\rm Y}_{00}\bigcap \widetilde{\rm Y}_{11}$ is an effective Cartier divisor of $\widetilde{\rm Y}_{00}.$ By the universal property of blow up, we get a morphism from $\widetilde{\rm Y}_{00}$ to the blow up ${\rm Y}_{00}$ at ${\rm Y}_{00}\bigcap{\rm Y}_{11}$ which is birational and isomorphic on local rings. Thus it is an isomorphism. The case for $\phi_{11}$ is the same. Hence we finish the proof.
\end{proof}
\begin{remark}
    The resolution of $\Spec W(k)[\![X_1,X_2,Y_1,Y_2,Z_1,\dots,Z_{2n-4}]\!]/(X_1Y_1-p,X_2Y_2-p)$ can be demonstrated by the following diagram which characterizes the intersection condition of the irreducible components of the special fiber:
    \[\begin{tikzcd}
    [
    column sep=small, 
    row sep=scriptsize,
    cells={font=\footnotesize} 
]
	{\mathbb{A}^{2n-2}} && {\mathbb{A}^{2n-2}} && {{\rm Bl}_{\mathbb{A}^{2n-4}}\mathbb{A}^{2n-2}} && {\mathbb{A}^{2n-2}} \\
	& {\mathbb{A}^{2n-4}} \\
	{\mathbb{A}^{2n-2}} && {\mathbb{A}^{2n-2}} && {\mathbb{A}^{2n-2}} && {{\rm Bl}_{\mathbb{A}^{2n-4}}\mathbb{A}^{2n-2}}
	\arrow["{\mathbb{A}^{2n-3}}", no head, from=1-1, to=1-3]
	\arrow["{\mathbb{A}^{2n-3}}"', no head, from=1-1, to=3-1]
	\arrow[""{name=0, anchor=center, inner sep=0}, "{\mathbb{A}^{2n-3}}"{description}, no head, from=1-3, to=3-3]
	\arrow["{\mathbb{A}^{2n-3}}", from=1-5, to=1-7]
	\arrow[""{name=1, anchor=center, inner sep=0}, "{\mathbb{A}^{2n-3}}"{description}, from=1-5, to=3-5]
	\arrow["{\mathbb{P}^1/\mathbb{A}^{2n-4}}"{description}, no head, from=1-5, to=3-7]
	\arrow["{\mathbb{A}^{2n-3}}", from=1-7, to=3-7]
	\arrow["{\mathbb{A}^{2n-3}}"', no head, from=3-1, to=3-3]
	\arrow["{\mathbb{A}^{2n-3}}"{description}, from=3-5, to=3-7]
	\arrow[shorten <=16pt, shorten >=16pt, Rightarrow, from=0, to=1]
\end{tikzcd}\]
\end{remark}

\begin{proposition}
\label{tangent sheaff}
    Let $(\mathcal{A}_1,\lambda_1,\eta_1,\mathcal{A}_2,\lambda_2,\eta_2,\phi)$ be the universal object of $\Sh_{1,n-1}(K_{\gothp}^{1}).$ Let $\psi:A_2\ra A_1$ be the isogeny such that $\psi\circ\phi=p$ and $\phi\circ\psi=p.$ Then the tangent space $T_{{\rm Y}_{ij}}$ of ${\rm Y}_{ij}$ is equal to $\mathcal{F}_{i}\oplus \mathcal{G}_{j}$ where $\mathcal{F}_{i}$ and $\mathcal{G}_{j}$ satisfy:
    \begin{enumerate}
        \item $ \mathcal{F}_{0}=\cHom\big( \omega^{\circ}_{\mathcal{A}^{\vee}_2/{{\Sh}}_{1,n-1},1},\Lie^{\circ}_{\mathcal{A}_2/{{\Sh}}_{1,n-1},1}\big) ,$
        \item $ 0\rightarrow \cHom\big( \omega^{\circ}_{\mathcal{A}_1^{\vee}/{{\Sh}}_{1,n-1},1},\frac{\phi^{-1}_{*,1}(\omega^{\circ}_{\mathcal{A}^{\vee}_1/{{\Sh}}_{1,n-1},1})}{\omega^{\circ}_{\mathcal{A}^{\vee}_1/{{\Sh}}_{1,n-1},1}}\big)  \rightarrow
        \mathcal{F}_{1} \rightarrow
        \cHom\big( \omega^{\circ}_{\mathcal{A}^{\vee}_2/{{\Sh}}_{1,n-1},1},\frac{\Im\phi_{*,1}}{\omega^{\circ}_{\mathcal{A}^{\vee}_2/{{\Sh}}_{1,n-1},1}}\big) \rightarrow 0$ is exact,
        \item $\mathcal{G}_{0}=\cHom\big( \omega^{\circ}_{\mathcal{A}^{\vee}_1/{{\Sh}}_{1,n-1},2},\Lie^{\circ}_{\mathcal{A}_1/{{\Sh}}_{1,n-1},2}\big) 
        ,$
        \item  $0\rightarrow \cHom\big( \frac{\omega^{\circ}_{\mathcal{A}^{\vee}_2/{{\Sh}}_{1,n-1},2}}{\phi_{*,2}(\omega^{\circ}_{\mathcal{A}_1^{\vee}/{{\Sh}}_{1,n-1},2})},\Lie^{\circ}_{\mathcal{A}_2/{{\Sh}}_{1,n-1},2}\big)
        \rightarrow \mathcal{G}_{1} \rightarrow
        \cHom\big( \frac{\omega^{\circ}_{\mathcal{A}_1^{\vee}/{{\Sh}}_{1,n-1},2}}{\Im({\psi_{*,2}})},\Lie^{\circ}_{\mathcal{A}_1/{{\Sh}}_{1,n-1},2}\big)\rightarrow 0$ is exact.
    \end{enumerate}
    The tangent sheaves $T_{{\rm Y}_{ij}}$ are all locally free of rank $2n-2.$ This also shows that ${\rm Y}_{ij}$ are all of dimension $2n-2.$
\end{proposition}
\begin{proof}
    We compute the tangent sheaf using deformation theory. Consider a closed immersion $S\hookrightarrow\hat S$ in $\Sh_{1,n-1}(K_{\gothp}^{1})$ defined by an ideal sheaf $\mathcal{I}$ satisfying $\mathcal{I}^{2}=0.$ 
    \begin{enumerate}
        \item Take a point $x=(A_1,\lambda_1,\eta_1,A_2,\lambda_2,\eta_2,\phi)\in {\rm Y}_{00}(S).$ Lifting $x$ to $\hat{S}$ is equivalent to 
        \begin{itemize}
            \item lifting $\omega_{A_2^{\vee}/S,1}^{\circ}$ to $\omega_{\hat{A}_{2}^{\vee}/S,1}$ such that  $\omega_{\hat{A}_{2}^{\vee}/S,1}$ is contained in ${\rm H}_1^{\rm cris}(A_2/\hat{S})^{\circ}_1;$
            \item lifting $\omega_{A_1^{\vee}/S,2}^{\circ}$ to $\omega_{\hat{A}_{1}^{\vee}/S,2}$ such that  $\omega_{\hat{A}_{1}^{\vee}/S,2}$ is contained in ${\rm H}_1^{\rm cris}(A_1/\hat{S})^{\circ}_2.$
        \end{itemize}
        Thus after passing to the universal object, $T_{\rm Y_{00}}=\mathcal{F}_{0}\oplus \mathcal{G}_{0}.$
        \item Take a point $x=(A_1,\lambda_1,\eta_1,A_2,\lambda_2,\eta_2,\phi)\in {\rm Y}_{01}(S).$ Lifting $x$ to $\hat{S}$ is equivalent to 
        \begin{itemize}
            \item lifting $\omega_{A_2^{\vee}/S,1}^{\circ}$ to $\omega_{\hat{A}_{2}^{\vee}/S,1}^{\circ}$ such that  $\omega_{\hat{A}_{2}^{\vee}/S,1}^{\circ}$ is contained in ${\rm H}_1^{\rm cris}(A_2/\hat{S})_1^{\circ};$
            \item lifting $\omega_{A_1^{\vee}/S,2}^{\circ}$ to $\omega_{\hat{A}_{1}^{\vee}/S,2}^{\circ}$ such that  $\omega_{\hat{A}_{1}^{\vee}/S,2}^{\circ}$ is contained in ${\rm H}_1^{\rm cris}(A_1/\hat{S})_{2}^{\circ}$ and contains $\psi_{*,2}({\rm H}_1^{\rm cris}(A_2/\hat{S})^{\circ}_2)$
            \item lifting $\omega_{A_2^{\vee}/S,2}^{\circ}$ to $\omega_{\hat{A}_{2}^{\vee}/S,2}^{\circ}$ such that  $\omega_{\hat{A}_{2}^{\vee}/S,2}^{\circ}$ is contained in ${\rm H}_1^{\rm cris}(A_2/\hat{S})_{2}^{\circ}$ and contains $\phi_{*,2}(\omega_{\hat{A}_{1}^{\vee}/S,2}^{\circ}).$
        \end{itemize}
        Thus after passing to the universal object, $T_{\rm Y_{01}}=\mathcal{F}_{0}\oplus \mathcal{G}_{1}.$
        \item Take a point $x=(A_1,\lambda_1,\eta_1,A_2,\lambda_2,\eta_2,\phi)\in {\rm Y}_{10}(S).$ Lifting $x$ to $\hat{S}$ is equivalent to 
        \begin{itemize}
            \item lifting $\omega_{A_2^{\vee}/S,1}^{\circ}$ to $\omega_{\hat{A}_{2}^{\vee}/S,1}$ such that  $\omega_{\hat{A}_{2}^{\vee}/S,1}$ is contained in ${\rm H}_1^{\rm cris}(A_1/\hat{S})^{\circ}_1;$
            \item lifting $\omega_{A_1^{\vee}/S,1}^{\circ}$ to $\omega_{\hat{A}_{1}^{\vee}/S,1}^{\circ}$ such that  $\omega_{\hat{A}_{1}^{\vee}/S,1}^{\circ}$ is contained in $\phi_{*,1}^{-1}(\omega_{\hat{A}_{2}^{\vee}/S,1});$
            \item lifting $\omega_{A_1^{\vee}/S,2}^{\circ}$ to $\omega_{\hat{A}_{1}^{\vee}/S,2}$ such that  $\omega_{\hat{A}_{1}^{\vee}/S,2}$ is contained in ${\rm H}_1^{\rm cris}(A_1/\hat{S})^{\circ}_2.$
        \end{itemize}
        Thus after passing to the universal object, $T_{\rm Y_{10}}=\mathcal{F}_{1}\oplus \mathcal{G}_{0}.$
        \item Take a point $x=(A_1,\lambda_1,\eta_1,A_2,\lambda_2,\eta_2,\phi)\in {\rm Y}_{11}(S).$ Lifting $x$ to $\hat{S}$ is equivalent to 
        \begin{itemize}
            \item lifting $\omega_{A_2^{\vee}/S,1}^{\circ}$ to $\omega_{\hat{A}_{2}^{\vee}/S,1}$ such that  $\omega_{\hat{A}_{2}^{\vee}/S,1}$ is contained in ${\rm H}_1^{\rm cris}(A_1/\hat{S})^{\circ}_1;$
            \item lifting $\omega_{A_1^{\vee}/S,1}^{\circ}$ to $\omega_{\hat{A}_{1}^{\vee}/S,1}^{\circ}$ such that  $\omega_{\hat{A}_{1}^{\vee}/S,1}^{\circ}$ is contained in $\phi_{*,1}^{-1}(\omega_{\hat{A}_{2}^{\vee}/S,1});$
             \item lifting $\omega_{A_1^{\vee}/S,2}^{\circ}$ to $\omega_{\hat{A}_{1}^{\vee}/S,2}^{\circ}$ such that  $\omega_{\hat{A}_{1}^{\vee}/S,2}^{\circ}$ is contained in ${\rm H}_1^{\rm cris}(A_1/\hat{S})_{2}^{\circ}$ and contains $\psi_{*,2}({\rm H}_1^{\rm cris}(A_2/\hat{S})^{\circ}_2)$
            \item lifting $\omega_{A_2^{\vee}/S,2}^{\circ}$ to $\omega_{\hat{A}_{2}^{\vee}/S,2}^{\circ}$ such that  $\omega_{\hat{A}_{2}^{\vee}/S,2}^{\circ}$ is contained in ${\rm H}_1^{\rm cris}(A_2/\hat{S})_{2}^{\circ}$ and contains $\phi_{*,2}(\omega_{\hat{A}_{1}^{\vee}/S,2}^{\circ}).$
        \end{itemize}
        Thus after passing to the universal object, $T_{\rm Y_{11}}=\mathcal{F}_{1}\oplus \mathcal{G}_{1}.$
    \end{enumerate}
    
\end{proof}

We introduce some new correspondences now.

\begin{definition}\label{G:irreducible components of supsingular locus of Sh_1.n-1}
    For $1\leq i \leq n-1,$ let $C_{i}$ be the moduli space over $\FF_{p^2}$ that associates to  each  locally Noetherian $\FF_{p^2}$-scheme $S,$ the set of isomorphism classes of  $(A,\lambda,\eta, A',\lambda',\eta',\phi, B,\lambda'',\eta'',\psi,\psi'),$ where
    \begin{itemize}
        \item $(A, \lambda , \eta , A', \lambda',\eta', \phi)\in {{\Sh}}_{1,n-1}(K_{\gothp}^{1}),$
        \item $(B,\lambda'',\eta'')\in {{\Sh}}_{0,n},$
        \item $\psi$ is an isogeny from $B$ to $A$ such that $(A,\lambda,\eta,B,\lambda'',\eta'',\psi)\in {\rm Y}_{i}$ and $\psi'$ is an isogeny from $B$ to $A$ such that $(A',\lambda',\eta',B,\lambda'',\eta'',\psi')\in {\rm Y}_{i+1}.$
        \item $\phi\circ\psi=\psi'$
    \end{itemize}
    
    We also let $C_{n}$ be the moduli space over $\FF_{p^2}$ that associates to  each  locally Noetherian $\FF_{p^2}$-scheme $S,$ the set of isomorphism classes of tuples  $(A,\lambda,\eta, A',\lambda',\eta',\phi, B_1,\lambda''_{1},\eta''_{1},B_2,\lambda''_{2},\eta''_{2},\psi,\psi'),$ where
    \begin{itemize}
        \item $(A, \lambda , \eta , A', \lambda',\eta', \phi)\in {{\Sh}}_{1,n-1}(K_{\gothp}^{1}),$
        \item $(B_1,\lambda''_{1},\eta''_{1}),(B_2,\lambda''_{2},\eta''_{2})\in {{\Sh}}_{0,n},$
        \item $\psi$ is an isogeny from $B_1$ to $A$ such that $(A,\lambda,\eta,B_1,\lambda''_1,\eta''_1,\psi)\in {\rm Y}_{n}$ and $\psi'$ is an isogeny from $B_2$ to $A'$ such that $(A',\lambda',\eta',B_2,\lambda''_2,\eta''_2,\psi')\in {\rm Y}_{1}.$
        \item $B_{1}\in {\rm S}_\gothp(B_{2}).$ Here ${\rm S}_\gothp$ is the Hecke action defined in Definition~\ref{S:Hecke action} and denote by $\phi_{12}:B_1\ra B_2$ the corresponding $p$-quasi-isogeny.
        \item $\phi\circ\psi=\psi'\circ\phi_{12}.$
    \end{itemize}
\end{definition}

For $1\leq i\leq n,$ $C_{i}$ is representable by a smooth and projective scheme over ${\Sh}_{0,n}.$ For $1\leq i\leq n-1,$  let $\pr'_i$ denote the projection from $C_{i}$ to ${{\Sh}}_{0,n}$ mapping $(A,\lambda,\eta, A',\lambda',\eta',\phi, B,\lambda'',\eta'',\psi,\psi')$ to $(B,\lambda'',\eta'')$ and  $\pr_{i}$ denote the projection from $C_{i}$ to ${{\Sh}}_{1,n-1}(K_{\gothp}^{1})$ with image to be $(A,\lambda,\eta, A',\lambda',\eta',\phi).$
Let $\pr'_n$ denote the projection from $C_n$ to ${\Sh}_{0,n}$ mapping $(A,\lambda,\eta, A',\lambda',\eta',\phi, B_1,\lambda''_{1},\eta''_{1},B_2,\lambda''_{2},\eta''_{2},\psi,\psi')$ to $(B_2,\lambda''_{2},\eta''_{2},\psi,\psi')$ and $\pr_n$ denote the projection from $C_n$ to ${\Sh}_{1,n-1}$ with image $(A,\lambda,\eta, A',\lambda',\eta',\phi).$

The moduli problem for $C_{i}$ for $1\leq i\leq n$ is slightly complicated. We will introduce a more explict moduli space and then show they are isomorphic.

For $1\leq i\leq n-1,$ consider the functor $\underline{C}'_i$ which associates to each locally Noetherian $\F_{p^2}$-scheme $S,$ the set of isomorphism classes of tuples $(B, \lambda'', \eta'', H_1,H_2,H'_1,H'_2),$ where
\begin{itemize}
\item $(B,\lambda'',\eta'')$ is an $S$-valued point of ${\Sh}_{0,n}$;

\item $H_1,H'_1\subset {\rm H}^{\dR}_1(B/S)^{\circ}_{1}$ are $\calO_S$-subbundles of rank $i$ and $i+1$ respectively and  $H_2,H'_2\subset {\rm H}^{\dR}_1(B/S)_{2}^{\circ}$ are  $\calO_S$-subbundles of rank $i-1$ and $i$ respectively. They satisfy:

\begin{enumerate}
    \item $V^{-1}(H_2^{(p)})\subseteq V^{-1}(H_2'^{(p)})\bigcap H_{1},V^{-1}(H_2'^{(p)})\bigcup H_{1}\subseteq H'_{1},$
    \item $H_2\subseteq H'_2\bigcap F(H_{1}^{(p)}),H_2'\bigcup F(H_{1}^{(p)})\subseteq F(H_{1}'^{(p)}).$
\end{enumerate}
Here, $F: {\rm H}^{\dR}_1(B/S)^{\circ,(p)}_1\xra{\sim} {\rm H}^{\dR}_1(B/S)_{2}^{\circ}$ and $V:{\rm H}^{\dR}_1(B/S)^{\circ}_{1}\xra{\sim} {\rm H}^{\dR}_1(B/S)_{2}^{\circ, (p)} $ are the Frobenius and Verschiebung homomorphisms respectively, which are actually isomorphisms because of the signature condition on ${\Sh}_{0,n}.$
\end{itemize}

There is a natural projection  $\pi'_i\colon \underline{C}'_i\ra {\Sh}_{0,n}$ given by $(B,\lambda',\eta',H_1,H_2,H'_1,H'_2)\mapsto (B,\lambda',\eta').$

\begin{proposition}\label{P:Yj-prime}
For $1\leq i \leq n-1,$ the functor $\underline{C}'_i$ is representable by a scheme $C'_i$ smooth and projective over ${\Sh}_{0,n}$ of  dimension $n.$
 Moreover, if $(\calB,\lambda'',\eta'',\cH_1,\cH_2,\cH'_1,\cH'_2)$ denotes the universal object over $C'_i,$ then  the tangent bundle of $C'_i$ is
$
T_{Y'_{j},y_0}\simeq \mathcal{F}\oplus\mathcal{G},
$ where $\mathcal{F}, \mathcal{G}$ satisfies:
\begin{itemize}
    \item $0\ra \calH om\big(\calH_1/V^{-1}(\calH_2^{(p)}),\calH'_{1}/\calH_1\big)\ra \mathcal{F}$
    $\ra \calH om\big(\calH'_1/V^{-1}(\calH_2'^{(p)}),\calH^{\dR}_1(\calB/{{\Sh}}_{0,n})^{\circ}_{1}/\calH'_1\big) 
    \ra 0$ is exact.
    \item $0\ra \calH om\big(\calH'_2/\calH_{2}, F(\calH^{',(p)}_1)/\calH'_2\big)
\ra \mathcal{G}\ra \calH om\big(\calH_2, F(\calH^{(p)}_1)/\calH_2\big)\ra 0$ is exact.
\end{itemize}
\end{proposition}
\begin{proof}
The proof is exactly the same as \cite[Proposition 4.4]{HTX17} for $\underline{Y}_j'.$
\end{proof}

To construct a morphism from $C_{i}$ to $C'_{i}$ for $1\leq i\leq n-1,$ we need the following lemma:
\begin{lemma}\label{L:omega-F}
Let $(A,\lambda, \eta, B,\lambda',\eta',\phi)$ be an $S$-point of $Y_{j}.$
 Then the image of $\phi_{*,1}$ contains both $\omega^{\circ}_{A^\vee/S,1}$ and $F\big({\rm H}^{\dR}_1(A/S)^{\circ,(p)}_{2}\big),$ and the image of $\phi_{*,2}$ is contained in $\omega^{\circ}_{A^\vee/S,2}$ and $F\big({\rm H}^{\dR}_1(A/S)^{\circ,(p)}_{1}\big).$
\end{lemma}
\begin{proof}
    See \cite[Lemma 4.6]{HTX17}.
\end{proof}

There is a natural  morphism $\alpha: C_i\ra C'_i$ for $1\leq j\leq n-1$ defined as follows. For a locally Noetherian $\FF_{p^2}$-scheme $S$ and an $S$-point $(A,\lambda,\eta, A',\lambda',\eta',\phi, B,\lambda'',\eta'',\psi,\psi')$ of $C_i,$  we define
\begin{equation*}
\label{EE:definition of H12}
H_1:=\psi_{*,1}^{-1}(\omega^{\circ}_{A^\vee/S,1})\subseteq {\rm H}^{\dR}_1(B/S)^{\circ}_{1},
\quad \text{and}\quad
 H_2:=p\psi_{*,2}^{'-1}(\omega^{\circ}_{A^\vee/S,2})\subseteq {\rm H}^{\dR}_{1}(B/S)^{\circ}_{2}.
\end{equation*}
\begin{equation*}
    \label{EE:definition of H'12}
H'_1:=\psi_{*,1}^{'-1}(\omega^{\circ}_{A'^{\vee}/S,1})\subseteq {\rm H}^{\dR}_1(B/S)^{\circ}_{1},
\quad \text{and}\quad
 H'_2:=p\psi_{*,2}^{'-1}(\omega^{\circ}_{A'^\vee/S,2})\subseteq {\rm H}^{\dR}_{1}(B/S)^{\circ}_{2}.
\end{equation*}
In particular, $H_1,H'_1,H_{2}$ and  $H'_2$ are $\calO_S$-subbundles of rank $i,i+1,i-1$ and $i,$ respectively.
By \ref{L:omega-F}, we can easily see $\alpha$ is well-defined.

Therefore, we have a well-defined map $\alpha\colon C_i\ra C'_i$ given by
 \[
 \alpha\colon (A,\lambda,\eta, A',\lambda',\eta',\phi, B,\lambda'',\eta'',\psi,\psi')\longmapsto (B,\lambda'',\eta'',H_1,H_2,H'_1,H'_2).
 \]
Moreover, it is clear from the definition that $\pi'_i \circ \alpha = \pr'_i.$

 \begin{proposition}\label{P:isom-Yj}
The morphism $\alpha$ is an isomorphism.
\end{proposition}
\begin{proof}
The proof is similar to \cite[Proposition 4.6]{HTX17} and we omit here.
\end{proof}
In the sequel, we will always identify $C_i$ with $C'_i$ and $\pr'_i$ with $\pi'_i$ for $1\leq i\leq n-1.$
Recall $pr_n$ as a morphism from $C_n$ to $\Sh_{1,n-1},$ Since $(A,\lambda,\eta,B_1,\lambda''_1,\eta''_1,\psi)\in Y_{n}$ and $(A',\lambda',\eta',B_2,\lambda''_2,\eta''_2,\psi')\in Y_{1},$ we get $\psi_{*,1}({\rm H}_{1}^{\rm dR}(B_{1}/S)_{1}^{\circ})=\omega_{A^{\vee},1}^{\circ},$ $\psi_{*,2}({\rm H}_{1}^{\rm dR}(B_{1}/S)_{2}^{\circ})=0,$ and $\psi'_{*,1}({\rm H}_{1}^{\rm dR}(B_{2}/S)_{1}^{\circ})={\rm H}_{1}^{\rm dR}(A'/S)_{1}^{\circ}$ and $\psi'_{*,2}({\rm H}_{1}^{\rm dR}(B_{2}/S)_{2}^{\circ})=\omega_{A'^{\vee},2}^{\circ}.$ Since $B_{1}\in {\rm S}_\gothp(B_{2}),$ we get $\omega^{\circ}_{A^{\vee},1}=\Ker(\phi_{*,1})$ and $\Im(\phi_{*,2})=\omega^{\circ}_{A^{'\vee},2}.$ Thus the image of $C_{n}$ is contained in ${\rm Y}_{00}.$

\begin{proposition}
    The morphism $\pr_{n}:C_{n}\ra {\rm Y}_{00}$ is an isomorphism.
\end{proposition}
\begin{proof}
    Let $k$ be a perfect field containing $\F_{p^2}.$ We first prove that $\pr_{n}$ induces a bijection of points $\pr_{n}: C_n(k)\xra{\sim} {\rm Y}_{00}(k).$
It suffices to show that there exists  a morphism of sets $\beta: {\rm Y}_{00}(k)\ra C_n(k)$ inverse to $\pr_n.$
Let $y=(A,\lambda,\eta, A',\lambda',\eta',\phi)\in {\rm Y}_{00}(k).$
Define $\beta(y)=(A,\lambda,\eta, A',\lambda',\eta',\phi, B_1,\lambda_1'',\eta_1'', B_2,\lambda_2'',\eta_2'',\psi,\psi')$ as follows.
Let $\tcE_1=V\tcD(A)_{2},\tcE'_1=p\tcD(A)_{2}$ and $\tcE_2=\tcD(A')_{2},\tcE'_2=V\tcD(A')_{2}.$
Applying Proposition~\ref{P:abelian-Dieud} with $m=1,$ we get two points $(B_1,\lambda_1'',\eta_1'')$ and $(B_2,\lambda_2'',\eta_2'')$ of ${\Sh}_{0,n}$ such that there exist isogenies $\psi: B_1\ra A,\psi': B_2\ra A'.$ It can be checked easily that $B_1={\rm S}_\gothp(B_2).$
This finishes the construction of $\beta(y).$
 It is direct to check that $\beta$ is the set theoretic inverse to $\pr_{n}: C_n(k)\ra {\rm Y}_{00}(k).$

By a simple argument on Serre--Tate and Grothendieck--Messing deformation theory, we can show that $\pr_{n}$ induces an isomorphism on the tangent spaces at each closed point; it follows that $\pr_{n}$ is an isomorphism.
\end{proof}

As a corollary, we can write the correspondence between ${{\Sh}}_{0,n}$ and ${{\Sh}}_{1,n-1}(K_{\gothp}^{1})$ determined by $C_{i}$ for $1\leq i\leq n$ in the following diagrams:
\begin{corollary}
\label{Corresss}
    \begin{enumerate}
    \item For $1\leq i\leq n-1,$ we have the following diagram:
    \[\begin{tikzcd}
	{{\rm Y}_{11}} && {C_i} && {{\rm Sh}_{0,n}} \\
	\\
	{{\rm Sh_{1,n-1}}\times{\rm Sh_{1,n-1}}} && {{\rm Y}_{i}\times {\rm Y}_{i+1}} && {{\rm Sh}_{0,n}\times {\rm Sh}_{0,n}}
	\arrow[from=1-1, to=3-1]
	\arrow["{\pr_{i}}"', from=1-3, to=1-1]
	\arrow["{\pr'_{i}}", from=1-3, to=1-5]
	\arrow[from=1-3, to=3-3]
	\arrow["{{(1,1)}}", from=1-5, to=3-5]
	\arrow["{{{\mathop{p}\limits^{\leftarrow}}_{i}\times {\mathop{p}\limits^{\leftarrow}}_{i+1}}}"', from=3-3, to=3-1]
	\arrow["{{{\mathop{p}\limits^{\rightarrow}}_{i}\times {\mathop{p}\limits^{\rightarrow}}_{i+1}}}", from=3-3, to=3-5]
    \end{tikzcd}\] where the first two vertical arrows are induced by natrual projections.
    \item For $C_{n},$ we have the following diagram:
    \[\begin{tikzcd}
	{{\rm Y}_{00}} && {C_n} && {{\rm Sh}_{0,n}} \\
	\\
	{{\rm Sh}_{1,n}\times{\rm Sh}_{1,n}} && {{\rm Y}_{n}\times {\rm Y}_{1}} && {{\rm Sh}_{0,n}\times {\rm Sh}_{0,n}}
	\arrow[from=1-1, to=3-1]
	\arrow["{\pr_{n},\simeq}"', from=1-3, to=1-1]
	\arrow["{\pr'_{n},{\mathbb{P}^{n-1}\times \mathbb{P}^{n-1} bundle}}", from=1-3, to=1-5]
	\arrow[from=1-3, to=3-3]
	\arrow["{{(1,{\rm S}_\gothp)}}", from=1-5, to=3-5]
	\arrow["{{{\mathop{p}\limits^{\leftarrow}}_{n}\times {\mathop{p}\limits^{\leftarrow}}_{1}}}"', from=3-3, to=3-1]
	\arrow["{{{\mathop{p}\limits^{\rightarrow}}_{n}\times {\mathop{p}\limits^{\rightarrow}}_{1}}}", from=3-3, to=3-5]
\end{tikzcd}\] where the first two vertical arrows are induced by natrual projections.
    \end{enumerate}
\end{corollary}
\begin{proof}
    What remains to show is that $\pr_{n}'$ makes $C_{n}$ be a $\mathbb{P}^{n-1}\times \mathbb{P}^{n-1}$ bundle over ${{\Sh}}_{0,n}.$ To see this, we construct a morphism $\alpha$ from ${\rm Y}_{00}$ to $\Gr\big({\rm H}^{\dR}_1(\calB/{\Sh}_{0,n})^{\circ}_{1}, 1\big)\times \Gr\big({\rm H}^{\dR}_1(\calB/{\Sh}_{0,n})^{\circ}_{2},n-1\big)$ with $(\calB,\lambda'',\eta'')$ the universal object of ${{\Sh}}_{0,n}.$ More explicitly, for any $S$-point $y=(A,\lambda,\eta, A',\lambda',\eta',\phi)\in {\rm Y}_{00},$ $\alpha(y)$ is defined to be the $S$-point $(\omega_{A'^\vee/S,1}^\circ,\omega_{A^\vee/S,2}^\circ).$

    By a similar argument as we done for $C_i, 1\leq i\leq n-1$ in Proposition~\ref{P:isom-Yj}, we can show $\alpha$ is an isomorphism. Hence we finish the proof.
\end{proof}

The correspondences $C_1,\dots,C_n$ can help us to analyze structures of ${\rm Y}_{00}, {\rm Y}_{11}.$ We are now going to analyze the strucures of ${\rm Y}_{01}, {\rm Y}_{10}.$ First, we need the following lemma:
\begin{lemma}
\begin{enumerate}
    \item The $S$-point $(A,\lambda,\eta)\in {{\Sh}}_{1,n-1}$ is in ${\mathop{p}\limits^{\leftarrow}}_{n}({\rm Y}_{n})$ if and only if $F\tcD(A)_{2}^{\circ}=V\tcD(A)_{2}^{\circ}.$
    \item The $S$-point $(A,\lambda,\eta)\in {{\Sh}}_{1,n-1}$ is in ${\mathop{p}\limits^{\leftarrow}}_{1}({\rm Y}_{1})$ if and only if $F\tcD(A)_{1}^{\circ}=V\tcD(A)_{1}^{\circ}.$
\end{enumerate}
\end{lemma}
\begin{proof}
    The proof of $(1)$ and $(2)$ are similar. For simplicity, we only give the proof of $(1).$

    For the `if' part, we take $\tcE_1=V\tcD(A)_{2}$ and $\tcE_{2}=p\tcD(A)_{2}.$ Applying Proposition~\ref{P:abelian-Dieud} with $m=1,$ we get a point $(B,\lambda'',\eta'')$ of ${\Sh}_{0,n}$ and an isogeny $\phi: B\ra A$ such that $(A,\lambda, \eta,B,\lambda'', \eta'',\phi)$ is a point of ${\rm Y}_{n}.$
    

    Conversely, if$(A,\lambda, \eta,B,\lambda'', \eta'',\phi)$ is an $S$-point of ${\rm Y}_{n},$ then by the defintion of ${\rm Y}_{n},$ we see that $\phi_{*,2}\tcD(B)_{2}^{\circ}=p\tcD(A)_{2}^{\circ}$ and $\phi_{*,1}\tcD(B)_{1}^{\circ}$ has corank $n-1$ contained in $\tcD(A)_{1}^{\circ}.$ By Lemma~\ref{L:omega-F}, $V\tcD(A)_{2}^{\circ}\subseteq \phi_{*,1}\tcD(B)_{1}^{\circ}.$ It forces that $V\tcD(A)_{2}^{\circ}\subseteq \phi_{*,1}\tcD(B)_{1}^{\circ}.$ By $F\tcD(B)_{2}^{\circ}=V\tcD(B)_{2}^{\circ}=p\tcD(B)_{1}^{\circ},$ $F\tcD(A)_{2}^{\circ}=V\tcD(A)_{2}^{\circ}.$ Thus we finish the proof.
\end{proof}

\begin{proposition}
\label{bloww up}
    \begin{enumerate}
        \item There is a natural morphism $\pi_{10}:{\rm Y}_{10}\ra {{\Sh}}_{1,n-1}$ mapping $(A,\lambda,\eta, A',\lambda',\eta',\phi)$ to $(A,\lambda,\eta).$ The closed subscheme ${\rm Y}_{10}$ of ${{\Sh}}_{1,n-1}(K_{\gothp}^{1})$ is isomorphic to ${\rm Bl}_{{\mathop{p}\limits^{\leftarrow}}_{n}({\rm Y}_{n})}{{\Sh}}_{1,n-1}$ with $\pi_{10}$ to be exacly the blowing-up map. The exceptional divisor is ${\rm Y}_{00}\bigcap {\rm Y}_{10}.$
        \item There is a natural morphism $\pi_{01}:{\rm Y}_{01}\ra {{\Sh}}_{1,n-1}$ by mapping any point $(A,\lambda,\eta, A',\lambda',\eta',\phi)$ to $(A',\lambda',\eta').$ The closed subscheme ${\rm Y}_{01}$ of ${{\Sh}}_{1,n-1}(K_{\gothp}^{1})$ is isomorphic to ${\rm Bl}_{{\mathop{p}\limits^{\leftarrow}}_{1}({\rm Y}_{1})}{{\Sh}}_{1,n-1}$ with $\pi_{01}$ to be exacly the blowing-up map. The exceptional divisor is ${\rm Y}_{00}\bigcap {\rm Y}_{01}.$
    \end{enumerate}
\end{proposition}
\begin{proof}
    The proof of $(1)$ and $(2)$ are similar. For simplicity, we only give the proof of $(1).$
    
    First, we introduce a new scheme $C_{10}$ and show it is isomorphic both to ${\rm Bl}_{{\mathop{p}\limits^{\leftarrow}}_{n}({\rm Y}_{n})}{{\Sh}}_{1,n-1}$ and to ${\rm Y}_{10}.$

    Let $\underline{C}_{10}$ be the moduli space over $\FF_{p^{2}}$ that associates to each locally Noetherian $\FF_{p^{2}}$-scheme $S,$ the set of isomorphism classes of tuples $(A,\lambda,\eta,H),$ where
    \begin{itemize}
        \item $(A,\lambda,\eta)\in {{\Sh}}_{1,n-1},$
        \item $H$ is a subbundle contained in $V^{-1}\big({\rm H}_{1}^{\rm dR}(A/S)_{2}^{\circ,(p)}\big)$ of rank $2.$
        \item $H$ satisfies $F\big({\rm H}_{1}^{\rm dR}(A/S)_{2}^{\circ,(p)}\big)\bigcup \omega_{A^\vee/S,2}^\circ\subseteq H.$
    \end{itemize}

    With a similar argument as we done for $\underline{C}_{i}, 1\leq i\leq n-1,$ we can show $\underline{C}_{10}$ is represented by a smooth, projective scheme over ${{\Sh}}_{1,n-1}$ of dimension $2(n-1).$ We denote it by $C_{10}.$ There is natural morphism $\alpha$ from ${\rm Y}_{10}$ to $C_{10}$ by mapping any point $(A,\lambda,\eta, A',\lambda',\eta',\phi)$ to $(A,\lambda,\eta, \omega_{A'^\vee,1}^{\circ}).$ It can be checked $\alpha$ is well defined and is an isomorphism using a method similar to Proposition~\ref{P:isom-Yj}.

    Given any $S$-point $(A,\lambda,\eta,H)\in C_{10},$ we can see by dimension counting that the subbundle $H=F\big({\rm H}_{1}^{\rm dR}(A/S)_{2}^{\circ,(p)}\big)\bigcup \omega_{A^\vee/S,2}$ if $F\big({\rm H}_{1}^{\rm dR}(A/S)_{2}^{\circ,(p)}\big)\neq \omega_{A^\vee/S,2}.$ The morphism $\beta:C_{10}\rightarrow {{\Sh}}_{1,n-1}$ mapping $(A,\lambda,\eta,H)$ to $(A,\lambda,\eta)$ inducing the isomorhism $C_{10}\simeq {\rm Bl}_{{\mathop{p}\limits^{\leftarrow}}_{n}({\rm Y}_{n})}{{\Sh}}_{1,n-1}.$ Under the isomorphism, the exceptional divisor consists of points $(A,\lambda,\eta, A',\lambda',\eta',\phi)\in {\rm Y}_{10}$ such that $(A,\lambda,\eta)\in {\mathop{p}\limits^{\leftarrow}}_{n}({\rm Y}_{n}),$ which is exactly ${\rm Y}_{00}\bigcap {\rm Y}_{10}.$ Thus we finish the proof. 
\end{proof}

Now we begin the computation of cohomology groups. Let $L$ be a $p$-coprime coefficient ring.
\begin{proposition} We have 
\label{cohomologyy}
\begin{enumerate}
    \item ${\rm H}^{*}_{\acute{e}t}({\rm Y}_{00}\bigcap {\rm Y}_{10},L)={\rm H}^{*}_{\acute{e}t}({\rm Y}_{00}\bigcap {\rm Y}_{01},L)={\rm H}^{*}_{\acute{e}t}(\mathbb{P}^{n-1}/{{\Sh}}_{0,n},L)\otimes {\rm H}^{*}_{\acute{e}t}(\mathbb{P}^{n-2}/\Sh_{0,n},L);$
    \item ${\rm H}^{2n-4}_{\acute{e}t}({\rm Y}_{00}\bigcap {\rm Y}_{11},L)={\rm H}^{0}({{\Sh}}_{0,n},L(2-n))^{\oplus n-2}\oplus {\rm H}^{0}_{\acute{e}t}({{\Sh}}_{0,n}(K_{\gothp}^{1}),L(2-n));$
    \item ${\rm H}^{2n-2}_{\acute{e}t}({\rm Y}_{00}\bigcap {\rm Y}_{11},L)={\rm H}^{0}({{\Sh}}_{0,n},L(1-n))^{\oplus n-2}$ 
    \item ${\rm H}^{i}_{\acute{e}t}({\rm Y}_{00}\bigcap {\rm Y}_{11},L)=0,$ for $i$ odd;
    \item There exists an injective map $\phi:{H^{0}_{\acute{e}t}(\Sh_{0,n},L)\ra H^0_{\acute{e}t}(\Sh_{0,n}(K_\gothp^1),L)}$ induced by the restriction $$H^{2n-4}_{\acute{e}t}({\rm Y}_{00}\bigcap{\rm Y}_{10})\ra H^{2n-4}_{\acute{e}t}({\rm Y}_{00}\bigcap{\rm Y}_{11})\footnote{Even this is not enough to decide what $\phi$ is, but it is all we need in the proof later.}$$
    such that
    \begin{itemize}
        \item ${\rm H}^{2n-3}_{c}({\rm Y}_{00}\bigcap {\rm Y}_{01}-{\rm Y}_{00}\bigcap {\rm Y}_{11},L)={\rm H}^{2n-3}_{c}({\rm Y}_{00}\bigcap {\rm Y}_{10}-{\rm Y}_{00}\bigcap {\rm Y}_{11},L)=\coker\phi;$
        \item ${\rm H}^{2n-4}_{c}({\rm Y}_{00}\bigcap {\rm Y}_{01}-{\rm Y}_{00}\bigcap {\rm Y}_{11},L)=0;$
        \item ${\rm H}^{2n-2}_{c}({\rm Y}_{00}\bigcap {\rm Y}_{01}-{\rm Y}_{00}\bigcap {\rm Y}_{11},L)={\rm H}^{0}_{\acute{e}t}({{\Sh}}_{0,n},L(1-n)).$
    \end{itemize}
    \item If $L$ is a finite extension of $\FF_\ell$ and $\ell \nmid \frac{p^{2n-2}-1}{p^2-1},$ then $${\rm H}^{2n-3}_{c}({\rm Y}_{00}\bigcap {\rm Y}_{01}-{\rm Y}_{00}\bigcap {\rm Y}_{11},L)={\rm H}^{2n-3}_{c}({\rm Y}_{00}\bigcap {\rm Y}_{10}-{\rm Y}_{00}\bigcap {\rm Y}_{11},L)={\rm H}^{2n-2}_{\acute{e}t}(\Sh_{1,n-1},\rho_{n-1,1}).$$
    Here $\rho_{n-1,1}$ is the nontrivial factor in the parabolic induction ${\rm Ind}_{K_\gothp^1}^{K_\gothp}\mathbbm{1}.$
\end{enumerate}
\end{proposition}
\begin{proof}
    Let $S$ be a scheme defined over $\FF_{p^{2}}.$
    Under the isomorphism ${\rm Y}_{00}\simeq \Gr\big({\rm H}^{\dR}_1(\calB/{\Sh}_{0,n})^{\circ}_{1}, n-1\big)\times \Gr\big({\rm H}^{\dR}_1(\calB'/{\Sh}_{0,n})^{\circ}_{2},1\big)$ as in Proposition~\ref{Corresss}, we have
    \begin{enumerate}
    \item  ${\rm Y}_{00}\bigcap {\rm Y}_{11}$ is isomorphic to a closed subscheme of $\Gr\big({\rm H}^{\dR}_1(\calB/{\Sh}_{0,n})^{\circ}_{1}, 1\big)\times \Gr\big({\rm H}^{\dR}_1(\calB/{\Sh}_{0,n})^{\circ}_{2},n-1\big)$ satisfying that for any $S$-point $(H_1,H_2),$ $H_1\subseteq V^{-1}(H_2^{(p)})\subseteq {\rm H}^{\dR}_1(\calB'/{\Sh}_{0,n})^{\circ}_{1}, \quad F(H_1^{(p)})\subseteq H_2\subseteq {\rm H}^{\dR}_1(\calB/{\Sh}_{0,n})^{\circ}_{2}.$
    \item ${\rm Y}_{00}\bigcap {\rm Y}_{01}$ is isomorphic to a closed subscheme of $\Gr\big({\rm H}^{\dR}_1(\calB/{\Sh}_{0,n})^{\circ}_{1}, 1\big)\times \Gr\big({\rm H}^{\dR}_1(\calB/{\Sh}_{0,n})^{\circ}_{2},n-1\big)$ satisfying that for any $S$-point $(H_1,H_2),$ $F(H_1^{(p)})\subseteq H_2\subseteq {\rm H}^{\dR}_1(\calB/{\Sh}_{0,n})^{\circ}_{2}.$
    \item ${\rm Y}_{00}\bigcap {\rm Y}_{10}$ is isomorphic to a closed subscheme of $\Gr\big({\rm H}^{\dR}_1(\calB/{\Sh}_{0,n})^{\circ}_{1}, 1\big)\times \Gr\big({\rm H}^{\dR}_1(\calB/{\Sh}_{0,n})^{\circ}_{2},n-1\big)$ satisfying that for any $S$-point $(H_1,H_2),$ $H_1\subseteq V^{-1}(H_2^{(p)})\subseteq {\rm H}^{\dR}_1(\calB/{\Sh}_{0,n})^{\circ}_{1}.$
\end{enumerate}

Let $(\mathcal{H}_1,\mathcal{H}_2)$ be the universal bundles in each closed subscheme and denote the first chern classes $c_1(\mathcal{H}_1)$ and $c_1(\mathcal{H}_2)$ by $\eta_1,\eta_2$ respectively.
Define $\phi$ to be the morphism mapping any $S$-point $(H_1,H_2)$ to $(H_1,V^{-1}(H_2^{(p)})).$ Then ${\rm Y}_{00}\bigcap {\rm Y}_{10}$ is mapped to the closed subscheme $\tilde Y$ of $\Gr\big({\rm H}^{\dR}_1(\calB/{\Sh}_{0,n})^{\circ}_{1}, 1\big)\times \Gr\big({\rm H}^{\dR}_1(\calB/{\Sh}_{0,n})^{\circ}_{1},n-1\big)$ consisting of line bundles $(H_1,H_2)$ satisfying $H_1\subseteq H_2\subseteq {\rm H}^{\dR}_1(\calB/{\Sh}_{0,n})^{\circ}_{1}.$ It is a $\mathbb{P}^{n-2}$-bundle over $\Gr\big({\rm H}^{\dR}_1(\calB/{\Sh}_{0,n})^{\circ}_{1},1\big)=\mathbb{P}^{n-1}/{{\Sh}}_{0,n},$ which consists of points $H_1$ such that $H_1\subseteq {\rm H}^{\dR}_1(\calB/{\Sh}_{0,n})^{\circ}_{1}.$ We denote it by $X\simeq \mathbb{P}^{n-1}/\Sh_{0,n}.$ Since $\phi$ induces an isomorphism on cohomology groups, by Kunn\'eth formula, we have ${\rm H}^{*}_{\acute{e}t}({\rm Y}_{00}\bigcap {\rm Y}_{10},L)={\rm H}_{\acute{e}t}^{*}(\mathbb{P}^{n-1}/{\Sh_{0,n}},L)\otimes {\rm H}^{*}_{\acute{e}t}(\mathbb{P}_{n-2}/{\Sh_{0,n}},L)$ to be a graded ring generated by $\eta_1$ and $\eta_2$ over $H_{\acute{e}t}^{0}(\Sh_{0,n},L).$ The same result holds for ${\rm H}^{*}_{\acute{e}t}({\rm Y}_{00}\bigcap {\rm Y}_{01},L).$

Under the morphism $\phi$, ${\rm Y}_{00}\bigcap {\rm Y}_{11}$ is mapped to a closed subscheme $W$ of $\Gr\big({\rm H}^{\dR}_1(\calB/{\Sh}_{0,n})^{\circ}_{1}, 1\big)\times \Gr\big({\rm H}^{\dR}_1(\calB/{\Sh}_{0,n})^{\circ}_{1},n-1\big)$ consisting of points $(H_1,H_2)$ satisfying \[H_1\subseteq H_2\subseteq {\rm H}^{\dR}_1(\calB/{\Sh}_{0,n})^{\circ}_{1},\quad F(V^{-1}(H_1^{(p)})^{(p)})\subseteq H_2\subseteq {\rm H}^{\dR}_1(\calB/{\Sh}_{0,n})^{\circ}_{1}.\] Blowing up at points satisfying $H_1=F(V^{-1}(L^{(p)})^{(p)})$ and denoting the locus with $T,$ we get a closed subscheme $Z$ of $\Gr\big({\rm H}^{\dR}_1(\calB/{\Sh}_{0,n})^{\circ}_{1}, 1\big)\times \Gr\big({\rm H}^{\dR}_1(\calB/{\Sh}_{0,n})^{\circ}_{1},2\big)\times
\Gr\big({\rm H}^{\dR}_1(\calB/{\Sh}_{0,n})^{\circ}_{2},n-1\big)
$ with points $(H_1,H_1',H_2)$ such that $H_1\cup F(V^{-1}(H_1^{(p)})^{(p)}) \subseteq H_1'\subseteq H_2\subseteq {\rm H}^{\dR}_1(\calB/{\Sh}_{0,n})^{\circ}_{2}.$ We denote the exceptional divisor by $E,$ consisting of points $(H_1,H_1',H_2)$ such that $F(V^{-1}(H_1^{(p)})^{(p)})=H_1.$
Denote the universal object by $(\mathcal{H}_1,\mathcal{H}_1',\mathcal{H}_2).$
$Z$ is a $\mathbb{P}^{n-3}$-bundle over the closed subscheme $Y$ of $\Gr\big({\rm H}^{\dR}_1(\calB/{\Sh}_{0,n})^{\circ}_{1}, 1\big)\times \Gr\big({\rm H}^{\dR}_1(\calB/{\Sh}_{0,n})^{\circ}_{1},2\big)$ consisting of points $(H_1,H_1')$ such that $H_1\cup F(V^{-1}(H_1^{(p)})^{(p)}) \subseteq H_1'\subseteq  {\rm H}^{\dR}_1(\calB/{\Sh}_{0,n})^{\circ}_{2},$ which is the blowing-up of $\Gr\big({\rm H}^{\dR}_1(\calB/{\Sh}_{0,n})^{\circ}_{1},1\big)=\mathbb{P}^{n-1}/{{\Sh}}_{0,n}.$ It is just the scheme $X$ we have defined above. We also denote the rational locus of $X$ to be $T''$ which is isomorphic to $\mathbb{P}^{n-1}(\FF_{p^{2}})/\Sh_{0,n}.$ We denote the exceptional divisor by $T',$ which consists of points $(H_1,H_1')$ such that $H_1=F(V^{-1}(H_1^{(p)})^{(p)}).$ Then $T'\simeq \bigsqcup\limits_{\#\mathbb{P}^{n-1}(\FF_{p^{2}})}\mathbb{P}^{n-2}/{{\Sh}}_{0,n}.$ Considering the $\GL_n$-action on $\PP^{n-1}(\FF_{p^2}),$ $T'$ is identified with a $\PP^{n-2}$-bundle over ${\Sh}_{0,n}(K_\gothp^1).$
Then $E$ is a $\mathbb{P}^{n-3}$-bundle over $T'$ corresponding to $\mathcal{H}_2/\mathcal{H}_1'.$

By blowing-up exact sequence for the pair $(X,Y,T'',T')$ we have 
${{\rm H}^{2n-2}_{\acute{e}t}(Y,L)}={{\rm H}^{2n-2}_{\acute{e}t}(X,L)},$ ${\rm H}^{0}_{\acute{e}t}(Y,L)={{\rm H}^{0}_{\acute{e}t}({{\Sh}}_{0,n},L)}$ and $H^{2i}_{\acute{e}t}(Y,L)={H^{2i}_{\acute{e}t}(X,L)}\oplus H^{2i}_{\acute{e}t}(T',L)={{\rm H}^{0}_{\acute{e}t}({{\Sh}}_{0,n},L(-i))}\oplus {\rm H}^{0}_{\acute{e}t}({{\Sh}}_{0,n}(K_{\gothp}^{1}),L(-i))$ for $1\leq i\leq n-2.$ 

Since $Z$ is the $\mathbb{P}^{n-3}$-bundle over $Y$ corresponding to ${\rm H}^{\dR}_1(\calB/{\Sh}_{0,n})^{\circ}_{1}/\mathcal{H}_2$ we have $H^{*}_{\acute{e}t}(Z,L)=H^{*}_{\acute{e}t}(Y,L)\otimes H^{*}_{\acute{e}t}(\mathbb{P}^{n-3}/\Sh_{0,n},L).$ In paritcular, ${\rm H}^{2n-4}_{\acute{e}t}(Z,L(n-2))=H_{\acute{e}t}^{0}(\Sh_{0,n},L)^{\oplus (n-2)}\oplus H_{\acute{e}t}^{2n-4}(E,L(n-2)).$

By blowing-up exact sequence, we have \[\begin{tikzcd}
[
    column sep=small, 
    row sep=scriptsize,
    cells={font=\footnotesize} 
]
	{} && \cdots && {{\rm H}^{2n-5}_{\acute{e}t}(E,L)} \\
	{{\rm H}^{2n-4}_{\acute{e}t}(W,L)} && {{\rm H}^{2n-4}_{\acute{e}t}(Z,L)\oplus {\rm H}^{2n-4}_{\acute{e}t}(T,L)} && {{\rm H}^{2n-4}_{\acute{e}t}(E,L)} \\
	{H^{2n-3}_{\acute{e}t}(W,L)} && {H^{2n-3}_{\acute{e}t}(Z,L)\oplus H^{2n-3}_{\acute{e}t}(T,L)} && \cdots
	\arrow[from=1-3, to=1-5]
	\arrow[from=1-5, to=2-1]
	\arrow[from=2-1, to=2-3]
	\arrow[from=2-3, to=2-5]
	\arrow[from=2-5, to=3-1]
	\arrow[from=3-1, to=3-3]
	\arrow[from=3-3, to=3-5]
\end{tikzcd}\] where ${\rm H}^{2n-5}_{\acute{e}t}(E,L)=0$ and ${H^{2n-3}_{\acute{e}t}(Z,L)\oplus H^{2n-3}_{\acute{e}t}(T,L)}=0.$ Since $T=\bigsqcup\limits_{\# \mathbb{P}^{n-1}(\FF_{p^{2}})} \mathbb{P}^{n-2}/{{\Sh}}_{0,n},$ the $\GL_n$-action on $T$ by switching the points in $\mathbb{P}^{n-1}(\FF_{p^{2}})$ identifies $T$ with a $\PP^{n-2}$-bundle over $\Sh_{0,n}(K_\gothp^1).$ Then $E$ is a $\PP^{n-3}$-bundle over $T$ corresponding to $\mathcal{H}_1'/\mathcal{H}_1.$

Thus the map ${{\rm H}^{2n-4}_{\acute{e}t}(Z,L(n-2))\oplus {\rm H}^{2n-4}_{\acute{e}t}(T,L(n-2))} \ra {{\rm H}^{2n-4}_{\acute{e}t}(E,L(n-2))}$ is surjective and the cohomology group ${\rm H}^{2n-4}_{\acute{e}t}(W,L(n-2))=H_{\acute{e}t}^{0}(\Sh_{0,n},L)^{\oplus (n-2)}\bigoplus H_{\acute{e}t}^{2n-4}(T,L(n-2)).$ Since $\phi$ induces isomorphism on cohomology groups, $H_{\acute{e}t}^{2n-4}({\rm Y}_{00}\bigcap{\rm Y}_{11},L(n-2))=H_{\acute{e}t}^{0}(\Sh_{0,n},L)^{\oplus n-2}\oplus H_{\acute{e}t}^{0}(\Sh_{0,n}(K_\gothp^1),L).$

By excision sequence  
${\rm H}^{2n-3}_{c}({\rm Y}_{00}\bigcap {\rm Y}_{10}-{\rm Y}_{00}\bigcap {\rm Y}_{11},L(n-2))$ is the cokernel of ${\rm H}^{2n-4}_{\acute{e}t}({\rm Y}_{00}\bigcap {\rm Y}_{10},L(n-2))\ra  {\rm H}^{2n-4}_{\acute{e}t}({\rm Y}_{00}\bigcap {\rm Y}_{11},L(n-2))$ induced by restriction.
Up to isomorphism induced by $\phi,$ it is isomorphic to the cokernel of the map ${\rm H}^{2n-4}_{\acute{e}t}(\tilde Y,L(n-2))\ra  {\rm H}^{2n-4}_{\acute{e}t}(W,L(n-2))$ induced by restriction.
Hence it is the cokernel of an injective map $H^{0}_{\acute{e}t}(\Sh_{0,n},L)\ra H^0_{\acute{e}t}(\Sh_{0,n}(K_\gothp^1),L),$ which we do not specify here. Moreover, it is easy to see that the map is injective and hence ${\rm H}^{2n-4}_{c}({\rm Y}_{00}\bigcap {\rm Y}_{10}-{\rm Y}_{00}\bigcap {\rm Y}_{11},L(n-1))=0.$

Similarly, be excision sequence, we have ${\rm H}^{2n-2}_{c}({\rm Y}_{00}\bigcap {\rm Y}_{10}-{\rm Y}_{00}\bigcap {\rm Y}_{11},L(n-1))={\rm H}^{0}_{\acute{e}t}({\rm {\Sh}}_{0,n},L).$
\end{proof}

We have shown that for any point $(A,\lambda,\eta,A',\lambda',\eta',\phi)\in {\rm Y}_{00},$ the abelian varieties $A,A'$ are supersingular. Now we are going to give an equivalent condition of when points in ${{\Sh}}_{1,n-1}(K_{\gothp}^{1})$ are supersingular
First we give an equivalent condition of when points lie in ${\rm Y}_{11}\backslash ({\rm Y }_{00}\bigcup {\rm Y}_{10}\bigcup {\rm Y}_{01}).$
\begin{lemma}\label{In y11}
     For any $S$-point $(A,\lambda,\eta, A',\lambda',\eta',\phi),$ it lies in $ {\rm Y}_{11}\backslash ({\rm Y }_{00}\bigcup {\rm Y}_{10}\bigcup {\rm Y}_{01})$ if and only if there exists $\alpha \in p\tcD(A')_{1}^{\circ}$ such that:
     \begin{itemize}
         \item $p\phi_{*,1}\tcD(A)_{1}^{\circ}+W(\overline{\FF}_{p})\alpha=p\phi_{*,1}\tcD(A')_{1}^{\circ},$
         \item 
         $F\alpha-V\alpha\in p\phi_{*,2}\tcD(A)_{2}^{\circ},$
         \item 
         $\alpha\in \phi_{*,1}\tcD(A)_{1}^{\circ}\bigcap \phi_{*,1}F^{-1}V\tcD(A)_{1}^{\circ}\bigcap \phi_{*,1}V^{-1}F\tcD(A)_{1}^{\circ}$
         \item 
         $\alpha\notin \phi_{*,1}V\tcD(A)_{2}^{\circ},\phi_{*,1}F\tcD(A)_{2}^{\circ}.$
     \end{itemize}
\end{lemma}
\begin{proof}
    If there exists such $\alpha \in p\tcD(A')_{1}^{\circ},$
    take $\tcE_1=\tcD(A)_{1}+W(\overline{\FF}_{p})\phi_{*,1}^{-1}(p^{-1}\alpha)$ and $\tcE_{2}=W(\overline{\FF}_{p})(V^{-1}(\phi_{*,1}^{-1}(\alpha)))+\tcD(A)_{2}.$ Applying Corollary~\ref{CORP}, we get a point $(A',\lambda',\eta')$ of ${\Sh}_{1,n-1}$ and 
    a p-quasi-isogeny $\phi: A\ra A',$ such that $(A,\lambda, \eta,A',\lambda', \eta',\phi)$ is a point of ${\rm Y}_{11}\backslash ({\rm Y }_{00}\bigcup {\rm Y}_{10}\bigcup {\rm Y}_{01}).$ since  $\omega^{\circ}_{A^{\vee},1}\neq \Ker(\phi_{*,1})$ and $\Im(\phi_{*,2})\neq \omega^{\circ}_{A^{\vee},2}.$
    

    Conversely, if $(A,\lambda, \eta,A',\lambda', \eta',\phi)$ is a point of ${\rm Y}_{11}\backslash ({\rm Y }_{00}\bigcup {\rm Y}_{10}\bigcup {\rm Y}_{01}),$ we can take $\alpha\in p\tcD(A')_{1}^{\circ}$ such that $p\phi_{*,1}\tcD(A)_{1}^{\circ}+W(\overline{\FF}_{p})\alpha=p\phi_{*,1}\tcD(A')_{1}^{\circ}.$ Then since $\omega^{\circ}_{A^{\vee},1}\neq \Ker(\phi_{*,1})$ and $\Im(\phi_{*,2})\neq \omega^{\circ}_{A^{\vee},2},$ we have $p\tcD(A')_{1}\neq \phi_{*,1}V\tcD(A)_{2}$ and $V\tcD(A')_{1}\neq \phi_{*,2}\tcD(A)_{2}.$ Thus $W(\overline{\FF}_{p})\alpha+\phi_{*,1}(V\tcD(A)_{2})=V\tcD(A')_{2}^{\circ},W(\overline{\FF}_{p})(p^{-1}\alpha)+\phi_{*,1}\tcD(A)_{1}=\tcD(A')_{1}^{\circ}$ and  $W(\overline{\FF}_{p})(F\alpha)+\phi_{*,2}(p\tcD(A)_{2})=p\tcD(A')_{2}^{\circ}.$ Moreover, we have $W(\overline{\FF}_{p})(F^{-1}\alpha)+\phi_{*,2}(V\tcD(A))_{1}=V\tcD(A')_{1}^{\circ}$ and $W(\overline{\FF}_{p})(V^{-1}\alpha)+\phi_{*,2}\tcD(A)_{2}=\tcD(A')_{2}^{\circ}.$ Moreover, we have $F^{-1}\alpha-V^{-1}\alpha=xV^{-1}\alpha$ for some $x\in W(\overline{\FF}_{p})$ since $V\tcD(A')_{1}\neq \phi_{*,2}\tcD(A)_{2}$ and $\phi_{*,2}\tcD(A)_{2}$ has corank $1$ in $\tcD(A')_{2}.$ If $x\notin pW(\overline{\FF}_{p}),$ we can modify $\alpha$ by $y\alpha$ for some $y\in W(\overline{\FF}_{p})$ such that $F^{-1}(y\alpha)-V^{-1}(y\alpha)=y^{\sigma^{-1}}F^{-1}\alpha-y^{\sigma}V^{-1}\alpha= ((x+1)y^{\sigma^{-1}}-y^{\sigma})V^{-1}\alpha.$ Take $y$ such that $(x+1)y^{\sigma^{-1}}-y^{\sigma}\in pW(\overline{\FF}_{p})$ we get $F^{-1}(y\alpha)-V^{-1}(y\alpha)\in p\phi_{*,2}\tcD(A)_{2}^{\circ}.$ Substituting $\alpha$ with $y\alpha,$ we finish the proof.
\end{proof}

To give a judgement when points in ${{\Sh}}_{1,n-1}(K_{\gothp}^{1})$ are supersingular, we need some definitions.
\begin{definition}
    Let $k$ be a perfect field of characteristic $p$ and $W(k)$ be the Witt vector ring corresponding to $k.$ Supoose $(P,F)$ is a $Q(k)=Frac(W(k))$-isocrystal defined in Definition~\ref{iisocrystal}. We say $(P,F)$ is average of slope 0 if there exist (thus for every) some (full) lattice $H$ in $P$ such that $\ell(H/H\bigcap F(H))=\ell(F(H)/H\bigcap F(H))\leq 1.$ We say $(P,F)$ is pure of slope $0$ if $P$ admits a $F$-invariant (full) lattice.
\end{definition}
\begin{definition}
     Let $k$ be a perfect field of characteristic $p$ and $W(k)$ be the Witt vector ring corresponding to $k.$ Supoose $(P,F)$ is a $Q(k)=Frac(W(k))$-isocrystal average of slope $0$ and $H\subseteq P$ is a sublattice. Suppose $\ell(H/H\bigcap F(H))=\ell( F(H)/H\bigcap F(H))\leq 1.$ Let $Lat_{\leq 1}(P)$ to be the set of $H$ satisfies the above conditions. For $i\geq 0,$ we define $S_{i}(H)=\sum\limits_{j=0}^{i}F^{j}(H)$ and $T_{i}(H)=\bigcap\limits_{j=0}^{i}F^{j}(H).$ Moreover, we define $S_{\infty}(H)=\lim\limits_{i\ra \infty}S_{i}(H)$ and $T_{\infty}(H)=\lim\limits_{i\ra \infty}T_{i}(H).$ We define $s(H)=\inf\{s~|~S_{s}(H)=S_{\infty}(H)\}$ and $t(H)=\inf\{t~|~T_{t}(H)=T_{\infty}(H)\}.$
\end{definition}

\begin{lemma}\label{corr}
    Assume $(P,F)$ is a $Q(k)=Frac(W(k))$-isocrystal average of slope $0.$ Let $H\in Lat_{\leq 1}(P).$ Then we have:
    \begin{enumerate}
        \item $s(H)=0\iff t(H)=0 \iff H=F(H),$
        \item $S_{i}, T_{i}$ commute with $F$ and mulitplication by $p.$ So $s(H)=s(F(H))=s(pH),$
        \item For $0\leq i,j<\infty,$ $S_{i}(H),T_{i}(H)\in Lat_{\leq 1}(P)$ and $S_i(S_j(H))=S_{i+j}(H),T_i(T_j(H))=T_{i+j}(H),$
        \item If $(P,F)$ is pure of slope $0,$ then $s(H),t(H)\leq \rank(P)-1.$ Otherwise, $s(H)=t(H)=\infty$
        \item Let $0\leq i,j<\infty.$ Then
        \[T_{j}(S_{i}(H))=
        \left\{ 
        \begin{array}{ll}
             S_{\infty}(H),& \text{if } i\geq s(H); \\
             F^{j}(S_{i-j}(H)),& \text{if } j\leq i<s(H); \\
             F^{i}(T_{j-i}(H)),&\text{if } i<s(H) \text{ and } i<j<i+t(H);\\
             T_{\infty}(H),& \text{if } i<s(H) \text{ and } j\geq i+t(H);\\
        \end{array}
        \right.
        \]
        
        So $t(S_{i}(H))=t(H)+i$ if $0\leq i<s(H).$
    \end{enumerate}
\end{lemma}
\begin{proof}
    It is easy to check $(1)$-$(2).$ For $(4),$ if $(P,F)$ is pure of slope $0,$ then $S_{\rank(P)-1}(H)$ is $F$-invariant by \cite[Proposition 2.17]{RZ96}, so $s(H)\leq \rank(P)-1.$ Otherwise, by our definition, $P$ has no $F$-invariant lattice, so $S_{i}(H)\subsetneq S_{i+1}(H)$ for every $i<\infty,$ that is, $s(H)=\infty.$

    For $(3),$ first prove by induction on $0\leq i<s(H)$ that $\ell(S_{i+1}(H)/S_{i}(H))=1.$ For $i=0,$it follows from $H\in Lat_{\leq 1}(P).$ For $0<i<s(H),$ $S_{i}(H)\neq F(S_{i}(H))$ and $F(S_{i-1}(H))\subseteq S_{i}(H)\cap F(S_{i-1}(H))\subsetneq F(S_{i}(H)).$ By the inductive bypothesis, $\ell(S_{i}(H)/S_{i-1}(H))=1$ and thus $\ell(F(S_{i}(H))/F(S_{i-1}(H)))=1.$ It forces that\[F(S_{i-1}(H))=S_{i}(H)\cap F(S_{i}(H)) \text{ and }\ell(F(S_{i}(H))/S_{i}(H)\cap(S_{i}(H)))=1.\]
    So $\ell(F(S_{i+1}(H))/F(S_{i}(H)))=\ell(F(S_{i}(H))/S_{i}(H)\cap(S_{i}(H)))=1.$ This completes the induction. It follows immediately that $S_{i}(H)\in Lat_{\leq 1}(P).$ The other assertion of $(3)$ is clear.

    We have seen $F(S_{i-1}(H))=T_{1}(S_{i}(H))$ for $0<i<s(H).$ So $T_{i}(S_{j}(H))=F^{j}(S_{i-j}(H))$ for $j\leq i<s(H).$ In particular, $T_{i}(S_{i}(H))=F^{i}(H)$ for $i<s(H).$ So for $i<j<\infty,$ $T_{j}(S_{i}(H))=T_{j-i}(T_{i}S_{i}(H))=T_{i-j}(F^{i}(H))=F^{i}(T_{j-i}(H)).$ So we get $(5).$
\end{proof}

\begin{proposition}\label{key lemma}
    For any $S$-point $(A,\lambda,\eta, A',\lambda',\eta',\phi)\in {{\Sh}}_{1,n-1}(K_{\gothp}^{1})^{\rm ss},$ if 
   $S_{\infty}(V\tcD(A')_{2}^{\circ})\subseteq T_{\infty}(\phi_{*,1}\tcD(A)_{1}^{\circ})$ or
$S_{\infty}(p\tcD(A')_{2}^{\circ})\subseteq T_{\infty}(\phi_{*,2}(V\tcD(A)_{1}^{\circ})),$
      then the point $(A,\lambda,\eta, A',\lambda',\eta',\phi)$ is contained in $\pr_n(C_{n})\bigcup \bigcup\limits_{i=1}^{n-1}\pr_{i}(C_i).$ 
\end{proposition}
\begin{proof}
 Consider the isocrystals  $(\tcD(A)_{i}^{\circ}\otimes_{W(k)}Q(k),FV^{-1}\otimes 1),(\tcD(A')_{i}^{\circ}\otimes_{W(k)}Q(k),FV^{-1}\otimes 1).$ Since $A,A'$ are supersingular, they are pure of slope $0.$
 Let $F'=FV^{-1}.$ If $S_{\infty}(V\tcD(A')_{2}^{\circ})\subseteq T_{\infty}(\phi_{*,1}\tcD(A)_{1}^{\circ}),$ we take $\tcE_1$ to be an $F'$-invariant lattice such that $S_{\infty}(V\tcD(A')_{2}^{\circ})\subseteq \tcE_1\subseteq T_{\infty}(\phi_{*,1}\tcD(A)_{1}^{\circ}).$ Suppose $\ell(\phi_{*,1}(\tcD(A)_{1}^{\circ})/\tcE_1)=i-1.$ Let $\tcE_{2}=V\tcE_1.$ Applying Proposition~\ref{P:abelian-Dieud} with $m=1,$ we get a point $(B,\lambda'',\eta'')$ of $\Sh_{0,n}$ and isogenies $\psi: B\ra A,\psi': B\ra A',$ such that $(A,\lambda,\eta,B,\lambda'', \eta'',\psi)\in {\rm Y}_{i},(A',\lambda',\eta',B,\lambda'', \eta'',\psi')\in {\rm Y}_{i+1}.$ By Proposition~\ref{S:corre(1,n-1)(0,n)}, we get $A,A'$ are supersingular.
The case $S_{\infty}(p\tcD(A')_{2}^{\circ})\subseteq T_{\infty}(\phi_{*,2}(V\tcD(A)_{1}^{\circ}))$ is similar and we omit here. Hence we finish the proof.

\end{proof}

\begin{proposition}
\label{supersingular case}
    With notations as above, we have ${{\Sh}}_{1,n-1}(K_{\gothp}^{1})^{\rm ss}=\pr_n(C_{n})\bigcup \bigcup\limits_{i=1}^{n-1}\pr_{i}(C_i)={\rm Y}_{00}\bigcup \bigcup\limits_{i=1}^{n-1}\pr_{i}(C_i).$
\end{proposition}
\begin{proof}
    We have seen ${\rm Y}_{00}\subseteq {{\Sh}}_{1,n-1}(K_{\gothp}^{1})^{\rm ss}.$  For any $S$-point $(A,\lambda,\eta, A',\lambda',\eta',\phi)\in {\rm Y}_{10}$ or $ {\rm Y}_{01},$ we show if $A,A'$ are supersingular, $(A,\lambda,\eta, A',\lambda',\eta',\phi)\in \pr_n(C_{n})\bigcup\bigcup\limits_{i=1}^{n-1}\pr_{i}(C_i).$ For simplicity, we only prove for ${\rm Y}_{10}.$ The case for ${\rm Y}_{01}$ is quite the same. 
    
    Take a supersingular point $(A,\lambda,\eta, A',\lambda',\eta',\phi)\in {\rm Y}_{10},$ then $\omega^{\circ}_{A^{'\vee},1}=\Im(\phi_{*,1})$ and $\Ker(\phi_{*,2})=\omega^{\circ}_{A^{\vee},2}.$ Thus $V\tcD(A')_{2}^{\circ}\subseteq \phi_{*,1}\tcD(A)_{1}^{\circ}$ and $V\tcD(A)_{1}^{\circ}=\phi_{*,2}(\tcD(A)_{2}^{\circ}).$ If $p\tcD(A')_{1}^{\circ}=\phi_{*,1}V\tcD(A)_{2}^{\circ},$ then $F\tcD(A)_{2}^{\circ}=V\tcD(A)_{2}^{\circ}.$ That is, $A\in {\rm Y}_{n}$ and $A'\in {\rm Y}_{1}.$
    If $p\tcD(A')_{1}^{\circ}\neq\phi_{*,1}V\tcD(A)_{2}^{\circ},$ we have $\phi_{*,1}V\tcD(A)_{2}^{\circ}\in Lat_{\leq 1}(\tcD(A)_{1}^{\circ}\otimes_{W(k)}Q(k))$ and $\phi_{*,1}(p\tcD(A)_{1}^{\circ})=\phi_{*,1}(V\tcD(A)_{2}^{\circ})\bigcap\phi_{*,1}(F\tcD(A)_{2}^{\circ}).$ Thus $t(\phi_{*,1}(\tcD(A)_{1}^{\circ}))=t(\phi_{*,1}(p\tcD(A)_{1}^{\circ}))=t(\phi_{*,1}(V\tcD(A)_{2}^{\circ}))-1.$ Hence $T_{t(\phi_{*,1}(\tcD(A)_{1}^{\circ}))}(\phi_{*,1}(V\tcD(A)_{2}^{\circ}))\subseteq T_{t(\phi_{*,1}(\tcD(A)_{1}^{\circ}))}(\phi_{*,1}(\tcD(A)_{1}^{\circ})).$ Therefore $$S_{\infty}(\phi_{*,1}(V\tcD(A)_{2}^{\circ}))=S_{\infty}(T_{t(\phi_{*,1}(\tcD(A)_{1}^{\circ}))}(\phi_{*,1}(V\tcD(A)_{2}^{\circ})))\subseteq S_{\infty}(T_{t(\phi_{*,1}(\tcD(A)_{1}^{\circ}))}(\phi_{*,1}(\tcD(A)_{1}^{\circ})))= T_{\infty}(\phi_{*,1}(\tcD(A)_{1}^{\circ})).$$ This shows $s(\phi_{*,1}(V\tcD(A)_{2}^{\circ}))+t(\phi_{*,1}(\tcD(A)_{1}^{\circ}))\leq n-1.$ Hence $s(\phi_{*,1}(V\tcD(A)_{2}^{\circ}))+t(\phi_{*,1}(V\tcD(A)_{2}^{\circ}))\leq n.$ Since $V\tcD(A')_{2}^{\circ}=\phi_{*,1}(V\tcD(A)_{2}^{\circ})\bigcup\phi_{*,1}(F\tcD(A)_{2}^{\circ}),$ we can easily get $s(V\tcD(A')_{2}^{\circ})=s(\phi_{*,1}(V\tcD(A)_{2}^{\circ}))-1.$ Hence $s(V\tcD(A')_{2}^{\circ})+t(\phi_{*,1}(\tcD(A)_{1}^{\circ}))\leq n-2,$ which means $S_{\infty}(V\tcD(A')_{2}^{\circ}))\subseteq T_{\infty}(\phi_{*,1}(\tcD(A)_{1}^{\circ})).$ There exists $1\leq i\leq n-1$ such that $(A,\lambda,\eta, A',\lambda',\eta',\phi)\in \pr_{i}(C_i)$ by the proof of Proposition~\ref{key lemma}.

    For any $S$-point $(A,\lambda,\eta, A',\lambda',\eta',\phi)\in {\rm Y}_{11}\backslash ({\rm Y }_{00}\bigcup {\rm Y}_{10}\bigcup {\rm Y}_{01}),$ we show $(A,\lambda,\eta, A',\lambda',\eta',\phi)\in \bigcup\limits_{i=1}^{n-1}\pr_{i}(C_i),$ if $A,A'$ are supersingular. There is a morphism $\delta: {{\Sh}}_{1,n-1}(K_{\gothp}^{1})\ra {{\Sh}}_{1,n-1}$ mapping $(A,\lambda,\eta, A',\lambda',\eta',\phi)$  to $(A,\lambda,\eta).$ Suppose the image $(A,\lambda,\eta)\in V^{(\omega_{1},\omega_{2})},$ with $(\omega_1,\omega_2)=(a,b)$ and $a\leq b.$ Then there is a basis of $\tcD(A)_{1}^{\circ}\oplus \tcD(A)_{2}^{\circ},$ denoted by $\{e_{i,j}|i=1,2;1\leq j\leq n\}$ such that $F,V$ act on $\tcD(A)_{1}^{\circ}\oplus \tcD(A)_{2}^{\circ}$ mod $p$ by \[
    F(e_{1,i})=\left\{\begin{array}{ll}
         e_{2,i}&\text{if } 1\leq i\leq a-1; \\
         0&\text{if }i=a;\\
         e_{2,i-1}&\text{if }i\geq a+1.
    \end{array}\right.
    F(e_{2,i})=\left\{\begin{array}{ll}
         0&\text{if } i\leq 1\leq b-1; \\
         e_{1,1}&\text{if }i=b;\\
         0&\text{if }i\geq b+1.\\
    \end{array}\right.
    \]
    \[
    V(e_{1,i})=\left\{\begin{array}{ll}
         0&\text{if } i=1; \\
         e_{2,i-1}&\text{if }2\leq i \leq b;\\
         e_{2,i}&\text{if }b+1\leq i\leq n;\\
    \end{array}\right.
    V(e_{2,i})=\left\{\begin{array}{ll}
         0&\text{if } i\leq 1\leq a-1; \\
         0&\text{if }a\leq i\leq n-1;\\
         e_{1,a}&\text{if }i=n.\\
    \end{array}\right.
    \]

    Suppose $\alpha$ in Lemma~\ref{In y11} to be written as $\alpha=\sum\limits_{i=1}^{n}x_{i}e_{1,i}$(Here we identify $\tcD(A)_{i}^{\circ}$ with its image in $\tcD(A')_{i}^{\circ}$ by $\phi_{*,i}$ for $i=1,2.$). Checking the conditions in Lemma~\ref{In y11} directly, we get $x_{n},x_{b+1}\in pW(\overline{\FF}_{p})$ and there exists $i\neq 1,a$ such that $x_{i}\notin pW(\overline{\FF}_{p}).$ Moreover, we have \[
    \left\{
    \begin{array}{ll}
         x_{i}^{\sigma}\equiv x_{i+1}^{\sigma^{-1}}\text{ } mod \text{ } p& \text{if } 1\leq i\leq a-1;  \\
         x_{i}^{\sigma}\equiv x_{i}^{\sigma^{-1}} \text{ } mod \text{ }p&\text{if } a+1\leq i\leq b;\\
         x_{i}\equiv 0 \text{ } mod \text{ }p&\text{if } b+1\leq i\leq n.\\
    \end{array}\right.
    \]

    Note that $p\phi_{*,1}(\tcD(A)_{2}^{\circ})\subseteq S_{\infty}(\phi_{*,1}(V\tcD(A)_{1}^{\circ}))\subseteq T_{\infty}(\phi_{*,1}\tcD(A)_{1}^{\circ}).$
    To show $S_{\infty}(V\tcD(A')_{2}^{\circ})\subseteq T_{\infty}(\phi_{*,1}\tcD(A)_{1}^{\circ}),$ it suffices to check that $\alpha \in \frac{T_{\infty}(\phi_{*,1}\tcD(A)_{1}^{\circ})}{p\phi_{*,1}(\tcD(A)_{2}^{\circ})}.$ First consider $\frac{FV^{-1}(\phi_{*,1}\tcD(A)_{1}^{\circ})}{p\phi_{*,1}(\tcD(A)_{2}^{\circ})}=\frac{\frac{1}{p}F^2(\phi_{*,1}\tcD(A)_{1}^{\circ})}{p\phi_{*,1}(\tcD(A)_{2}^{\circ})}.$ 

    The action $F,V$ on $\tcD(A)_1^\circ$ can be described as \[
    F(e_{1,i})=\left\{\begin{array}{ll}
         e_{2,i}+pa_{2,i}&\text{if } 1\leq i\leq a-1; \\
         pe_{2,n}+pa_{2,a}&\text{if }i=a;\\
         e_{2,i-1}+pa_{2,i}&\text{if }i\geq a+1.
    \end{array}\right.
    V(e_{1,i})=\left\{\begin{array}{ll}
         pe_{2,b}+pb_{2,1}&\text{if } i=1; \\
         e_{2,i-1}+pb_{2,i}&\text{if }2\leq i \leq b;\\
         e_{2,i}+pb_{2,i}&\text{if }b+1\leq i\leq n;\\
    \end{array}\right.
    \]
    Here $a_{2,i}$ and $b_{2,i}$ are element in $\tcD(A)_{2}^\circ.$ Then we have \[
    F(e_{2,i})=\left\{\begin{array}{ll}
         pe_{1,i+1}-pFb_{2,i+1}&\text{if } 1\leq i\leq b-1; \\
         e_{1,1}-Fb_{2,1}&\text{if }i=b;\\
         pe_{1,i}-pFb_{2,i}&\text{if }i\geq b+1.
    \end{array}\right.\]
    Now consider the action $F$ on $\tcD(A)_{2}^\circ\mod p,$ we can see easily that $Fb_{2,i}=d_{2,i}e_{1,1}+pc_{2,i}$ for some $c_{2,i}$ in $\tcD(A)_{1}^\circ$ and $d_{2,i}\in W(\overline{\FF}_p).$ Thus  
    \[
    F(e_{2,i})=\left\{\begin{array}{ll}
         pe_{1,i+1}-pd_{2,i}e_{1,1}-p^2c_{2,i+1}&\text{if } 1\leq i\leq b-1; \\
         e_{1,1}-d_{2,b}e_{1,1}-pc_{2,b}&\text{if }i=b;\\
         pe_{1,i}-pd_{2,i}e_{1,1}-p^2c_{2,i}&\text{if }i\geq b+1.
    \end{array}\right.\] 
    This shows $\frac{FV^{-1}(\phi_{*,1}\tcD(A)_{1}^{\circ})}{p\phi_{*,1}(\tcD(A)_{2}^{\circ})}=\overline{\FF}_p\{\frac{1}{p}e_{1,1},e_{1,2},\cdots,e_{1,n-1}\}.$
    
    By induction, $\frac{T_{\infty}(\phi_{*,1}\tcD(A)_{1}^{\circ})}{p\phi_{*,1}(\tcD(A)_{2}^{\circ})}=\overline{\FF}_{p}\{e_{1,i}, i\leq b\}.$ Thus we can see obviously that $\alpha \in T_{\infty}(\phi_{*,1}\tcD(A)_{1}^{\circ}).$ Hence we finish the proof.
\end{proof}

For the morphisms from ${\rm Y_{01}}$ and ${\rm Y_{10}}$ to ${{\Sh}}_{1,n-1}$ induced by isomorphisms in Proposition~\ref{bloww up} and their relations with stratification on ${{\Sh}}_{1,n-1},$ we have the following proposition:
\begin{proposition}
\label{connections}
    \begin{enumerate}
        \item The morphism $\pi_{10}$ from ${\rm Y}_{10}$ to ${{\Sh}}_{1,n-1}$ is surjective and maps ${\rm Y}_{10}\bigcap {\rm Y}_{11}$ to the complement of union of Newton strata of first slope less than $\frac{1}{4},$ that is, equal to $0.$ We denote it by $N_{10}.$ Moreover the morphism maps $\pr_{i}(C_{i})\bigcap {\rm Y}_{10}$ to $Y_{i}$ surjectively for $1\leq i\leq n.$
        \item The morphism from ${\rm Y}_{10}$ to ${{\Sh}}_{1,n-1}$ is surjective and maps ${\rm Y}_{10}\bigcap {\rm Y}_{11}$ to the complement of union of Newton strata of last slope bigger than $\frac{3}{4},$ that is, equal to $1.$ We denote it by $N_{01}.$ Moreover the morphism maps $\pr_{i}(C_{i})\bigcap {\rm Y}_{10}$ to ${\rm Y}_{i}$ surjectively.
    \end{enumerate}
\end{proposition}
\begin{proof}
    By symmetry, we only prove $(1).$
    Since ${\rm Y}_{10}\simeq {\rm Bl}_{{\mathop{p}\limits^{\leftarrow}}_{n}({\rm Y}_{n})}{{\Sh}}_{1,n-1}$ by Proposition~\ref{bloww up}, the morphism is surjective obviously. 
    
    For any $S$-point $(A,\lambda,\eta,A',\lambda',\eta',\phi)\in {\rm Y}_{10}\bigcap {\rm Y}_{11},$ we have $V\tcD(A')_{2}^{\circ}\subseteq \phi_{*,1}(\tcD(A)_{1}^{\circ}),$ $p\tcD(A')_{2}^{\circ}\subseteq \phi_{*,2}(V\tcD(A)_{1}^{\circ})$ and $V\tcD(A')_{1}^{\circ}=\phi_{*,2}(\tcD(A)_{2}^{\circ}).$ Therefore $\phi_{*,1}(F\tcD(A)_{2}^{\circ})=p\tcD(A')_{1}^{\circ}\subseteq V\tcD(A')_{2}^{\circ}\subseteq \phi_{*,1}(\tcD(A)_{1}^{\circ})\bigcap \phi_{*,1}(F^{-1}V\tcD(A)_{1}^{\circ}).$ This shows $F^4\tcD(A)_2^\circ\subseteq pF\tcD(A)_1^\circ\subseteq p\tcD(A)_2^\circ$ and $F^4\tcD(A)_1\subseteq F^3\tcD(A)_2^\circ\subseteq p\tcD(A)_1^\circ.$
    Hence we have shown the image of ${\rm Y}_{10}\bigcap {\rm Y}_{11}$ is contained in $N_{10}.$ 
    
    Conversely, to show $N_{10}$ is contained in the image of ${\rm Y}_{10}\bigcap {\rm Y}_{11},$ it suffices to show $V^{(\omega_{1},\omega_{2})}$ is contained in the image of ${\rm Y}_{10}\bigcap {\rm Y}_{11}$ where $(\omega_{1},\omega_{2})=(n,2).$ Then we can get $N_{10}=\overline{V}^{(\omega_{1},\omega_{2})}$ from the morphism is proper. 
 For any $S$-point $(A,\lambda,\eta)\in V^{(\omega_{1},\omega_{2})},$ by Proposition~\ref{Formula}, $\tcD(A)_{1}^{\circ}\bigoplus \tcD(A)_{2}^{\circ}$ has a basis $\{e_{i,j}|i=1,2 \text{ and } 1\leq j\leq n\}$ such that $F,V$ act on $\tcD(A)_{1}^{\circ}\oplus \tcD(A)_{2}^{\circ}$ mod $p$ by\[
    F(e_{1,i})=\left\{\begin{array}{ll}
         e_{2,i}&\text{if } i\leq 1\leq n-1; \\
         0&\text{if }i=n;\\
    \end{array}\right.
    F(e_{2,i})=\left\{\begin{array}{ll}
         0&\text{if } i=1; \\
         e_{1,1}&\text{if }i=2;\\
         0&\text{if }3\leq i\leq n.\\
    \end{array}\right.
    \]
    \[
    V(e_{1,i})=\left\{\begin{array}{ll}
         0&\text{if } i=1; \\
         e_{2,1}&\text{if } i=2;\\
        e_{2,i}&\text{if } 3\leq i\leq n;\\
    \end{array}\right.
    V(e_{2,i})=\left\{\begin{array}{ll}
         0&\text{if } 1\leq i\leq n-1; \\
         e_{1,n}&\text{if }i=n.\\
    \end{array}\right.
    \]

    We take $\tcE_1=W(\overline{\FF}_{p})\{\frac{1}{p}e_{1,1}, e_{1,2},\dots, e_{1,n}\}$ and $\tcE_{2}=W(\overline{\FF}_{p})\{\frac{1}{p}e_{2,1}, e_{2,2},\dots, e_{2,n}\}.$
Applying Corollary \ref{CORP}, we get a point $(A',\lambda',\eta')$ of ${\Sh}_{1,n-1}$ and 
an isogeny $\phi: A\ra A',$ such that $(A,\lambda, \eta,A',\lambda', \eta',\phi)$ is a point of ${\rm Y}_{10}\bigcap {\rm Y}_{11}.$

To show the morphism maps $\pr_{i}(C_{i})\bigcap {\rm Y}_{10}$ to ${\rm Y}_{i}$ surjectively, we note that $\pr_{i}(C_{i})\bigcap {\rm Y}_{10}$ is contained in ${\rm Y}_{i}.$ By Proposition~\ref{supersingular case}, we see that $V^{(\omega_{1},\omega_{2})}$ with $(\omega_{1},\omega_{2})=(n+1-i,n+1-i)$ is contained in $\pr_{i}(C_{i})\bigcap {\rm Y}_{10}.$ Hence ${\rm Y}_{i},$ as its closure is contained in $\pr_{i}(C_{i})\bigcap {\rm Y}_{10}.$ We finish the proof.
\end{proof}

\begin{construction}\label{frob:Y_10 Y_01}
To illustrate the relation between ${\rm Y}_{10}$ and ${\rm Y}_{01},$ we note there are two morphisms between ${\rm Y}_{10}$ and ${\rm Y}_{01},$ which are called `essential Frobenius' as in \cite{Zho23}. Let $S$ be a locally Noetherian $\FF_{p^2}$-scheme $S.$

We first construct $Fr':{\rm Y}_{10}\ra {\rm Y}_{01}.$ For  any $S$-point $y=(A,\lambda,\eta, A',\lambda',\eta',\phi)$ satisfies $V\tcD(A')_{1}^{\circ}=\phi_{*,2}(\tcD(A)_{2}^{\circ})$ and $V\tcD(A')_{2}^{\circ}\subseteq\phi_{*,1}(\tcD(A)_{1}^{\circ}).$  
We take $\tcE_1=FV^{-1}\tcD(A)_{1}$ and $\tcE_{2}=FV^{-1}\tcD(A)_{2}^{\circ}.$
Applying Corollary \ref{CORP}, we get $(A'',\lambda'', \eta'')$ a point of ${\Sh}_{1,n-1}$ and an isogeny $\phi'': A\rightarrow A''$ such that $(A'',\lambda'', \eta'',A',\lambda', \eta',\phi\circ \phi''^{-1})$ is a point of ${{\Sh}}_{1,n-1}(K_{\gothp}^{1}).$ It can be checked that $(A'',\lambda'', \eta,A',\lambda', \eta',\phi\circ \phi''^{-1})\in {\rm Y}_{01}.$ Thus we let $Fr'(y)=(A'',\lambda'', \eta,A',\lambda', \eta',\phi'')$ and finish constructing $Fr'.$


Next, we construct $Fr'':{\rm {\rm Y_{01}}}\ra {\rm {\rm Y_{10}}}.$ For any $S$-point $y=(A,\lambda,\eta, A',\lambda',\eta',\phi)$ satisfies $p\tcD(A')_{2}^{\circ}\subseteq V\tcD(A)_{1}^{\circ}$ and $p\tcD(A')_{1}^{\circ}=V\tcD(A)_{2}^{\circ}.$ We take $\tcE_1=FV^{-1}\tcD(A')_{1}$ and $\tcE_{2}=FV^{-1}\tcD(A')_{2}^{\circ}.$
Applying Corollary \ref{CORP}, we get 
$(A'',\lambda'', \eta'')$ a point of ${\Sh}_{1,n-1}$ and an isogeny $\phi'': A'\ra A''$ such that
$(A,\lambda, \eta,A'',\lambda'', \eta'',\phi''\circ \phi)$ is a point of ${{\Sh}}_{1,n-1}(K_{\gothp}^{1}).$ It can be checked that $(A,\lambda, \eta,A'',\lambda'', \eta'',\phi''\circ \phi)\in {\rm Y}_{10}.$ Thus we let $Fr''(y)=(A,\lambda, \eta,A'',\lambda'', \eta'',\phi''\circ \phi)$ and finish constructing $Fr''.$

\end{construction}

It can be checked directly that $Fr'\circ Fr''=\Frob_{p^2}$ on ${\rm Y_{01}}$ and $Fr''\circ Fr'=\Frob_{p^2}$ on ${\rm Y_{10}}.$ Furthermore, for the action of $Fr'$ and $Fr'',$ we have the following proposition:

\begin{proposition}
\label{image}
    \begin{enumerate}
        \item For any $1\leq i\leq n-1,$ the morphism $Fr'$ induces a morphism from $\pr_{i}(C_{i})\bigcap {\rm Y}_{10}$ to $\pr_{i}(C_{i})\bigcap {\rm Y}_{01}.$ Moreover, it induces a morphism from $\pr_{n}(C_{n})\bigcap {\rm Y}_{10}={\rm Y}_{00}\bigcap {\rm Y}_{10}$ to ${\rm Y}_{00}\bigcap {\rm Y}_{01}$ and a morphism from $\pr_{n}(C_{n})\bigcap {\rm Y}_{11}={\rm Y}_{00}\bigcap {\rm Y}_{11}$ to itself.
        \item For any $1\leq i\leq n-1,$ the morphism $Fr''$ induces a morphism from $\pr_{i}(C_{i})\bigcap {\rm Y}_{01}$ to $\pr_{i}(C_{i})\bigcap {\rm Y}_{10}.$ It also induces a morphism from $\pr_{n}(C_{n})\bigcap {\rm Y}_{01}={\rm Y}_{00}\bigcap {\rm Y}_{01}$ to ${\rm Y}_{00}\bigcap {\rm Y}_{10}$ and a morphism from $\pr_{n}(C_{n})\bigcap {\rm Y}_{11}={\rm Y}_{00}\bigcap {\rm Y}_{11}$ to itself.
    \end{enumerate}
\end{proposition}
\begin{proof}
    It is just a direct check with Proposition~\ref{P:isom-Yj} and Proposition~\ref{G:stalk of points}.
\end{proof}

\section{Ihara lemma for \texorpdfstring{$n\geq 3$}{n3}
}\label{I3}
In this section we prove the Ihara lemma for $n\geq 3.$ Let $k$ be a finite extension of $\FF_{p^2}$ or ${\QQ}_{p^2}.$
For any proper scheme $X$ over $k,$ let $\overline{X}$ be the geometric fiber of $X.$
\begin{theorem}
\label{IIhara lemma}
Under the Hypothesis~\ref{Main hypo}, we have: 
    \begin{enumerate}
        \item (Definite Ihara) The map 
        \begin{equation*}
        {\mathrm{H}}_{\acute{e}t}^{0}(\overline{\Sh}_{0,n}(K_{\gothp}^{1}),k_\lambda)_{\mathfrak{m}} \xrightarrow{\psi} {\mathrm{H}^{0}_{\acute{e}t}}({\overline{\Sh}}_{0,n},k_\lambda)^{\oplus n}_{\mathfrak{m}}
        \end{equation*} is surjective, where $\psi=(\lp,\rp, \rp {\A}, \cdots, \rp {\A}^{n-2}).$         
        \item (Indefinite Ihara) The map
        \begin{equation*}
        {\rm H}_{\acute{e}t}^{2(n-1)}({\overline{Sh}}_{1,n-1}(K_{\gothp}^{1}),k_\lambda(n))_{\mathfrak{m}} \xrightarrow{\psi} {\rm H^{2(n-1)}_{\acute{e}t}}({\overline{Sh}}_{1,n-1},k_\lambda(n))^{\oplus n}_{\mathfrak{m}}
        \end{equation*} is surjective with $\overline{Sh}_{1,n-1}, \overline{Sh}_{1,n-1}(K_{\gothp}^{1})$ the geometric generic fibers of $\calS h_{1,n-1},\calS h_{1,n-1}(K_{\gothp}^{1})$ and $\psi=(\lp,\rp, \rp {\A}, \cdots, \rp {\rm A}^{n-2})$ defined similarly. 
    \end{enumerate}
\end{theorem}

The calculation in Proposition~\ref{cohomologyy} still works out to the geometric fiber.
We have the following diagram:
\[\begin{tikzcd}
[
    column sep=small, 
    row sep=scriptsize,
    cells={font=\footnotesize} 
]
	{\overline{{\rm Y}}_{00}\bigcap\overline{{\rm Y}}_{10}} & {} & {\overline{{\rm Y}}_{10}} & {} & {\overline{{\rm Y}}_{10}\bigcap \overline{\rm Y}_{11}} && {\overline{{\rm Y}}_{00}\bigcap \overline{\rm Y}_{11}} \\
	{\overline{\rm Y}_{n}} & {} & {\overline{{\Sh}}_{1,n-1}} && {\overline{N}_{10}} && {\overline{\rm Y}_{n}}
	\arrow[hook, from=1-1, to=1-3]
	\arrow[from=1-1, to=2-1]
	\arrow[from=1-3, to=2-3]
	\arrow[hook', from=1-5, to=1-3]
	\arrow[from=1-5, to=2-5]
	\arrow[hook', from=1-7, to=1-5]
	\arrow[from=1-7, to=2-7]
	\arrow[hook, from=2-1, to=2-3]
	\arrow[hook', from=2-5, to=2-3]
	\arrow[hook', from=2-7, to=2-5]
\end{tikzcd}\] where the vertical morphisms are induced by the blowing-up under ${\rm Bl}_{{\mathop{p}\limits^{\leftarrow}}_{n}(\overline{\rm Y}_{n})}{\overline{\Sh}}_{1,n-1}$ in Proposition~\ref{bloww up} and $\overline N_{10}$ is the complement of union of Newton strata of first slope equal to $0$ defined in Proposition~\ref{connections}.

Take $\overline{\mathcal{U}}_{10}=\overline{{\rm Y}}_{10}-\overline{{\rm Y}}_{10}\bigcap \overline{\rm Y}_{11},$ we have the following excision exact sequence:
\[\begin{tikzcd}
[
    column sep=small, 
    row sep=scriptsize,
    cells={font=\footnotesize} 
]
	& \cdots & {{\rm H}^{2n-5}(\overline{{\rm Y}}_{10}\bigcap \overline{\rm Y}_{11},k_\lambda)} \\
	{{\rm H}^{2n-4}_{c}(\overline{\mathcal{U}}_{10},k_\lambda)} & {{\rm H}^{2n-4}(\overline{{\rm Y}}_{10},k_\lambda)} & {{\rm H}^{2n-4}(\overline{{\rm Y}}_{10}\bigcap \overline{\rm Y}_{11},k_\lambda)} \\
	{{\rm H}^{2n-3}_{c}(\overline{\mathcal{U}}_{10},k_\lambda)} & \cdots
	\arrow[from=1-2, to=1-3]
	\arrow[from=1-3, to=2-1]
	\arrow[from=2-1, to=2-2]
	\arrow["Res", from=2-2, to=2-3]
	\arrow["\partial"', from=2-3, to=3-1]
	\arrow[from=3-1, to=3-2]
\end{tikzcd}\]

Therefore the sequence ${{\rm H}^{2n-4}(\overline{{\rm Y}}_{10},k_\lambda)} \xra{Res} {{\rm H}^{2n-4}(\overline{{\rm Y}}_{10}\bigcap \overline{\rm Y}_{11},k_\lambda)}\xra{\partial} {{\rm H}^{2n-3}_{c}(\overline{\mathcal{U}}_{10},k_\lambda)}$ is exact. Moreover, the closed immersion of $\overline{{\rm Y}}_{00}\bigcap \overline{\rm Y}_{11}$ and $\pr_{i}(\overline{C}_{i})\bigcap \overline{{\rm Y}}_{10}$ into $\overline{{\rm Y}}_{10}\bigcap\overline{\rm Y}_{11}$ induces Gysin maps \[{\rm H}^{2n-6}_{c}(\overline{{\rm Y}}_{00}\bigcap \overline{\rm Y}_{11},k_\lambda(n-3))\oplus \bigoplus\limits_{i=1}^{n-1} {\rm H}^{0}(\overline{C}_{i}\bigcap \overline{{\rm Y}}_{10},k_\lambda)\xra{Gys} {\rm H}^{2n-4}(\overline{{\rm Y}}_{10}\bigcap \overline{\rm Y}_{11},k_\lambda(n-2)).\] Putting together, we have the following diagram with the vertical sequence exact:
\[\begin{tikzcd}
[
    column sep=small, 
    row sep=scriptsize,
    cells={font=\footnotesize} 
]
	{{\rm H}^{2n-3}_{c}(\overline{\mathcal{U}}_{10},k_\lambda(n-2))} \\
	{{\rm H}^{2n-4}(\overline{{\rm Y}}_{10}\bigcap \overline{\rm Y}_{11},k_\lambda(n-2))} & {H^{(2n-6)}(\overline{{\rm Y}}_{00}\bigcap \overline{\rm Y}_{11},k_\lambda(n-3))\oplus \bigoplus\limits_{i=1}^{n-1} {\rm H}^{0}(\overline{C}_{i}\bigcap \overline{{\rm Y}}_{10},k_\lambda)} \\
	{{\rm H}^{2n-4}(\overline{{\rm Y}}_{10},k_\lambda(n-2))}
	\arrow["\partial", from=2-1, to=1-1]
	\arrow["{{{Gys,\pr_{i,!}}}}", from=2-2, to=2-1]
	\arrow["Res", from=3-1, to=2-1]
\end{tikzcd}\]
\begin{proposition}
    The map $\partial\circ \text{Gys}: \bigoplus\limits_{i=1}^{n-1} {\rm H}^{0}_{\acute{e}t}(\overline{C}_{i}\bigcap \overline{{\rm Y}}_{10},k_\lambda)_{\mathfrak{m}} \ra {{\rm H}^{2n-3}_{c}(\overline{\mathcal{U}}_{10},k_\lambda(n-2))}_{\mathfrak{m}}$ is injective.
\end{proposition}
\begin{proof}
    Consider the closed immersion of $\overline{\rm Y}_{00}\bigcap\overline{\rm Y}_{10}\hookrightarrow \overline{\rm Y}_{10},$ we have the following diagram
    \[\begin{tikzcd}
	{{\rm H}^{2n-4}_{\acute{e}t}(\overline{{\rm Y}}_{10},k_\lambda(n-2))} && {{\rm H}^{2n-4}_{\acute{e}t}(\overline{{\rm Y}}_{00}\bigcap \overline{{\rm Y}}_{10},k_\lambda(n-2))} \\
	{{\rm H}^{2n-6}_{\acute{e}t}(\overline{{\rm Y}}_{00}\bigcap \overline{{\rm Y}}_{10},k_\lambda(n-3))}
	\arrow["Res", from=1-1, to=1-3]
	\arrow["Gys", from=2-1, to=1-1]
	\arrow["{{\cup c_1(\mathcal{N}_{\overline{{\rm Y}}_{00}\bigcap\overline{{\rm Y}}_{10}/\overline{{\rm Y}}_{10}})}}"', from=2-1, to=1-3]
\end{tikzcd}\] Since $\overline{\rm Y}_{00}\bigcap\overline{\rm Y}_{10}$ is the exceptional divisor of $\overline{\rm Y}_{10},$ the cup product with $c_1(\mathcal{N}_{\overline{{\rm Y}}_{00}\bigcap\overline{{\rm Y}}_{10}/\overline{{\rm Y}}_{10}})$ is an injection and the image is the same as the image of ${{\rm H}^{2n-4}_{\acute{e}t}(\overline{{\rm Y}}_{10},k_\lambda(n-2))}$ in ${{\rm H}^{2n-4}_{\acute{e}t}(\overline{{\rm Y}}_{00}\bigcap \overline{{\rm Y}}_{10},k_\lambda(n-2))}.$ Since $\overline{\rm Y}_{10}\simeq {\rm Bl}_{{\rm Y}_{n}}\Sh_{1,n-1},$ we have ${\rm H}^{2n-4}_{\acute{e}t}(\overline{\rm Y}_{10},k_\lambda(n-2))_\mathfrak{m}={\rm H}^0_{\acute{e}t}(\overline{\Sh}_{0,n},k_\lambda)_\mathfrak{m}^{\oplus n-2}.$
Thus by the commutativity of the following diagram: \[\begin{tikzcd}
[
    column sep=small, 
    row sep=scriptsize,
    cells={font=\footnotesize} 
]
	{{\rm H}^{2n-3}_{c}(\overline{\mathcal{U}}_{10},k_\lambda(n-2))}_{\mathfrak{m}} \\
	{{\rm H}^{2n-4}_{\acute{e}t}(\overline{{\rm Y}}_{10}\bigcap \overline{\rm Y}_{11},k_\lambda(n-2))_{\mathfrak{m}}} & {{\rm H}_{\acute{e}t}^{2n-6}(\overline{{\rm Y}}_{00}\bigcap \overline{\rm Y}_{11},k_\lambda(n-3))_{\mathfrak{m}}\oplus \bigoplus\limits_{i=1}^{n-1} {\rm H}_{\acute{e}t}^{0}(\overline{C}_{i}\bigcap \overline{{\rm Y}}_{10},k_\lambda)_{\mathfrak{m}}} \\
	{{\rm H}^{2n-4}_{\acute{e}t}(\overline{{\rm Y}}_{10},k_\lambda(n-2))_{\mathfrak{m}}} & {{\rm H}^{2n-6}_{\acute{e}t}(\overline{{\rm Y}}_{00}\bigcap \overline{{\rm Y}}_{10},k_\lambda(n-3))_{\mathfrak{m}}}
	\arrow["\partial", from=2-1, to=1-1]
	\arrow["Gys", from=2-2, to=2-1]
	\arrow["Res", from=3-1, to=2-1]
	\arrow["Res", from=3-2, to=2-2]
	\arrow["Gys", from=3-2, to=3-1]
\end{tikzcd}.\] Therefore, the image of ${{\rm H}^{2n-4}_{\acute{e}t}(\overline{{\rm Y}}_{10},k_\lambda(n-2))_{\mathfrak{m}}}$ in ${{\rm H}_{\acute{e}t}^{2n-4}(\overline{{\rm Y}}_{10}\bigcap \overline{\rm Y}_{11},,k_\lambda(n-2))_{\mathfrak{m}}}$ induced by the restriction map is contained in the image of ${\rm H}^{2n-6}_{\acute{e}t}(\overline{{\rm Y}}_{00}\bigcap \overline{\rm Y}_{11},,k_\lambda(n-3))_{\mathfrak{m}}$ induced by the Gysin map. 

Now, by Proposition~\ref{connections} and Theorem~\ref{T:main-theorem}, we have the following commutative diagram: \[\begin{tikzcd}
[
    column sep=small, 
    row sep=scriptsize,
    cells={font=\footnotesize} 
]
	{{\rm H}^{2n-2}_{\acute{e}t}(\overline{{\Sh}}_{1,n-1},k_\lambda(n-1))_{\mathfrak{m}}} & {{\rm H}^{2n-2}_{\acute{e}t}(\overline{{\rm Y}}_{10},k_\lambda(n-1))_{\mathfrak{m}}} \\
	& {{\rm H}^{2n-4}_{\acute{e}t}(\overline{{\rm Y}}_{00}\bigcap \overline{{\rm Y}}_{10},k_\lambda(n-2))_{\mathfrak{m}}} \\
	\\
	{{\rm H}^{0}(\overline{\rm Y}_{n})_{\mathfrak{m}}\oplus\bigoplus\limits_{i=1}^{n-1}{\rm H}^{0}({\rm Y}_{i})_{\mathfrak{m}}} & {{\rm H}^{2n-6}_{\acute{e}t}(\overline{{\rm Y}}_{00}\bigcap \overline{\rm Y}_{11},k_\lambda(n-3))_{\mathfrak{m}}\oplus\bigoplus\limits_{i=1}^{n-1}{\rm H}^{0}_{\acute{e}t}(\overline{C}_{i}\bigcap \overline{{\rm Y}}_{10},k_\lambda)_{\mathfrak{m}}}
	\arrow["\delta"', from=1-2, to=1-1]
	\arrow["Gys"', from=2-2, to=1-2]
	\arrow["{{{{{{\mathop{p}\limits^{\rightarrow}}_{i,!}}}}}}", from=4-1, to=1-1]
	\arrow["{{(Gys,pr_{i,!})}}"', from=4-2, to=2-2]
	\arrow["{{{(\alpha,\beta)}}}", from=4-2, to=4-1]
\end{tikzcd}\]where the vertical maps are Gysin maps induced by the blowing-up. It is easy to see  ${\rm H}^{0}_{\acute{e}t}(\overline{C}_{i}\bigcap \overline{{\rm Y}}_{10},k_\lambda)_{\mathfrak{m}}={\rm H}^{0}(\overline{\rm Y}_{i},k_\lambda)_{\mathfrak{m}}.$ If there is $(x,y)\in {{\rm H}^{2n-3}_{c}(\overline{{\rm Y}}_{00}\bigcap \overline{\rm Y}_{11},k_\lambda(n-3))_{\mathfrak{m}}\oplus\bigoplus\limits_{i=1}^{n-1}{\rm H}^{0}_{\acute{e}t}(\overline{C}_{i}\bigcap \overline{{\rm Y}}_{10},k_\lambda)_{\mathfrak{m}}}$ such that the image of $(x,y)$ through the horizontal map is zero, then ${{{{\mathop{p}\limits^{\rightarrow}}_{i,!}}}}\circ(\alpha(x),\beta(y))=0$ for any $1\leq i\leq n.$ By the injectivity of ${{{{\mathop{p}\limits^{\rightarrow}}_{i,!}}}},$ we see $\beta(y)=0.$ Hence $y=0.$ This shows the image of $\bigoplus\limits_{i=1}^{n-1}{\rm H}^{0}_{\acute{e}t}({\rm Y}_{i},k_\lambda)_{\mathfrak{m}}$ in ${{\rm H}^{2n-4}_{\acute{e}t}(\overline{{\rm Y}}_{10}\bigcap \overline{\rm Y}_{11},k_\lambda(n-2))}_{\mathfrak{m}}$ has trivial intersection with the image of ${{\rm H}^{2n-4}_{\acute{e}t}(\overline{{\rm Y}}_{10},k_\lambda)}_{\mathfrak{m}}.$ Thus we finish the proof.
\end{proof}

Symmetrically, we define $\overline{\mathcal{U}}_{01}=\overline{\rm Y}_{01}-\overline{\rm Y}_{01}\bigcap\overline{\rm Y}_{11}.$ By Proposition~\ref{image}, we have $Fr''(\overline{\mathcal{U}}_{01})$ contains $\overline{C}_{i}\bigcap \overline{{\rm Y}}_{10}$ for $1\leq i\leq n.$ Thus we have a simlar diagram as above:
\[\begin{tikzcd}
[
    column sep=small, 
    row sep=scriptsize,
    cells={font=\footnotesize} 
]
	{\overline{{\rm Y}}_{00}\bigcap\overline{{\rm Y}}_{10}} & {} & {\overline{{\rm Y}}_{10}} & {} & {\overline{{\rm Y}}_{10}\backslash{Fr''(\overline{\mathcal{U}}_{01})}} && {\overline{{\rm Y}}_{00}\bigcap \overline{\rm Y}_{11}} \\
	{\overline{\rm Y}_{n}} & {} & {\overline{{\Sh}}_{1,n-1}} && \overline{N}_{01} && {\overline{\rm Y}_{n}}
	\arrow[hook, from=1-1, to=1-3]
	\arrow[from=1-1, to=2-1]
	\arrow[from=1-3, to=2-3]
	\arrow[hook', from=1-5, to=1-3]
	\arrow[from=1-5, to=2-5]
	\arrow[hook', from=1-7, to=1-5]
	\arrow[from=1-7, to=2-7]
	\arrow[hook, from=2-1, to=2-3]
	\arrow[hook', from=2-5, to=2-3]
	\arrow[hook', from=2-7, to=2-5]
\end{tikzcd}.\] we also have the following diagram with the vertical sequence exact:
\[\begin{tikzcd}
[
    column sep=small, 
    row sep=scriptsize,
    cells={font=\footnotesize} 
]
	{{\rm H}^{2n-3}_{c}(Fr''(\overline{\mathcal{U}}_{01}),k_\lambda(n-2))} \\
	{{\rm H}^{2n-4}_{\acute{e}t}({\overline{{\rm Y}}_{10}\backslash{Fr''(\overline{\mathcal{U}}_{01})}},k_\lambda(n-2))} & {H^{2n-6}_{\acute{e}t}(\overline{{\rm Y}}_{00}\bigcap \overline{\rm Y}_{11},k_\lambda(n-3))\oplus \bigoplus\limits_{i=1}^{n-1} {\rm H}^{0}_{\acute{e}t}(\overline{C}_{i}\bigcap \overline{{\rm Y}}_{10},k_\lambda)} \\
	{{\rm H}^{2n-4}_{\acute{e}t}(\overline{{\rm Y}}_{10},k_\lambda(n-2))}
	\arrow["\partial", from=2-1, to=1-1]
	\arrow["{{Gys,pr_{i,!}}}", from=2-2, to=2-1]
	\arrow["Res", from=3-1, to=2-1]
\end{tikzcd}\] 
Similar as above, we have the following proposition:
\begin{proposition}
The map $\partial\circ \text{Gys}: \bigoplus\limits_{i=1}^{n-1} {\rm H}^{0}_{\acute{e}t}(\overline{C}_{i}\bigcap \overline{{\rm Y}}_{10},k_\lambda)_{\mathfrak{m}} \ra {{\rm H}^{2n-3}_{c}(Fr''(\overline{\mathcal{U}}_{01}),k_\lambda(n-2))_{\mathfrak{m}}}$ is injective.    
\end{proposition}

Consider $Fr''(\overline{\mathcal{U}}_{01})\bigcup \overline{\mathcal{U}}_{10}$ with the Mayer-Vietoris sequence, we have the following diagram:
\[\begin{tikzcd}
[
    column sep=small, 
    row sep=scriptsize,
    cells={font=\footnotesize} 
]
	&& {{{\rm H}^{2n-3}_{c}(Fr''(\overline{\mathcal{U}}_{01})\bigcup \overline{\mathcal{U}}_{10},k_\lambda(n-2))_{\mathfrak{m}}}} \\
	{\bigoplus\limits_{i=1}^{n-1} {\rm H}^{0}_{\acute{e}t}(\overline{C}_{i}\bigcap \overline{{\rm Y}}_{10},k_\lambda)_{\mathfrak{m}}} && {{{\rm H}^{2n-3}_{c}(Fr''(\overline{\mathcal{U}}_{01}),k_\lambda(n-2))_{\mathfrak{m}}}\bigoplus{{\rm H}^{2n-3}_{c}(\overline{\mathcal{U}}_{10},k_\lambda(n-2))_{\mathfrak{m}}}} \\
	&& {{{\rm H}^{2n-3}_{c}(Fr''(\overline{\mathcal{U}}_{01})\bigcap \overline{\mathcal{U}}_{10},k_\lambda(n-2))_{\mathfrak{m}}}}
	\arrow["\Phi", from=2-1, to=2-3]
	\arrow["\phi"', dashed, from=2-1, to=3-3]
	\arrow[from=2-3, to=1-3]
	\arrow["i", from=3-3, to=2-3]
\end{tikzcd}\] Since two parts of the last horizontal map has different signs, the composition of the last horizontal map with $\Phi$ is zero. Hence $\Phi$ factors through the fisrt horizontal map: $\Phi=i\circ \phi.$ Thus $\phi$ is injective.

We claim that $\phi$ gives the desired map appeared in Theorem~\ref{IIhara lemma}. First, we note that the morphism from $\overline{C}_{i}$ to ${\rm Y}_{i}$ induced by the blowing-up for $1\leq i\leq n-1$ gives the isomorphism ${\rm H}^{0}_{\acute{e}t}(\overline{C}_{i}\bigcap \overline{{\rm Y}}_{10},k_\lambda)={\rm H}^{0}_{\acute{e}t}({\rm Y}_{i},k_\lambda).$ Since ${\rm H}^{0}_{\acute{e}t}({\rm Y}_{i},k_\lambda)_{\mathfrak{m}}={\rm H}^{0}_{\acute{e}t}(\overline{{\Sh}}_{0,n},k_\lambda)_{\mathfrak{m}},$ we have ${\rm H}^{0}_{\acute{e}t}(\overline{C}_{i}\bigcap \overline{{\rm Y}}_{10},k_\lambda)_{\mathfrak{m}}={\rm H}^{0}_{\acute{e}t}(\overline{{\Sh}}_{0,n},k_\lambda)_{\mathfrak{m}}.$ Moreover, $Fr''(\overline{\mathcal{U}}_{01})\bigcup \overline{\mathcal{U}}_{10}=V^{(\omega_{1},\omega_{2})}\bigsqcup (\overline{{\rm Y}}_{00}\bigcap \overline{{\rm Y}}_{10}-\overline{{\rm Y}}_{00}\bigcap \overline{\rm Y}_{11}),$ with $(\omega_{1},\omega_{2})=(n,1),$ that is, $V^{(\omega_{1},\omega_{2})}$ is the $\mu$-ordinary locus of $\overline{{\Sh}}_{1,n-1}.$ By Proposition~\ref{affineness}, $V^{(\omega_{1},\omega_{2})}$ is affine. Thus this shows ${\rm H}^{2n-3}_{c}(V^{(\omega_{1},\omega_{2})},k_\lambda)_\mathfrak{m}=0.$ This gives an injective map $i:{{\rm H}^{2n-3}_{c}(Fr''(\overline{\mathcal{U}}_{01})\bigcap \overline{\mathcal{U}}_{10},k_\lambda)_{\mathfrak{m}}}\ra {{\rm H}^{2n-3}_{c}(\overline{{\rm Y}}_{00}\bigcap \overline{{\rm Y}}_{10}-\overline{{\rm Y}}_{00}\bigcap \overline{\rm Y}_{11},k_\lambda)_{\mathfrak{m}}}$ by excision sequence. Composite these maps together, the following map is injecitve:
\[\begin{tikzcd}
[
    column sep=small, 
    row sep=scriptsize,
    cells={font=\footnotesize} 
]
	{\bigoplus\limits_{i=1}^{n-1} {\rm H}^{0}_{\acute{e}t}(\overline{{\Sh}}_{0,n},k_\lambda)_{\mathfrak{m}}} & {\bigoplus\limits_{i=1}^{n-1} {\rm H}^{0}_{\acute{e}t}(\overline{C}_{i}\bigcap \overline{{\rm Y}}_{10},k_\lambda)_{\mathfrak{m}}} & {{{\rm H}^{2n-3}_{c}(Fr''(\overline{\mathcal{U}}_{01})\bigcap \overline{\mathcal{U}}_{10},k_\lambda(n-2))_{\mathfrak{m}}}} \\
	&& {{{\rm H}^{2n-3}_{c}(\overline{{\rm Y}}_{00}\bigcap \overline{{\rm Y}}_{10}-\overline{{\rm Y}}_{00}\bigcap \overline{\rm Y}_{11},k_\lambda(n-2))_{\mathfrak{m}}}}
	\arrow[from=1-1, to=1-2]
	\arrow["\phi", from=1-2, to=1-3]
	\arrow["i", from=1-3, to=2-3]
\end{tikzcd}\] where the last map is an identity following Proposition~\ref{cohomologyy}. 

Recall our calculation in Proposition~\ref{cohomologyy}, 
${\rm H}^{2n-3}_{c}(\overline{{\rm Y}}_{00}\bigcap \overline{{\rm Y}}_{10}-\overline{{\rm Y}}_{00}\bigcap \overline{\rm Y}_{11},k_\lambda(n-2))_{\mathfrak{m}}$ can be identified with a direct summand of ${\rm H}^{2n-4}_{c}(\overline{{\rm Y}}_{00}\bigcap \overline{\rm Y}_{11},k_\lambda(n-2))_{\mathfrak{m}}.$ Considering the direct sum with the map  $\phi:{\rm H}^0_{\acute{e}t}(\overline{\Sh}_{0,n},k_\lambda)_\mathfrak{m}\ra {\rm H}^0_{\acute{e}t}(\overline{\Sh}_{0,n}(K_\gothp^1),k_\lambda)_{\mathfrak{m}}$ in Proposition~\ref{cohomologyy}, we get an injective map
\begin{equation}\label{Ihara constructed}
    {\bigoplus\limits_{i=1}^{n} {\rm H}^{0}_{\acute{e}t}(\overline{{\Sh}}_{0,n},k_\lambda)_{\mathfrak{m}}}\ra {{\rm H}^{0}_{\acute{e}t}(\overline{{\Sh}}_{0,n}(K_{\gothp}^{1}),k_\lambda)}_{\mathfrak{m}}
\end{equation}
We denote the map by $\Psi.$ Now it suffices to check the relation of $\Psi$ with the dual map of that in Theorem~\ref{IIhara lemma}. 

The dual map of that in Theorem~\ref{IIhara lemma} is 
$${\mathrm{H}^{0}}({{\Sh}}_{0,n},k_\lambda)^{\oplus n} \xrightarrow{\alpha} {\mathrm{H}}^{0}({\Sh_{0,n}(K_\gothp^1)},k_\lambda) ,
$$
where $\alpha$ can be expressed as 
$$
\begin{pmatrix}
    {\rm T}_{0,1}^{(1)}&{\rm T}_{0,1}^{(2)}&\cdots&{\rm T}_{0,1}^{(n)}\\
\end{pmatrix}
$$ Here for $1\leq i\leq n,$ ${\rm T}_{0,1}^{(i)}$ are Hecek actions defined in Definition~\ref{S:Hecke action}.

As in the proof of Proposition~\ref{cohomologyy}, $\overline{\rm Y}_{00}\bigcap\overline{\rm Y}_{10}$ and $\overline{\rm Y}_{00}\bigcap\overline{\rm Y}_{11}$ can be realized as a closed subscheme of product of Grassmannians. With the same notation here, denote by $\overline T$ the closed subscheme of $\overline{\rm Y}_{00}\bigcap\overline{\rm Y}_{11}$ such that $\mathcal{H}_1=\mathcal{H}_1^{(p^2)}.$
By construction, $\Psi$ is the direct sum of 
\begin{equation}\label{Ihara i}
    \iota_i:{\rm H}^0_{\acute{e}t}(\overline C_i\bigcap\overline{\rm Y}_{10},k_\lambda)_\mathfrak{m}\ra {\rm H}^{2n-4}_{\acute{e}t}(\overline {\rm Y}_{10}\bigcap\overline{\rm Y}_{11},k_\lambda)_\mathfrak{m}\ra  {\rm H}^{2n-4}_{\acute{e}t}(\overline {\rm Y}_{00}\bigcap\overline{\rm Y}_{11},k_\lambda)_\mathfrak{m}\ra {\rm H}^{2n-4}_{\acute{e}t}(\overline{T},k_\lambda)_\mathfrak{m}
\end{equation}
with the first map induced by closed immersions and the last two maps induced by restrictions for $1\leq i\leq n-1$ and 
\begin{equation}\label{Ihara n}
    \iota_n:{\rm H}^0_{\acute{e}t}(\overline{\Sh}_{0,n},k_\lambda)_\mathfrak{m}\ra {\rm H}^{2n-4}_{\acute{e}t}(\overline{\rm Y}_{00}\bigcap\overline{\rm Y}_{10},k_\lambda)_\mathfrak{m}\ra {\rm H}^{2n-4}_{\acute{e}t}(\overline {\rm Y}_{00}\bigcap\overline{\rm Y}_{11},k_\lambda)_\mathfrak{m}\ra {\rm H}^{2n-4}_{\acute{e}t}(\overline{T},k_\lambda)_\mathfrak{m}
\end{equation}
as in Proposition~\ref{cohomologyy} with the last two maps induced by restrictions.

Note that $\overline T$ is a $\PP^{n-2}$-bundle over $\overline{\Sh}_{0,n}(K_\gothp^1)$ and thus ${\rm H}^{2n-4}_{\acute{e}t}(\overline{T},k_\lambda)_\mathfrak{m}={\rm H}^0_{\acute{e}t}(\overline{\Sh}_{0,n}(K_\gothp^1),k_\lambda)_\mathfrak{m}.$ For any point $(z_1,z_2)\in \overline{\Sh}_{0,n}(K_\gothp^1),$ denote by $f_{(z_1,z_2)}$ the function mapping $(z_1,z_2)$ to $1$ and other points to $0.$ Then $\{f_{(z_1,z_2)}~|~(z_1,z_2)\in \overline{\Sh}_{0,n}(K_\gothp^1)\}$ forms a base of ${\rm H}^{2n-4}_{\acute{e}t}(\overline{T},k_\lambda)_\mathfrak{m}.$

The map $\iota_n$ can be described as in the proof of Proposition~\ref{cohomologyy}. The map $\iota_i$ can be described as the following diagram:
\[\begin{tikzcd}
[
    column sep=small, 
    row sep=scriptsize,
    cells={font=\footnotesize} 
]
	{{\rm Ch}^0(\overline{C}_{i}\bigcap\overline{\rm Y}_{10},k_\lambda)_\mathfrak{m}} & {{\rm Ch}^{n-2}(\overline{\rm Y}_{11}\bigcap\overline{\rm Y}_{10},k_\lambda)_\mathfrak{m}} & {{\rm Ch}^{n-2}(\overline{\rm Y}_{00}\bigcap\overline{\rm Y}_{11},k_\lambda)_\mathfrak{m}} & {{\rm Ch}^{n-2}_{\acute{e}t}(\overline{T},k_\lambda)_\mathfrak{m}} \\
	{{\rm H}^0_{\acute{e}t}(\overline{C}_{i}\bigcap\overline{\rm Y}_{10},k_\lambda)_\mathfrak{m}} & {{\rm H}^{2n-4}_{\acute{e}t}(\overline{\rm Y}_{11}\bigcap\overline{\rm Y}_{10},k_\lambda)_\mathfrak{m}} & {{\rm H}^{2n-4}_{\acute{e}t}(\overline {\rm Y}_{00}\bigcap\overline{\rm Y}_{11},k_\lambda)_\mathfrak{m}} & {{\rm H}^{2n-4}_{\acute{e}t}(\overline{T},k_\lambda)_\mathfrak{m}}
	\arrow[from=1-1, to=1-2]
	\arrow["{cl_{\overline{\Sh}_{0,n}}^0}", from=1-1, to=2-1]
	\arrow[from=1-2, to=1-3]
	\arrow["{cl_{\overline{\rm Y}_{11}\bigcap\overline{\rm Y}_{10}}^{n-2}}", from=1-2, to=2-2]
	\arrow[from=1-3, to=1-4]
	\arrow["{cl_{\overline{\rm Y}_{00}\bigcap\overline{\rm Y}_{11}}^{n-2}}", from=1-3, to=2-3]
	\arrow["{cl_{\overline{T}}^{n-2}}", from=1-4, to=2-4]
	\arrow[from=2-1, to=2-2]
	\arrow[from=2-2, to=2-3]
	\arrow[from=2-3, to=2-4].
\end{tikzcd}\]
Then $\iota_i$ maps $1$ to $cl_{\overline{T}}^{n-2}([\overline{C}_{i}\bigcap\overline{\rm Y}_{10}\bigcap\overline{T}]).$ Thus $\Psi$ can be expressed as 
$$
\begin{pmatrix}
    c_1{\rm T}_{0,1}^{(1)}&c_2{\rm T}_{0,1}^{(2)}&\cdots&c_n{\rm T}_{0,1}^{(n)}\\
\end{pmatrix}
$$
with $c_i$ coefficients which can be calculated by intersection theory but we do not need here. Since $\Psi$ is injective, $c_i$ is prime to $\ell$ for $1\leq i\leq n$ and thus after an appropriate base change, the map $\Psi$ is identified with $\alpha.$ This shows $\alpha$ is injective. This finish the proof of the definite Ihara lemma in Theorem~\ref{IIhara lemma}.


Now we give the proof of the Indefinite Ihara lemma.
Recall the prime $p'$ in Section~\ref{S:Hecke algebra} and $K,K'$ defined in Section~\ref{S:defn of Shimura var}.
we use notation `$(K)$' and `$(K')$' to strengthen the level structure of Shimura varieties since we need to consider Shimura varieties for different primes. We use subscript $p'$ to denote Shimura varieties defined at $p'.$

By Proposition~\ref{S:mod l cohomology} and Hypothesis~\ref{Main hypo}, we get:
\begin{align*}
{\rm H}^{0}_{\acute{e}t}(\overline{\Sh}_{0,n,p'},k_{\ell})_{\mathfrak{m}}&={\bar M'_{K}};\\
{\rm H}^{2(n-1)}_{\acute{e}t}(\overline{\Sh}_{1,n-1,p'},k_{\ell})_{\mathfrak{m}}&=\bar M'_{K}\otimes_{k_\lambda}\big(({\bar\rho_{\mathfrak m}}\otimes_{k_\lambda}\wedge^{n-1}{\bar\rho_{\mathfrak m}})\otimes_{k_\lambda} {k}_\lambda(\frac{(n-1)(n-2)}{2})\big);
\end{align*}
such that ${\bar M'_{K}}$ is a $k_\lambda$-module with trival $\Gamma_E$ action.
Here we use subscript $p'$ to denote the Shimura varieties defined at $p'.$

By Proposition~\ref{S:mod l cohomology}, we get 
\begin{align*}
{\rm H}^{0}_{\acute{e}t}(\overline{\Sh}_{0,n,p'}(K'),k_{\ell})_{\mathfrak{m}}&={\bar M'_{K'}};\\
{\rm H}^{2(n-1)}_{\acute{e}t}(\overline{\Sh}_{1,n-1,p'}(K'),k_{\ell})_{\mathfrak{m}}&=\bar M'_{K'}\otimes_{k_\lambda}\big(({\bar\rho_{\mathfrak m}}\otimes_{k_\lambda}\wedge^{n-1}{\bar\rho_{\mathfrak m}})\otimes_{k_\lambda} {k}_\lambda(\frac{(n-1)(n-2)}{2})\big);
\end{align*}
such that ${\bar N'_{K'}}$ is a $k_\lambda$-module with trival $\Gamma_E$ action.

By definite Ihara lemma and smooth proper base change theorem, we have a surjection \[{\rm H}^{0}_{\acute{e}t}(\overline{Sh}_{0,n}(K')_{\overline{\QQ}_{p}},k_\lambda)_{\mathfrak{m}}\ra {\rm H}^{0}_{\acute{e}t}(\overline{Sh}_{0,n}(K)_{\overline{\QQ}_{p}},k_\lambda)_{\mathfrak{m}}^{n}.\] Under the isomorphism $\overline{\QQ}_{p}\simeq \CC\simeq \overline{\QQ}_{p'},$ we have a surjection \[{\rm H}^{0}_{\acute{e}t}(\overline{Sh}_{0,n}(K')_{\overline{\QQ}_{p'}},k_\lambda)_{\mathfrak{m}}\ra {\rm H}^{0}_{\acute{e}t}(\overline{Sh}_{0,n}(K)_{\overline{\QQ}_{p'}},k_\lambda)_{\mathfrak{m}}^{n}.\] Again by torion-freeness and smooth proper base change, we get \[{\rm H}^{0}_{\acute{e}t}(\overline{\Sh}_{0,n,p'}(K'),k_\lambda)_{\mathfrak{m}}\ra {\rm H}^{0}_{\acute{e}t}(\overline{Sh}_{0,n,p'}(K),k_\lambda)_{\mathfrak{m}}^{n}.\] Then by torsion-freeness, we get a surjection $\bar M'_{K}\ra (\bar M'_{K'})^{\oplus n}.$ Hence 
$$\bar M'_{K}\otimes_{k_\lambda}\big(({\bar\rho_{\mathfrak m}}\otimes_{k_\lambda}\wedge^{n-1}{\bar\rho_{\mathfrak m}})\otimes_{k_\lambda} {k}_\lambda(\frac{(n-1)(n-2)}{2})\big)\ra \bigg(\bar M'_{K'}\otimes_{k_\lambda}\big(({\bar\rho_{\mathfrak m}}\otimes_{k_\lambda}\wedge^{n-1}{\bar\rho_{\mathfrak m}})\otimes_{k_\lambda} {k}_\lambda(\frac{(n-1)(n-2)}{2})\big)\bigg)^{\oplus n}$$ 
is surjective. After proper base change and a proper Tate twist, we get the indefinite Ihara lemma since the geometric generic fibers at $p$ and $p'$ are isomorphic.

\begin{corollary}
\label{level raising resultt}
    Fixed an unramified RACSDC representation $\Pi$ satisfies Hypothesis \ref{Main hypo} and \ref{Main hypo3}. Then there exists an irreducible representation $\Pi'$ of $\GL_n(\AAA_E)$ such that the associated Galois representation $\rho_{\Pi'}$ is residually isomorphic to $\rho_\Pi$ and the monodromy operator of $\rho_{\Pi'}$ is conjugate to $\begin{pmatrix}
        1&1\\
        0&1
    \end{pmatrix} \oplus 1_{n-2}.$
\end{corollary}
\begin{proof}
    We only give a sketch of proof here. Denote by $\psi': {\mathrm{H}^{0}_{\acute{e}t}}({\overline{\Sh}}_{0,n},k_\lambda)^{\oplus n}_{\mathfrak{m}}\rightarrow {\mathrm{H}}_{\acute{e}t}^{0}(\overline{\Sh}_{0,n}(K_{\gothp}^{1}),k_\lambda)_{\mathfrak{m}}$ the dual of $\psi.$ The Corollary is equivalent to $\psi'$ is not surjective by Mackey theory and Proposition~\ref{S:level raising}. Considering $\psi\circ \psi'$ as an $n\times n$ matrix with each element a Hecke operator, it can be shown the determinant is zero under Hypothesis~\ref{Main hypo3}. And thus we get Corollary~\ref{level raising result}. 
\end{proof}

\section{Geometry of \texorpdfstring{${\Sh}_{1,2}(Iw_{\gothp})$}{Sh_1,n-1Iw}
}\label{GI}

Recall that the Hecke action $\T$ is related to the special fiber of the unitary Shimura variety $\Sh_{1,n-1}(K_\gothp^1)$ for any $n.$ There is also a relation between the Hecke action ${\rm A}$ defined in Definition~\ref{S:Hecke action} and unitary Shimura varieties. For our use, we only consider the $n=3$ case. Recall the open compact subgroup $K=K^p(\ZZ_p^\times,K_\gothp)$ of $G_{1,2}(\AAA_\infty)$ we fixed in Section~\ref{S:Shimura Varieties}, let $K^2=K^p(\ZZ_p^\times,Iw_\gothp)$ with $Iw_\gothp$ the standard Iwahori subgroup of $\GL_n(\QQ_{p^2}).$
\begin{definition}\label{G:Iwahori level}
    Let $\cSh_{a_{\bullet}}(Iw_\gothp)$ be the unitary Shimura variety defined over $\ZZ_{p^{2}}$ which represents the functor the takes a locally Noetherian $\ZZ_{p^{2}}$-scheme $S$ to the set of isomorphism classes of tuples $$(A_1,\lambda_1,\eta_1,A_2,\lambda_2,\eta_2,A_3,\lambda_3,\eta_3,\phi_{12},\phi_{23}),$$ where
 $(A_1,\lambda_1,\eta_1,A_2,\lambda_2,\eta_2,\phi_{12})$ and $(A_2,\lambda_2,\eta_2,A_3,\lambda_3,\eta_3,\phi_{23})$ are $S$-points of ${\cSh}_{a_{\bullet}}(K_\gothp^1)$ such that the cokernels of the maps
 \[
 \phi_{23,*,1}\circ\phi_{12,*,1}: {\rm H}^\dR_1(A_1/S)^\circ_1 \to {\rm H}^\dR_1(A_3/S)^\circ_1 \quad\textrm{and}\quad
\phi_{23,*,2}\circ\phi_{12,*,2}: {\rm H}^\dR_1(A_1/S)^\circ_2 \to {\rm H}^\dR_1(A_3/S)^\circ_2
\]
are both locally free $\cO_S$-modules of rank $2.$
\end{definition}
Similarly as above, we use $\Sh_{a_{\bullet}}(Iw_{\gothp})$ to express the special fiber of $\cSh_{a_{\bullet}}(Iw_\gothp)$ and $Sh_{a_{\bullet}}(Iw_{\gothp})$ to express the generic fiber.

Denote by ${{\mathop{p}\limits^{\rightarrow}}_{Iw}}$ and ${{\mathop{p}\limits^{\leftarrow}}_{Iw}}$ the morphisms mapping any $S$-point of $\Sh_{a_\bullet}(Iw_\gothp)$ to 
$(A_2,\lambda_2,\eta_2,A_3,\lambda_3,\eta_3,\phi_{23})$ and $(A_1,\lambda_1,\eta_1,A_2,\lambda_2,\eta_2,\phi_{12}).$ 
When $a_\bullet=(0,3),$ it is easy to see the correspondence $({{\mathop{p}\limits^{\leftarrow}}_{Iw}},{{\mathop{p}\limits^{\rightarrow}}_{Iw}})$ gives the Hecke action ${\rm A}$ defined in Definition~\ref{S:Hecke action}.
\begin{definition}\label{GI:Condition on closed subschemes}
    Let $(\mathcal{A}_1,\lambda_1,\eta_1,\mathcal{A}_2,\lambda_2,\eta_2,\mathcal{A}_3,\lambda_3,\eta_3,\phi_{12},\phi_{23})$ be the universal object of $\cSh_{1,2}(Iw_{\gothp}).$
    For $0\leq i,j\leq 2,$ let ${\rm Z}_{ij}$ be the locus of $\Sh_{1,n-1}(K_{\gothp}^{1})$ on which the universal object satisfies $(1.i)(2.j)$ in the following:
    \begin{itemize}
        \item 
        $(1.0)$ $\omega^{\circ}_{\mathcal{A}_2^{\vee},1}=\Ker(\phi_{23,*,1}),$
        $(1.1)$ $\omega^{\circ}_{\mathcal{A}_1^{\vee},1}=\Ker(\phi_{12,*,1}),$
        $(1.2)$ $\omega^{\circ}_{\mathcal{A}_3^{\vee},1}=\Im(\phi_{23,*,1}\circ\phi_{12,*,1}).$
        \item 
        $(2.0)$ $\Im(\phi_{23,*,2})=\omega^{\circ}_{\mathcal{A}_3^{\vee},2},$
        $(2.1)$ $\Im(\phi_{12,*,2})=\omega^{\circ}_{\mathcal{A}_2^{\vee},2},$  
        $(2.2)$ $\Ker(\phi_{23,*,2}\circ\phi_{12,*,2})=\omega^{\circ}_{\mathcal{A}_1^{\vee},2},$
    \end{itemize}
\end{definition}
\begin{proposition}\label{GI:stalk of points}
    \begin{enumerate}
        \item The scheme $\cSh_{1,2}(Iw_{\gothp})$ is quasi-projective over $\ZZ_{p^{2}}$ of dimension $4;$ and we have $${\Sh}_{1,n-1}(Iw_{\gothp})=\bigcup\limits_{0\leq i,j\leq 2}{\rm Z}_{ij}$$
        with points in $Z_{ij}$ satisfying $(1.i),(2.j);$
        \item For $0\leq i,j\leq 2,$ ${\rm Z}_{ij}$ is smooth over $\FF_{p^{2}}$ of dimension $4;$
        \item Let $k$ be a perfect field containing $\FF_{p^{2}}.$ $x$ be a closed point of $\Sh_{1,n-1}(Iw_{\gothp})(k).$ Let $S_j$ be the set of $i$ such that the universal object satisfies condition $(i,j)$ defined in Definition~\ref{G:Condition on closed subschemes} at $x$ Then the completed local ring of $\cSh_{1,n-1}(K_{\gothp}^{1})$ at $x$ is $$W(k)[\![X_i,i\in S_1;Y_j,j\in S_2;Z_k,k\in S_3;T_1,\dots,T_{r}]\!]/\mathcal{I}$$ with $r$ a positive integer to make the dimension of the local ring is $4$ and the ideal $\mathcal{I}$ satisfies for $0\leq i\leq 2,$
        \begin{itemize}
            \item if $i\in S_1\bigcap S_2\bigcap S_3,$ then $X_iY_iZ_i-p\in \mathcal{I};$ otherwise 
            \item if $i\in S_1\bigcap S_2,$ then $X_iY_i-p\in \mathcal{I};$
            \item if $i\in S_1\bigcap S_3,$ then $X_iZ_i-p\in \mathcal{I};$
            \item if $i\in S_2\bigcap S_3,$ then $Y_iZ_i-p\in \mathcal{I}.$
        \end{itemize}
    \end{enumerate}
\end{proposition}
\begin{proposition}\label{GI:blow up}
    After blowing up ${\cSh}_{1,2}(Iw_{\gothp})$ at ${\rm Z}_{11},{\rm Z}_{22}$ and ${\rm Z}_{00}$ successively, we can get a strictly semistable scheme, denoted by $\widetilde{\cSh}_{1,2}(Iw_\gothp).$ Denote the process of the resolution by $\cSh_{1,2}(Iw_\gothp)=V^0\xleftarrow{\pi_1} V^1\xleftarrow{\pi_2} V^2\xleftarrow{\pi_3} V^3.$
    For $1\leq i\leq 3,$ $\pi_i:V_{i}\ra V_{i-1}$ is induced by the blow up of $V_{i-1}$ at $(\pi_{i-1}\circ \dots \circ \pi_{1})^{-1}({\rm Z}_{(i-1)(i-1)}).\footnote{If $i=1,$ it is just the identity.}$
    Let $\widetilde{\Sh}_{1,2}(Iw_\gothp)$ denote its special fiber, then
    $
    \widetilde{\Sh}_{1,2}(Iw_\gothp)=\bigcup\limits_{0\leq i,j\leq 2}\widetilde{\rm Z}_{ij}
    $ with 
    \begin{itemize}
        \item $\widetilde{\rm Z}_{00}$ obtained from ${\rm Z}_{00}$ by blow ups at 
        \begin{enumerate}
            \item $({\rm Z}_{00}\cap {\rm Z}_{11})\bigcup ({\rm Z}_{00}\cap {\rm Z}_{22})\bigcup ({\rm Z}_{00}\cap {\rm Z}_{12})\bigcup ({\rm Z}_{00}\cap {\rm Z}_{21});$
            \item $\pi^{-1}_{1}({\rm Z}_{11}\cap {\rm Z}_{00});$
            \item $(\pi_{2}\circ \pi_{1})^{-1}({\rm Z}_{22}\cap {\rm Z}_{00})$ successively;
        \end{enumerate}
        \item $\widetilde{\rm Z}_{11}$ obtained from ${\rm Z}_{11}$ by blow ups at 
        \begin{enumerate}
            \item ${\rm Z}_{11}\cap {\rm Z}_{00};$
            \item $\pi_1^{-1}(({\rm Z}_{11}\cap {\rm Z}_{22})\bigcup ({\rm Z}_{11}\cap {\rm Z}_{02})\bigcup ({\rm Z}_{11}\cap {\rm Z}_{20}));$
            \item $(\pi_{2}\circ \pi_{1})^{-1}({\rm Z}_{22}\cap {\rm Z}_{11})$ successively;
        \end{enumerate}
        \item $\widetilde{\rm Z}_{22}$ obtained from ${\rm Z}_{22}$ by blow ups at 
        \begin{enumerate}
            \item ${\rm Z}_{22}\cap {\rm Z}_{00};$
            \item $\pi_1^{-1}({\rm Z}_{22}\cap {\rm Z}_{11});$
            \item $(\pi_{2}\circ \pi_{1})^{-1}(({\rm Z}_{22}\cap {\rm Z}_{10})\bigcup({\rm Z}_{22}\cap{\rm Z}_{01}))$ successively;
        \end{enumerate}
        \item $\widetilde{\rm Z}_{01}$ obtained from ${\rm Z}_{01}$ by
        \begin{enumerate}
            \item picking an irreducible component $\pi_{1}^{-1}{\rm Z}_{01}$ isomorphic to ${\rm Z}_{01};$
            \item picking an irreducible component $(\pi_{2}^{-1}\circ\pi_{1}^{-1}){\rm Z}_{01}$ isomorphic to ${\rm Z}_{01};$
            \item blowing up at the intersection locus with $(\pi_{2}\circ \pi_{1})^{-1}{\rm Z}_{22}$ successively; 
        \end{enumerate}
        \item $\widetilde{\rm Z}_{10}$ obtained from ${\rm Z}_{01}$ by
        \begin{enumerate}
            \item picking an irreducible component $\pi_{1}^{-1}{\rm Z}_{10}$ isomorphic to ${\rm Z}_{10};$
            \item picking an irreducible component $(\pi_{2}^{-1}\circ\pi_{1}^{-1}){\rm Z}_{10}$ isomorphic to ${\rm Z}_{10};$
            \item blowing up at the intersection locus with $(\pi_{2}\circ \pi_{1})^{-1}{\rm Z}_{22}$ successively; 
        \end{enumerate}
        \item $\widetilde{\rm Z}_{12}$ obtained from ${\rm Z}_{12}$ by
        \begin{enumerate}
            \item blowing up at the intersection locus with ${\rm Z}_{00}$ and denoting the scheme we get by ${\rm Z}_{12}'$ 
            \item picking an irreducible component $\pi_{2}^{-1}{\rm Z}'_{12}$ isomorphic to ${\rm Z}_{12}';$
            \item  picking an irreducible component $(\pi_{3}^{-1}\circ\pi_{2}^{-1}){\rm Z}'_{12}$ isomorphic to ${\rm Z}_{12}'$ successively;
        \end{enumerate}
        \item $\widetilde{\rm Z}_{21}$ obtained from ${\rm Z}_{21}$ by
        \begin{enumerate}
            \item blowing up at the intersection locus with ${\rm Z}_{00}$ and denoting the scheme we get by ${\rm Z}_{21}'$ 
            \item picking an irreducible component $\pi_{2}^{-1}{\rm Z}'_{21}$ isomorphic to ${\rm Z}_{21}';$
            \item  picking an irreducible component $(\pi_{3}^{-1}\circ\pi_{2}^{-1}){\rm Z}'_{21}$ isomorphic to ${\rm Z}_{21}'$ successively;
        \end{enumerate}
        \item $\widetilde{\rm Z}_{20}$ obtained from ${\rm Z}_{20}$ by
        \begin{enumerate}
            \item picking an irreducible component $\pi_{1}^{-1}{\rm Z}_{20}$ isomorphic to ${\rm Z}_{20};$
            \item
            blowing up at the intersection locus with $\pi_{1}^{-1}{\rm Z}_{11}$ and denoting the scheme we get by ${\rm Z}_{20}'$ 
            \item  picking an irreducible component $\pi_{3}^{-1}{\rm Z}'_{20}$ isomorphic to ${\rm Z}_{20}'$ successively;
        \end{enumerate}
        \item $\widetilde{\rm Z}_{02}$ obtained from ${\rm Z}_{02}$ by
        \begin{enumerate}
            \item picking an irreducible component $\pi_{1}^{-1}{\rm Z}_{02}$ isomorphic to ${\rm Z}_{02};$
            \item
            blowing up at the intersection locus with $\pi_{1}^{-1}{\rm Z}_{11}$ and denoting the scheme we get by ${\rm Z}_{02}'$ 
            \item  picking an irreducible component $\pi_{3}^{-1}{\rm Z}'_{02}$ isomorphic to ${\rm Z}_{02}'$ successively;
        \end{enumerate}
    \end{itemize}
\end{proposition}
\begin{remark}
    The proof of Proposition~\ref{GI:blow up} and ~\ref{GI:stalk of points} uses the same method as in Proposition~\ref{G:blow up} and ~\ref{G:stalk of points} only with more difficulties in calculation. We omit here for simplicity.
\end{remark}
Recall the closed subschemes $C_i$ of $\Sh_{1,n-1}(K_\gothp^1)$ constructed in Definition~\ref{G:irreducible components of supsingular locus of Sh_1.n-1} for $1\leq i\leq n.$ For $n=3,$ it can be checked easily that $\pr_i$ is bijective on points for $1\leq i\leq 3,$ hence an isomorphism. From now on, we use $C_i$ to stand for $C_i$ and $\pr_i(C_i)$ both. 
\begin{proposition}
    The morphisms ${{\mathop{p}\limits^{\leftarrow}}_{Iw}}:\Sh_{1,2}(Iw_\gothp)\ra \Sh_{1,2}(K_\gothp^1)$ and ${{\mathop{p}\limits^{\rightarrow}}_{Iw}}:\Sh_{1,2}(Iw_\gothp)\ra \Sh_{1,2}(K_\gothp^1)$ induces morphisms:
   \begin{itemize}
       \item ${{\mathop{p}\limits^{\leftarrow}}_{Iw,00}}:{\rm Z}_{00}\ra{C}_2$ and ${{\mathop{p}\limits^{\rightarrow}}_{Iw,00}}:{\rm Z}_{00}\ra{\rm Y}_{00}$ with ${{\mathop{p}\limits^{\leftarrow}}_{Iw,00}}$ inducing ${\rm Z}_{00}$ to be a $\PP^1$-bundle over $C_2$ and ${{\mathop{p}\limits^{\rightarrow}}_{Iw,00}}$ surjective;
       \item 
       ${{\mathop{p}\limits^{\leftarrow}}_{Iw,01}}:{\rm Z}_{01}\ra{\rm Y}_{10}$ and ${{\mathop{p}\limits^{\rightarrow}}_{Iw,01}}:{\rm Z}_{01}\ra{\rm Y}_{01}$ with ${{\mathop{p}\limits^{\leftarrow}}_{Iw,01}}$ inducing ${\rm Z}_{01}$ to be he blow up of ${\rm Y}_{10}$ at $C_2\bigcap{\rm Y}_{10}$ and ${{\mathop{p}\limits^{\rightarrow}}_{Iw,01}}$ inducing ${\rm Z}_{01}$ to be the blow up of ${\rm Y}_{01}$ at $C_1\bigcap{\rm Y}_{01}$ both differ by a Frobenius twist of degree $p^2;$
       \item  ${{\mathop{p}\limits^{\leftarrow}}_{Iw,02}}:{\rm Z}_{02}\ra{\rm Y}_{11}$ and ${{\mathop{p}\limits^{\rightarrow}}_{Iw,02}}:{\rm Z}_{02}\ra{\rm Y}_{01}$ with ${{\mathop{p}\limits^{\leftarrow}}_{Iw,02}}$ inducing ${\rm Z}_{02}$ to be isomorphic ${\rm Y}_{11}$ with a Frobenius twist of degree $p^2$ and ${{\mathop{p}\limits^{\rightarrow}}_{Iw,02}}$ inducing ${\rm Z}_{02}$ to be the blow up of ${\rm Y}_{01}$ at $C_2\bigcap{\rm Y}_{01};$
       \item 
       ${{\mathop{p}\limits^{\leftarrow}}_{Iw,10}}:{\rm Z}_{10}\ra{\rm Y}_{01}$ and ${{\mathop{p}\limits^{\rightarrow}}_{Iw,10}}:{\rm Z}_{10}\ra{\rm Y}_{10}$ with ${{\mathop{p}\limits^{\leftarrow}}_{Iw,10}}$ inducing ${\rm Z}_{10}$ to be the blow up of ${\rm Y}_{01}$ at $C_2\bigcap{\rm Y}_{01}$ and ${{\mathop{p}\limits^{\rightarrow}}_{Iw,10}}$ inducing ${\rm Z}_{10}$ to be the blow up of ${\rm Y}_{10}$ at $C_1\bigcap{\rm Y}_{10};$
       \item ${{\mathop{p}\limits^{\leftarrow}}_{Iw,11}}:{\rm Z}_{11}\ra{\rm Y}_{00}$ and ${{\mathop{p}\limits^{\rightarrow}}_{Iw,11}}:{\rm Z}_{11}\ra C_1$ with ${{\mathop{p}\limits^{\leftarrow}}_{Iw,11}}$ surjective and ${{\mathop{p}\limits^{\rightarrow}}_{Iw,11}}$ inducing ${\rm Z}_{11}$ to be a $\PP^1$-bundle over $C_1;$
      \item ${{\mathop{p}\limits^{\leftarrow}}_{Iw,12}}:{\rm Z}_{12}\ra{\rm Y}_{01}$ and ${{\mathop{p}\limits^{\rightarrow}}_{Iw,12}}:{\rm Z}_{12}\ra {\rm Y}_{11}$ with ${{\mathop{p}\limits^{\leftarrow}}_{Iw,12}}$ inducing ${\rm Z}_{12}$ isomorphic to the blow up of ${\rm Y}_{01}$ at $C_1\bigcap {\rm Y}_{01}$ with a Frobenius twist of degree $p^2$ and ${{\mathop{p}\limits^{\rightarrow}}_{Iw,12}}$ inducing ${\rm Z}_{12}$ isomorphic to ${\rm Y}_{11};$
    \item ${{\mathop{p}\limits^{\leftarrow}}_{Iw,20}}:{\rm Z}_{20}\ra{\rm Y}_{11}$ and ${{\mathop{p}\limits^{\rightarrow}}_{Iw,20}}:{\rm Z}_{20}\ra {\rm Y}_{10}$ with ${{\mathop{p}\limits^{\leftarrow}}_{Iw,20}}$ inducing ${\rm Z}_{20}$ isomorphic to ${\rm Y}_{11}$ and ${{\mathop{p}\limits^{\rightarrow}}_{Iw,20}}$ inducing ${\rm Z}_{20}$ isomorphic to the blow up of ${\rm Y}_{10}$ at $C_2\bigcap{\rm Y}_{10}$ with a Frobenius twist of degree $p^2.$
     \item ${{\mathop{p}\limits^{\leftarrow}}_{Iw,21}}:{\rm Z}_{21}\ra{\rm Y}_{10}$ and ${{\mathop{p}\limits^{\rightarrow}}_{Iw,21}}:{\rm Z}_{21}\ra {\rm Y}_{11}$ with ${{\mathop{p}\limits^{\leftarrow}}_{Iw,21}}$ inducing ${\rm Z}_{21}$ isomorphic to the blow up of ${\rm Y}_{10}$ at $C_1\bigcap {\rm Y}_{10}$ and ${{\mathop{p}\limits^{\rightarrow}}_{Iw,21}}$ inducing ${\rm Z}_{21}$ isomorphic to the blow up of ${\rm Y}_{11}$ with a Frobenius twist of degree $p^2.$
     \item ${{\mathop{p}\limits^{\leftarrow}}_{Iw,22}}:{\rm Z}_{22}\ra C_1$ and ${{\mathop{p}\limits^{\rightarrow}}_{Iw,11}}:{\rm Z}_{22}\ra C_2$ with ${{\mathop{p}\limits^{\rightarrow}}_{Iw,22}}$ inducing ${\rm Z}_{22}$ to be a $\PP^1$-bundle over $C_1;$ and ${{\mathop{p}\limits^{\rightarrow}}_{Iw,22}}$ inducing ${\rm Z}_{22}$ to be a $\PP^1$-bundle over $C_2;$
   \end{itemize}
\end{proposition}
\begin{proof}
    We take the first point as an example. Let $(\mathcal{A}_1,\lambda_1,\eta_1,\mathcal{A}_2,\lambda_2,\eta_2,\phi)$ be the universal object of $C_2$ as a closed subscheme of ${\Sh}_{1,2}(K_\gothp^1).$
    Consider $\Gr(\phi_{*,1}(\mathcal{H}_1^{\dR}(\mathcal{A}_1/C_2)_2^\circ),1)$ the Grassmannian scheme over $C_2.$ Let $S$ be a scheme over $\FF_{p^2}.$ Define a morphism $\alpha_{00}:{\rm Z}_{00}\ra \Gr(\phi_{*,1}(\mathcal{H}_1^{\dR}(\mathcal{A}_1/C_2)_2^\circ),1)$ by mapping any $S$-point $(A_1,\lambda_1,\eta_1,A_2,\lambda_2,\eta_2,A_3,\lambda_3,\eta_3,\phi_{12},\phi_{23})$ to $(A_1,\lambda_1,\eta_1,A_2,\lambda_2,\eta_2,\phi_{12},\psi_{23,*,2}(\Lie_{A_3/S,2}^\circ)).$ Here $\psi_{23}$ is the quasi-isogeny $A_3\ra A_2$ such that $\psi_{23}\circ\phi_{23}=p$ and $\phi_{23}\circ\psi_{23}=p.$ Since $\phi_{23,*,2}\circ \phi_{12,*,2}({\rm H}_{1}^{\rm dR}(A_1/S)^\circ_2) \subseteq \phi_{23,*,2}({\rm H}_{1}^{\rm dR}(A_2/S)^\circ_2)=\omega_{A^\vee_{3}/S,2}^\circ$ is a line bundle of ${\rm H}_{1}^{\rm dR}(A_3/S)^\circ_2,$  $\psi_{23,*,2}(\Lie_{A_3/S,2}^\circ)\subseteq \phi_{12,*,2}({\rm H}_{1}^{\rm dR}(A_1/S)^\circ_2).$ Thus $\alpha_{00}$ is well defined. For any field $k$ over $\FF_{p^2},$ to show it is bijective on $k$-points, it suffices to construct an invertible map $\beta_{00}:\Gr(\phi_{*,1}(\mathcal{H}_1^{\dR}(\mathcal{A}_1/C_2)_2^\circ),1)(k)\ra {\rm Z}_{00}(k).$ For any $k$-point $(A_1,\lambda_1,\eta_1,A_2,\lambda_2,\eta_2,\phi_{12},H),$ let $\tilde{\mathcal{E}}_1=V\tcD(A_2)^\circ_2\subseteq \tcD(A_2)^\circ_2$ and $\tilde{\mathcal{E}}_2$ be the preimage of $H$ under the natural projection from $\tcD(A_2)^\circ_2$ to $H_2^{\dR}(A_2/k)_2^\circ.$ Then by Proposition~\ref{P:abelian-Dieud}, there exists an abelian variety $A_3$ over $k$ equipped with an $\cO_D$-action, a prime-to-$p$ polarization $\lambda_3$ and an $\cO_D$-equivariant $p$-isogeny $\psi_{23}:A_3\ra A_2$ such that the natural inclustion $\tilde{\mathcal{E}}_i\subseteq \tcD(A_2)^\circ_2$ is identified with the map $\psi_{23,*,i}:\tcD(A_3)^\circ_2\ra \tcD(A_2)^\circ_2$ induced by $\psi_{23}$ and such that $\psi_{23}^\vee\circ\lambda_2\circ\psi_{23}=p\lambda_3.$ We can also equip $A_3$ with a level structure $\eta_3$ such that $\psi_{23}\circ\eta_3=p\eta_2.$ Take the isogeny $\phi_{23}:A_2\ra A_3$ such that $\psi_{23}\circ\phi_{23}=p$ and $\phi_{23}\circ\psi_{23}=p$, we have $\phi_{23}^\vee\circ\lambda_3\phi_{23}=p\lambda_2$ and $\phi_{23}\circ\eta_2=\eta_3.$ Then $(A_1,\lambda_1,\eta_1,A_2,\lambda_2,\eta_2,A_3,\lambda_3,\eta_3,\phi_{12},\phi_{23})$ gives a $k$-point in ${\rm Z}_{00}$ and we define it to be the image under $\beta_{00}.$ It is easy to see $\beta_{00}\circ\alpha_{00}=\id$ and $\alpha_{00}\circ\beta_{00}=\id.$ Then by deformation theory, we can see $\alpha_{00}$ induces a bijection on tangent spaces, hence an isomorphism. Since ${{\mathop{p}\limits^{\leftarrow}}_{Iw,00}}$ is equal to the composition of the natural projection of the Grassmannian scheme to $C_2$ with $\alpha_{00},$ we see that ${{\mathop{p}\limits^{\leftarrow}}_{Iw,00}}$ induces ${\rm Z}_{00}$ to be a $\PP^1$-bundle over $C_2.$ The surjectivity of ${{\mathop{p}\limits^{\rightarrow}}_{Iw,00}}$ can be checked direcly.

    The other parts of the proposition can also be obtained similarly, by constructing a map to Grassmannian schemes and we omit here. The degree of the map can be obtained easily only to notice that the degree of ${{\mathop{p}\limits^{\leftarrow}}_{Iw}}:\Sh_{1,2}(Iw_\gothp)\ra \Sh_{1,2}(K_\gothp^1)$ and ${{\mathop{p}\limits^{\rightarrow}}_{Iw}}:\Sh_{1,2}(Iw_\gothp)\ra \Sh_{1,2}(K_\gothp^1)$ are all $1+p^2.$
\end{proof}
As a direct corollary, we have
\begin{corollary}\label{GI:characterize of Y_11}
    Up to Frobenius twist, ${\rm Y}_{11}$ is isormorphic to the blow up of ${\rm Y}_{01}$ at $C_2\bigcap{\rm Y}_{01}$ or at $C_1\bigcap{\rm Y}_{01}$ or the blow up of ${\rm Y}_{01}$ at $C_2\bigcap{\rm Y}_{10}$ or at $C_2\bigcap{\rm Y}_{01}.$
\end{corollary}
\begin{remark}
    By Corollary~\ref{GI:characterize of Y_11}, the cohomology groups of ${\rm Y}_{11}$ is computable. This is the main difficulty when proving the arithmetic level raising theorem for $n\geq 4.$
\end{remark}
Now we construct some `essential Frobenius' morphisms as in Construction~\ref{frob:Y_10 Y_01}.
\begin{construction}\label{frob:Z_10 Z_01}
    Let $S$ be a $\FF_{p^2}$-scheme. For any $S$-point $(A_1,\lambda_1,\eta_1,A_2,\lambda_2,\eta_2,A_3,\lambda_3,\eta_3,\phi_{12},\phi_{23})$ of ${\rm Z}_{10},$ we construct a $S$-point in ${\rm Z}_{01}.$ Take $\tcE_1=FV^{-1}\tcD(A_2)_1^\circ$ and $\tcE_2=FV^{-1}\tcD(A_2)_2^\circ.$ Applying Corollary~\ref{CORP}, we get a pair $(A_2',\lambda_2',\eta_2')$ and a $p$-quasi-isogeny $\psi_{2}:A_2'\ra A_3$ such that $\psi_{2}^\vee\circ\lambda_2'\circ\psi_{23}=\lambda_2$ and $\psi_{2}'\circ\eta_2=\eta_2'.$ 
    Let $\phi_{2}:A_2'\ra A_2$ be the $p$-quasi-isogeny such that $\psi_{2}\circ\phi_{2}=1$ and $\phi_{2}\circ\psi_{2}=1.$ Let $\phi'_{12}=\psi_2\circ\phi_{12}$ and $\phi_{23}'=\phi_{23}\circ\phi_2$
    It can be checked easily that $(A_1,\lambda_1,\eta_1,A_2',\lambda_2',\eta_2',A_3,\lambda_3,\eta_3,\phi'_{12},\phi'_{23})\in {\rm Z}_{01}.$ This gives a morphism $Fr_{10\ra 01}:{\rm Z}_{10}\ra {\rm Z}_{01},$ which is bijective on points.

    For any $S$-point $(A_1,\lambda_1,\eta_1,A_2,\lambda_2,\eta_2,A_3,\lambda_3,\eta_3,\phi_{12},\phi_{23})$ of ${\rm Z}_{01},$ we construct a $S$-point in ${\rm Z}_{10}.$ Take $\tcE_1=FV^{-1}\tcD(A_1)_1^\circ$ and $\tcE_2=FV^{-1}\tcD(A_1)_2^\circ\subseteq \tcD(A_2)_2^\circ.$ Applying Corollary~\ref{CORP}, we get a pair $(A_1',\lambda_1',\eta_1')$ and a $p$-quasi-isogeny $\psi_{1}:A_1\ra A_1'$ such that $\psi_{1}^\vee\circ\lambda_1'\circ\psi_{1}=\lambda_1$ and $\psi_{1}\circ\eta_1=\eta_1'.$ 
    Let $\phi_{1}:A_1\ra A_1'$ be the $p$-quasi-isogeny such that $\psi_{1}\circ\phi_{1}=1$ and $\phi_{1}\circ\psi_{1}=1.$
    Take $\tcE'_1=FV^{-1}\tcD(A_3)_1^\circ$ and $\tcE'_2=\tcD(A_3)_2^\circ.$ Applying Corollary~\ref{CORP}, we get a pair $(A_3',\lambda_3',\eta_3')$ and $\phi_{3}:A_3\ra A_3'$ such that $\phi_{3}^\vee\circ\lambda_3'\circ\phi_{3}=\lambda_3$ and $\phi_{3}'\circ\eta_3=\eta_3'.$ Let $\phi_{12}'=\phi_{12}\circ\phi_{1}$ and $\phi_{23}'=\phi_{3}\circ\phi_{23}.$
    It can be checked easily that $(A'_1,\lambda'_1,\eta'_1,A_2,\lambda_2,\eta_2,A'_3,\lambda'_3,\eta'_3,\phi'_{12},\phi'_{23})\in {\rm Z}_{10}.$ This gives a morphism $Fr_{01\ra 10}:{\rm Z}_{10}\ra {\rm Z}_{01},$ which is bijective on points. Note that $Fr_{01\ra10}\circ Fr_{10\ra01}=Fr_{p^2}$ and $Fr_{10\ra01}\circ Fr_{01\ra10}=Fr_{p^2}.$
\end{construction}
\begin{construction}\label{frob:Z_12 Z_21}
    Let $S$ be a $\FF_{p^2}$-scheme. For any $S$-point $(A_1,\lambda_1,\eta_1,A_2,\lambda_2,\eta_2,A_3,\lambda_3,\eta_3,\phi_{12},\phi_{23})$ of ${\rm Z}_{21},$ we construct a $S$-point in ${\rm Z}_{12}.$ Take $\tcE_1=FV^{-1}\tcD(A_1)_1^\circ$ and $\tcE_2=FV^{-1}\tcD(A_1)_2^\circ\subseteq \tcD(A_2)_2^\circ.$ Applying Corollary~\ref{CORP}, we get a pair $(A_1',\lambda_1',\eta_1')$ and a $p$-quasi-isogeny $\psi_{1}:A_1\ra A_1'$ such that $\psi_{1}^\vee\circ\lambda_1'\circ\psi_{1}=\lambda_1$ and $\psi_{1}\circ\eta_1=\eta_1'.$ 
    Let $\phi_{1}:A_1\ra A_1'$ be the $p$-quasi-isogeny such that $\psi_{1}\circ\phi_{1}=1$ and $\phi_{1}\circ\psi_{1}=1.$ Let $\phi_{12}'=\phi_{12}\circ\phi_1.$
    It can be checked easily that the point $(A'_1,\lambda'_1,\eta'_1,A_2,\lambda_2,\eta_2,A_3,\lambda_3,\eta_3,\phi'_{12},\phi_{23})\in {\rm Z}_{12}.$ This gives a morphism $Fr_{21\ra 12}:{\rm Z}_{21}\ra {\rm Z}_{12},$ which is bijective on points.

    For any $S$-point $(A_1,\lambda_1,\eta_1,A_2,\lambda_2,\eta_2,A_3,\lambda_3,\eta_3,\phi_{12},\phi_{23})$ of ${\rm Z}_{12},$ we construct a $S$-point in ${\rm Z}_{21}.$  Take $\tcE'_1=FV^{-1}\tcD(A_3)_1^\circ$ and $\tcE'_2=\tcD(A_3)_2^\circ.$ Applying Corollary~\ref{CORP}, we get a pair $(A_3',\lambda_3',\eta_3')$ and $\phi_{3}:A_3\ra A_3'$ such that $\phi_{3}^\vee\circ\lambda_3'\circ\phi_{3}=\lambda_3$ and $\phi_{3}'\circ\eta_3=\eta_3'.$ Let $\psi_3:A_3'\ra A_3$ be the $p$-quasi-isogeny such that $\psi_3\circ\phi_3=1$ and $\phi_3\circ\psi_3=1.$Take $\tcE_1=FV^{-1}\tcD(A_2)_1^\circ$ and $\tcE_2=FV^{-1}\tcD(A_2)_2^\circ.$ Applying Corollary~\ref{CORP}, we get a pair $(A_2',\lambda_2',\eta_2')$ and a $p$-quasi-isogeny $\psi_{2}:A_2'\ra A_3$ such that $\psi_{2}^\vee\circ\lambda_2'\circ\psi_{23}=\lambda_2$ and $\psi_{2}'\circ\eta_2=\eta_2'.$ 
    Let $\phi_{2}:A_2'\ra A_2$ be the $p$-quasi-isogeny such that $\psi_{2}\circ\phi_{2}=1$ and $\phi_{2}\circ\psi_{2}=1.$ Let $\phi'_{12}=\psi_2\circ\phi_{12}$ and $\phi_{23}'=\phi_3\circ\phi_{23}\circ\phi_2$
    It can be checked easily that $(A_1,\lambda_1,\eta_1,A'_2,\lambda'_2,\eta'_2,A'_3,\lambda'_3,\eta'_3,\phi'_{12},\phi'_{23})\in {\rm Z}_{21}.$ This gives a morphism $Fr_{12\ra 21}:{\rm Z}_{12}\ra {\rm Z}_{21},$ which is bijective on points. Note that $Fr_{12\ra21}\circ Fr_{21\ra12}=Fr_{p^2}$ and $Fr_{21\ra12}\circ Fr_{12\ra21}=Fr_{p^2}.$
\end{construction}
\begin{construction}\label{frob:Z_20 Z_02}
    Let $S$ be a $\FF_{p^2}$-scheme. For any $S$-point $(A_1,\lambda_1,\eta_1,A_2,\lambda_2,\eta_2,A_3,\lambda_3,\eta_3,\phi_{12},\phi_{23})$ of ${\rm Z}_{02},$ we construct a $S$-point in ${\rm Z}_{20}.$ Take $\tcE'_1=FV^{-1}\tcD(A_3)_1^\circ$ and $\tcE'_2=\tcD(A_3)_2^\circ.$ Applying Corollary~\ref{CORP}, we get a pair $(A_3',\lambda_3',\eta_3')$ and $\phi_{3}:A_3\ra A_3'$ such that $\phi_{3}^\vee\circ\lambda_3'\circ\phi_{3}=\lambda_3$ and $\phi_{3}'\circ\eta_3=\eta_3'.$ Let $\psi_3:A_3'\ra A_3$ be the $p$-quasi-isogeny such that $\psi_3\circ\phi_3=1$ and $\phi_3\circ\psi_3=1.$ Let $\phi_{23}'=\phi_3\circ\phi_{23}.$
    It can be checked easily that the point $(A_1,\lambda_1,\eta_1,A_2,\lambda_2,\eta_2,A'_3,\lambda'_3,\eta'_3,\phi_{12},\phi'_{23})\in {\rm Z}_{12}.$ This gives a morphism $Fr_{02\ra 20}:{\rm Z}_{02}\ra {\rm Z}_{20},$ which is bijective on points.

    For any $S$-point $(A_1,\lambda_1,\eta_1,A_2,\lambda_2,\eta_2,A_3,\lambda_3,\eta_3,\phi_{12},\phi_{23})$ of ${\rm Z}_{20},$ we construct a $S$-point in ${\rm Z}_{02}.$ Take $\tcE_1=FV^{-1}\tcD(A_1)_1^\circ$ and $\tcE_2=FV^{-1}\tcD(A_1)_2^\circ\subseteq \tcD(A_2)_2^\circ.$ Applying Corollary~\ref{CORP}, we get a pair $(A_1',\lambda_1',\eta_1')$ and a $p$-quasi-isogeny $\psi_{1}:A_1\ra A_1'$ such that $\psi_{1}^\vee\circ\lambda_1'\circ\psi_{1}=\lambda_1$ and $\psi_{1}\circ\eta_1=\eta_1'.$ 
    Let $\phi_{1}:A_1\ra A_1'$ be the $p$-quasi-isogeny such that $\psi_{1}\circ\phi_{1}=1$ and $\phi_{1}\circ\psi_{1}=1.$ 
    Take $\tcE_1=FV^{-1}\tcD(A_2)_1^\circ$ and $\tcE_2=FV^{-1}\tcD(A_2)_2^\circ.$ Applying Corollary~\ref{CORP}, we get a pair $(A_2',\lambda_2',\eta_2')$ and a $p$-quasi-isogeny $\psi_{2}:A_2'\ra A_3$ such that $\psi_{2}^\vee\circ\lambda_2'\circ\psi_{23}=\lambda_2$ and $\psi_{2}'\circ\eta_2=\eta_2'.$ 
    Let $\phi_{2}:A_2'\ra A_2$ be the $p$-quasi-isogeny such that $\psi_{2}\circ\phi_{2}=1$ and $\phi_{2}\circ\psi_{2}=1.$ Let $\phi'_{12}=\psi_2\circ\phi_{12}\circ\phi_1$ and $\phi_{23}'=\phi_3\circ\phi_{23}\circ\phi_2$
    It can be checked easily that $(A_1',\lambda_1',\eta_1',A'_2,\lambda'_2,\eta'_2,A_3,\lambda_3,\eta_3,\phi'_{12},\phi'_{23})\in {\rm Z}_{02}.$ This gives a morphism $Fr_{20\ra 02}:{\rm Z}_{20}\ra {\rm Z}_{02},$ which is bijective on points. Note that $Fr_{02\ra20}\circ Fr_{20\ra02}=Fr_{p^2}$ and $Fr_{20\ra02}\circ Fr_{02\ra20}=Fr_{p^2}.$
\end{construction}
\begin{construction}\label{12-20}
    Let $S$ be a $\FF_{p^2}$-scheme. For any $S$-point $(A_1,\lambda_1,\eta_1,A_2,\lambda_2,\eta_2,A_3,\lambda_3,\eta_3,\phi_{12},\phi_{23})$ of ${\rm Z}_{12},$ we construct a $S$-point in ${\rm Z}_{20}.$ Take $\tcE_1=p\tcD(A_2)_{1}^{\circ}, \tcE_2=p\tcD(A_2)_{2}^{\circ}.$ Then by Corollary~\ref{CORP},
    we get
    Let $(A'_2,\lambda'_2,\eta'_2)$ with $p$-quasi-isogeny $\psi_2:A_2\ra A_2'$ such that $\psi_2^\vee\circ\lambda_2'\circ\psi_2=\lambda_2$ and $\psi_2\circ\eta_2=\eta_2'.$ Moreover, $\psi_{2,*,i}:\tcD(A_2)_i^\circ\ra\tcD(A_2')_{i}^\circ$ has image $p\tcD(A_2')_i^\circ.$
   Similarly we can get $(A'_3,\lambda'_3,\eta'_3)$ with $p$-quasi-isogeny $\psi_3:A_3\ra A_3'$ such that $\psi_3^\vee\circ\lambda_3'\circ\psi_3=\lambda_3$ and $\psi_2\circ\eta_3=\eta_3'.$ Moreover, $\psi_{3,*,i}:\tcD(A_3)_i^\circ\ra\tcD(A_3')_{i}^\circ$ has image $p\tcD(A_3')_i^\circ.$
   Let $\psi_{23}:A_3\ra A_2$ be the isogeny such that $\psi_{23}\circ\phi_{23}=p$ and $\phi_{23}\circ\psi_{23}=p.$ Let $\psi_{12}:A_2\ra A_1$ be the isogeny such that $\psi_{12}\circ\phi_{12}=p$ and $\phi_{12}\circ\psi_{12}=p.$ Take $\phi_2:A_2'\ra A_2$ and $\phi_3:A_3'\ra A_3$ the $p$-quasi-isogeny as in Remark~\ref{dual_isogeny}. Let $\phi_{31}'=\frac{1}{p^2}\psi_{12}\circ\psi_{23}\circ\phi_3$ and $\phi_{23}'=\psi_{3}\circ\phi_{23}\circ\phi_2$
   Then $(A_2',\lambda_2',\eta_2',A_3',\lambda_3',\eta_3',A_1,\eta_1,\lambda_1,\phi_{23}',\phi_{31}')\in {\rm Z}_{20}.$
   This gives a morphism $\Phi_{12\ra20}:{\rm Z}_{12}\ra{\rm Z}_{20},$ which can be checked to be an isomorphism by deformation theory. 

    With exacly the same construction, we can also get a morphism $\Phi_{21\ra02}:{\rm Z}_{21}\ra{\rm Z}_{02},$ which maps a $S$-point $(A_1,\lambda_1,\eta_1,A_2,\lambda_2,\eta_2,A_3,\lambda_3,\eta_3,\phi_{12},\phi_{23})$ of ${\rm Z}_{21}$ to $(A_2',\lambda_2',\eta_2',A_3',\lambda_3',\eta_3',A_1,\eta_1,\lambda_1,\phi_{23}',\phi_{31}')$ with $(A'_2,\lambda'_2,\eta'_2)$ and $(A'_3,\lambda'_3,\eta'_3)$ obtained in the same way using Corollary~\ref{CORP}. 
\end{construction}
\begin{construction}\label{Y_1010}
    Let $S$ be a $\FF_{p^2}$-scheme. For any $S$-point $(A_2,\lambda_2,\eta_2,A_3,\lambda_3,\eta_3,\phi_{23})$ of ${\rm Y}_{10},$ we construct a $S$-point in ${\rm Y}_{10}.$ Take $\tcE_1=p\tcD(A_2)_{1}^{\circ}, \tcE_2=p\tcD(A_2)_{2}^{\circ}.$ Then by Corollary~\ref{CORP},
    we get
    Let $(A'_2,\lambda'_2,\eta'_2)$ with $p$-quasi-isogeny $\psi_2:A_2\ra A_2'$ such that $\psi_2^\vee\circ\lambda_2'\circ\psi_2=\lambda_2$ and $\psi_2\circ\eta_2=\eta_2'.$ Moreover, $\psi_{2,*,i}:\tcD(A_2)_i^\circ\ra\tcD(A_2')_{i}^\circ$ has image $p\tcD(A_2')_i^\circ.$
   Similarly we can get $(A'_3,\lambda'_3,\eta'_3)$ with $p$-quasi-isogeny $\psi_3:A_3\ra A_3'$ such that $\psi_3^\vee\circ\lambda_3'\circ\psi_3=\lambda_3$ and $\psi_2\circ\eta_3=\eta_3'.$ Moreover, $\psi_{3,*,i}:\tcD(A_3)_i^\circ\ra\tcD(A_3')_{i}^\circ$ has image $p\tcD(A_3')_i^\circ.$
   Let $\psi_{23}:A_3\ra A_2$ be the isogeny such that $\psi_{23}\circ\phi_{23}=p$ and $\phi_{23}\circ\psi_{23}=p.$ Take $\phi_2:A_2'\ra A_2$ and $\phi_3:A_3'\ra A_3$ the $p$-quasi-isogeny as in Remark~\ref{dual_isogeny}. Let $\phi_{23}'=\psi_{3}\circ\phi_{23}\circ\phi_2$
   Then $(A_2',\lambda_2',\eta_2',A_3',\lambda_3',\eta_3',\phi_{23}')\in {\rm Y}_{10}.$
   This gives a morphism $\Phi_{10\ra10}:{\rm Y}_{10}\ra{\rm Y}_{10},$ which can be checked to be an isomorphism by deformation theory. In fact, this construction can be viewed as a kind of Hecke action and this can be defined for any closed subschemes mentioned above. We omit the construction here for simplicity and use notation $\Phi_{ij\ra ij}:{\rm Y}_{ij}\ra{\rm Y}_{ij}$ to denote the action on ${\rm Y}_{ij}$ and $\Psi_{ij\ra ij}:{\rm Z}_{ij}\ra{\rm Z}_{ij}$ to denote the action on ${\rm Z}_{ij}$ for every possible $i,j.$ Denote by $\Phi:\Sh_{1,2}\ra\Sh_{1,2}$ the action on $\Sh_{1,2}.$
\end{construction}

\begin{proposition}\label{GI:Diagram commuting}
With the constructions above, it is easily to check that the following diagrams are commutative:
\[\begin{tikzcd}
[
    column sep=small, 
    row sep=scriptsize,
    cells={font=\footnotesize} 
]
	& {{\rm Z}_{10}} && {{\rm Z}_{01}} \\
	\\
	{{\rm Y}_{01}} && {{\rm Y}_{10}} & {{\rm Y}_{10}} && {{\rm Y}_{01}}
	\arrow["{Fr_{10\ra01}}", from=1-2, to=1-4]
	\arrow["{{{\mathop{p}\limits^{\leftarrow}}_{Iw,10}}}"', from=1-2, to=3-1]
	\arrow["{{{\mathop{p}\limits^{\rightarrow}}_{Iw,10}}}"', from=1-2, to=3-4]
	\arrow["{{{\mathop{p}\limits^{\leftarrow}}_{Iw,01}}}", from=1-4, to=3-3]
	\arrow["{{{\mathop{p}\limits^{\rightarrow}}_{Iw,01}}}", from=1-4, to=3-6]
	\arrow["{Fr''}", from=3-1, to=3-3]
	\arrow["{Fr'}", from=3-4, to=3-6]
\end{tikzcd}
\begin{tikzcd}
[
    column sep=small, 
    row sep=scriptsize,
    cells={font=\footnotesize} 
]
	& {{\rm Z}_{10}} && {{\rm Z}_{01}} \\
	\\
	{{\rm Y}_{01}} && {{\rm Y}_{10}} & {{\rm Y}_{10}} && {{\rm Y}_{01}}
	\arrow["{{{\mathop{p}\limits^{\leftarrow}}_{Iw,10}}}"', from=1-2, to=3-1]
	\arrow["{{{\mathop{p}\limits^{\rightarrow}}_{Iw,10}}}"', from=1-2, to=3-4]
	\arrow["{Fr_{01\ra10}}"', from=1-4, to=1-2]
	\arrow["{{{\mathop{p}\limits^{\leftarrow}}_{Iw,01}}}", from=1-4, to=3-3]
	\arrow["{{{\mathop{p}\limits^{\rightarrow}}_{Iw,01}}}", from=1-4, to=3-6]
	\arrow["{Fr'}"', from=3-3, to=3-1]
	\arrow["{Fr''}"', from=3-6, to=3-4]
\end{tikzcd}\]
\[\begin{tikzcd}
[
    column sep=small, 
    row sep=scriptsize,
    cells={font=\footnotesize} 
]
	& {{\rm Z}_{20}} && {{\rm Z}_{02}} \\
	\\
	{{\rm Y}_{11}} && {{\rm Y}_{11}} & {{\rm Y}_{10}} && {{\rm Y}_{01}}
	\arrow["{{{\mathop{p}\limits^{\leftarrow}}_{Iw,20}}}"', from=1-2, to=3-1]
	\arrow["{{{\mathop{p}\limits^{\rightarrow}}_{Iw,20}}}"', from=1-2, to=3-4]
	\arrow["{Fr_{02\ra20}}"', from=1-4, to=1-2]
	\arrow["{{{\mathop{p}\limits^{\leftarrow}}_{Iw,02}}}", from=1-4, to=3-3]
	\arrow["{{{\mathop{p}\limits^{\rightarrow}}_{Iw,02}}}", from=1-4, to=3-6]
	\arrow[no head, from=3-1, to=3-3]
	\arrow[shift left, no head, from=3-1, to=3-3]
	\arrow["{Fr''}"', from=3-6, to=3-4]
\end{tikzcd}
\begin{tikzcd}
[
    column sep=small, 
    row sep=scriptsize,
    cells={font=\footnotesize} 
]
	& {{\rm Z}_{20}} && {{\rm Z}_{02}} \\
	\\
	{{\rm Y}_{11}} && {{\rm Y}_{11}} & {{\rm Y}_{10}} && {{\rm Y}_{01}}
	\arrow["{Fr_{20\ra02}}", from=1-2, to=1-4]
	\arrow["{{{\mathop{p}\limits^{\leftarrow}}_{Iw,20}}}"', from=1-2, to=3-1]
	\arrow["{{{\mathop{p}\limits^{\rightarrow}}_{Iw,20}}}"', from=1-2, to=3-4]
	\arrow["{{{\mathop{p}\limits^{\leftarrow}}_{Iw,02}}}", from=1-4, to=3-3]
	\arrow["{{{\mathop{p}\limits^{\rightarrow}}_{Iw,02}}}", from=1-4, to=3-6]
	\arrow["{Fr_{p^2}}", from=3-1, to=3-3]
	\arrow["{Fr'}"', from=3-4, to=3-6]
\end{tikzcd}\]
\[\begin{tikzcd}
[
    column sep=small, 
    row sep=scriptsize,
    cells={font=\footnotesize} 
]
	& {{\rm Z}_{21}} && {{\rm Z}_{12}} \\
	\\
	{{\rm Y}_{10}} && {{\rm Y}_{01}} & {{\rm Y}_{11}} && {{\rm Y}_{11}}
	\arrow["{Fr_{21\ra12}}", from=1-2, to=1-4]
	\arrow["{{{\mathop{p}\limits^{\leftarrow}}_{Iw,21}}}"', from=1-2, to=3-1]
	\arrow["{{{\mathop{p}\limits^{\rightarrow}}_{Iw,21}}}"', from=1-2, to=3-4]
	\arrow["{{{\mathop{p}\limits^{\leftarrow}}_{Iw,12}}}", from=1-4, to=3-3]
	\arrow["{{{\mathop{p}\limits^{\rightarrow}}_{Iw,12}}}", from=1-4, to=3-6]
	\arrow["{Fr'}", from=3-1, to=3-3]
	\arrow[no head, from=3-4, to=3-6]
	\arrow[shift left, no head, from=3-4, to=3-6]
\end{tikzcd}
\begin{tikzcd}
[
    column sep=small, 
    row sep=scriptsize,
    cells={font=\footnotesize} 
]
	& {{\rm Z}_{21}} && {{\rm Z}_{12}} \\
	\\
	{{\rm Y}_{10}} && {{\rm Y}_{01}} & {{\rm Y}_{11}} && {{\rm Y}_{11}}
	\arrow["{{{\mathop{p}\limits^{\leftarrow}}_{Iw,21}}}"', from=1-2, to=3-1]
	\arrow["{{{\mathop{p}\limits^{\rightarrow}}_{Iw,21}}}"', from=1-2, to=3-4]
	\arrow["{Fr_{12\ra21}}"', from=1-4, to=1-2]
	\arrow["{{{\mathop{p}\limits^{\leftarrow}}_{Iw,12}}}", from=1-4, to=3-3]
	\arrow["{{{\mathop{p}\limits^{\rightarrow}}_{Iw,12}}}", from=1-4, to=3-6]
	\arrow["{Fr''}"', from=3-3, to=3-1]
	\arrow["{Fr_{p^2}}", from=3-6, to=3-4]
\end{tikzcd}\]
\[\begin{tikzcd}
	{{\rm Z}_{21}} && {{\rm Z}_{12}} \\
	{{\rm Z}_{02}} && {{\rm Z}_{20}}
	\arrow["{Fr_{21\ra12}}", from=1-1, to=1-3]
	\arrow["{{\Phi_{21\ra02}}}"', from=1-1, to=2-1]
	\arrow["{{\Phi_{12\ra20}}}", from=1-3, to=2-3]
	\arrow["{Fr_{02\ra20}}", from=2-1, to=2-3]
\end{tikzcd}
\begin{tikzcd}
	{{\rm Z}_{02}} & {{\rm Z}_{20}} & {{\rm Z}_{12}} & {{\rm Y}_{11}} & {{\rm Z}_{20}} \\
	{{\rm Y}_{01}} & {{\rm Y}_{10}} &&& {{\rm Y}_{10}}
	\arrow["{Fr_{02\ra20}}", from=1-1, to=1-2]
	\arrow["{{{\mathop{p}\limits^{\rightarrow}}_{Iw,02}}}", from=1-1, to=2-1]
	\arrow["{\Phi_{12\ra20}^{-1}}", from=1-2, to=1-3]
	\arrow["{{{\mathop{p}\limits^{\rightarrow}}_{Iw,20}}}", from=1-2, to=2-2]
	\arrow["{{{\mathop{p}\limits^{\rightarrow}}_{Iw,12}}}", from=1-3, to=1-4]
	\arrow["{{{{\mathop{p}\limits^{\leftarrow}}_{Iw,20}}}^{-1}}", from=1-4, to=1-5]
	\arrow["{{{\mathop{p}\limits^{\rightarrow}}_{Iw,20}}}", from=1-5, to=2-5]
	\arrow["{Fr''}", from=2-1, to=2-2]
	\arrow["{\Phi_{10\ra10}}", from=2-2, to=2-5]
\end{tikzcd}\]
\[\begin{tikzcd}
[
    column sep=small, 
    row sep=scriptsize,
    cells={font=\footnotesize} 
]
	& {{\rm Y}_{10}} && {{\rm Z}_{21}} && {{\rm Z}_{12}} && {{\rm Y}_{11}} \\
	{\Sh_{1,2}} &&&& {{\rm Z}_{02}} && {{\rm Z}_{20}} \\
	& {{\rm Y}_{01}} && {{\rm Z}_{02}} &&&& {{\rm Y}_{11}} \\
	&&&&& {{\rm Z}_{20}}
	\arrow["{{{\mathop{p}\limits^{\leftarrow}}_{10}}}"', from=1-2, to=2-1]
	\arrow["{{{\mathop{p}\limits^{\leftarrow}}_{Iw,21}}}"', from=1-4, to=1-2]
	\arrow["{Fr_{21\ra12}}", from=1-4, to=1-6]
	\arrow["{\Phi_{21\ra02}}"', from=1-4, to=3-4]
	\arrow["{{{\mathop{p}\limits^{\rightarrow}}_{Iw,12}}}", from=1-6, to=1-8]
	\arrow["{\Phi_{12\ra20}}"', from=1-6, to=4-6]
	\arrow["{Fr_{02\ra20}}"', from=2-5, to=2-7]
	\arrow["{{{\mathop{p}\limits^{\leftarrow}}_{Iw,20}}}"', from=2-7, to=1-8]
	\arrow["{{{\mathop{p}\limits^{\rightarrow}}_{01}}}"', from=3-2, to=2-1]
	\arrow["{\Phi_{02\ra02}^{-1}}"', from=3-4, to=2-5]
	\arrow["{{{\mathop{p}\limits^{\rightarrow}}_{Iw,02}}}"', from=3-4, to=3-2]
	\arrow["{Fr_{02\ra20}}"', from=3-4, to=4-6]
	\arrow["{\Phi_{11\ra 11}^{-1}}"', from=3-8, to=1-8]
	\arrow["{\Phi_{11\ra 11}^{-1}}"', from=4-6, to=2-7]
	\arrow["{{{\mathop{p}\limits^{\leftarrow}}_{Iw,20}}}"', from=4-6, to=3-8]
\end{tikzcd}
\begin{tikzcd}
[
    column sep=small, 
    row sep=scriptsize,
    cells={font=\footnotesize} 
]
	& {{\rm Z}_{12}} && {{\rm Z}_{21}} && {{\rm Y}_{10}} & {{\rm Y}_{10}} \\
	{{\rm Y}_{11}} &&&&&&& {\Sh_{1,2}} \\
	{{\rm Y}_{11}} &&& {{\rm Z}_{02}} &&&& {\Sh_{1,2}} \\
	& {{\rm Z}_{20}} && {{\rm Z}_{20}} && {{\rm Y}_{01}}
	\arrow["{Fr_{12\ra21}}", from=1-2, to=1-4]
	\arrow["{{\mathop{p}\limits^{\rightarrow}}_{Iw,12}}"', from=1-2, to=2-1]
	\arrow["{\Phi_{12\ra20}}", from=1-2, to=4-2]
	\arrow["{{\mathop{p}\limits^{\leftarrow}}_{Iw,21}}", from=1-4, to=1-6]
	\arrow["{\Phi_{21\ra02}}", from=1-4, to=3-4]
	\arrow["{\Phi_{10\ra10}}", from=1-6, to=1-7]
	\arrow["{{\mathop{p}\limits^{\leftarrow}}_{10}}", from=1-7, to=2-8]
	\arrow["{\Phi_{11\ra11}}"', from=2-1, to=3-1]
	\arrow["{Fr_{02\ra20}}", from=3-4, to=4-4]
	\arrow["{\Phi}"', from=3-8, to=2-8]
	\arrow["{{\mathop{p}\limits^{\leftarrow}}_{Iw,20}}", from=4-2, to=3-1]
	\arrow["{Fr_{p^2}}", from=4-2, to=4-4]
	\arrow["{{\mathop{p}\limits^{\rightarrow}}_{Iw,02}}"', from=4-4, to=4-6]
	\arrow["{{\mathop{p}\limits^{\leftarrow}}_{01}}"', from=4-6, to=3-8]
\end{tikzcd}\]
\end{proposition}

\section{Arithmetic level raising theorem for \texorpdfstring{$n\geq 3$}{n3}
}\label{A3}
We first construct the Arithmetic level raising map for general $n.$

Recall in Proposition~\ref{property of higher Chow groups} ${\rm Ch}^{n}({{\Sh}_{1,n-1}},1,k_\lambda)={{\rm{H}}}^{2n-1}_{\mathcal{M}} ({{{\Sh}_{1,n-1}}},k_\lambda(n)).$ Consider the cycle class map from ${{\rm{H}}}^{2n-1}_{\mathcal{M}} ({{{\Sh}_{1,n-1}}},k_\lambda(n))$ to $ {\rm H}_{\acute{e} t}^{2n-1}({\Sh}_{1,n-1},k_\lambda(n))),$ we get a map from ${\rm Ch}^{n}({{\Sh}_{1,n-1}},1,k_\lambda)$ to ${\rm H}_{\acute{e} t}^{2n-1}({\Sh}_{1,n-1},k_\lambda(n))).$ Composing with the Gysin map ${\rm Ch}^{1}({{\Sh}^{\rm ss}_{1,n-1}},1,k_\lambda)\rightarrow{\rm Ch}^{n}({{\Sh}_{1,n-1}},1,k_\lambda)$ induced by the closed immersion of ${\Sh}^{\rm ss}_{1,n-1}$ into ${\Sh}_{1,n-1},$ we get the map ${\rm Ch}^{1}({{\Sh}_{1,n-1}},1,k_\lambda)\rightarrow{\rm H}_{\acute{e} t}^{2n-1}({\Sh}_{1,n-1},k_\lambda(n))).$

On the other hand, considering the Galois action on the special fiber ${\Sh}_{1,n-1},$ we get a short exact sequence:
\[0\rightarrow{\rm H}^{1}(\mathbb{F}_{p^{2}}, {\rm H}_{\acute{e} t}^{2n-2}({\overline{\Sh}}_{1,n-1},k_\lambda(n))) \rightarrow {\rm H}_{\acute{e} t}^{2n-1}({\Sh}_{1,n-1},k_\lambda(n))) \rightarrow
{\rm H}^{0}(\mathbb{F}_{p^{2}}, {\rm H}_{\acute{e} t}^{2n-1}({\overline{\Sh}}_{1,n-1},k_\lambda(n))).\]

After localizing at $\mathfrak{m}$, we get that ${\rm H}^{i}({\overline{\Sh}}_{1,n-1},1, k_\lambda(n))_{\mathfrak{m}}$ is nonzero if and only if $i=2n-2.$ Hence we get two maps by lifting:
${\rm Ch}^{n}({{\Sh}_{1,n-1}},1,k_\lambda)_{\mathfrak{m}}\rightarrow{\rm H}^{1}(\mathbb{F}_{p^{2}}, {\rm H}_{\acute{e} t}^{2n-2}({{\overline{\Sh}}_{1,n-1}},k_\lambda(n))_{\mathfrak{m}},$ which is the so-called Abel-Jacobi map; and 
${\rm Ch}^{1}({{\Sh}_{1,n-1}},1,k_\lambda)_{\mathfrak{m}}\rightarrow{\rm H}^{1}(\mathbb{F}_{p^{2}}, {\rm H}_{\acute{e} t}^{2n-2}({{\overline{\Sh}}_{1,n-1}},k_\lambda(n))_{\mathfrak{m}},$ which is the so-called level raising map.

Summing up, we have the diagram as below:
\[\begin{tikzcd}
	&& {0={\rm{H}}^0(\mathbb{F}_{p^2},{\rm{H}}_{\acute{e}t}^{2n-1}(\overline{{\Sh}}_{1,n-1},k_\lambda(n))_\mathfrak{m})} & {} & {} \\
	{{\rm{Ch}}^{n}({{\Sh}_{1,n-1}},1,k_\lambda)_{\mathfrak{m}}} & {{\rm{H}}^{2n-1}_{\mathcal{M}} ({{\Sh}_{1,n-1}},k_\lambda(n))_{\mathfrak{m}}} & {{\rm{H}}_{\acute{e}t}^{2n-1}({{\Sh}_{1,n-1}},k_\lambda(n))_{\mathfrak{m}}} & {} & {} \\
	{{\rm{Ch}}^1({{\Sh}_{1,n-1}^{{\rm{ss}}}},1,k_\lambda)_{\mathfrak{m}}} & {} & {{\rm{H}}^1(\mathbb{F}_{p^2},{\rm{H}}_{\acute{e}t}^{2n-2}(\overline{{\Sh}}_{1,n-1},k_\lambda(n))_\mathfrak{m})} & {} & {}
	\arrow[Rightarrow, no head, from=2-2, to=2-1]
	\arrow[from=2-2, to=2-3]
	\arrow["{{{Abel-Jacobi~~map}}}"'{pos=0.05}, from=2-2, to=3-3]
	\arrow[two heads, from=2-3, to=1-3]
	\arrow[from=3-1, to=2-1]
	\arrow["{{{level -raising ~~map}}}", from=3-1, to=3-3]
	\arrow[hook, from=3-3, to=2-3]
\end{tikzcd}\]

\begin{theorem}\label{main theorem n=3}
    Under Hypothesis \ref{Main hypo} and \ref{Main hypo2}, the map ${\rm Ch}^{1}({{\Sh}_{1,2}},1,k_\lambda)_{\mathfrak{m}}\rightarrow{\rm H}^{1}(\mathbb{F}_{p^{2}}, {\rm H}_{\acute{e}t}^{4}({\overline{\Sh}_{1,2}}, 1, k_\lambda(2))_{\mathfrak{m}}$ is surjective.
\end{theorem}
\begin{notation}
    Recall the blow up of $\cSh_{1,2}(K_{\gothp}^1)$ at ${\rm Y}_{00}$ in Proposition~\ref{G:blow up} denoted by $\widetilde{\cSh}_{1,2}(K_{\gothp}^1).$  Denote its special fiber by $\widetilde{\Sh}_{1,2}(K_\gothp^1),$ which can be expressed as the union of four irreducible components, denoted by $\widetilde{\rm Y}_{00},\widetilde{\rm Y}_{10},\widetilde{\rm Y}_{01}$ and $\widetilde{\rm Y}_{11}.$ Let $\widetilde{Sh}_{1,2}(K_{\gothp}^1)$ denote the generic fiber.
    Let $Y^{(i)}$ be the disjoint union of all the intersection of $i+1$ irreducible closed components of $\widetilde{\Sh}_{1,2}(K_\gothp^1).$
    For a $\FF_{p^2}$-scheme $X$, we use $\overline{X}$ to denote the geometric fiber of $X$ defined over $\overline\FF_p.$
\end{notation}

By Proposition~\ref{G:blow up}, $\widetilde{\Sh}_{1,2}(K_{\gothp}^1)$ is strictly semistable. Then by \cite[Corollary 2.8]{Sai03}, we have the weight spectral sequence
$
E_1^{p,q}=\bigoplus\limits_{i\geq\max(0,-p)}{\rm H}^{q-2i}_{\acute{e}t}(\overline{Y}^{(p+2i)},k_\lambda(-i))
\Rightarrow {\rm H}^{p+q}_{\acute{e}t}(\overline{\widetilde{Sh}}_{1,2}(K_{\gothp}^1),k_\lambda),
$
with $\overline{\widetilde{Sh}}_{1,2}(K_{\gothp}^1)$ the geoemtric generic fiber of $\widetilde{\Sh}_{1,2}(K_\gothp^1).$ Since blow up does not change the generic fiber, ${\rm H}^{p+q}_{\acute{e}t}(\overline{\widetilde{Sh}}_{1,2}(K_{\gothp}^1),k_\lambda)={\rm H}^{p+q}_{\acute{e}t}(\overline{Sh}_{1,2}(K_{\gothp}^1),k_\lambda).$ The weight spectral sequence gives a filtration $\Fil_{\bullet}{\rm H}^{p+q}_{\acute{e}t}(\overline{\widetilde{Sh}}_{1,2}(K_\gothp^1),k_\lambda).$ We denote the $i$-th graded piece by $Gr_i {\rm H}^{p+q}_{\acute{e}t}(\widetilde{Sh}_{1,2}(K_\gothp^1),k_\lambda).$ From now on, we consider the the spectral sequence after localizing at $\mathfrak{m}.$
\begin{proposition}
\label{gr_2 calculation}
    We have
    $
    Gr_2 {\rm H}^{p+q}_{\acute{e}t}(\widetilde{Sh}_{1,2}(K_\gothp^1),k_\lambda)_{\mathfrak{m}}\subseteq {\rm Ch}^{1}(\overline{\Sh}_{1,2}^{\rm ss},1,k_\lambda(-2))_{\mathfrak{m}}={\rm Ch}^{1}(\Sh_{1,2}^{\rm ss},1,k_\lambda(-2))_{\mathfrak{m}}.
    $
\end{proposition}
\begin{proof}
    We first show that $
    Gr_2 {\rm H}^{p+q}_{\acute{e}t}(\widetilde{Sh}_{1,2},k_\lambda)_{\mathfrak{m}}\subseteq {\rm Ch}^{1}(\overline{\Sh}_{1,2},1,k_\lambda(-2))_{\mathfrak{m}}.$ It suffices to show $E_{2,\mathfrak{m}}^{-2,6}=\Ker d_{1,\mathfrak{m}}^{-2,6}\subseteq  {\rm Ch}^{1}(\overline{\Sh}_{1,2},1,k_\lambda(-2))_{\mathfrak{m}}.$ 
    
    By proposition~\ref{G:blow up}, we have 
    $
        E_{1,\mathfrak{m}}^{-2,6}={\rm H}^{2}_{\acute{e}t}(\overline{Y}^{(2)},k_\lambda(-2))_{\mathfrak{m}}=H^2_{\acute{e}t}(\overline{\rm Y}_{00}\bigcap \overline{\rm Y}_{11},k_\lambda(-2))_{\mathfrak{m}}^{\oplus 2}={\rm H}^0_{\acute{e}t}(\overline{\Sh}_{0,3},k_\lambda(-3))_{\mathfrak{m}}^{\oplus 2}\oplus {\rm H}^{0}_{\acute{e}t}(\overline{\Sh}_{0,3}(K_\gothp^1),k_\lambda(-3))_{\mathfrak{m}}^{\oplus 2}
        $
        with the last equation following from Proposition~\ref{cohomologyy}
    and
    \begin{align*}
    E_{1,\mathfrak{m}}^{-1,6}=&{\rm H}^{4}_{\acute{e}t}(\overline{Y}^{(1)},k_\lambda(-1))_{\mathfrak{m}}={\rm H}^4_{\acute{e}t}(\overline{\rm Y}_{00}\bigcap \overline{\rm Y}_{10},k_\lambda(-1))_{\mathfrak{m}}\oplus {\rm H}^4_{\acute{e}t}(\overline{\rm Y}_{00}\bigcap \overline{\rm Y}_{01},k_\lambda(-1))_{\mathfrak{m}}\\&\oplus {\rm H}^4_{\acute{e}t}(\overline{\rm Y}_{11}\bigcap \overline{\rm Y}_{10},k_\lambda(-1))_{\mathfrak{m}}\oplus {\rm H}^4_{\acute{e}t}(\overline{\rm Y}_{11}\bigcap \overline{\rm Y}_{01},k_\lambda(-1))_{\mathfrak{m}}\oplus {\rm H}^4_{\acute{e}t}(\overline{\widetilde{\rm Y}}_{00}\bigcap \overline{\widetilde{\rm Y}}_{11},k_\lambda(-1))_{\mathfrak{m}};
    \end{align*}
    with $\overline{\widetilde{\rm Y}}_{00}\bigcap \overline{\widetilde{\rm Y}}_{11}$ the ${\overline{\PP}}^1$-bundle over $\overline{{\rm Y}}_{00}\bigcap \overline{{\rm Y}}_{11}.$ By Proposition~\ref{cohomologyy}, 
    \[
    E_{1,\mathfrak{m}}^{-1,6}={\rm H}^0_{\acute{e}t}(\overline{\Sh}_{0,3},k_\lambda(-3))_{\mathfrak{m}}^{\oplus 6}\oplus {\rm H}^{0}_{\acute{e}t}(\overline{\Sh}_{0,3}(K_\gothp^1),k_\lambda(-3))_{\mathfrak{m}}\oplus {\rm H}^4_{\acute{e}t}(\overline{\rm Y}_{11}\bigcap \overline{\rm Y}_{10},k_\lambda(-1))_{\mathfrak{m}}\oplus {\rm H}^4_{\acute{e}t}(\overline{\rm Y}_{11}\bigcap \overline{\rm Y}_{01},k_\lambda(-1))_{\mathfrak{m}}.
    \]
    If $d_{1,\mathfrak{m}}^{-2,6}(x,y,z,w)=0$ with $x,y\in {\rm H}^0_{\acute{e}t}(\overline{\Sh}_{0,3},k_\lambda(-3))_{\mathfrak{m}}$ and $z,w \in {\rm H}^{0}_{\acute{e}t}(\overline{\Sh}_{0,3}(K_\gothp^1),k_\lambda(-3))_{\mathfrak{m}},$ then composition with the projection to ${\rm H}^4_{\acute{e}t}(\overline{\widetilde{\rm Y}}_{00}\bigcap \overline{\widetilde{\rm Y}}_{11},k_\lambda(-1))_{\mathfrak{m}}$ in $E_{1,\mathfrak{m}}^{-1,6}$ shows that $x=y$ and $z=w.$ Consider the composition of the map $H^2_{\acute{e}t}(\overline{\rm Y}_{00}\bigcap \overline{\rm Y}_{11},k_\lambda(-2))_{\mathfrak{m}}\ra {\rm H}^4_{\acute{e}t}(\overline{\rm Y}_{00}\bigcap \overline{\rm Y}_{10},k_\lambda(-1))_{\mathfrak{m}}\oplus {\rm H}^4_{\acute{e}t}(\overline{\rm Y}_{11}\bigcap \overline{\rm Y}_{10},k_\lambda(-1))_{\mathfrak{m}}$ induced by the closed immersion as a part of $d_{1,\mathfrak{m}}^{-2,6}$ with the  natural map
    ${\rm H}^4_{\acute{e}t}(\overline{\rm Y}_{11}\bigcap \overline{\rm Y}_{10},k_\lambda(-1))_{\mathfrak{m}}\ra {\rm H}^4_{\acute{e}t}(\overline{C}_1\bigcap \overline{\rm Y}_{10},k_\lambda(-1))_{\mathfrak{m}}\oplus {\rm H}^4_{\acute{e}t}(\overline{C}_2\bigcap \overline{\rm Y}_{10},k_\lambda(-1))_{\mathfrak{m}}\oplus {\rm H}^4_{\acute{e}t}(\overline{\rm Y}_{00}\bigcap \overline{\rm Y}_{11},k_\lambda(-1))_{\mathfrak{m}}$ also induced by closed immersions and denote it by $\alpha$, we get $x=y=0,\quad z=w\in\Ker({\rm H}^{0}_{\acute{e}t}(\overline{\Sh}_{0,3}(K_\gothp^1),k_\lambda(-3))_{\mathfrak{m}}\xra{\alpha}{\rm H}^{0}_{\acute{e}t}(\overline{\Sh}_{0,3},k_\lambda(-3))_{\mathfrak{m}})^{\oplus 3}.\footnote{Here we still use $\alpha$ to denote the map, even though it is the composition of $\alpha$ with a projection map.} $ It can be checked easily that $\alpha$ is exact the map appearing in Theorem~\ref{IIhara lemma} and Theorem~\ref{Chow}. Hence $E_{2,\mathfrak{m}}^{-2,6}=\Ker d_{1,\mathfrak{m}}^{-2,6}\subseteq  {\rm Ch}^{1}(\overline{\Sh}_{1,2}^{\rm ss},1,k_\lambda(-2))_{\mathfrak{m}}.$

    The equation ${\rm Ch}^{1}(\overline{\Sh}_{1,2}^{\rm ss},1,k_\lambda(-2))_{\mathfrak{m}}={\rm Ch}^{1}(\Sh_{1,2}^{\rm ss},1,k_\lambda(-2))_{\mathfrak{m}}$ follows from Theorem~\ref{S:mod l cohomology} which shows that the Galois action on ${\rm Ch}^{1}(\overline{\Sh}_{1,2}^{\rm ss},1,k_\lambda(-2))_{\mathfrak{m}}={\rm Ch}^{1}(\Sh_{1,2}^{\rm ss},1,k_\lambda(-2))_{\mathfrak{m}}$ is trivial.
\end{proof}
\begin{remark}
In fact, with a torsion-freeness argument as in \cite[Proposition 6.3.1]{LTXZZ}, we can show $Gr_2 {\rm H}^{p+q}_{\acute{e}t}(\widetilde{Sh}_{1,2}(K_\gothp^1),k_\lambda)_{\mathfrak{m}}={\rm Ch}^{1}(\Sh_{1,2}^{\rm ss},1,k_\lambda(-2))_{\mathfrak{m}}.$
But it turns out that we do not need to know what $Gr_2 {\rm H}^{p+q}_{\acute{e}t}(\widetilde{Sh}_{1,2}(K_\gothp^1),k_\lambda)_{\mathfrak{m}}$ is for the proof of Theorem~\ref{main theorem n=3} as can be seen below.    
\end{remark}

To calculate $Gr_1 {\rm H}^{p+q}_{\acute{e}t}(\widetilde{Sh}_{1,2}(K_\gothp^1),k_\lambda)_{\mathfrak{m}},$ we need to analyze the weight spectral sequence with $\cO_\lambda$-coefficient:
$$
\mathcal{E}_1^{p,q}=\bigoplus\limits_{i\geq\max(0,-p)}{\rm H}^{q-2i}_{\acute{e}t}(\overline{Y}^{(p+2i)},\cO_\lambda(-i))
\Rightarrow {\rm H}^{p+q}_{\acute{e}t}(\overline{\widetilde{Sh}}_{1,2}(K_{\gothp}^1),\cO_\lambda),
$$

We have the following lemma which can be checked directly:
\begin{lemma}\label{torsion-free lemma}
    Let $A,B,C$ be two abelian groups satisfying the short exact sequence $0\ra A\ra B\ra C.$ Suppose $A,C$ are torsion-free, then $B$ is torsion-free.
\end{lemma}

\begin{lemma}\label{torsion-free of supersingular locus}
    Let $\Sh_{1,2}^{\rm ss}$ denote the supersingular locus of $\Sh_{1,2}.$ Then ${\rm H}^0_{\acute{e}t}(\overline\Sh_{1,2}^{\rm ss},\cO_\lambda)_{\mathfrak{m}}=0$ and ${\rm H}^i_{\acute{e}t}(\overline\Sh_{1,2}^{\rm ss},\cO_\lambda)_{\mathfrak{m}}$ is torsion-free for $i=1,2.$
\end{lemma}
\begin{proof}
    For $1\leq i\leq 3$ and $z\in \overline\Sh_{0,3},$ denote the closed immersion of $\overline{\rm Y}_{i,z}$ into $\overline\Sh_{1,2}^{\rm ss}$ by $\iota_{i,z}.$ For $1\leq i_1,i_2\leq 3$ and $z_1, z_2\in \overline\Sh_{0,3}$ such that $(i_1,z_1)\neq (i_2,z_2)$ and $\overline{\rm Y}_{i_1,z_1}\cap\overline{\rm Y}_{i_2,z_2}$ is nonempty, denote the closed immersion of $\overline{\rm Y}_{i_1,z_1}\cap\overline{\rm Y}_{i_2,z_2}$ into $\overline\Sh_{1,2}^{\rm ss}$ by $\iota_{i_1,i_2,z_1,z_2}.$ For $1\leq i_1,i_2,i_3\leq 3$ and $z_1, z_2, z_3\in \overline\Sh_{0,3}$ such that $(i_1,z_1)\neq (i_2,z_2)\neq (i_3,z_3)$ and $\overline{\rm Y}_{i_1,z_1}\cap\overline{\rm Y}_{i_2,z_2}\cap\overline{\rm Y}_{i_3,z_3}$ is nonempty, denote the closed immersion of $\overline{\rm Y}_{i_1,z_1}\cap\overline{\rm Y}_{i_2,z_2}\cap\overline{\rm Y}_{i_3,z_3}$ into $\overline\Sh_{1,2}^{\rm ss}$ by $\iota_{i_1,i_2,i_3,z_1,z_2,z_3}.$ For $1\leq i_1,i_2,i_3,i_4\leq 3$ and $z_1, z_2, z_3,  z_4\in \overline\Sh_{0,3}$ such that $(i_1,z_1)\neq (i_2,z_2)\neq (i_3,z_3)\neq (i_4,z_4)$ and $\overline{\rm Y}_{i_1,z_1}\cap\overline{\rm Y}_{i_2,z_2}\cap\overline{\rm Y}_{i_3,z_3}\cap\overline{\rm Y}_{i_4,z_4}$ is nonempty, denote the closed immersion of $\overline{\rm Y}_{i_1,z_1}\cap\overline{\rm Y}_{i_2,z_2}\cap\overline{\rm Y}_{i_3,z_3}\cap\overline{\rm Y}_{i_4,z_4}$ into $\overline\Sh_{1,2}^{\rm ss}$ by $\iota_{i_1,i_2,i_3,i_4,z_1,z_2,z_3,z_4}.$ 

    By \cite[Proposition 4.9]{HTX17}, $\overline\Sh_{0,3}$ is finite. Then we have the following exact sequence
    \[\begin{tikzcd}
	0 & {\cO_\lambda} & {\bigoplus\limits_{1\leq i\leq3}\bigoplus\limits_{z\in \Sh_{0,3}}\iota_{i,z,*}\cO_\lambda} & {\bigoplus\limits_{1\leq i_1,i_2\leq3}\bigoplus\limits_{z_1, z_2\in \Sh_{0,3}}\iota_{i_1,i_2,z_1,z_2,*}\cO_\lambda} \\
	&&& {\bigoplus\limits_{1\leq i_1,i_2,i_3\leq3}\bigoplus\limits_{z_1, z_2, z_3\in \Sh_{0,3}}\iota_{i_1,i_2,i_3,z_1,z_2,z_3,*}\cO_\lambda} \\
	&&& {\bigoplus\limits_{1\leq i_1,i_2,i_3,i_4\leq3}\bigoplus\limits_{z_1, z_2, z_3, z_4\in \Sh_{0,3}}\iota_{i_1,i_2,i_3,i_4,z_1,z_2,z_3,z_4,*}\cO_\lambda}
	\arrow[from=1-1, to=1-2]
	\arrow[from=1-2, to=1-3]
	\arrow[from=1-3, to=1-4]
	\arrow[from=1-4, to=2-4]
	\arrow[from=2-4, to=3-4]
\end{tikzcd}\]

Applying $\Gamma(\overline\Sh_{1,2}^{\rm ss},\bullet)_\mathfrak{m}$ to the exact sequence and since it is left exact, the short sequence 
\[\begin{tikzcd}
	0 & {{\rm H}^0_{\acute{e}t}(\overline{\Sh}_{1,2}^{\rm ss},\cO_\lambda)_\mathfrak{m}} & {\bigoplus\limits_{1\leq i\leq3}{\rm H}^0_{\acute{e}t}(\overline{\rm Y}_i,\cO_\lambda)_\mathfrak{m}} & {\bigoplus\limits_{1\leq i_1,i_2\leq3}\bigoplus\limits_{z_1, z_2\in \Sh_{0,3}}{\rm H}^0_{\acute{e}t}(\overline{\rm Y}_{i_1,z_1}\cap\overline{\rm Y}_{i_2,z_2},\cO_\lambda)_\mathfrak{m}}
	\arrow[from=1-1, to=1-2]
	\arrow[from=1-2, to=1-3]
	\arrow["\alpha", from=1-3, to=1-4]
\end{tikzcd}\] is exact.

By Proposition~\ref{S:corre(1,n-1)(0,n)}, for any point $z\in \overline\Sh_{0,3},$ $\overline{\rm Y}_{1,z}$ and $\overline{\rm Y}_{3,z}$ are isomorphic to $\PP^2$ and $\overline{\rm Y}_{2,z}$ is isomorphic to $Z^{\langle 3\rangle}_2.$ 
By \cite[Proposition 6.4]{HTX17}, for any two different points $z_1,z_2\in \overline{\Sh}_{0,3},$ we have
\begin{itemize}
    \item $\overline{\rm Y}_{1,z_1}\cap\overline{\rm Y}_{1,z_2}$ and $\overline{\rm Y}_{3,z_1}\cap\overline{\rm Y}_{3,z_2}$ are empty;
    \item $\overline{\rm Y}_{2,z_1}\cap\overline{\rm Y}_{2,z_2}$ is nonempty if and only if $z_1\in {\rm R}_{\gothp}^{(1,2)}{\rm S}_\gothp^{-1}(z)$ and in this case $\overline{\rm Y}_{2,z_1}\cap\overline{\rm Y}_{2,z_2}$ is a point;
    \item $\overline{\rm Y}_{1,z_1}\cap\overline{\rm Y}_{2,z_2}$ is nonempty if and only if $z_2\in {\rm T}_\gothp^{(1)}(z_1),$ in which case $\overline{\rm Y}_{1,z_1}\cap\overline{\rm Y}_{2,z_2}$ is isomorphic to $\PP^1;$
    \item $\overline{\rm Y}_{2,z_1}\cap\overline{\rm Y}_{3,z_2}$ is nonempty if and only if $z_2\in {\rm T}_\gothp^{(1)}(z_1),$ in which case $\overline{\rm Y}_{2,z_1}\cap\overline{\rm Y}_{3,z_2}$ is isomorphic to $\PP^1;$
    \item $\overline{\rm Y}_{1,z_1}\cap\overline{\rm Y}_{3,z_2}$ is nonempty if and only if $z_3\in {\rm T}_\gothp^{(2)}(z_1),$ in which case $\overline{\rm Y}_{1,z_1}\cap\overline{\rm Y}_{3,z_2}$ is a point.
\end{itemize}

Consider $\alpha_1,\alpha_2$ induced by ${\bigoplus\limits_{1\leq i\leq3}\bigoplus\limits_{z\in \Sh_{0,3}}\iota_{i,z,*}\cO_\lambda} \ra \bigoplus\limits_{z_2\in {\rm T}_\gothp^{(1)}(z_1)}\iota_{1,2,z_1,z_2,*}\cO_\lambda$ and ${\bigoplus\limits_{1\leq i\leq3}\bigoplus\limits_{z\in \Sh_{0,3}}\iota_{i,z,*}\cO_\lambda} \ra \bigoplus\limits_{z_2\in {\rm T}_\gothp^{(1)}(z_1)}\iota_{2,3,z_1,z_2,*}\cO_\lambda$ respectively, then $\alpha_1$ can be expressed as 
$
	{{\rm H}^0_{\acute{e}t}(\overline{\Sh}_{0,3},\cO_\lambda)^{\oplus 3}_\mathfrak{m}} \xra{{({\mathop{p}\limits^{\rightarrow}}^{*},{\mathop{p}\limits^{\leftarrow}}^{*},0)}} {{\rm H}^0_{\acute{e}t}(\overline{\Sh}_{0,3}(K_\gothp^1),\cO_\lambda)_\mathfrak{m}}
$ and $\alpha_2$ can be expressed as
$
	{{\rm H}^0_{\acute{e}t}(\overline{\Sh}_{0,3},\cO_\lambda)^{\oplus 3}_\mathfrak{m}} \xra{{(0,{\mathop{p}\limits^{\rightarrow}}^{*},{\mathop{p}\limits^{\leftarrow}}^{*})}} {{\rm H}^0_{\acute{e}t}(\overline{\Sh}_{0,3}(K_\gothp^1),\cO_\lambda)_\mathfrak{m}}
$

By Theorem~\ref{I3}, $\Ker\alpha_1$ is the third component of the direct sum on the left and $\Ker\alpha_2$ is the first component of the direct sum on the left. Hence
$\Ker\alpha\subseteq\Ker\alpha_1\cap\Ker\alpha_2=0.$ This shows ${{\rm H}^0_{\acute{e}t}(\overline{\Sh}_{1,2}^{\rm ss},\cO_\lambda)_\mathfrak{m}}=0.$ Since the maps $\alpha_1,\alpha_2,\alpha$ are induced by immersions of connected components, the same result holds for ${{\rm H}^0_{\acute{e}t}(\overline{\Sh}_{1,2}^{\rm ss},k_\lambda)_\mathfrak{m}},$ which is also zero. Then by the universal coefficient theorem, ${{\rm H}^1_{\acute{e}t}(\overline{\Sh}_{1,2}^{\rm ss},\cO_\lambda)_\mathfrak{m}}$ is torsion-free.
    Denote the kernel of $$
    \bigoplus\limits_{1\leq i_1,i_2\leq3}\bigoplus\limits_{z_1, z_2\in \Sh_{0,3}}\iota_{i_1,i_2,z_1,z_2,*}\cO_\lambda\ra\bigoplus\limits_{1\leq i_1,i_2,i_3\leq3}\bigoplus\limits_{z_1, z_2, z_3\in \Sh_{0,3}}\iota_{i_1,i_2,i_3,z_1,z_2,z_3,*}\cO_\lambda
    $$ by $C,$ the kernel of $$\bigoplus\limits_{1\leq i_1,i_2,i_3\leq3}\bigoplus\limits_{z_1, z_2,z_3\in \Sh_{0,3}}\iota_{i_1,i_2,i_3,z_1,z_2,z_3,*}\cO_\lambda\ra \bigoplus\limits_{1\leq i_1,i_2,i_3,i_4\leq3}\bigoplus\limits_{z_1, z_2, z_3,z_4\in \Sh_{0,3}}\iota_{i_1,i_2,i_3,i_4,z_1,z_2,z_3,z_4,*}\cO_\lambda$$ by $D.$

    Then we have short exact sequences
   $
	0 \ra {\cO_\lambda} \ra {\bigoplus\limits_{1\leq i\leq3}\bigoplus\limits_{z\in \Sh_{0,3}}\iota_{i,z,*}\cO_\lambda} \ra C \ra 0 $ and $
	0 \ra C \ra {\bigoplus\limits_{1\leq i_1,i_2\leq3}\bigoplus\limits_{z_1,z_2\in \Sh_{0,3}}\iota_{i_1,i_2,z_1,z_2,*}\cO_\lambda} \ra D \ra 0
	$

Taking cohomology of the exact sequences, we get
\[\begin{tikzcd}
[
    column sep=small, 
    row sep=scriptsize,
    cells={font=\footnotesize} 
]
	0 & {{\rm H}^1_{\acute{e}t}(\overline\Sh_{1,2}^{\rm ss},C)_{\mathfrak{m}}} & {{\rm H}^2_{\acute{e}t}(\overline\Sh_{1,2}^{\rm ss},\cO_\lambda)_{\mathfrak{m}}} & {\bigoplus\limits_{1\leq i\leq3}{\rm H}^2_{\acute{e}t}(\overline{\rm Y}_i,\cO_\lambda)_\mathfrak{m}} \\
	0 & {{\rm H}^0_{\acute{e}t}(\overline\Sh_{1,2}^{\rm ss},C)_{\mathfrak{m}}} & {\bigoplus\limits_{1\leq i_1,i_2\leq3}\bigoplus\limits_{z_1, z_2\in \Sh_{0,3}}{\rm H}^0_{\acute{e}t}(\overline{\rm Y}_{i_1,z_1}\cap\overline{\rm Y}_{i_2,z_2},\cO_\lambda)_{\mathfrak{m}}} & {{\rm H}^0_{\acute{e}t}(\overline\Sh_{1,2}^{\rm ss},D)_{\mathfrak{m}}} \\
	&&& {{\rm H}^1_{\acute{e}t}(\overline\Sh_{1,2}^{\rm ss},C)_{\mathfrak{m}}} \\
	&&& 0
	\arrow[from=1-1, to=1-2]
	\arrow[from=1-2, to=1-3]
	\arrow[from=1-3, to=1-4]
	\arrow[from=2-1, to=2-2]
	\arrow[from=2-2, to=2-3]
	\arrow["{\beta_1}", from=2-3, to=2-4]
	\arrow[from=2-4, to=3-4]
	\arrow[from=3-4, to=4-4]
\end{tikzcd}\]
The left zero of the first exact sequence follows from the vanishing of the first cohomology groups of ${\rm Y}_i$ for $1\leq i\leq 3.$ The right zero of the second row follows from the intersection of ${\rm Y}_{i_1,z_1}$ and ${\rm Y}_{i_2,z_2}$ can only be point or $\PP^1.$ By lemma~\ref{torsion-free lemma}, it suffices to show ${{\rm H}^1_{\acute{e}t}(\overline\Sh_{1,2}^{\rm ss},C)_{\mathfrak{m}}}$ is torsion-free.

Applying the left exact functor $\Gamma(\overline{\Sh}_{1,2}^{\rm ss},\bullet)_\mathfrak{m}$ to the exact sequence 
$$0\ra D\ra \bigoplus\limits_{1\leq i_1,i_2,i_3\leq3}\bigoplus\limits_{z_1,z_2,z_3\in \Sh_{0,3}}\iota_{i_1,i_2,i_3,z_1,z_2,z_3,*}\cO_\lambda\ra \bigoplus\limits_{1\leq i_1,i_2,i_3,i_4\leq3}\bigoplus\limits_{z_1, z_2,z_3,z_4\in \Sh_{0,3}}\iota_{i_1,i_2,i_3,i_4,z_1,z_2,z_3,z_4,*}\cO_\lambda,$$ there is an injection
$
{{\rm H}^0_{\acute{e}t}(\overline\Sh_{1,2}^{\rm ss},D)_{\mathfrak{m}}}\ra {\bigoplus\limits_{1\leq i_1,i_2,i_3\leq3}\bigoplus\limits_{z_1\neq z_2\neq z_3\in \Sh_{0,3}}{\rm H}^0_{\acute{e}t}(\overline{\rm Y}_{i_1,z_1}\cap\overline{\rm Y}_{i_2,z_2}\cap\overline{\rm Y}_{i_3,z_3},\cO_\lambda)_{\mathfrak{m}}}.
$
Composing with $\beta_1$, we get the commutative diagram
\begin{center}
\begin{sideways}
\begin{adjustbox}{width=\textheight, keepaspectratio}
\begin{tikzcd}[column sep=scriptsize, row sep=scriptsize]
0 & 
\scriptsize{
\bigoplus\limits_{\substack{1\leq i_1,i_2\leq3 \\ z_1\neq z_2\in \Sh_{0,3}}}
{\rm H}^0_{\acute{e}t}(\overline{\mathrm{Y}}_{i_1,z_1}\cap\overline{\mathrm{Y}}_{i_2,z_2},\mathcal{O}_\lambda)_{\mathfrak{m}} 
} & 
\scriptsize{
\bigoplus\limits_{\substack{1\leq i_1,i_2\leq3 \\ z_1\neq z_2\in \Sh_{0,3}}}
{\rm H}^0_{\acute{e}t}(\overline{\mathrm{Y}}_{i_1,z_1}\cap\overline{\mathrm{Y}}_{i_2,z_2},\mathcal{O}_\lambda)_{\mathfrak{m}}
} & 0 & 0 \\
0 & 
\scriptsize{{\rm H}^0_{\acute{e}t}(\overline{\Sh}_{1,2}^{\mathrm{ss}},D)_{\mathfrak{m}}} & 
\scriptsize{
\bigoplus\limits_{\substack{1\leq i_1,i_2,i_3\leq3 \\ z_1\neq z_2\neq z_3\in \Sh_{0,3}}}
{\rm H}^0_{\acute{e}t}(\overline{\mathrm{Y}}_{i_1,z_1}\cap\overline{\mathrm{Y}}_{i_2,z_2}\cap\overline{\mathrm{Y}}_{i_3,z_3},\mathcal{O}_\lambda)_{\mathfrak{m}}
} & 
\scriptsize{
\bigoplus\limits_{\substack{1\leq i_1,i_2,i_3,i_4\leq3 \\ z_1\neq z_2\neq z_3\neq z_4\in \Sh_{0,3}}}
{\rm H}^0_{\acute{e}t}(\overline{\mathrm{Y}}_{i_1,z_1}\cap\cdots\cap\overline{\mathrm{Y}}_{i_4,z_4},\mathcal{O}_\lambda)_{\mathfrak{m}}
} \\
& 
\scriptsize{H^1_{\acute{e}t}(\overline{\Sh}_{1,2}^{\mathrm{ss}},C)_{\mathfrak{m}}} & 
\scriptsize{\mathrm{Coker}\,\beta_2}
\arrow[from=1-1, to=1-2]
\arrow[from=1-2, to=1-3]
\arrow["{{\beta_1}}", from=1-2, to=2-2]
\arrow[from=1-3, to=1-4]
\arrow["{{\beta_2}}", from=1-3, to=2-3]
\arrow[from=1-4, to=1-5]
\arrow[from=1-4, to=2-4]
\arrow[from=2-1, to=2-2]
\arrow[from=2-2, to=2-3]
\arrow[from=2-2, to=3-2]
\arrow[from=2-3, to=2-4]
\arrow[from=2-3, to=3-3]
\arrow["{{\beta_3}}", from=3-2, to=3-3]
\end{tikzcd}
\end{adjustbox}
\end{sideways}
\end{center}

Since $\beta_2$ is induced by immersions of connected components, ${\rm Coker}\beta_2$ counts the difference of the number of the connected components and hence is torsion free. By snake lemma, $\beta_3$ is injective. This shows ${{\rm H}^1_{\acute{e}t}(\overline\Sh_{1,2}^{\rm ss},C)_{\mathfrak{m}}}$ is torsion free. Here we finish the proof.
\end{proof}

\begin{proposition}
    $\mathcal{E}_{1,\mathfrak{m}}^{p,q}$ is torsion free except $(p,q)=(1,3),(1,4),(-1,5),(-1,6).$
\end{proposition}
\begin{proof}
    The torsion-freeness can be checked easily for $p\neq \pm1$ by Hyposthesis~\ref{Main hypo}, Proposition~\ref{bloww up}, Proposition~\ref{Corresss}, Proposition~\ref{cohomologyy} and Proposition~\ref{GI:characterize of Y_11}. To show the torsion-freeness for $p=\pm1,$ it remains to show ${\rm H}^i_{\acute{e}t}(\overline{\rm Y}_{10}\bigcap \overline{\rm Y}_{11},\cO_{\lambda})_{\mathfrak{m}}$ and ${\rm H}^i_{\acute{e}t}(\overline{\rm Y}_{01}\bigcap \overline{\rm Y}_{11},\cO_{\lambda})_{\mathfrak{m}}$ are torsion free for $i\neq 3,4.$ By Poincare duality, it suffices to show ${\rm H}^i_{\acute{e}t}(\overline{\rm Y}_{10}\bigcap \overline{\rm Y}_{11},\cO_{\lambda})_{\mathfrak{m}}$ and ${\rm H}^i_{\acute{e}t}(\overline{\rm Y}_{01}\bigcap \overline{\rm Y}_{11},\cO_{\lambda})_{\mathfrak{m}}$ are torsion free for $i=0,1,2.$

    Recall the subschemes $N_{10},N_{01},\mathcal{U}_{10},\mathcal{U}_{01}$ in Section~\ref{I3}. For $i\geq 7$ and $i=0,1,3,5,$ ${\rm H}^{i}_{\acute{e}t}(\overline{\rm Y}_{10},\cO_\lambda)_{\mathfrak{m}}$ and ${\rm H}^{i}_{\acute{e}t}(\overline{\rm Y}_{01},\cO_\lambda)_{\mathfrak{m}}$ are all zero.

    By excision sequence, there is an injective map from ${\rm H}^{i}_{\acute{e}t}(\overline{\rm Y}_{10}\bigcap\overline{\rm Y}_{11},\cO_\lambda)_{\mathfrak{m}}$ to $H^{i+1}_{c}(\overline{\mathcal{U}}_{10},\cO_\lambda)_{\mathfrak{m}}$ for $i=0,1.$ Also by excision sequence, we get an exact sequence
$
	{{\rm H}^{2}_{c}(\overline{\mathcal{U}}_{10},\cO_\lambda)_{\mathfrak{m}}} \ra {{\rm H}^{2}_{\acute{e}t}(\overline{\rm Y}_{10},\cO_\lambda)_{\mathfrak{m}}} \ra {{\rm H}^{2}_{\acute{e}t}(\overline{\rm Y}_{10}\bigcap \overline{\rm Y}_{11},\cO_\lambda)_{\mathfrak{m}}} \ra {H^{3}_{c}(\overline{\mathcal{U}}_{10},\cO_\lambda)_{\mathfrak{m}}} \ra {0}
	$
    Now we show that ${\rm H}^{0}_{c}(\overline{\mathcal{U}}_{10},\cO_\lambda)_{\mathfrak{m}}=0$ and ${\rm H}^{i}_{c}(\overline{\mathcal{U}}_{10},\cO_\lambda)_{\mathfrak{m}}$ are torsion-free for $i=1,2.$ 

    By excision sequence and Proposition~\ref{cohomologyy}, ${\rm H}^{i}_{c}(\overline\Sh_{1,2}-\overline{N}_{10},\cO_\lambda)_{\mathfrak{m}}={\rm H}^{i}_{c}(\overline{\mathcal{U}}_{10},\cO_\lambda)_{\mathfrak{m}}$ for $i=1,2.$ And the sequence 
$
	0 \ra {{\rm H}^{3}_{\acute{e}t}(\overline{\Sh}_{12}-\overline{\rm N}_{10},\cO_\lambda)_{\mathfrak{m}}} \ra {{\rm H}^{3}_{c}(\overline{\mathcal{U}}_{10},\cO_\lambda)_{\mathfrak{m}}} \ra {{\rm H}^{3}_{\acute{e}t}(\overline{\rm Y}_{00}\bigcap \overline{\rm Y}_{10}-\overline{\rm Y}_{00}\bigcap \overline{\rm Y}_{11},\cO_\lambda)_{\mathfrak{m}}}
$ is exact.

    By Mayer-Vietoris sequence and Proposition~\ref{affineness}, ${\rm H}^{i}_{c}(\overline\Sh_{1,2}-\overline{N}_{10},\cO_\lambda)_{\mathfrak{m}}\oplus {\rm H}^{i}_{c}(\overline\Sh_{1,2}-\overline{N}_{01},\cO_\lambda)_{\mathfrak{m}}$ is isomorphic to ${\rm H}^{i}_{c}(\overline\Sh_{1,2}-\overline{N}_{10}\bigcap \overline{N}_{01},\cO_\lambda)_{\mathfrak{m}}$ for $0\leq i\leq 3.$ By excision sequence and Hypothesis~\ref{Main hypo}, this is isomorphic to $H^{i-1}_{c}(\overline{N}_{10}\bigcup Fr''(\overline{N}_{01}),\cO_\lambda)_{\mathfrak{m}},$ which is exactly $H^{i-1}_{\acute{e}t}(\overline\Sh_{1,2}^{\rm ss},\cO_\lambda)_{\mathfrak{m}}.$ Then it comes from Lemma~\ref{torsion-free of supersingular locus}.
    
    The torsion-freeness of ${{\rm H}^{2}_{\acute{e}t}(\overline{\rm Y}_{10}\bigcap \overline{\rm Y}_{11},\cO_\lambda)_{\mathfrak{m}}}$ comes from 
the commutative diagram
    \[\begin{tikzcd}
	{{{\rm H}^{2}_{c}(\overline{\mathcal{U}}_{10},\cO_\lambda)_{\mathfrak{m}}}} && {{{\rm H}^{2}_{\acute{e}t}(\overline{\rm Y}_{10},\cO_\lambda)_{\mathfrak{m}}}} \\
	{{\rm H}^{2}_{c}(\overline\Sh_{1,2}-\overline{N}_{10},\cO_\lambda)_{\mathfrak{m}}} && {{{\rm H}^{2}_{\acute{e}t}(\overline{\rm Sh}_{1,2},\cO_\lambda)_{\mathfrak{m}}}=0}
	\arrow[from=1-1, to=1-3]
	\arrow[from=2-1, to=1-1]
	\arrow[from=2-1, to=2-3]
	\arrow[from=2-3, to=1-3]
\end{tikzcd}\]
    
    By excision sequence and blow up exact sequence, ${{\rm H}^{2}_{c}(\overline{\mathcal{U}}_{10},\cO_\lambda)_{\mathfrak{m}}}={\rm H}^{2}_{c}(\overline\Sh_{1,2}-\overline{N}_{10},\cO_\lambda)_{\mathfrak{m}}
    ,$
    ${{\rm H}^{2}_{\acute{e}t}(\overline{\rm Y}_{10},\cO_\lambda)_{\mathfrak{m}}}={\rm H}^0_{\acute{e}t}(\overline{\rm Y}_3,H^2_{\acute{e}t}(\PP^1,\cO_\lambda))_\mathfrak{m}.$ 
    Thus the image of 
    ${{\rm H}^{2}_{c}(\overline{\mathcal{U}}_{10},\cO_\lambda)_{\mathfrak{m}}}$ in ${{\rm H}^{2}_{\acute{e}t}(\overline{\rm Y}_{10},\cO_\lambda)_{\mathfrak{m}}}$ is zero. By Lemma~\ref{torsion-free lemma}, we get the torsion-freeness of ${{\rm H}^{2}_{\acute{e}t}(\overline{\rm Y}_{10}\bigcap \overline{\rm Y}_{11},\cO_\lambda)_{\mathfrak{m}}}.$ 
    Then the Proposition follows from the symmetry between ${\rm Y}_{10}$ and ${\rm Y}_{01}.$
\end{proof}

\begin{proposition}\label{vanishing of gr_1}
    Under Hypothesis~\ref{Main hypo2}, $E_{1,\mathfrak{m}}^{-1,5}=0.$ As a direct corollary,
    $Gr_1 {\rm H}^{p+q}_{\acute{e}t}(\widetilde{Sh}_{1,2}(K_\gothp^1),k_\lambda)_{\mathfrak{m}}=0.$
\end{proposition}
\begin{proof}
    It can be checked directly that $\mathcal{E}_{1,\mathfrak{m}}^{-1,5}={\rm H}^3_{\acute{e}t}(\overline{\rm Y}_{10}\bigcap \overline{\rm Y}_{11},\cO_\lambda)_{\mathfrak{m}}\oplus {\rm H}^3_{\acute{e}t}(\overline{\rm Y}_{01}\bigcap \overline{\rm Y}_{11},\cO_\lambda)_{\mathfrak{m}}$ and the torsion part of $\mathcal{E}_{1,\mathfrak{m}}^{-1,6}$ comes from 
    ${\rm H}^4_{\acute{e}t}(\overline{\rm Y}_{10}\bigcap \overline{\rm Y}_{11},\cO_\lambda)_{\mathfrak{m}}\oplus {\rm H}^4_{\acute{e}t}(\overline{\rm Y}_{01}\bigcap \overline{\rm Y}_{11},\cO_\lambda)_{\mathfrak{m}}.$
    By the universal coefficient theorem, $E_{1,\mathfrak{m}}^{-1,5}=\mathcal{E}_{1,\mathfrak{m}}^{-1,6}\otimes_{\cO_\lambda}k_\lambda\oplus {\rm Tor}_1(\mathcal{E}_{1,\mathfrak{m}}^{-1,6},k_\lambda).$
    By Weight monodromy conjecture \cite{Sch12}, the eigenvalue of ${\rm Fr}_{p^2}$ on  $Gr_1 {\rm H}^{p+q}_{\acute{e}t}(\widetilde{Sh}_{1,2}(K_\gothp^1),{\rm L}_\lambda)_{\mathfrak{m}}$ has absolute value $p^{5}.$ By Hypothesis~\ref{Main hypo2}, it must be zero. This shows the free part of $\mathcal{E}_{1,\mathfrak{m}}^{-1,5}$ is zero.
    
    Since $\mathcal{E}_{\infty,\mathfrak{m}}^{-1,6}=0$ and $\mathcal{E}_{1,\mathfrak{m}}^{1,i}$ are torsion-free for $i=5,6,$ the possible eigenvalue of $\Frob_{p^2}$ on ${\rm Tor}_1(\mathcal{E}_{1,\mathfrak{m}}^{-1,6},k_\lambda)$ can have absolute value $p^6,p^5.$ By Poincare duality and universal coefficient theorem, the possible eigenvalue of ${\rm Fr}_{p^2}$ on ${\rm Tor}_1(\mathcal{E}_{1,\mathfrak{m}}^{-1,5},k_\lambda)$ can have absolute value $p^4,p^5.$ Since $E_{1,\mathfrak{m}}^{1,4}=E_{1,\mathfrak{m}}^{-1,6}(1),$ the possible eigenvalue of ${\rm Fr}_{p^2}$ on $E_{1,\mathfrak{m}}^{1,4}$ can have absolute value $p^4,p^3.$ Then since $E_{1,\mathfrak{m}}^{2,3}=0,$ if the eigenvalue of ${\rm Fr}_{p^2}$ on ${\rm Tor}_1(\mathcal{E}_{1,\mathfrak{m}}^{-1,5},k_\lambda)$ and ${\rm Tor}_1(\mathcal{E}_{1,\mathfrak{m}}^{-1,6},k_\lambda)$ can be $p^5$ indeed, the eigenvalues of ${\rm Fr}_{p^2}$ on $Gr_1 {\rm H}^{p+q}_{\acute{e}t}(\widetilde{Sh}_{1,2}(K_\gothp^1),{k}_\lambda)_{\mathfrak{m}}$ must contain $p^5,$ which contradicts to Hypothesis~\ref{Main hypo2}.
    
    By Poincare duality, it suffices to show that the  eigenvalue of $\Frob_{p^2}$ on ${\rm Tor}_1(\mathcal{E}_{1,\mathfrak{m}}^{-1,6},k_\lambda)$ can not have absolute value $p^6,$ which is equivalent to show the eigenvalue of $\Frob_{p^2}$ on the torsion part of $\mathcal{E}_{1,\mathfrak{m}}^{-1,6}$ can not have absolute valus $p^6.$

    Since $\mathcal{E}_{1,\mathfrak{m}}^{0,6}$ is torsion free, the $p^6$-torsion part of $\mathcal{E}_{1,\mathfrak{m}}^{-1,6}$ is contained in $\ker d_{1,\mathfrak{m}}^{-1,6}.$ Note that $\mathcal{E}_{1,\mathfrak{m}}^{1,5}$ is torsion-free, the eigenvalue of ${\rm Fr}_{p^2}$ can not have absolute value $p^6.$ It can be checked directly that $\mathcal{E}^{2,4}_{2,\mathfrak{m}}=0.$ Then since $\mathcal{E}_{\infty,\mathfrak{m}}^{-1,6}=0,$ the $p^6$-torsion part of $\mathcal{E}_{1,\mathfrak{m}}^{-1,6}$ is contained in $\Im d_{1,\mathfrak{m}}^{0,6}.$ By definition 
    $\mathcal{E}_{1,\mathfrak{m}}^{-2,6}={\rm H}^{2}_{\acute{e}t}(\overline{\rm Y}_{00}\cap\overline{\rm Y}_{11},\cO_\lambda(-2))^{\oplus 2}_{\mathfrak{m}}.
    $
    Let $(x_1,x_2)$ be an element in $\mathcal{E}_{1,\mathfrak{m}}^{-2,6}.$ Suppose $d_{1,\mathfrak{m}}^{-2,6}((x_1,x_2))$ is contained in the $p^6$-torsion part of $\mathcal{E}_{1,\mathfrak{m}}^{-1,6}.$ Consider the projection of $\mathcal{E}_{1,\mathfrak{m}}^{-1,6}$ to the component ${\rm H}^{4}_{\acute{e}t}(\overline{\widetilde{\rm Y}}_{00}\bigcap\overline{\widetilde{\rm Y}}_{11},\cO_\lambda(-1))_{\mathfrak{m}},$ this forces $x_1=-x_2.$ Then $d_{1,\mathfrak{m}}^{-2,6}((x_1,x_2))$ is contained in the anti-diagonal part of the $p^6$-torsion part of $\mathcal{E}_{1,\mathfrak{m}}^{-1,6},$ which contradicts to the $p^6$-torsion part of $\mathcal{E}_{1,\mathfrak{m}}^{-1,6}$ is contained in $\Im d_{1,\mathfrak{m}}^{0,6}.$ Hence it must be zero. Here we finish the proof.
\end{proof}
\begin{remark}\label{torsion vanishing of 10-01}
    By Proposition~\ref{vanishing of gr_1} and some easy calculations, it can be checked that the spectral sequences degenerate at the second page. A direct corollary of this is that ${\rm H}^3_{\acute{e}t}(\overline{\rm Y}_{10}\bigcap\overline{\rm Y}_{11})_{\mathfrak{m}}={\rm H}^3_{\acute{e}t}(\overline{\rm Y}_{01}\bigcap\overline{\rm Y}_{11})_{\mathfrak{m}}=0.$
\end{remark}

\begin{lemma}\label{A:Supsingular vanishing}
    Let $\Frob_{p^2}$ denote the Frobenius morphism on schemes defined over $\overline\FF_{p^2}$. For $1\leq i\leq 3,$ the map ${\rm H}^{0}_{\acute{e}t}(\overline{\rm Y}_{i},k_{\lambda}(1))_{\mathfrak{m}}\ra {\rm H}^{4}_{\acute{e}t}(\overline\Sh_{1,2},k_{\lambda}(3))_{\mathfrak{m}}$ induced by the closed immersion $\iota_i:\overline{\rm Y}_{i}\hookrightarrow\overline\Sh_{1,2}$ is contained in $(1-p^{-6}{\Frob}_{p^2}^{*}){\rm H}^{4}_{\acute{e}t}(\overline\Sh_{1,2},k_{\lambda}(3))_{\mathfrak{m}}.$
\end{lemma}
\begin{proof}
    Take $x\in {\rm H}^{0}_{\acute{e}t}(\overline{\rm Y}_{i},k_{\lambda}(1))_{\mathfrak{m}}.$ Since $\iota_i$ is Frobenius equivariant, $\iota_{i,*}(p^{-2}{\Frob}_{p^2}^{*}(x))=p^{-6}{\Frob}_{p^2}^{*}\iota_{i,*}(x).$ Since the Frobenius action on ${\rm H}^{0}_{\acute{e}t}({\rm Y}_{i},k_{\lambda})_{\mathfrak{m}}$ is trivial, $\iota_{i,*}(x)=p^{-4}{\Frob}_{p^2}^{*}\iota_{i,*}(x).$ Take $x'=\frac{p^2}{p^2-1}\iota_{i,*}(x)\in {\rm H}^{4}_{\acute{e}t}(\overline\Sh_{1,2},k_{\lambda}(2))_{\mathfrak{m}},$ then $x=(1-p^{-6}{\Frob}_{p^2}^{*})x'$ and hence 
    $x\in (1-p^{-6}{\Frob}_{p^2}^{*}){\rm H}^{4}_{\acute{e}t}(\overline\Sh_{1,2},k_{\lambda}(2))_{\mathfrak{m}}.$ Hence we finish the proof.
\end{proof}

\begin{notation}
    Recall the strictly semistable scheme $\widetilde{\cSh}_{1,2}(Iw_\gothp)$ we get by blow up in Proposition~\ref{GI:blow up}.  $\widetilde{\Sh}_{1,2}(Iw_\gothp)$ can be expressed as the union of nine irreducible components, denoted by $\tilde{\rm Z}_{ij}$ for $0\leq i,j\leq 2.$

    Let $Z^{(i)}$ be the disjoint union of all the intersection of $i+1$ irreducible closed components of $\widetilde{\Sh}_{1,2}(Iw_\gothp).$
\end{notation}

By \cite[Corollary 2.8]{Sai03}, we have the weight spectral sequence
$$
{E'}_1^{p,q}=\bigoplus\limits_{i\geq\max(0,-p)}{\rm H}^{q-2i}_{\acute{e}t}(\overline{Z}^{(p+2i)},k_\lambda(-i))
\Rightarrow {\rm H}^{p+q}_{\acute{e}t}(\overline{\widetilde{Sh}}_{1,2}(Iw_\gothp),k_\lambda),
$$
with $\overline{\widetilde{Sh}}_{1,2}(Iw_\gothp)$ the geoemtric generic fiber of $\widetilde{\Sh}_{1,2}(Iw_\gothp).$ since blow up does not change the generic fiber, ${\rm H}^{p+q}_{\acute{e}t}(\overline{\widetilde{Sh}}_{1,2}(Iw_\gothp),k_\lambda)={\rm H}^{p+q}_{\acute{e}t}(\overline{Sh}_{1,2}(Iw_\gothp),k_\lambda).$ The spectral sequence gives a filtration $\Fil_{\bullet}{\rm H}^{p+q}_{\acute{e}t}(\overline{\widetilde{Sh}}_{1,2}(Iw_\gothp^1),k_\lambda).$ We denote the $i$-th graded piece by $Gr_i {\rm H}^{p+q}_{\acute{e}t}(\overline{\widetilde{Sh}}_{1,2}(Iw_\gothp^1),k_\lambda).$ 

 Let $\pi_{Iw}:\widetilde{\cSh}_{1,2}(Iw_\gothp)\ra \cSh_{1,2}(Iw_\gothp)$ and $\pi_{K_\gothp^1}:\widetilde{\cSh}_{1,2}(K_\gothp^1)\ra \cSh_{1,2}(K_\gothp^1)$ denote the blow ups.
By the universal property of blow up, we have the following commutative diagram:
\[\begin{tikzcd}
[
    column sep=small, 
    row sep=scriptsize,
    cells={font=\footnotesize} 
]
	&& {\widetilde{\rm Sh}_{1,2}(Iw_\gothp)} \\
	& {\widetilde{\rm Sh}_{1,2}(K_\gothp^1)} & {{\rm Sh}_{1,2}(Iw_\gothp)} & {\widetilde{\rm Sh}_{1,2}(K_\gothp^1)} \\
	& {{\rm Sh}_{1,2}(K_\gothp^1)} && {{\rm Sh}_{1,2}K_\gothp^1)} \\
	{{\rm Sh}_{1,2}} && {{\rm Sh}_{1,2}} && {{\rm Sh}_{1,2}}
	\arrow["{{{{{{\mathop{\widetilde{p}}\limits^{\leftarrow}}_{Iw}}}}}}"', from=1-3, to=2-2]
	\arrow["{\pi_{Iw_\gothp}}", from=1-3, to=2-3]
	\arrow["{{{{{{\mathop{\widetilde{p}}\limits^{\rightarrow}}_{Iw}}}}}}", from=1-3, to=2-4]
	\arrow["{\pi_{K_\gothp^1}}", from=2-2, to=3-2]
	\arrow["{{{{\mathop{p}\limits^{\leftarrow}}}}}"', from=2-2, to=4-1]
	\arrow["{{{{{\mathop{p}\limits^{\leftarrow}}_{Iw}}}}}"', from=2-3, to=3-2]
	\arrow["{{{{{\mathop{p}\limits^{\rightarrow}}_{Iw}}}}}", from=2-3, to=3-4]
	\arrow["{\pi_{K_\gothp^1}}", from=2-4, to=3-4]
	\arrow["{{{{\mathop{p}\limits^{\rightarrow}}}}}", from=2-4, to=4-5]
	\arrow["{{{{\mathop{p}\limits^{\leftarrow}}}}}", from=3-2, to=4-1]
	\arrow["{{{{\mathop{p}\limits^{\rightarrow}}}}}", from=3-2, to=4-3]
	\arrow["{{{{\mathop{p}\limits^{\leftarrow}}}}}"', from=3-4, to=4-3]
	\arrow["{{{{\mathop{p}\limits^{\rightarrow}}}}}"', from=3-4, to=4-5]
\end{tikzcd}\]

Since blow up does not change the generic fiber, the map in the Indefinite Ihara lemma in Theorem~\ref{IIhara lemma} can be translated as the following three maps:
\begin{enumerate}
    \item $f_1: {\rm H}^4_{\acute{e}t}(\overline{\widetilde{Sh}}_{1,2}(K_\gothp^1),k_{\lambda}(3))_{\mathfrak{m}}\xra{\lp}{\rm H}^4_{\acute{e}t}(\overline{Sh}_{1,2},k_{\lambda}(3))_{\mathfrak{m}};$
    \item $f_2:{\rm H}^4_{\acute{e}t}(\overline{\widetilde{Sh}}_{1,2}(K_\gothp^1),k_{\lambda}(3))_{\mathfrak{m}}\xra{\rp}{\rm H}^4_{\acute{e}t}(\overline{Sh}_{1,2},k_{\lambda}(3))_{\mathfrak{m}};$
    \item $f_3:{\rm H}^4_{\acute{e}t}(\overline{\widetilde{Sh}}_{1,2}(K_\gothp^1),k_{\lambda}(3))_{\mathfrak{m}}\xra{{{\mathop{\widetilde{p}}\limits^{\leftarrow}}^{*}_{Iw}}}{\rm H}^4_{\acute{e}t}(\overline{\widetilde{Sh}}_{1,2}(Iw_\gothp),k_{\lambda}(3))_{\mathfrak{m}}\xra{{{\mathop{\widetilde{p}}\limits^{\rightarrow}}_{*,Iw}}}{\rm H}^4_{\acute{e}t}(\overline{\widetilde{Sh}}_{1,2}(K_\gothp^1),k_{\lambda}(3))_{\mathfrak{m}}\xra{\rp} {\rm H}^4_{\acute{e}t}(\overline{Sh}_{1,2},k_{\lambda}(3))_{\mathfrak{m}}.$
\end{enumerate}
By Theorem~\ref{IIhara lemma}, $(f_1,f_2,f_3)$ is surjective.

By \cite[Proposition 2.11, 2.13]{Sai03}, pushforward and pullback of a morphism preserves the monodromy filtraion. Since ${\Sh}_{1,2}$ is smooth and irreducible, the monodromy filtration concentrates. We get three maps
\begin{enumerate}
    \item $Gr_0f_{1}:Gr_0 {\rm H}^{4}_{\acute{e}t}(\overline{\widetilde{Sh}}_{1,2}(Iw_\gothp),k_\lambda(3))\xra{Gr_0\lp}Gr_0 {\rm H}^{4}_{\acute{e}t}(\overline{{Sh}}_{1,2},k_\lambda(3))_{\mathfrak{m}};$
    \item $Gr_0f_{2}:Gr_0 {\rm H}^{4}_{\acute{e}t}(\overline{\widetilde{Sh}}_{1,2}(Iw_\gothp),k_\lambda(3))\xra{Gr_0\rp}Gr_0 {\rm H}^{4}_{\acute{e}t}(\overline{{Sh}}_{1,2},k_\lambda(3))_{\mathfrak{m}};$
     \item $Gr_0f_3:Gr_0{\rm H}^4_{\acute{e}t}(\overline{\widetilde{Sh}}_{1,2}(K_\gothp^1),k_{\lambda}(3))_{\mathfrak{m}}\xra{Gr_0{{\mathop{\widetilde{p}}\limits^{\leftarrow}}^{*}_{Iw}}}Gr_0{\rm H}^4_{\acute{e}t}(\overline{\widetilde{Sh}}_{1,2}(Iw_\gothp),k_{\lambda}(3))_{\mathfrak{m}}\xra{Gr_0{{\mathop{\widetilde{p}}\limits^{\rightarrow}}_{*,Iw}}}$\\$Gr_0{\rm H}^4_{\acute{e}t}(\overline{\widetilde{Sh}}_{1,2}(K_\gothp^1),k_{\lambda}(3))_{\mathfrak{m}}\xra{Gr_0\rp} Gr_0{\rm H}^4_{\acute{e}t}(\overline{Sh}_{1,2},k_{\lambda}(3))_{\mathfrak{m}}.$
\end{enumerate}
\begin{proposition}\label{Gr_0:non-Frob-p^2}
    We have $(E_{1,\mathfrak{m}}^{0,4}(3))^{{\rm Fr}_{p^2}\neq p^{-2}}=({\rm H}^4_{\acute{e}t}(\overline\Sh_{1,2},k_\lambda(3))_{\mathfrak{m}}^{^{{\rm Fr}_{p^2}\neq p^{-2}}})^{\oplus 3}$ with each of $\overline{\rm Y}_{01},\overline{\rm Y}_{10},\overline{\rm Y}_{11}$ contributing to one of the factor and 
    $({E'}_{1,\mathfrak{m}}^{0,4}(3))^{{\rm Fr}_{p^2}\neq p^{-2}}=({\rm H}^4_{\acute{e}t}(\overline\Sh_{1,2},k_\lambda(3))_{\mathfrak{m}}^{^{{\rm Fr}_{p^2}\neq p^{-2}}})^{\oplus 6},$ with each of $\{\overline{\rm Z}_{ij};0\leq i,j\leq 2,i\neq j\}$ contributing to one of the factor. 
\end{proposition}
\begin{proof}
    First, we calculate $E_{1,\mathfrak{m}}^{0,4}(3).$ By definition, $E_{1,\mathfrak{m}}^{0,4}(3)={\rm H}^4_{\acute{e}t}(\overline{Y}^{(0)},k_\lambda(3))_{\mathfrak{m}}\oplus {\rm H}^{2}_{\acute{e}t}(\overline{Y}^{(2)},k_{\lambda}(2))_{\mathfrak{m}}.$ 
    By Corollary~\ref{Corresss}, Proposition~\ref{bloww up}, Proposition~\ref{G:blow up} and Corollary~\ref{GI:characterize of Y_11}, 
    \begin{align*}
        {\rm H}^4_{\acute{e}t}(\overline{Y}^{(0)},k_\lambda(3))_{\mathfrak{m}}=&{\rm H}^{4}_{\acute{e}t}(\overline{\widetilde{\rm Y}}_{11})_{\mathfrak{m}}\oplus {\rm H}^{4}_{\acute{e}t}(\overline{\widetilde{\rm Y}}_{10})_{\mathfrak{m}}\oplus {\rm H}^{4}_{\acute{e}t}(\overline{\widetilde{\rm Y}}_{01})_{\mathfrak{m}}\oplus {\rm H}^{4}_{\acute{e}t}(\overline{\widetilde{\rm Y}}_{00})_{\mathfrak{m}}\\=&{\rm H}^{0}(\overline\Sh_{0,3},k_{\lambda}(1))_{\mathfrak{m}}^{\oplus 7}\oplus {\rm H}^{4}_{\acute{e}t}(\overline\Sh_{1,2},k_{\lambda}(3))^{\oplus 3}_{\mathfrak{m}}.
    \end{align*}
    By Proposition~\ref{cohomologyy} and Proposition~\ref{G:blow up}, $${\rm H}^{2}_{\acute{e}t}(\overline{Y}^{(2)},k_{\lambda}(2))_{\mathfrak{m}}={\rm H}^{0}(\overline\Sh_{0,3},k_{\lambda}(1))_{\mathfrak{m}}^{\oplus 2}\oplus {\rm H}^{0}(\overline\Sh_{1,2}(K_\gothp^1),k_\lambda(1))_{\mathfrak{m}}^{\oplus 2}.$$
    Thus by Proposition~\ref{S:mod l cohomology}, $(E_{1,\mathfrak{m}}^{0,4}(3))^{{\rm Fr}_{p^2}\neq p^{-2}}=({\rm H}^4_{\acute{e}t}(\overline\Sh_{1,2},k_\lambda(3))_{\mathfrak{m}}^{^{{\rm Fr}_{p^2}\neq p^{-2}}})^{\oplus 3}.$

    Next, we calculate ${E'}_{1,\mathfrak{m}}^{0,4}(3).$ Since the morphism ${\mathop{p}\limits^{\rightarrow}}\circ{{\mathop{p}\limits^{\rightarrow}}_{Iw}}$ is ${\rm Fr}_{p^2}$-equivariant and the images of $\overline{\widetilde{\rm Z}}_{00},\overline{\widetilde{\rm Z}}_{11},\overline{\widetilde{\rm Z}}_{22}$  are contained in $\overline{\rm Y}_1,\overline{\rm Y}_{2},\overline{\rm Y}_3$ respectively, it is easy to see $${\rm H}^4_{\acute{e}t}(\overline{\widetilde{\rm Z}}_{00},k_\lambda(3))_{\mathfrak{m}},{\rm H}^4_{\acute{e}t}(\overline{\widetilde{\rm Z}}_{11},k_\lambda(3))_{\mathfrak{m}},{\rm H}^4_{\acute{e}t}(\overline{\widetilde{\rm Z}}_{22},k_\lambda(3))_{\mathfrak{m}}\subseteq ({E'}_{1,\mathfrak{m}}^{0,4}(3))^{{\rm Fr}_{p^2}=p^2}$$ by Proposition~\ref{A:Supsingular vanishing}. By Proposition~\ref{GI:blow up} and a similar reason as above, we can see 
    $$
    {\rm H}^4_{\acute{e}t}(\overline{Y}^{(0)},k_\lambda(3))_{\mathfrak{m}}^{{\rm Fr}_{p^2}\neq p^2}= \bigoplus\limits_{0\leq i,j\leq 2,i\neq j}{\rm H}^{4}_{\acute{e}t}(\overline{\widetilde{\rm Z}}_{ij},k_{\lambda}(3))_{\mathfrak{m}}^{{\rm Fr}_{p^2}\neq p^2}=({\rm H}^4_{\acute{e}t}(\overline\Sh_{1,2},k_\lambda(3))_{\mathfrak{m}}^{^{{\rm Fr}_{p^2}\neq p^{-2}}})^{\oplus 6}.
    $$ 
\end{proof}

\begin{proposition}\label{Cal of Gr_2}
 With appropriate basis on ${\rm H}^{4}_{\acute{e}t}(\overline{{Sh}}_{1,2},k_\lambda(3))^{\oplus 3}_{\mathfrak{m}},$ the image of the maps $$(Gr_0f_1,Gr_0f_2,Gr_0f_3):Gr_0 {\rm H}^{4}_{\acute{e}t}(\overline{\widetilde{Sh}}_{1,2}(Iw_\gothp),k_\lambda(3))\ra Gr_0 {\rm H}^{4}_{\acute{e}t}(\overline{{Sh}}_{1,2},k_\lambda(3))^{\oplus 3}_{\mathfrak{m}}={\rm H}^{4}_{\acute{e}t}(\overline{{Sh}}_{1,2},k_\lambda(3))^{\oplus 3}_{\mathfrak{m}}$$ is contained in ${\rm H}^{4}_{\acute{e}t}(\overline{{Sh}}_{1,2},k_\lambda(3))^{\oplus 2}_{\mathfrak{m}}\oplus (1-{\rm Fr}_{p^2}){\rm H}^{4}_{\acute{e}t}(\overline{{Sh}}_{1,2},k_\lambda(3))^{\oplus 3}_{\mathfrak{m}}.$
\end{proposition}
\begin{proof}
    For $1\leq i\leq 3,$ we use $f_{i,1},f_{i,2},f_{i,3}$ to denote the map from ${\rm H}^4_{\acute{e}t}(\overline{\Sh}_{1,2},k_\lambda(3))_{\mathfrak{m}}$ factor of $E_{1,\mathfrak{m}}^{0,4}$ to ${\rm H}^4_{\acute{e}t}(\overline{\Sh}_{1,2},k_\lambda(3))_{\mathfrak{m}}$ induced by $f_i$ and contributed by ${\rm Y}_{10},{\rm Y}_{01},{\rm Y}_{11}$ respectively. 

    Then it suffices to show that the mapping matrix $(f_{j,i})_{1\leq i,j\leq 3}$\footnote{Here rows are indexed by $i$ and columns are indexed by $j$} has image contained in $${\rm H}^{4}_{\acute{e}t}(\overline{{Sh}}_{1,2},k_\lambda(3))^{\oplus 2}_{\mathfrak{m}}\oplus (1-{\rm Fr}_{p^2}){\rm H}^{4}_{\acute{e}t}(\overline{{Sh}}_{1,2},k_\lambda(3))_{\mathfrak{m}}$$ with appropriate basis on ${\rm H}^{4}_{\acute{e}t}(\overline{{Sh}}_{1,2},k_\lambda(3))^{\oplus 3}_{\mathfrak{m}}.$ 
    Express the $\{f_{i,j}\}_{1\leq i,j\leq 3}$ explicitly, we have
    \begin{itemize}
        \item $f_{1,1}={\mathop{p}\limits^{\leftarrow}}_{10,*}\circ {\mathop{p}\limits^{\leftarrow}}_{10}^{*};$
        \item $f_{1,2}={\mathop{p}\limits^{\rightarrow}}_{10,*}={\mathop{p}\limits^{\rightarrow}}_{01,*}\circ (Fr')_{*}\circ {\mathop{p}\limits^{\leftarrow}}_{10}^{*};$
        \item $f_{1,3}=({\mathop{p}\limits^{\rightarrow}}_{01,*}\circ{{\mathop{p}\limits^{\rightarrow}}_{Iw,01,*}}\circ{{\mathop{p}\limits^{\leftarrow}}_{Iw,01}^{*}}+{{\mathop{p}\limits^{\rightarrow}}_{11,*}}\circ {{\mathop{p}\limits^{\rightarrow}}_{Iw,21,*}}\circ{{\mathop{p}\limits^{\leftarrow}}_{Iw,21}^{*}})\circ {\mathop{p}\limits^{\leftarrow}}_{10}^{*};$
        \item $f_{2,1}={\mathop{p}\limits^{\leftarrow}}_{01,*}={\mathop{p}\limits^{\leftarrow}}_{10,*}\circ (Fr'')_{*}\circ {\mathop{p}\limits^{\rightarrow}}_{01}^{*};$
        \item $f_{2,2}={\mathop{p}\limits^{\rightarrow}}_{01,*}\circ {\mathop{p}\limits^{\rightarrow}}_{01}^{*};$
        \item $f_{2,3}=({\mathop{p}\limits^{\rightarrow}}_{10,*}\circ{{\mathop{p}\limits^{\rightarrow}}_{Iw,10,*}}\circ{{\mathop{p}\limits^{\leftarrow}}_{Iw,10}^{*}}+{{\mathop{p}\limits^{\rightarrow}}_{11,*}}\circ {{\mathop{p}\limits^{\rightarrow}}_{Iw,12,*}}\circ{{\mathop{p}\limits^{\leftarrow}}_{Iw,12}^{*}})\circ {\mathop{p}\limits^{\rightarrow}}_{01}^{*};
        $
        \item $f_{3,1}={{\mathop{p}\limits^{\leftarrow}}_{11,*}}\circ i_{11};$
        \item $f_{3,2}={{\mathop{p}\limits^{\rightarrow}}_{11,*}}\circ i_{11};$
        \item $f_{3,3}=({\mathop{p}\limits^{\rightarrow}}_{10,*}\circ{{\mathop{p}\limits^{\rightarrow}}_{Iw,20,*}}\circ{{\mathop{p}\limits^{\leftarrow}}_{Iw,20}^{*}}+{{\mathop{p}\limits^{\rightarrow}}_{01,*}}\circ {{\mathop{p}\limits^{\rightarrow}}_{Iw,02,*}}\circ{{\mathop{p}\limits^{\leftarrow}}_{Iw,02}^{*}})\circ i_{11}.$
    \end{itemize}
    Here we denote by $i_{11}$ the inclusion of ${\rm H}^{4}_{\acute{e}t}(\overline{{Sh}}_{1,2},k_\lambda(3))^{\oplus 3}_{\mathfrak{m}}$ to ${\rm H}^{4}_{\acute{e}t}({\rm Y}_{11},k_\lambda(3))_{\mathfrak{m}}.$

    By Proposition~\ref{GI:Diagram commuting}, ${{\mathop{p}\limits^{\leftarrow}}_{Iw,01}}\circ Fr_{10\ra01}=Fr''\circ{{\mathop{p}\limits^{\leftarrow}}_{Iw,10}}$ and ${\mathop{p}\limits^{\rightarrow}}_{Iw,01}\circ Fr_{10\ra01}=Fr'\circ{\mathop{p}\limits^{\rightarrow}}_{Iw,10}.$ Then 
    \begin{align*}(Fr'\circ{\mathop{p}\limits^{\rightarrow}}_{Iw,10})_{*}\circ (Fr''\circ{{\mathop{p}\limits^{\leftarrow}}_{Iw,10}})^{*}=&({\mathop{p}\limits^{\rightarrow}}_{Iw,01}\circ Fr_{10\ra01})_{*}\circ({{\mathop{p}\limits^{\leftarrow}}_{Iw,01}}\circ Fr_{10\ra01})^{*}\\=& {\mathop{p}\limits^{\rightarrow}}_{Iw,01,*}\circ Fr_{10\ra01,*}\circ Fr_{10\ra01}^{*}\circ {{\mathop{p}\limits^{\leftarrow}}_{Iw,01}^{*}}\\=&p^2({\mathop{p}\limits^{\rightarrow}}_{Iw,01,*}\circ {{\mathop{p}\limits^{\leftarrow}}_{Iw,01}^{*}}),
    \end{align*} where the second equation comes from Proposition~\ref{GI:blow up}.
    Hence 
    \begin{equation}\label{f_11_1}
    {\mathop{p}\limits^{\rightarrow}}_{Iw,01,*}\circ {{\mathop{p}\limits^{\leftarrow}}_{Iw,01}^{*}}=\frac{1}{p^2}(Fr'\circ{\mathop{p}\limits^{\rightarrow}}_{Iw,10})_{*}\circ (Fr''\circ{{\mathop{p}\limits^{\leftarrow}}_{Iw,10}})^{*}.
    \end{equation}

    By Proposition~\ref{GI:Diagram commuting}, ${\mathop{p}\limits^{\rightarrow}}_{Iw,21}={\mathop{p}\limits^{\rightarrow}}_{Iw,12}\circ Fr_{21\ra12}.$ Hence
    \begin{equation}\label{f_11_2}{{\mathop{p}\limits^{\rightarrow}}_{Iw,21,*}}\circ{{\mathop{p}\limits^{\leftarrow}}_{Iw,21}^{*}}=({\mathop{p}\limits^{\rightarrow}}_{Iw,12}\circ Fr_{21\ra12})_{*}\circ {{\mathop{p}\limits^{\leftarrow}}_{Iw,21}^{*}}.
    \end{equation}

    By Proposition~\ref{GI:Diagram commuting}, ${\mathop{p}\limits^{\rightarrow}}_{Iw,21}={\mathop{p}\limits^{\rightarrow}}_{Iw,12}\circ Fr_{21\ra12}$ and $Fr'\circ {\mathop{p}\limits^{\leftarrow}}_{Iw,21}={\mathop{p}\limits^{\leftarrow}}_{Iw,12}\circ Fr_{21\ra12}.$ Then 
    \begin{align*}({\mathop{p}\limits^{\rightarrow}}_{Iw,12}\circ Fr_{21\ra12})_{*}\circ {{\mathop{p}\limits^{\leftarrow}}_{Iw,21}^{*}\circ Fr'^{*}}=&({\mathop{p}\limits^{\rightarrow}}_{Iw,12}\circ Fr_{21\ra12})_{*}\circ({{\mathop{p}\limits^{\leftarrow}}_{Iw,12}}\circ Fr_{21\ra12})^{*}\\=& {\mathop{p}\limits^{\rightarrow}}_{Iw,12,*}\circ Fr_{21\ra12,*}\circ Fr_{21\ra12}^{*}\circ {{\mathop{p}\limits^{\leftarrow}}_{Iw,12}^{*}}\\=&p^2({\mathop{p}\limits^{\rightarrow}}_{Iw,12,*}\circ {{\mathop{p}\limits^{\leftarrow}}_{Iw,12}^{*}}),
    \end{align*} where the second equation comes from Proposition~\ref{GI:blow up}.
    Hence
    \begin{equation}\label{f_23}{\mathop{p}\limits^{\rightarrow}}_{Iw,12,*}\circ {{\mathop{p}\limits^{\leftarrow}}_{Iw,12}^{*}}=\frac{1}{p^2}{\mathop{p}\limits^{\rightarrow}}_{Iw,12,*}\circ Fr_{21\ra12,*}\circ {\mathop{p}\limits^{\leftarrow}}_{Iw,21}^{*}\circ {Fr'}^{*}.
    \end{equation}

    By Proposition~\ref{GI:Diagram commuting},
    ${\mathop{p}\limits^{\leftarrow}}_{Iw,02}={\mathop{p}\limits^{\leftarrow}}_{Iw,20}\circ Fr_{02\ra20}.$ Hence
    \begin{equation}\label{f_33_1}
        {\mathop{p}\limits^{\rightarrow}}_{Iw,02,*}\circ{\mathop{p}\limits^{\leftarrow}}_{Iw,02}^{*}={\mathop{p}\limits^{\rightarrow}}_{Iw,02,*}\circ({\mathop{p}\limits^{\leftarrow}}_{Iw,20}\circ Fr_{02\ra20})^{*}.
    \end{equation}

    By proposition~\ref{GI:Diagram commuting},
    ${\mathop{p}\limits^{\leftarrow}}_{Iw,02}={\mathop{p}\limits^{\leftarrow}}_{Iw,20}\circ Fr_{02\ra20}$ and ${\mathop{p}\limits^{\rightarrow}}_{Iw,20}\circ Fr_{02\ra20}=Fr''\circ {\mathop{p}\limits^{\rightarrow}}_{Iw,02}.$ Then \begin{align*}
    (Fr''\circ {\mathop{p}\limits^{\rightarrow}}_{Iw,02})_{*}\circ({\mathop{p}\limits^{\leftarrow}}_{Iw,20}\circ Fr_{02\ra20})^{*}=&({\mathop{p}\limits^{\rightarrow}}_{Iw,20}\circ Fr_{02\ra20})_{*}\circ({\mathop{p}\limits^{\leftarrow}}_{Iw,20}\circ Fr_{02\ra20})^{*}\\=& {\mathop{p}\limits^{\rightarrow}}_{Iw,20,*}\circ Fr_{02\ra20,*}\circ Fr_{02\ra20}^{*}\circ {{\mathop{p}\limits^{\leftarrow}}_{Iw,20}^{*}}\\=&p^2({\mathop{p}\limits^{\rightarrow}}_{Iw,20,*}\circ {{\mathop{p}\limits^{\leftarrow}}_{Iw,20}^{*}}),
    \end{align*} where the second equation comes from Proposition~\ref{GI:blow up}.
    Hence
    \begin{equation}\label{f_33_2}{\mathop{p}\limits^{\rightarrow}}_{Iw,20,*}\circ {{\mathop{p}\limits^{\leftarrow}}_{Iw,20}^{*}}=\frac{1}{p^2}(Fr''\circ {\mathop{p}\limits^{\rightarrow}}_{Iw,02})_{*}\circ({\mathop{p}\limits^{\leftarrow}}_{Iw,20}\circ Fr_{02\ra20})^{*}.
    \end{equation}

    By proposition~\ref{GI:Diagram commuting}, ${\mathop{p}\limits^{\rightarrow}}_{11}\circ{\mathop{p}\limits^{\leftarrow}}_{Iw,20}={\mathop{p}\limits^{\leftarrow}}_{10}\circ{\mathop{p}\limits^{\rightarrow}}_{Iw,20}$ and ${\mathop{p}\limits^{\leftarrow}}_{11}\circ{\mathop{p}\limits^{\rightarrow}}_{Iw,21}={\mathop{p}\limits^{\rightarrow}}_{10}\circ{\mathop{p}\limits^{\leftarrow}}_{Iw,21}.$ Since ${\mathop{p}\limits^{\leftarrow}}_{Iw,20}$ and ${\mathop{p}\limits^{\rightarrow}}_{Iw,12}$ are invertible by Proposition~\ref{GI:blow up},
    \begin{equation}\label{ra p_11}{\mathop{p}\limits^{\rightarrow}}_{11}={\mathop{p}\limits^{\leftarrow}}_{10}\circ{\mathop{p}\limits^{\rightarrow}}_{Iw,20}\circ {\mathop{p}\limits^{\leftarrow}}_{Iw,20}^{-1}
    \end{equation} and 
    \begin{equation}\label{la p_11}{\mathop{p}\limits^{\leftarrow}}_{11}\circ Fr_{p^2}={\mathop{p}\limits^{\rightarrow}}_{10}\circ{\mathop{p}\limits^{\leftarrow}}_{Iw,21}\circ Fr_{12\ra21}\circ{\mathop{p}\limits^{\rightarrow}}_{Iw,12}^{-1}.
    \end{equation}

    Multiplying the matrix
    $\begin{pmatrix}1&-{\mathop{p}\limits^{\leftarrow}}_{10,*}\circ Fr''_{*}\circ{\mathop{p}\limits^{\rightarrow}}_{01}^{*}&0\\
        0&1&0\\
        0&0&1\\
    \end{pmatrix}$ on the left, we get $(f_{j,i,1})_{1\leq i,j\leq 3}$ satisfying only the elements of the second column changed with $f_{2,1,1}=0,f_{2,2,1}=1-\Frob_{p^2,*}$ and 
    \begin{equation}\label{f_2,3,1}
        \begin{aligned}
        f_{2,3,1}=&f_{2,3}+\frac{1}{p^2}{\mathop{p}\limits^{\rightarrow}}_{01,*}\circ(Fr'\circ{\mathop{p}\limits^{\rightarrow}}_{Iw,10})_{*}\circ (Fr''\circ{{\mathop{p}\limits^{\leftarrow}}_{Iw,10}})^{*}\circ{\mathop{p}\limits^{\leftarrow}}_{10}^{*}\circ{\mathop{p}\limits^{\leftarrow}}_{10,*}\circ  (-Fr''_{*}\circ{\mathop{p}\limits^{\rightarrow}}_{01}^{*})+\\&{{\mathop{p}\limits^{\rightarrow}}_{11,*}}\circ ({\mathop{p}\limits^{\rightarrow}}_{Iw,12}\circ Fr_{21\ra12})_{*}\circ {{\mathop{p}\limits^{\leftarrow}}_{Iw,21}^{*}}\circ {\mathop{p}\limits^{\leftarrow}}_{10}^{*}\circ{\mathop{p}\limits^{\leftarrow}}_{10,*}\circ  (-Fr''_{*}\circ{\mathop{p}\limits^{\rightarrow}}_{01}^{*})
    \end{aligned}
    \end{equation}
    
    by \eqref{f_11_1} and \eqref{f_11_2}.

    Note that for $x\in {\rm H}^{4}_{\acute{e}t}(\overline{{Sh}}_{1,2},k_\lambda(3))_{\mathfrak{m}},$ $({\mathop{p}\limits^{\rightarrow}}_{01,*}\circ Fr''^{*}\circ{\mathop{p}\limits^{\leftarrow}}_{10}^{*}\circ{\mathop{p}\limits^{\leftarrow}}_{10,*}\circ (Fr''_{*}\circ{\mathop{p}\limits^{\rightarrow}}_{01}^{*}))x=0$ if $(Fr''_{*}\circ{\mathop{p}\limits^{\rightarrow}}_{01}^{*})x$ is contained in the ${\rm H}^{0}_{\acute{e}t}(\overline{\Sh}_{0,3},k_\lambda(1))_{\mathfrak{m}}$ factor of ${\rm H}^{4}_{\acute{e}t}(\overline{\rm Y}_{10},k_\lambda(3))_{\mathfrak{m}}.$ Then since ${\mathop{p}\limits^{\rightarrow}}_{01,*}\circ Fr''^{*}\circ (Fr''_{*}\circ{\mathop{p}\limits^{\rightarrow}}_{01}^{*})=p^4,$ we have $x\in {\rm H}^{4}_{\acute{e}t}(\overline{{Sh}}_{1,2},k_\lambda(3))^{{\rm Fr}_{p^2}= p^{-2}}_{\mathfrak{m}}.$ Otherwise if the image of $x$ is not zero, $({\mathop{p}\limits^{\rightarrow}}_{01,*}\circ Fr''^{*}\circ{\mathop{p}\limits^{\leftarrow}}_{10}^{*}\circ{\mathop{p}\limits^{\leftarrow}}_{10,*}\circ (Fr''_{*}\circ{\mathop{p}\limits^{\rightarrow}}_{01}^{*}))x={\mathop{p}\limits^{\rightarrow}}_{01,*}\circ Fr''^{*}\circ (Fr''_{*}\circ{\mathop{p}\limits^{\rightarrow}}_{01}^{*})x=p^4(x)$ if we consider in the image modulo $(1-{\rm Fr}_{p^2})$ Thus
    it suffices for us to prove the proposition by enlarging the map to be $p^4.$ Thus the second factor in \eqref{f_2,3,1} can be enlarged to be $p^2{\mathop{p}\limits^{\rightarrow}}_{01,*}\circ(Fr'\circ{\mathop{p}\limits^{\rightarrow}}_{Iw,10})_{*}\circ{\mathop{p}\limits^{\leftarrow}}_{Iw,10}^*
    \circ  {\mathop{p}\limits^{\rightarrow}}_{01}^{*}=p^2{\mathop{p}\limits^{\rightarrow}}_{01,*}\circ Fr'_*\circ{{\mathop{p}\limits^{\leftarrow}}_{10}^*}$ with the last equation coming from ${\mathop{p}\limits^{\rightarrow}}_{Iw,10,*}\circ{\mathop{p}\limits^{\leftarrow}}_{Iw,10}^*=1.$ 

    By Proposition~\ref{GI:Diagram commuting} and Construction~\ref{12-20}, $Fr_{21\ra12}=\Phi^{-1}_{12\ra20}\circ Fr_{02\ra20}\circ\Phi_{21\ra02}$ and ${{\mathop{p}\limits^{\rightarrow}}_{Iw,20}}\circ ({{\mathop{p}\limits^{\leftarrow}}_{Iw,20}})^{-1}\circ {{\mathop{p}\limits^{\rightarrow}}_{Iw,12}}\circ \Phi_{12\ra20}^{-1}\circ Fr_{02\ra20}=\Phi_{10\ra10}\circ Fr''\circ {{\mathop{p}\limits^{\rightarrow}}_{Iw,02}}.$ Thus by \eqref{ra p_11}, the third factor of \eqref{f_2,3,1} is equal to ${\mathop{p}\limits^{\leftarrow}}_{10,*}\circ(\Phi_{10\ra10}\circ Fr''\circ {{\mathop{p}\limits^{\rightarrow}}_{Iw,02}})_{*}\circ \Phi_{21\ra02,*}\circ {{\mathop{p}\limits^{\leftarrow}}_{Iw,21}^{*}}\circ {\mathop{p}\limits^{\leftarrow}}_{10}^{*}\circ{\mathop{p}\limits^{\leftarrow}}_{10,*}\circ(-Fr''_{*}\circ{\mathop{p}\limits^{\rightarrow}}_{01}^{*})$
    By the same reason as above, it can be enlarged to be ${\mathop{p}\limits^{\leftarrow}}_{10,*}\circ(\Phi_{10\ra10}\circ Fr''\circ {{\mathop{p}\limits^{\rightarrow}}_{Iw,02}})_{*}\circ \Phi_{21\ra02,*}\circ {{\mathop{p}\limits^{\leftarrow}}_{Iw,21}^{*}}\circ(-Fr''_{*}\circ{\mathop{p}\limits^{\rightarrow}}_{01}^{*}).$ 

    Since ${{\mathop{p}\limits^{\leftarrow}}_{10}}\circ {{\mathop{p}\limits^{\rightarrow}}_{Iw,10}}={{\mathop{p}\limits^{\rightarrow}}_{01}}\circ {{\mathop{p}\limits^{\leftarrow}}_{Iw,10}},$ ${{\mathop{p}\limits^{\rightarrow}}_{Iw,10,*}}\circ {{\mathop{p}\limits^{\leftarrow}}_{Iw,10}^*}\circ {{\mathop{p}\limits^{\rightarrow}}_{01}^*}={{\mathop{p}\limits^{\rightarrow}}_{Iw,10,*}}\circ({{\mathop{p}\limits^{\leftarrow}}_{10}}\circ {{\mathop{p}\limits^{\rightarrow}}_{Iw,10}})^{*}={{\mathop{p}\limits^{\leftarrow}}_{10}^*}.$ Hence 
    \begin{equation}\label{f_2,3}
        f_{2,3}=
    {{\mathop{p}\limits^{\rightarrow}}_{01,*}}\circ Fr'_{*}\circ {{\mathop{p}\limits^{\leftarrow}}_{10}}^{*}+\frac{1}{p^2} (\Phi_{10\ra10}\circ Fr''\circ {{\mathop{p}\limits^{\rightarrow}}_{Iw,02}})_{*}\circ \Phi_{21\ra02,*}\circ {{\mathop{p}\limits^{\leftarrow}}_{Iw,21}^{*}}\circ {Fr'}^{*}\circ {{\mathop{p}\limits^{\rightarrow}}_{01}}^{*}
    \end{equation}

     Multiplying the matrix
    $\begin{pmatrix}1&0&-{{\mathop{p}\limits^{\leftarrow}}_{11,*}}\circ i_{11}\\
        0&1&0\\
        0&0&1\\
    \end{pmatrix}$ on the right, we get $(f_{j,i,2})_{1\leq i,j\leq 3}$ satisfying only the elements of the third column changed with $f_{3,1,2}=0;$
    \begin{equation}\label{f_{3,2,2,}}
        f_{3,2,2}={{\mathop{p}\limits^{\rightarrow}}_{11,*}}\circ i_{11}-{\mathop{p}\limits^{\rightarrow}}_{01,*}\circ (Fr')_{*}\circ {\mathop{p}\limits^{\leftarrow}}_{10}^{*}\circ{{\mathop{p}\limits^{\leftarrow}}_{11,*}}\circ i_{11};
    \end{equation}
    and 
    \begin{equation}
    f_{3,3,2}=f_{3,3}-f_{1,3}\circ{{\mathop{p}\limits^{\leftarrow}}_{11,*}}\circ i_{11}.
    \end{equation}

    By \eqref{f_33_2} and \eqref{f_33_1}, we have 
    \begin{equation}\label{f_33}
        \begin{aligned}
            f_{3,3}=&({\mathop{p}\limits^{\rightarrow}}_{10,*}\circ{{\mathop{p}\limits^{\rightarrow}}_{Iw,20,*}}\circ{{\mathop{p}\limits^{\leftarrow}}_{Iw,20}^{*}}+{{\mathop{p}\limits^{\rightarrow}}_{01,*}}\circ {{\mathop{p}\limits^{\rightarrow}}_{Iw,02,*}}\circ{{\mathop{p}\limits^{\leftarrow}}_{Iw,02}^{*}})\circ i_{11}\\=&\frac{1}{p^2}{\mathop{p}\limits^{\rightarrow}}_{01,*}\circ \Frob_{p^2,*}\circ {\mathop{p}\limits^{\rightarrow}}_{Iw,02,*}\circ Fr_{02\ra20}^{*}\circ{\mathop{p}\limits^{\leftarrow}}_{Iw,20}^*\circ i_{11}+{{\mathop{p}\limits^{\rightarrow}}_{01,*}}\circ{\mathop{p}\limits^{\rightarrow}}_{Iw,02,*}\circ({\mathop{p}\limits^{\leftarrow}}_{Iw,20}\circ Fr_{02\ra20})^{*}\circ i_{11}\\=&{\mathop{p}\limits^{\rightarrow}}_{01,*}\circ (\frac{1}{p^2}\Frob_{p^2,*}+1)\circ {\mathop{p}\limits^{\rightarrow}}_{Iw,02,*}\circ Fr_{02\ra20}^{*}\circ{\mathop{p}\limits^{\leftarrow}}_{Iw,20}^*\circ i_{11}.
        \end{aligned}
    \end{equation}

    By \eqref{f_11_1}, we have
    \begin{equation}
    f_{1,3}\circ{{\mathop{p}\limits^{\leftarrow}}_{11,*}}\circ i_{11}={{\mathop{p}\limits^{\rightarrow}}_{11,*}}\circ {{\mathop{p}\limits^{\rightarrow}}_{Iw,21,*}}\circ{{\mathop{p}\limits^{\leftarrow}}_{Iw,21}^{*}}\circ {\mathop{p}\limits^{\leftarrow}}_{10}^{*}\circ {{\mathop{p}\limits^{\leftarrow}}_{11,*}}\circ i_{11}+\frac{1}{p^2}{\mathop{p}\limits^{\rightarrow}}_{01,*}\circ Fr'_{*}\circ {Fr''}^*\circ {\mathop{p}\limits^{\leftarrow}}_{10}^{*}\circ{{\mathop{p}\limits^{\leftarrow}}_{11,*}}\circ i_{11}.
    \end{equation}
    
    Multiplying the matrix
    $\begin{pmatrix}
        1&0&0\\
        -f_{1,2,2}&1&0\\
        -f_{1,3,2}&-{{\mathop{p}\limits^{\rightarrow}}_{01,*}}\circ Fr'_{*}\circ {{\mathop{p}\limits^{\leftarrow}}_{10}}^{*}&1\\
    \end{pmatrix}$ on the left, we get $(f_{j,i,3})_{1\leq i,j\leq 3}$ such that $f_{1,1,3}=1,f_{1,2,3}=f_{1,3,3}=f_{2,1,3}=f_{3,1,3}=0$ and $f_{2,2,3}=f_{2,2,2}$ and $f_{3,2,3}=f_{3,2,2}.$ Moreover,
    \begin{equation}\label{f_2,3,3}
        \begin{aligned}
            f_{2,3,3}=&(1-p^2){{\mathop{p}\limits^{\rightarrow}}_{01,*}}\circ Fr'_{*}\circ {\mathop{p}\limits^{\leftarrow}}_{10}^{*}-{{\mathop{p}\limits^{\rightarrow}}_{01,*}}\circ Fr'_{*}\circ {{\mathop{p}\limits^{\rightarrow}}_{01}}^{*}\circ (1-\Frob_{p^2,*})
            \\&+{\mathop{p}\limits^{\leftarrow}}_{10,*}\circ(\Phi_{10\ra10}\circ Fr''\circ {{\mathop{p}\limits^{\rightarrow}}_{Iw,02}})_{*}\circ \Phi_{21\ra02,*}\circ {{\mathop{p}\limits^{\leftarrow}}_{Iw,21}^{*}}\circ (\frac{1}{p^2} {Fr'}^{*}-Fr''_{*})\circ {{\mathop{p}\limits^{\rightarrow}}_{01}}^{*}.
        \end{aligned}
    \end{equation}
    Since ${Fr''}^*,{\Frob_{p^2}^*}$ are invertible on cohomology groups,
    \begin{align*}
         (1-p^2){{\mathop{p}\limits^{\rightarrow}}_{01,*}}\circ Fr'_{*}\circ {\mathop{p}\limits^{\leftarrow}}_{10}^{*}-{{\mathop{p}\limits^{\rightarrow}}_{01,*}}\circ Fr'_{*}\circ {{\mathop{p}\limits^{\rightarrow}}_{01}}^{*}\circ (1-\Frob_{p^2,*})=&{{\mathop{p}\limits^{\rightarrow}}_{01,*}}\circ Fr'_{*}\circ {\mathop{p}\limits^{\leftarrow}}_{10}^{*}(p^2-\Frob_{p^2,*})\\=&-p^8(1-p^{-6}\Frob_{p^2}^*)\circ{{\mathop{p}\limits^{\rightarrow}}_{01,*}}\circ Fr'_{*}\circ {\mathop{p}\limits^{\leftarrow}}_{10}^{*}\circ {\Frob_{p^2}^*}^{-1},
    \end{align*} 
    and \begin{align*}
        (\Phi_{10\ra10}\circ Fr''\circ {{\mathop{p}\limits^{\rightarrow}}_{Iw,02}})_{*}\circ \Phi_{21\ra02,*}\circ {{\mathop{p}\limits^{\leftarrow}}_{Iw,21}^{*}}\circ (\frac{1}{p^2} {Fr'}^{*}-Fr''_{*})\circ {{\mathop{p}\limits^{\rightarrow}}_{01}}^{*}\\=-p^4(1-p^{-6}\Frob_{p^2}^*)(\Phi_{10\ra10}\circ Fr''\circ {{\mathop{p}\limits^{\rightarrow}}_{Iw,02}})_{*}\circ \Phi_{21\ra02,*}\circ {{\mathop{p}\limits^{\leftarrow}}_{Iw,21}^{*}}\circ {{Fr''}^*}^{-1}\circ {{\mathop{p}\limits^{\rightarrow}}_{01}}^{*}.
    \end{align*}
    They all have image contained in $(1-p^{-6}\Frob_{p^2}^*){\rm H}^{4}_{\acute{e}t}(\overline{{Sh}}_{1,2},k_\lambda(3))_\mathfrak{m}=(1-{\rm Fr}_{p^2}){\rm H}^{4}_{\acute{e}t}(\overline{{Sh}}_{1,2},k_\lambda(3))_{\mathfrak{m}}.$ Thus $f_{2,3,3}$ has image 
    contained in $(1-\Frob_{p^2}^{*}){\rm H}^{4}_{\acute{e}t}(\overline{{Sh}}_{1,2},k_\lambda(3))_{\mathfrak{m}}.$ 
    
   Now we calculate $f_{3,3,3}.$ By \eqref{ra p_11} and \eqref{f_33_2},
   \begin{equation}
       \begin{aligned}{{\mathop{p}\limits^{\rightarrow}}_{01,*}}\circ Fr'_{*}\circ {{\mathop{p}\limits^{\leftarrow}}_{10}}^{*}\circ {{\mathop{p}\limits^{\rightarrow}}_{11,*}}\circ i_{11}
       =&{{\mathop{p}\limits^{\rightarrow}}_{01,*}}\circ Fr'_{*}\circ {{\mathop{p}\limits^{\leftarrow}}_{10}}^{*}\circ ({\mathop{p}\limits^{\leftarrow}}_{10}\circ{\mathop{p}\limits^{\rightarrow}}_{Iw,20}\circ {\mathop{p}\limits^{\leftarrow}}_{Iw,20}^{-1})_{*}\circ i_{11}
       \\=&{{\mathop{p}\limits^{\rightarrow}}_{01,*}}\circ Fr'_{*}\circ{\mathop{p}\limits^{\rightarrow}}_{Iw,20,*}\circ {\mathop{p}\limits^{\leftarrow}}_{Iw,20}^{*}\circ i_{11}\\=&\frac{1}{p^2}{{\mathop{p}\limits^{\rightarrow}}_{01,*}}\circ Fr'_{*}\circ{Fr''}_{*}\circ{\mathop{p}\limits^{\rightarrow}}_{Iw,02,*}\circ {\mathop{p}\limits^{\leftarrow}}_{Iw,02}^{*}\circ i_{11}.
       \end{aligned}
   \end{equation} 
    Then by \eqref{ra p_11},
    \begin{equation}
        \begin{aligned}
            {{\mathop{p}\limits^{\rightarrow}}_{01,*}} \circ{\mathop{p}\limits^{\rightarrow}}_{Iw,02,*}\circ {\mathop{p}\limits^{\leftarrow}}_{Iw,02}^{*}\circ i_{11}-\frac{1}{p^2}{\mathop{p}\limits^{\rightarrow}}_{01,*}\circ Fr'_{*}\circ {Fr''}^*\circ {\mathop{p}\limits^{\leftarrow}}_{10}^{*}\circ{{\mathop{p}\limits^{\leftarrow}}_{11,*}}\circ i_{11}-\\{{\mathop{p}\limits^{\rightarrow}}_{01,*}}\circ Fr'_{*}\circ {{\mathop{p}\limits^{\leftarrow}}_{10}}^{*}\circ ({{\mathop{p}\limits^{\rightarrow}}_{11,*}}\circ i_{11}-{\mathop{p}\limits^{\rightarrow}}_{01,*}\circ (Fr')_{*}\circ {\mathop{p}\limits^{\leftarrow}}_{10}^{*}\circ{{\mathop{p}\limits^{\leftarrow}}_{11,*}}\circ i_{11})\\=(1-\frac{1}{p^2}\Frob_{p^2,*})\circ {{\mathop{p}\limits^{\rightarrow}}_{01,*}}\circ{\mathop{p}\limits^{\rightarrow}}_{Iw,02,*}\circ {\mathop{p}\limits^{\leftarrow}}_{Iw,02}^{*}\circ i_{11}-{\mathop{p}\limits^{\rightarrow}}_{01,*}\circ Fr'_{*}\circ (\frac{1}{p^2}{Fr''}^*-Fr'_{*})\circ {\mathop{p}\limits^{\leftarrow}}_{10}^{*}\circ{{\mathop{p}\limits^{\leftarrow}}_{11,*}}\circ i_{11}
        \end{aligned}
    \end{equation}
    Hence as what we do in \eqref{f_2,3,3}, this map has image in $(1-\Frob_{p^2}^{*}){\rm H}^{4}_{\acute{e}t}(\overline{{Sh}}_{1,2},k_\lambda(3))_{\mathfrak{m}}.$

    Now it suffices to show that 
    ${\mathop{p}\limits^{\rightarrow}}_{01,*}\circ \frac{1}{p^2}\Frob_{p^2,*}\circ {\mathop{p}\limits^{\rightarrow}}_{Iw,02,*}\circ Fr_{02\ra20}^{*}\circ{\mathop{p}\limits^{\leftarrow}}_{Iw,20}^*\circ i_{11}-{{\mathop{p}\limits^{\rightarrow}}_{11,*}}\circ {{\mathop{p}\limits^{\rightarrow}}_{Iw,21,*}}\circ{{\mathop{p}\limits^{\leftarrow}}_{Iw,21}^{*}}\circ {\mathop{p}\limits^{\leftarrow}}_{10}^{*}\circ {{\mathop{p}\limits^{\leftarrow}}_{11,*}}\circ i_{11}$ has image in $(1-{\rm Fr}_{p^2}){\rm H}^{4}_{\acute{e}t}(\overline{{Sh}}_{1,2},k_\lambda(3))_{\mathfrak{m}}.$

   By \eqref{f_33_2} and \eqref{ra p_11}, ${{\mathop{p}\limits^{\rightarrow}}_{11,*}}\circ {{\mathop{p}\limits^{\rightarrow}}_{Iw,21,*}}\circ{{\mathop{p}\limits^{\leftarrow}}_{Iw,21}^{*}}\circ {\mathop{p}\limits^{\leftarrow}}_{10}^{*}\circ {{\mathop{p}\limits^{\leftarrow}}_{11,*}}\circ i_{11}=\frac{1}{p^2}{{\mathop{p}\limits^{\leftarrow}}_{10,*}}\circ Fr''_*\circ {{\mathop{p}\limits^{\rightarrow}}_{Iw,02,*}}\circ Fr_{02\ra20}^* \circ {{\mathop{p}\limits^{\leftarrow}}_{Iw,20}^{*}}\circ{{\mathop{p}\limits^{\rightarrow}}_{Iw,21,*}}\circ{{\mathop{p}\limits^{\leftarrow}}_{Iw,21}^{*}}\circ {\mathop{p}\limits^{\leftarrow}}_{10}^{*}\circ {{\mathop{p}\limits^{\leftarrow}}_{11,*}}\circ i_{11}.$
   
    By Proposition~\ref{GI:Diagram commuting} and \eqref{f_11_2},
    \begin{equation}
        \begin{aligned}
            Fr_{02\ra20}^* \circ {{\mathop{p}\limits^{\leftarrow}}_{Iw,20}^{*}}\circ{{\mathop{p}\limits^{\rightarrow}}_{Iw,21,*}}\circ{{\mathop{p}\limits^{\leftarrow}}_{Iw,21}^{*}}=&Fr_{02\ra20}^* \circ {{\mathop{p}\limits^{\leftarrow}}_{Iw,20}^{*}}\circ{\mathop{p}\limits^{\rightarrow}}_{Iw,12,*}\circ Fr_{21\ra12,*}\circ {{\mathop{p}\limits^{\leftarrow}}_{Iw,21}^{*}}\\=&
            Fr_{02\ra20}^* \circ {{\mathop{p}\limits^{\leftarrow}}_{Iw,20}^{*}}\circ{\mathop{p}\limits^{\leftarrow}}_{Iw,20,*}\circ{\Phi_{11\ra11,*}^{-1}}\circ\Phi_{12\ra20,*}\circ Fr_{21\ra12,*}\circ {{\mathop{p}\limits^{\leftarrow}}_{Iw,21}^{*}}\\=&
            Fr_{02\ra20}^* \circ Fr_{02\ra20,*}\circ{\Phi_{02\ra02,*}^{-1}}\circ  \Phi_{21\ra02,*}\circ{{\mathop{p}\limits^{\leftarrow}}_{Iw,21}^{*}}\\=&p^2{\Phi_{02\ra02,*}^{-1}}\circ \Phi_{21\ra02,*}\circ{{\mathop{p}\limits^{\leftarrow}}_{Iw,21}^{*}}.
        \end{aligned}
    \end{equation}

    Thus by \eqref{la p_11} and Proposition~\ref{GI:Diagram commuting} and an appropriate base change such that $\Phi_{02\ra02,*},\Phi_{11\ra11,*}$ are identity by Hypothesis~\ref{Main hypo}, 
    \begin{equation}
        \begin{aligned}
            {{\mathop{p}\limits^{\rightarrow}}_{11,*}}\circ {{\mathop{p}\limits^{\rightarrow}}_{Iw,21,*}}\circ{{\mathop{p}\limits^{\leftarrow}}_{Iw,21}^{*}}\circ {\mathop{p}\limits^{\leftarrow}}_{10}^{*}\circ {{\mathop{p}\limits^{\leftarrow}}_{11,*}}\circ i_{11}=&{{\mathop{p}\limits^{\leftarrow}}_{10,*}}\circ Fr''_*\circ {{\mathop{p}\limits^{\rightarrow}}_{Iw,02,*}}\circ{\Phi_{02\ra02,*}^{-1}}\circ p^2\circ  {{\mathop{p}\limits^{\rightarrow}}_{Iw,02}^{*}}\circ Fr'_{*} \circ{{\mathop{p}\limits^{\leftarrow}}_{Iw,21}^{*}}\circ Fr_{12\ra 21,*}\\&
            \circ {\mathop{p}\limits^{\rightarrow}}_{Iw,12,*}^{-1}\circ \Frob_{p^2,*}^{-1}
            \circ i_{11}\\=&{{\mathop{p}\limits^{\leftarrow}}_{10,*}}\circ p^2 \circ{{\mathop{p}\limits^{\leftarrow}}_{Iw,21}^{*}}\circ Fr_{12\ra 21,*}
            \circ {\mathop{p}\limits^{\rightarrow}}_{Iw,12,*}^{-1}\circ i_{11}\\=&{\mathop{p}\limits^{\rightarrow}}_{01,*}\circ \Frob_{p^2,*}\circ {\mathop{p}\limits^{\rightarrow}}_{Iw,02,*}\circ Fr_{02\ra20}^{*}\circ{\mathop{p}\limits^{\leftarrow}}_{Iw,20}^*\circ i_{11}
        \end{aligned}
    \end{equation}
    Thus \begin{equation}
        \begin{aligned}
            {\mathop{p}\limits^{\rightarrow}}_{01,*}\circ \frac{1}{p^2}\Frob_{p^2,*}\circ {\mathop{p}\limits^{\rightarrow}}_{Iw,02,*}\circ Fr_{02\ra20}^{*}\circ{\mathop{p}\limits^{\leftarrow}}_{Iw,20}^*\circ i_{11}-\\{{\mathop{p}\limits^{\rightarrow}}_{11,*}}\circ {{\mathop{p}\limits^{\rightarrow}}_{Iw,21,*}}\circ{{\mathop{p}\limits^{\leftarrow}}_{Iw,21}^{*}}\circ {\mathop{p}\limits^{\leftarrow}}_{10}^{*}\circ {{\mathop{p}\limits^{\leftarrow}}_{11,*}}\circ i_{11}=&{\mathop{p}\limits^{\rightarrow}}_{01,*}\circ (1-\frac{1}{p^2}\Frob_{p^2,*})\circ {\mathop{p}\limits^{\rightarrow}}_{Iw,02,*}\circ Fr_{02\ra20}^{*}\circ{\mathop{p}\limits^{\leftarrow}}_{Iw,20}^*\circ i_{11}
        \end{aligned}
    \end{equation}
    This shows that it has image in $(1-{\rm Fr}_{p^2}){\rm H}^{4}_{\acute{e}t}(\overline{{Sh}}_{1,2},k_\lambda(3))_{\mathfrak{m}}.$ Hence we finish the proof
\end{proof}
\begin{lemma}
\label{fil}
    Suppose we are given two groups $A$ and $B$ and their filtrations $\Fil_{A,\bullet}$ and $\Fil_{B,\bullet}$ satisfying:
    \begin{enumerate}
        \item $0=\Fil_{-3}A\subseteq \Fil_{-2}A =\Fil_{-1} A\subseteq \Fil_{0} A= \Fil_{1} A\subseteq \Fil_{2} A=A.$
        \item $0=\Fil_{-1} B\subseteq \Fil_{0} B= B.$
    \end{enumerate}
    Every surjective map $A\to B$ preserving filtrations induces a surjective map
    $gr_{2}A\twoheadrightarrow\coker(gr_{0}A\rightarrow gr_{0}B).$
\end{lemma}
\begin{proof}
    In fact, we have
    \[gr_{2}(A)=\frac{\Fil_{2}A}{\Fil_{1}A}=\frac{A}{\Fil_{0}A}\twoheadrightarrow\frac{B}{\Im(\Fil_{0}A\rightarrow B)}=\frac{gr_{0}B}{\Im(gr_{0}A\rightarrow gr_{0}B)}=\coker(gr_{0}A\rightarrow gr_{0}B).\]
\end{proof}

By Lemma~\ref{fil}, there is a sujective map $\Phi:E_{2,\mathfrak{m}}^{-2,6}=E_{\infty,\mathfrak{m}}^{-2,6}\twoheadrightarrow H^1(\FF_{p^2},{\rm H}^4_{\acute{e}t}(\overline\Sh_{1,2},k_\lambda(3))_{\mathfrak{m}}).$

\begin{notation}
    Recall that $\widetilde\cSh_{1,2}(K_\gothp^1)$ is a stricly semistable scheme over $\ZZ_{p^2}.$ Let $\ZZ_{p^2}^{\rm ur}$ be the maximal unramified extension of $\ZZ_{p^2}.$ Denote by $i:\widetilde\Sh_{1,2}(K_\gothp^1)\ra \widetilde\cSh_{1,2}(K_\gothp^1),j:\widetilde {Sh}_{1,2}(K_\gothp^1)\ra \widetilde\cSh_{1,2}(K_\gothp^1), \bar i:\overline{\widetilde\Sh}_{1,2}(K_\gothp^1)\ra \widetilde\cSh_{1,2}(K_\gothp^1)_{\ZZ_{p^2}^{\rm un}},\bar j:\overline{\widetilde {Sh}}_{1,2}(K_\gothp^1)\ra \widetilde\cSh_{1,2}(K_\gothp^1)_{\ZZ_{p^2}^{\rm un}}.$ Define the nearby cycle to be $R\psi k_\lambda=\bar i_*R\bar j_*k_\lambda$ as an object of the derived category $D^+(\overline{\widetilde{\Sh}}(K_\gothp^1)_{1,2},k_\lambda).$ By \cite[Lemma 2.5]{Sai03}, $R\psi k_\lambda$ is an object in ${\rm Perv}(\overline{\widetilde{\Sh}}(K_\gothp^1)_{1,2},k_\lambda)[-4].$ Denote by $M_\bullet R\psi k_\lambda$ the monodromy filtration on $R\psi k_\lambda$ and for any $k$ an integer, $M_{\geq k} R\psi k_\lambda=R\psi k_\lambda/M_{k-1}R\psi k_\lambda.$ Let $gr_k R\psi k_\lambda=M_kR\psi k_\lambda/M_{k-1}R\psi k_\lambda.$
    
    For $0\leq p\leq 2,$ denote by $a_{p}: \overline Y^{(p)}\ra \overline{\widetilde{\Sh}}_{1,2}(K_\gothp^1)$ induced by closed immersions. Denote by $\Sh_{1,2}^{\rm \mu-ord}$ the $\mu$-ordinary locus of $\Sh_{1,2}$ and $\Sh_{1,2}^{\rm \mu-ss}$ its complement. Denote by $\Sh_{1,2}(K_\gothp^1)^{\rm \mu-ord}$ the preimage of $\Sh_{1,2}^{\rm \mu-ord}.$
\end{notation}

Consider the short exact sequence $0\ra gr_0 R\psi k_\lambda\ra M_{\geq 0}R\psi k_\lambda\ra M_{\geq 1}R\psi k_\lambda\ra 0,$ we have the following commutative diagram:
\[\begin{tikzcd}
[
    column sep=small, 
    row sep=scriptsize,
    cells={font=\footnotesize} 
]
	{{\tiny {\rm H}^{4}_{\acute{e}t}(\overline{\widetilde{\Sh}}_{1,2}(K_\gothp^1),gr_0R\psi k_\lambda(3))_\mathfrak{m}}} & {\tiny {\rm H}^{4}_{\acute{e}t}(\overline{\widetilde{\Sh}}_{1,2}(K_\gothp^1),gr_0 R\psi k_\lambda(3))_\mathfrak{m}} & {\tiny {\rm H}^4_{\acute{e}t}(\overline{\Sh}_{1,2},k_\lambda(3))_\mathfrak{m}^{\oplus 3}} \\
	{\tiny {\rm H}^{4}_{\acute{e}t}(\overline{\widetilde{\Sh}}_{1,2}(K_\gothp^1),M_{\geq 0}R\psi k_\lambda(3))_\mathfrak{m}} & {\tiny {\rm H}^{4}_{\acute{e}t}(\overline{\widetilde{\Sh}}_{1,2}(K_\gothp^1)^{\rm ord},gr_{0}R\psi k_\lambda(3)|_{\overline{\widetilde{\Sh}}_{1,2}(K_\gothp^1)^{\rm ord}})_\mathfrak{m}} & {\tiny {\rm H}^4_{\acute{e}t}(\overline{\Sh}_{1,2}^{\rm ord},k_\lambda(3))_\mathfrak{m}^{\oplus 3}} \\
	{\tiny E_{2,\mathfrak{m}}^{-2,6}(3)} & {\tiny X} & {\tiny H^5_{\overline{\Sh}^{\rm \mu-ss}_{1,2}}(\overline{\Sh}_{1,2},k_\lambda(3))_\mathfrak{m}^{\oplus 3}} \\
	0 & 0 & 0
	\arrow["\sim", from=1-1, to=1-2]
	\arrow[from=1-1, to=2-1]
	\arrow["{{{{\psi_1}}}}", from=1-2, to=1-3]
	\arrow[from=1-2, to=2-2]
	\arrow[from=1-3, to=2-3]
	\arrow["{{{{{\rho_1}}}}}", from=2-1, to=2-2]
	\arrow[from=2-1, to=3-1]
	\arrow["{{{{\psi_2}}}}", from=2-2, to=2-3]
	\arrow[from=2-2, to=3-2]
	\arrow[from=2-3, to=3-3]
	\arrow["{{{{{\rho_2}}}}}", from=3-1, to=3-2]
	\arrow[from=3-1, to=4-1]
	\arrow["{{{{\psi_3}}}}", from=3-2, to=3-3]
	\arrow[from=3-2, to=4-2]
	\arrow[from=3-3, to=4-3]
\end{tikzcd}\]
Here $\rho_1$ is the restriction map and $\psi_1,\psi_2$ are induced by the map in Theorem~\ref{I3}.
Moreover, we have the following commutative diagram:
\[\begin{tikzcd}
[
    column sep=small, 
    row sep=scriptsize,
    cells={font=\footnotesize} 
]
	{\tiny {\rm H}^4_{\acute{e}t}(\overline{\Sh}_{1,2},k_\lambda(3))_\mathfrak{m}\oplus Z } & {\tiny {\rm H}^{4}_{\acute{e}t}(\overline{\widetilde{\Sh}}_{1,2}(K_\gothp^1),gr_0 R\psi k_\lambda(3))_\mathfrak{m}} & 0 \\
	{\tiny {\rm H}^4_{\acute{e}t}(\overline{\Sh}^{\rm ord}_{1,2},k_\lambda(3))_\mathfrak{m}^{\oplus 3}} & {\tiny {\rm H}^{4}_{\acute{e}t}(\overline{\widetilde{\Sh}}_{1,2}(K_\gothp^1)^{\rm ord},M_{\geq 0}R\psi k_\lambda(3)|_{\overline{\widetilde{\Sh}}_{1,2}(K_\gothp^1)^{\rm ord}})_\mathfrak{m}} & {\tiny U} \\
	{\tiny Y} & {\tiny X} & {\tiny W} \\
	0 & 0 & 0
	\arrow["{{{\sigma_1}}}", from=1-1, to=1-2]
	\arrow["{{{\phi_1}}}"', from=1-1, to=2-1]
	\arrow[from=1-2, to=1-3]
	\arrow["{{{\psi_1}}}", from=1-2, to=2-2]
	\arrow[from=1-3, to=2-3]
	\arrow["{{{\sigma_2}}}", from=2-1, to=2-2]
	\arrow["{{{\phi_2}}}"', from=2-1, to=3-1]
	\arrow[from=2-2, to=2-3]
	\arrow[from=2-2, to=3-2]
	\arrow[from=2-3, to=3-3]
	\arrow["{{{\sigma_3}}}", from=3-1, to=3-2]
	\arrow[from=3-1, to=4-1]
	\arrow[from=3-2, to=3-3]
	\arrow[from=3-2, to=4-2]
	\arrow[from=3-3, to=4-3]
\end{tikzcd}\]

Here $Z$ denotes the contribution of the supersingular locus.
The left column is induced by Gysin sequence with $\phi_1(Z)=0.$ Hence $Y=H^5_{\overline{\Sh}_{1,2}^{\rm \mu-ss}}(\overline{\Sh}_{1,2},k_\lambda(3))_\mathfrak{m}^{\oplus 3}.$ 
Since $gr_2R\psi k_\lambda(3)|_{\overline{\widetilde{\Sh}}_{1,2}(K_\gothp^1)^{\rm ord}}=0,$ $U$ is the kernel of $\theta:{\rm H}^3_{\acute{e}t}(\overline N_{10}^{\rm ord},k_\lambda(2))_\mathfrak{m}\oplus {\rm H}^3_{\acute{e}t}(\overline N_{01}^{\rm ord},k_\lambda(2))_\mathfrak{m}\ra H^5_{\acute{e}t}(\overline{\Sh}^{\rm ord}_{1,2},k_\lambda(3))_\mathfrak{m}^{\oplus 3}$ induced by Gysin maps. Then $\rho_2$ factors $\sigma_2$ by the following lemma:
\begin{lemma}
    The maps
        ${\rm H}^3_{\acute{e}t}(\overline N_{10}^{\rm ord},k_\lambda(2))_\mathfrak{m}\ra H^5_{\acute{e}t}(\overline{\Sh}^{\rm ord}_{1,2},k_\lambda(3))_\mathfrak{m}^{\oplus 3}$ and $
        {\rm H}^3_{\acute{e}t}(\overline N_{01}^{\rm ord},k_\lambda(2))_\mathfrak{m}\ra H^5_{\acute{e}t}(\overline{\Sh}^{\rm ord}_{1,2},k_\lambda(3))_\mathfrak{m}^{\oplus 3} $
    are injective.
\end{lemma}
\begin{proof}
    By symmetry, it suffices to show that ${\rm H}^3_{\acute{e}t}(\overline N_{10}^{\rm ord},k_\lambda(2))_\mathfrak{m}\ra H^5_{\acute{e}t}(\overline{\Sh}^{\rm ord}_{1,2},k_\lambda(3))_\mathfrak{m}^{\oplus 3}$ is injective.
    The following daigram is commutative with each row and column exact.
    \[\begin{tikzcd}
    [
    column sep=small, 
    row sep=scriptsize,
    cells={font=\footnotesize} 
]
	{H^6(\overline{\rm Y}_{00}\bigcap\overline{\rm Y}_{10}-\overline{\rm Y}_{00}\bigcap\overline{\rm Y}_{11})_\mathfrak{m}=0} & {H^6_{(\overline{\rm Y}_{10}\bigcap\overline{\rm Y}_{11})^{\rm ss}}(\overline{\rm Y}_{10},k_\lambda(2))_\mathfrak{m}} & {H^6_{\overline{\rm Y}_{10}^{\rm ss}}(\overline{\rm Y}_{10},k_\lambda(3))_\mathfrak{m}} \\
	& {{\rm H}^4_{(\overline{\rm Y}_{10}\bigcap\overline{\rm Y}_{11})^{\rm ss}}(\overline{\rm Y}_{10}\bigcap\overline{\rm Y}_{11},k_\lambda(2))_\mathfrak{m}} \\
	& {{\rm H}^3_{\acute{e}t}(\overline{N}_{10}^{\rm ord},k_\lambda(2))_\mathfrak{m}} & {H^5_{\acute{e}t}(\overline{\Sh}_{1,2}^{\rm ord},k_\lambda(3))_\mathfrak{m}} \\
	& {{\rm H}^3_{\acute{e}t}(\overline{\rm Y}_{10}\bigcap\overline{\rm Y}_{11},k_\lambda(2))_\mathfrak{m}=0} & {H^5_{\acute{e}t}(\overline{\rm Y}_{10},k_\lambda(3))=0}
	\arrow[from=1-1, to=1-2]
	\arrow[hook, from=1-2, to=1-3]
	\arrow["{=}", from=2-2, to=1-2]
	\arrow[dashed, hook, from=3-2, to=1-3]
	\arrow[hook, from=3-2, to=2-2]
	\arrow["\theta_1", from=3-2, to=3-3]
	\arrow[from=3-3, to=1-3]
	\arrow[from=4-2, to=3-2]
	\arrow[from=4-3, to=3-3]
\end{tikzcd}\]
The equality of ${\rm H}^4_{(\overline{\rm Y}_{10}\bigcap\overline{\rm Y}_{11})^{\rm ss}}(\overline{\rm Y}_{10}\bigcap\overline{\rm Y}_{11},k_\lambda(2))_\mathfrak{m}$ and $H^6_{(\overline{\rm Y}_{10}\bigcap\overline{\rm Y}_{11})^{\rm ss}}(\overline{\rm Y}_{10},k_\lambda(2))_\mathfrak{m}$ comes from the smoothness of $\overline{\rm Y}_{10}\bigcap\overline{\rm Y}_{11}.$
Thus $\theta_1$ is injective from the diagram. The case for $\theta_2$ is the same and here we finish the proof.
\end{proof}

As in Proposition~\ref{Cal of Gr_2}, after changing a proper basis, we may suppose $\psi_i\circ\rho_i$ are $\diag\{1,1,1-\Frob_{p^2}^{*}\}$ for $i=1,2,3.$ Let $x\in E_{2,\mathfrak{m}}^{-2,6}.$ Take a lift $\tilde x\in {\rm H}^{4}_{\acute{e}t}(\overline{\widetilde{\Sh}}_{1,2}(K_\gothp^1),M_{\geq 0}R\psi k_\lambda(3))_\mathfrak{m}$ of $x.$ The third factor of $\psi_2\circ \rho_1(\tilde x)$ is mapped to zero in $H^5_{\overline{\Sh}^{\rm ss}_{1,2}}(\overline{\Sh}_{1,2},k_\lambda(3))_\mathfrak{m}^{\oplus 3}.$ Hence it can be lifted to an element $c(x)$ in ${\rm H}^4_{\acute{e}t}(\overline{\Sh}_{1,2},k_\lambda(3))_\mathfrak{m}.$ $\Phi(x)$ is exactly $c(x)$ in $H^1(\FF_{p^2},{\rm H}^4_{\acute{e}t}(\overline{\Sh}_{1,2},k_\lambda(3))_\mathfrak{m}).$ 

We define a map $\Psi:H^5_{\overline{\Sh}^{\rm ss}_{1,2}}(\overline{\Sh}_{1,2},k_\lambda(3))_\mathfrak{m}\ra H^1(\FF_{p^2},{\rm H}^4_{\acute{e}t}(\overline\Sh_{1,2},k_\lambda(3))_{\mathfrak{m}}).$ For $(y_1,y_2,y_3)\in H^5_{\overline{\Sh}^{\rm ss}_{1,2}}(\overline{\Sh}_{1,2},k_\lambda(3))_\mathfrak{m}^{\oplus 3},$ it has a lift $(y_1',y_2',y_3')\in {\rm H}^4_{\acute{e}t}(\overline{\Sh}^{\rm ord}_{1,2},k_\lambda(3))_\mathfrak{m}^{\oplus 3}.$ $\psi_2\circ\sigma_2((y_1',y_2',y_3'))=(y_1',y_2',(1-\Frob_{p^2}^{*})y_3').$ The third factor is mapped to zero in $H^5_{\overline{\Sh}^{\rm ss}_{1,2}}(\overline{\Sh}_{1,2},k_\lambda(3))_\mathfrak{m}$ and hence has a lift $c(y_3)$ in 
${\rm H}^4_{\acute{e}t}(\overline{\Sh}_{1,2},k_\lambda(3))_\mathfrak{m},$ which is independent of the lift. We define $\Psi(y_3)=c(y_3).$ Since the image of ${\rm H}^{4}_{\acute{e}t}(\overline{\widetilde{\Sh}}_{1,2}(K_\gothp^1),gr_0 R\psi k_\lambda(3))_\mathfrak{m}$ under $\psi_1$ is contained in ${\rm H}^4_{\acute{e}t}(\overline{\Sh}^{\rm ord}_{1,2},k_\lambda(3))_\mathfrak{m}^{\oplus 3}$ and $\Phi$ is surjective, $\Psi$ is surjective. 

Now we consider the relation of the level raising map in Theorem~\ref{main theorem n=3} with $\Psi.$
By definition of the cycle class map of the higher chow group, the following diagram is commutative:
\[\begin{tikzcd}[
    column sep=small, 
    row sep=scriptsize,
    cells={font=\footnotesize} 
]
	&& 0 && 0 \\
	0 & {{\rm L}} & {\bigoplus\limits_{i=1}^{3}H^1_{\acute{e}t}({\rm Y}_i^\circ,k_\lambda(1))_\mathfrak{m}} && {\bigoplus\limits_{1\leq i_1,i_2\leq 3}\bigoplus\limits_{z_1\neq z_2\in \Sh_{0,3}}{\rm H}^0_{\acute{e}t}({\rm Y}_{i_1,z_1}\cap {\rm Y}_{i_2,z_2},k_\lambda)_\mathfrak{m}} \\
	0 & {{\rm Ch}^{1}({\Sh}_{1,2}^{\rm ss},1,k_\lambda)_{\mathfrak{m}}} & {\bigoplus\limits_{i=1}^{3}{\rm H}^0_{\acute{e}t}({\rm Y}_i^\circ,\mathbb{G}_m)_\mathfrak{m}} && {\bigoplus\limits_{1\leq i_1,i_2\leq 3}\bigoplus\limits_{z_1\neq z_2\in \Sh_{0,3}}{\rm H}^0_{\acute{e}t}({\rm Y}_{i_1,z_1}\cap {\rm Y}_{i_2,z_2},k_\lambda)_\mathfrak{m}}
	\arrow[from=2-1, to=2-2]
	\arrow[from=2-2, to=2-3]
	\arrow[from=2-3, to=1-3]
	\arrow[from=2-3, to=2-5]
	\arrow[from=2-5, to=1-5]
	\arrow[from=3-1, to=3-2]
	\arrow["{\alpha_1}", from=3-2, to=2-2]
	\arrow[from=3-2, to=3-3]
	\arrow["{\alpha_2}", from=3-3, to=2-3]
	\arrow["div", from=3-3, to=3-5]
	\arrow["\sim"', from=3-5, to=2-5]
\end{tikzcd}\]
The map $\alpha_2$ is induced by the Kummer map and is surjective since ${\rm Y}_i^\circ$ are union of open subschemes of $\mathbb{A}^2$ and have trivial picard groups for $1\leq i\leq 3.$ Thus $\alpha_1$ is surjecitve. By definition, the image of $div$ is contained in the component such that ${{\rm Y}_{i_1,z_1}\cap {\rm Y}_{i_2,z_2}}$ has dimension $1.$ Thus we can suppose $i_2=i_1+1$ and $z_2\in {\rm T}_\gothp^{(1)}(z_1)$ as in Section~\ref{D}.

By purity, ${\bigoplus\limits_{i=1}^{3}H^1_{\acute{e}t}({\rm Y}_i^\circ,k_\lambda(1))_\mathfrak{m}}={\bigoplus\limits_{i=1}^{3}H^5_{{\rm Y}_i^\circ}({\Sh}_{1,2},k_\lambda(3))_\mathfrak{m}}$ and $\bigoplus\limits_{i=1}^{2}\bigoplus\limits_{z_2\in {\rm T}_{\gothp}^{(1)}(z_1)}{\rm H}^0_{\acute{e}t}({\rm Y}_{i_1,z_1}\cap {\rm Y}_{i_2,z_2},k_\lambda)_\mathfrak{m}=\bigoplus\limits_{i=1}^{2}\bigoplus\limits_{z_2\in {\rm T}_{\gothp}^{(1)}(z_1)}H^6_{{\rm Y}_{i,z_1}\cap {\rm Y}_{i+1,z_2}}(\Sh_{1,2},k_\lambda)_\mathfrak{m}.$ Thus it can be checked that ${\rm L}$ is isormophic to $H^5_{{\Sh}^{\rm ss}_{1,2}}(\Sh_{1,2},k_\lambda(3))_\mathfrak{m}.$  Denote by $\gamma$ the Gysin map $H^5_{{\Sh}^{\rm ss}_{1,2}}(\Sh_{1,2},k_\lambda(3))_\mathfrak{m}\ra H^5_{\acute{e}t}(\Sh_{1,2},k_\lambda(3))_\mathfrak{m}.$

By Hochschild-Serre spectral sequence, the following diagram is commutative with the two rows exact.
\[\begin{tikzcd}[
    column sep=small, 
    row sep=scriptsize,
    cells={font=\footnotesize} 
]
	0 & {H^1(\FF_{p^2},{\rm H}^4_{{\Sh}^{\rm ss}_{1,2}}(\Sh_{1,2},k_\lambda(3))_\mathfrak{m})} & {H^5_{{\Sh}^{\rm ss}_{1,2}}(\Sh_{1,2},k_\lambda(3))_\mathfrak{m}} && {{\rm H}^0(\FF_{p^2},H^5_{\overline{\Sh}^{\rm ss}_{1,2}}(\overline\Sh_{1,2},k_\lambda(3))_\mathfrak{m})} & 0 \\
	0 & {H^1(\FF_{p^2},{\rm H}^4_{\acute{e}t}(\Sh_{1,2},k_\lambda(3))_\mathfrak{m})} & {{\rm H}^4_{\acute{e}t}(\Sh_{1,2},k_\lambda(3))_\mathfrak{m}} && {{\rm H}^0(\FF_{p^2},H^5_{\acute{e}t}(\overline\Sh_{1,2},k_\lambda(3))_\mathfrak{m})} & 0
	\arrow[from=1-1, to=1-2]
	\arrow[from=1-2, to=1-3]
	\arrow[from=1-2, to=2-2]
	\arrow[from=1-3, to=1-5]
	\arrow["\gamma", from=1-3, to=2-3]
	\arrow[from=1-5, to=1-6]
	\arrow[from=1-5, to=2-5]
	\arrow[from=2-1, to=2-2]
	\arrow[from=2-2, to=2-3]
	\arrow[from=2-3, to=2-5]
	\arrow[from=2-5, to=2-6]
\end{tikzcd}\]

Since ${\rm H}^0(\FF_{p^2},H^5_{\acute{e}t}(\overline\Sh_{1,2},k_\lambda(3))_\mathfrak{m})=0,$ the surjectivity of the level raising map follows from the surjectivity of 
$$\gamma':{\rm H}^0(\FF_{p^2},H^5_{\overline{\Sh}^{\rm ss}_{1,2}}(\overline\Sh_{1,2},k_\lambda(3))_\mathfrak{m})\ra H^1(\FF_{p^2},{\rm H}^4_{\acute{e}t}(\Sh_{1,2},k_\lambda(3))_\mathfrak{m})/H^1(\FF_{p^2},{\rm H}^4_{{\Sh}^{\rm ss}_{1,2}}(\Sh_{1,2},k_\lambda(3))_\mathfrak{m}).
$$

Meanwhile, by Gysin sequence, the following sequence is exact.
\[\begin{tikzcd}[
    column sep=small, 
    row sep=scriptsize,
    cells={font=\footnotesize} 
]
	{{\rm H}^4_{\overline{\Sh}^{\rm ss}_{1,2}}(\overline\Sh_{1,2},k_\lambda(3))_\mathfrak{m}} & {{\rm H}^4_{\acute{e}t}(\overline\Sh_{1,2},k_\lambda(3))_\mathfrak{m}} && {{\rm H}^4_{\acute{e}t}(\overline\Sh_{1,2}^{\rm ord},k_\lambda(3))_\mathfrak{m}} & {H^5_{\overline{\Sh}^{\rm ss}_{1,2}}(\overline\Sh_{1,2},k_\lambda(3))_\mathfrak{m}} & 0
	\arrow["\rho", from=1-1, to=1-2]
	\arrow[from=1-2, to=1-4]
	\arrow[from=1-4, to=1-5]
	\arrow[from=1-5, to=1-6]
\end{tikzcd}\]

Substitute the left element of the sequence by $\ker \rho$ and applying the left exact functor $(\bullet)^{\Gal(\overline\FF_{p^2}/\FF_{p^2})},$ we get $\gamma'$ coincides with first boundary map, which can be checked by definition is exactly $\Psi$ and hence is surjective. This shows the level raising map is surjective and we finish the proof.

\appendix
\section{Ihara Lemma for \texorpdfstring{$n=2$}{n=2}}\label{I}
In this section, we give the proof of Ihara lemma for $n=2.$ Recall the Hecke action ${\rm T}={\rm T}_\gothp^{(1)}$ and ${\rm S}_\gothp$ defined in Definiton~\ref{S:Hecke action}. Defintion~\ref{G:Condition on closed subschemes}, Proposition~\ref{G:stalk of points} and Proposition~\ref{G:blow up} also holds for $n=2.$
    
Our main result in this section is:
\begin{theorem}
\label{Ihara}
Under the Hypothesis \ref{Main hypo} for $n=2,$ we have
    \begin{enumerate}
        \item (Definite Ihara) The map 
        \begin{equation*}
        {\mathrm{H}}^{0}({\overline{\Sh}_{0,2}(K_\gothp^1)},k_\lambda)_{\mathfrak{m}} \xrightarrow{(\lp,\rp)} {\mathrm{H}^{0}}({\overline{\Sh}}_{0,2},k_\lambda)_{\mathfrak{m}}^{\oplus 2}
        \end{equation*} is surjective, with the map induced by projection of ${{{\overline{\Sh}}}}_{0,2}(K_{\gothp}^{1})$ to ${{{\overline{\Sh}}}}_{1,1}.$
        \item (Indefinite Ihara) The map
        \begin{equation*}
        {\mathrm{H}}^{2}({{\overline{Sh}}}_{1,1}(K_{\gothp}^{1}),k_\lambda(2))_{\mathfrak{m}} \xrightarrow {(\lp,\rp)}({{{\overline{Sh}}}}_{1,1},k_\lambda(2))_{\mathfrak{m}}^{\oplus 2}
        \end{equation*}
        is surjective, with the map induced by projection of ${{{\overline{Sh}}}}_{1,1}(K_{\gothp}^{1})$ to ${{{\overline{Sh}}}}_{1,1}.$
    \end{enumerate}
\end{theorem}

To prove Theorem~\ref{Ihara}, we need to analyze the structure of ${\rm Y}_{ij}$ more carefully. 

Using deformation theory in Subsection~\ref{Deformaion}, we get:
\begin{proposition}
\label{tangent sheaf}
    For $1\leq i,j \leq 2,$ the tangent sheaf $T_{{\rm Y}_{ij}}=\mathcal{F}_{i}\oplus \mathcal{G}_{j}$ where $\mathcal{F}_{i}$ and $\mathcal{G}_{j}$ are:
    \begin{enumerate}
        \item $ \mathcal{F}_{0}=\cHom\big( \omega^{\circ}_{\mathcal{A}^{\vee}_2/{{\Sh}}_{1,1},1},\Lie^{\circ}_{\mathcal{A}^{'}/{{\Sh}}_{1,1},1}\big) ,$ $ \mathcal{F}_{1}=\cHom\big( \omega^{\circ}_{\mathcal{A}^{\vee}/{{\Sh}}_{1,1},1},\frac{\phi^{-1}_{*,1}(\omega^{\circ}_{\mathcal{A}^{\vee}/{{\Sh}}_{1,1},1})}{\omega^{\circ}_{\mathcal{A}^{\vee}/{{\Sh}}_{1,1},1}}\big) $ ,
        \item $\mathcal{G}_{0}=\cHom\big( \omega^{\circ}_{\mathcal{A}^{\vee}/{{\Sh}}_{1,1},2},\Lie^{\circ}_{\mathcal{A}/{{\Sh}}_{1,1},2}\big) 
        ,$ $ \mathcal{G}_{1}=\cHom\big( \frac{\omega^{\circ}_{\mathcal{A}^{\vee}_2/{{\Sh}}_{1,1},2}}{\phi_{*,2}(\omega^{\circ}_{\mathcal{A}^{\vee}/{{\Sh}}_{1,1},2})},\Lie^{\circ}_{\mathcal{A}^{'}/{{\Sh}}_{1,1},2}\big).$
    \end{enumerate}
     Here we suppose $(\mathcal{A},\lambda,\eta,\mathcal{A'}.\lambda',\eta',\phi)$ is the universal object over ${{\Sh}}_{1,1}.$
    The tangent sheaves $T_{{\rm Y}_{ij}}$ are all locally free of rank $2.$
\end{proposition}
\begin{proof}
    The proof for Proposition~\ref{tangent sheaff} still works.
\end{proof}

We need to define an action on ${{\Sh}}_{1,1}$ which is exactly the ``essential Frobenius'' as in \cite{Zho23}.
\begin{definition}
    We define $F$ to be the morphism on $\Sh_{1,1}$ which maps its any $S$-point $(A, \lambda , \eta )$ to $(A^{(p)},\lambda',\eta')$ satisfying $F$ acts on $A$ as the Frobenius morphism, $F\circ \eta'=\eta$ and $F^{\vee}\circ \lambda' \circ F=\lambda.$ Such $\lambda'$ and $\eta'$ exist by \cite[Theorem 2, Section 23]{Mum08} and $F$ on $A$ corresponds to the Frobenius map which is purely inseparable with trivial kernel.
\end{definition}

\begin{proposition}
\label{geometry1}
    The four closed subscheme $\rm {\rm Y_{00}}, \rm {\rm Y_{11}}, \rm {\rm Y_{01}}, \rm {\rm Y_{10}}$ have the following properties:
    \begin{itemize}
        \item $\rm {\rm Y_{00}}$ is a $\mathbb{P}^1\times\mathbb{P}^1$-bundle over ${\Sh}_{0,2}.$ More explicitly, ${\rm Y}_{00}$ is isomorphic to $C_2$ defined as a closed subscheme of ${\Sh}_{1,1}(K_{\gothp}^{1})$ satisfying for any $\mathbb{F}_{p^{2}}$ scheme $S,$ any $S$-point $(A, \lambda , \eta , A', \lambda',\eta', \phi),$ there exists $S$-points of ${{\Sh}}_{0,2}$ $(B_1,\lambda_{1},\eta_{1})$ and $(B_2,\lambda_{2},\eta_{2})$ such that there exists isogenies $B_{1}\rightarrow A\in Y_{2},$ $B_{2}\rightarrow A'\in Y_{1}$ and $B_{1}\in {\rm S}_{\gothp}(B_{2}).$

\item $\rm {\rm Y_{11}}$ is a $\mathbb{P}^1\times\mathbb{P}^1$ bundle over ${\Sh}_{0,2}.$ More explicitly, ${\rm Y}_{11}$ is isomorphic to $C_1$ defined as a closed subscheme of ${\Sh}_{1,1}(K_{\gothp}^{1})$ satisfying for any $\mathbb{F}_{p^{2}}$ scheme $S,$ any $S$-point $(A, \lambda , \eta , A', \lambda',\eta', \phi),$ there exists a $S$-point of ${{\Sh}}_{0,2}$ $(B_{1},\lambda_{1},\eta_{1})$ such that there exists isogenies $B_{1}\rightarrow A\in Y_{1}$ and $B_{1}\rightarrow A'\in Y_{2},$ 

        \item $\rm {\rm Y_{01}}$ and $\rm {\rm Y_{10}}$ are all isomorphic to ${\Sh}_{1,1}$ and they induce a morphism on ${\Sh}_{1,1}^{\oplus 2}$ characterized by $\begin{pmatrix} 1&F\\{\rm S}_\gothp^{-1}F &1\\
      \end{pmatrix}.$ 
    \end{itemize}
\end{proposition}
\begin{proof}
    Firstly, we show $\rm {\rm Y_{00}}$ and $\rm {\rm Y_{11}}$ are all $\mathbb{P}^1\times\mathbb{P}^1$ bundles over ${\Sh}_{0,2}.$
    For simplicity, we only prove it for $\rm {\rm Y_{00}}.$ There is a natrual map from $\rm {\rm Y_{00}}$ to ${\Sh}_{0,2}$ such that for any $\mathbb{F}_{p^{2}}$-scheme $S,$ an $S$-point $(A,\lambda,\eta, A',\lambda',\eta',\phi)$ of $\rm {\rm Y_{00}}$ is sent to $B$ which is given by Proposition~\ref{P:abelian-Dieud} with $\tcD(B)^{\circ}_{1}=\tcD(A')^{\circ}_{1}$ and $\tcD(B)^{\circ}_{2}=V\tcD(A')^{\circ}_{2}.$ With a simple argument of deformation theory, we can see such a map gives $\rm {\rm Y_{00}}$ the structure of $\mathbb{P}^1\times\mathbb{P}^1$ bundles over ${\Sh}_{0,2}.$

    Secondly, we show ${{\rm Y}_{00}}\simeq C_2.$ The proof of ${{\rm Y}_{11}}\simeq C_{1}$ is also similar. In fact, given any $S$-point $(A, \lambda , \eta , A', \lambda',\eta', \phi),$ we can construct $B,B'$ as follows. By Proposition~\ref{P:abelian-Dieud}, we get two $S$-points of ${{\Sh}}_{0,2}$ $B,B'$ from two pairs of dieudonne\'e modules $(\tcD(A')^{\circ}_{1},V\tcD(A')^{\circ}_{1})$ and $(V\tcD(A)^{\circ}_{2}, p\tcD(A)^{\circ}_{2})$ respectively. It is easy to check this gives us the desired isomophism. Thus we get the diagram.

    Thirdly, we show $\rm {\rm Y_{01}}$ and $\rm {\rm Y_{10}}$ are all isomorphic to ${\Sh}_{1,1}.$ Let $\alpha_{1}$ and $\beta_{2}$ be morphisms from $\rm {\rm Y_{01}}$ and $\rm {\rm Y_{10}}$ to ${\Sh}_{1,1}$ such that for any $\mathbb{F}_{p^{2}}$-scheme $S,$ an $S$-point $(A,\lambda,\eta, A',\lambda',\eta',\phi)$ of $\rm {\rm Y_{01}}$ or $\rm {\rm Y_{10}}$ is sent to $A'$ and let $\beta_{1}$ and $\alpha_{2}$ be morphisms from $\rm {\rm Y_{01}}$ and $\rm {\rm Y_{10}}$ to ${\Sh}_{1,1}$ such that for any $\mathbb{F}_{p^{2}}$-scheme $S,$ an $S$-point $(A,\lambda,\eta, A',\lambda',\eta',\phi)$ of $\rm {\rm Y_{01}}$ or $\rm {\rm Y_{10}}$ is sent to $A.$
    Then $\alpha_{1}$ and $\alpha_{2}$ are all isomorphisms and $\beta_{1}$ and $\beta_{2}$ are all purely inseparable morphisms which are bijective on points.

    What remains to show is that $\begin{pmatrix}
        1& \beta_{2}\circ\alpha_{2}^{-1}\\
        \beta_{1}\circ\alpha_{1}^{-1}&1\\
    \end{pmatrix}$
    on ${\Sh}_{1,1}^{\oplus 2}$ induces $\begin{pmatrix} 1&F\\{\rm S}_\gothp^{-1}F &1\\
      \end{pmatrix}.$

In fact, for any $\mathbb{F}_{p^{2}}$-scheme $S,$ we have the following claim:
\begin{enumerate}
    \item Any $S$ point $(A,\lambda,\eta, A',\lambda',\eta',\phi)$ in $\rm {\rm Y_{01}}$ is isomorphic to $(A,\lambda,\eta, A',\lambda',\eta',F^{-1}{\rm S}_\gothp)$\footnote{It may not be valid to write $F^{-1}$ as an isogeny. It just means after the Hecke action ${\rm S}$ acts on $A,$ it is isomorphic to the image of $A'$ under $F.$}
    \item Any $S$ point $(A,\lambda,\eta, A',\lambda',\eta',\phi)$ in $\rm {\rm Y_{10}}$ is isomorphic to $(A,\lambda,\eta, A^{(p)},\lambda',\eta',F)$
\end{enumerate}

As to (1), we find that $F\circ \phi$ maps the dieudonn\'e of $A$ $(\tcD(A)_{1}^{\circ}, \tcD(A)_{2}^{\circ})$ to $(p\tcD({A'}^{(p)})_{1}^{\circ}, p\tcD({A'}^{(p)})_{2}^{\circ}),$ since $\phi_{1,*}\tcD({A})_{1}^{\circ}=F\tcD(A')_{2}^{\circ}$ and $\phi_{2,*}\tcD({A})_{2}^{\circ}=F\tcD(A')_{1}^{\circ}.$ Moreover, we can see $F\circ \phi$ gives an isogeny from $A$ to $A'^{(p)}$ such that $A\in {\rm S}_\gothp(A').$
(2) can be proved similarly by consider the Frobenius action on $A$ and the uniqueness of Proposition~\ref{P:abelian-Dieud}.
With the claim, we finish the proof of the proposition.
\end{proof}

With Proposition~\ref{geometry1}, we can describe the intersections of the four closed subschemes as below:
\begin{proposition}
\label{geometry2}
For the intersections of the four closed subschemes, we have:
\begin{enumerate}
    \item $\rm {\rm Y_{00}}\bigcap \rm {\rm Y_{01}}\simeq\rm Y_{1}$; $\rm {\rm Y_{00}}\bigcap \rm {\rm Y_{10}}\simeq\rm Y_{2}$;
    \item $\rm {\rm Y_{11}}\bigcap \rm {\rm Y_{01}}\simeq\rm Y_{2}$; $\rm {\rm Y_{11}}\bigcap \rm {\rm Y_{10}}\simeq\rm Y_{1}$;
    \item ${\rm Y_{00}}\bigcap {\rm Y_{11}}\simeq{\Sh}_{0,3}(K_\gothp^{1}).$
\end{enumerate}
\end{proposition}
\begin{proof}
    The proof of $(1)$ and $(2)$ are similar. For simplicity, we only give the proof of ${\rm Y}_{00}\bigcap {\rm Y}_{01}\simeq {\rm Y}_{01}$ and $(3).$
    
    First, we show ${\rm Y}_{00}\bigcap {\rm Y}_{01}\simeq {\rm Y}_{1}.$ We define a morphism $\alpha:{\rm Y}_{00}\bigcap{\rm Y}_{01}\ra {\rm Y}_{1}$ as following: Let $k$ be a perfect field containinng $\FF_{p^{2}},$ suppose $y=(A,\lambda,\eta,A',\lambda',\eta',\phi)$ is a $k$-point of ${\rm Y}_{00}\bigcap {\rm Y}_{01}.$ We let $\tcE_1=\tcD(A')_{1}^{\circ}$ and $\tcE_2=V\tcD(A')_{1}^{\circ}\subseteq \tcD(A')^{\circ}_{2}.$
Then it can be checked that $F(\tilde \calE_i) \subseteq \tilde \calE_{3-i}$ and $V(\tilde \calE_i) \subseteq \tilde \calE_{3-i}$ for $i = 1,2.$
Applying PropositionProposition~\ref{P:abelian-Dieud} with $\tcE_1,\tcE_2,$ we get $(B,\lambda'', \eta'')$ a point of ${\Sh}_{0,n}$ and an isogeny $\phi': B\ra A',$ such that $(A',\lambda', \eta',B,\lambda'',\eta'',\phi')\in {\rm Y}_{1}$
In this way, we define $\alpha(y)=(A',\lambda', \eta',B,\lambda'',\eta'',\phi').$ Now we construct $\beta$ and check it is the converse of $\alpha.$ Suppose $y'=(A',\lambda',\eta',B,\lambda'',\eta''.\phi')$ is a $k$-point of ${\rm Y}_{1}.$ Let $\tcE_1=F\tcD(A')_{2}^{\circ}$ and $\tcE_2=F\tcD(A')_{1}^{\circ}\subseteq \tcD(A')^{\circ}_{2}.$
Then it can be checked that $F(\tilde \calE_i) \subseteq \tilde \calE_{3-i}$ and $V(\tilde \calE_i) \subseteq \tilde \calE_{3-i}$ for $i = 1,2.$
Applying Proposition~\ref{P:abelian-Dieud} with $\tcE_1,\tcE_2,$ we get $(A,\lambda,\eta)$ a point of ${\Sh}_{1,1}$ and an isogeny $\phi: A\ra A',$ such that $(A,\lambda, \eta,A',\lambda',\eta',\phi)\in {\rm Y}_{00}\bigcap{\rm Y}_{01}.
$
In this way, we define $\beta(y')=(A,\lambda, \eta,A',\lambda',\eta',\phi').$ It can be checked directly $\alpha$ and $\beta$ are inverse of each other on points.
By deformation theory, $\alpha$ induces a bijection on tangent spaces. Thus $\alpha$ is an isomorphism.

Now we give the proof of $(3).$ Similar as above, we construct a morpshim $\alpha'$ as follows: Let $k$ be a perfect field containing $\FF_{p^{2}},$ suppose $y=(A,\lambda,\eta,A',\lambda',\eta',\phi)\in {\rm Y}_{00}\bigcap {\rm Y}_{11}(k).$ We let $\tcE_1=V\tcD(A)_{2}^{\circ},\tcE'_1=\tcD(A')_{1}^{\circ}$ and $\tcE_2=p\tcD(A)_{2}^{\circ}, \tcE_{2}'=V\tcD(A')_{1}^{\circ}.$
Applying Proposition~\ref{P:abelian-Dieud} with $\tcE_1,\tcE_2$ and $\tcE_{1}',\tcE'_2$ repectively, we get two points $(B_1,\lambda_1,\eta_1),(B_2,\lambda_{2},\eta_{2})$ of $\Sh_{0,2}$ and two isogenies $\phi_{1}: B_{1}\ra A,\phi_{2}:B_{2}\ra A'$ such that $(A,\lambda, \eta,B_1,\lambda_1,\eta_1,\phi_1)
$ and $(A',\lambda', \eta',B_2,\lambda_2,\eta_2,\phi_2)\in {\rm Y}_{2}.$ In this way, we define $\alpha'(y)=(B_1,\lambda_1,\eta_1,B_2,\lambda_2,\eta_2,\phi_2^{-1}\circ\phi\circ\phi_{1})\in {\Sh}_{0,2}(K_\gothp^{1}).$ Using a similar argument as above, we can show $\alpha'$ is an isomorphism. This concludes the proof.
\end{proof}

From now on, we consider the blow up $\widetilde\cSh_{1,1}(K_\gothp^1)$ and $\widetilde\Sh_{1,1}(K_\gothp^1)$ as considered in Proposition~\ref{G:blow up} and \cite[Section 6]{Hel12}. Applying \cite[Corollary 2.8]{Sai03} to $\widetilde\cSh_{1,1}(K_{\gothp}^{1}),$ we can get the weight spectral sequence.

Let $Y_{\overline{\FF}_p}^{(i)}$ be the union of intersection of $i+1$ irreducible components of the special fiber. We can write $E_{1}$-page of the weight spectral sequence as the following diagram:
\[\begin{tikzcd}[
    column sep=small, 
    row sep=scriptsize,
    cells={font=\footnotesize} 
]
	{{\rm H}_{\acute{e}t}^{0}(Y^{(2)}_{\overline{\FF}_{p}})(-2)} & {{\rm H}_{\acute{e}t}^{2}(Y^{(1)}_{\overline{\FF}_{p}})(-1)} & {{\rm H}_{\acute{e}t}^{4}(Y^{(0)}_{\overline{\FF}_{p}})} & 0 & 0 \\
	0 & {{\rm H}_{\acute{e}t}^{1}(Y^{(1)}_{\overline{\FF}_{p}})(-1)} & {{\rm H}_{\acute{e}t}^{3}(Y^{(0)}_{\overline{\FF}_{p}})} & 0 & 0 \\
	0 & {{\rm H}_{\acute{e}t}^{0}(Y^{(1)}_{\overline{\FF}_{p}})(-1)} & {{\rm H}_{\acute{e}t}^{2}(Y^{(0)}_{\overline{\FF}_{p}}) \oplus {\rm H}_{\acute{e}t}^{0}(Y^{(2)}_{\overline{\FF}_{p}})(-1)} & {{\rm H}_{\acute{e}t}^{2}(Y^{(1)}_{\overline{\FF}_{p}})} & 0 \\
	0 & 0 & {{\rm H}_{\acute{e}t}^{1}(Y^{(0)}_{\overline{\FF}_{p}})} & {{\rm H}_{\acute{e}t}^{1}(Y^{(1)}_{\overline{\FF}_{p}})} & 0 \\
	0 & 0 & {{\rm H}_{\acute{e}t}^{0}(Y^{(0)}_{\overline{\FF}_{p}})} & {{\rm H}_{\acute{e}t}^{0}(Y^{(1)}_{\overline{\FF}_{p}})} & {{\rm H}_{\acute{e}t}^{0}(Y^{(2)}_{\overline{\FF}_{p}})}
	\arrow[from=1-1, to=1-2]
	\arrow[from=1-2, to=1-3]
	\arrow[from=2-2, to=2-3]
	\arrow[from=3-2, to=3-3]
	\arrow[from=3-3, to=3-4]
	\arrow[from=4-3, to=4-4]
	\arrow[from=5-3, to=5-4]
	\arrow[from=5-4, to=5-5]
\end{tikzcd}
\]
with the middle term of the last row at index $(0,0)$ and $E_{1}^{p,q}=0$ for $q>2$ or $q<-2.$ Here we omit the coefficient ring $k_\lambda$ and use `$(a)$' to stand for the Tate twist $k_\lambda(a)$ for some integer $a.$

It is easy to see that at the index $(1,0),$ the spectral sequence degenerates at $E_{2}$ page. Hence the short sequence at the bottom of the $E_{1}$-page is exact after localizing at $\mathfrak{m}.$
Now we give the proof of the definite Ihara lemma by analyzing such a short exact sequence:
\begin{proof}[Proof of theorem 6.2(1)]
    After localization, the short exact sequence above is equivalent to 
    \begin{equation*}
        {\rm H}^{0}({\overline{\Sh}}_{0,2},k_\lambda)_{m}^{\oplus 2}\xrightarrow{\alpha} {\rm H}^{0}({\overline{\Sh}}_{0,2},k_\lambda)_{m}^{\oplus 4}\oplus {\rm H}^{0}({\overline{{\Sh}}_{0,3}(K_\gothp^{1})},k_\lambda)_{m}\xrightarrow{\beta} {\rm H}^{0}({\overline{{\Sh}}_{0,3}(K_\gothp^{1})},k_\lambda)_{m}^{\oplus 2}
    \end{equation*}
    with $\alpha^{t}=\left(\begin{array}{ccccc}
       -{\rm S}_\gothp & -1 & 0 & 0 & -{\mathop{p}\limits^{\leftarrow}}^{*} \\
        0 & 0 & 1 & 1 & {\mathop{p}\limits^{\rightarrow}}^{*}
    \end{array}\right)$ and $\beta=\left(\begin{array}{ccccc}
       {\mathop{p}\limits^{\leftarrow}}^{*}  & 0 & {\mathop{p}\limits^{\rightarrow}}^{*} & 0 & -{\rm T}{\rm S}_\gothp \\
        0 & {\mathop{p}\limits^{\rightarrow}} & 0 & {\mathop{p}\limits^{\leftarrow}}^{*} & -1
    \end{array}\right),$ where we use $\alpha^{t}$ to express the transverse of $\alpha.$
    Since $\Im \alpha=\Ker \beta,$ we get for any $(x,y,z,w,r)\in {\rm H}^{0}({\overline{\Sh}}_{0,2},k_\lambda)_{m}^{\oplus 4}\oplus {\rm H}^{0}({\overline{{\Sh}}_{0,3}(K_\gothp^{1})},k_\lambda)_{m}$ satisfying ${\mathop{p}\limits^{\leftarrow}}^{*} x+{\mathop{p}\limits^{\rightarrow}}^{*} z-{\rm T}{\rm S}_\gothp r={\mathop{p}\limits^{\leftarrow}}^{*} y+{\mathop{p}\limits^{\rightarrow}}^{*} w-r=0,$ i.e $(x,y,z,w,r)\in \Ker \beta,$ we get that there exists $(s,t)\in {\rm H}^{0}({\overline{\Sh}}_{0,2},k_\lambda)_{m}^{\oplus 2}$ such that $(x,y,z,w,r)=(-{\rm S}_\gothp s,-s,t,t,,-{\mathop{p}\limits^{\leftarrow}}^{*} s+{\mathop{p}\limits^{\rightarrow}}^{*} t).$ Therefore, $x={\rm S}_\gothp y$ and $z=w.$ Since ${\rm S}_\gothp$ is an isormorphism as a morphism, we get if $-{\mathop{p}\limits^{\leftarrow}}^{*} s+{\mathop{p}\limits^{\rightarrow}}^{*} t=0,$ then $s=t=0.$ Hence the map ${\rm H}^{0}({\overline{\Sh}}_{0,2})_{m}^{\oplus 2}\xrightarrow{({\mathop{p}\limits^{\leftarrow}}^{*},{\mathop{p}\limits^{\rightarrow}}^{*})} {\rm H}^{0}({\overline{{\Sh}}_{0,3}(K_\gothp^{1})})_{m}$ is injective. 
    By Poincar\'e duality, we get the definite Ihara lemma.
\end{proof} 

The proof of Indefinite Ihara lemma from the definite Ihara lemma makes no difference from the case $n\geq 3,$ so we omit the proof of Indefinite Ihara lemma here.

\section{Arithmetic level raising theorem for  \texorpdfstring{$n=2$}{n=2}}\label{A}
Recall the level raising map we constructed in Section~\ref{I3}. We have the following theorem for $n=2$:
\begin{theorem}\label{main theorem n=2}
    Under Hypothesis \ref{Main hypo} and \ref{Main hypo2}, the level raising map \[{\rm Ch}^{1}({{\Sh}_{1,1}},1,k_\lambda)_{m}\rightarrow{\rm H}^{1}(\mathbb{F}_{p^{2}}, {\rm H}^{2}({\overline{\Sh}_{1,1}}, k_\lambda(2))_{\mathfrak{m}})\] is surjective.
\end{theorem}
By Theorem~\ref{Ihara}, we get a surjective map ${\rm H}^{2}({\overline{Sh}}_{1,1}(K_{\gothp}^{1}),k_\lambda(2))_{\mathfrak{m}}\twoheadrightarrow {\rm H}^{2}_{\acute{e}t}({\overline{Sh}}_{1,1},k_\lambda(2))^{\oplus 2}_{\mathfrak{m}}.$ Since $\cSh_{1,1}$ is smooth and irreducible, the monodromy filtration of $\cSh_{1,1}$ concentrates on itself. The surjective map above preserves monodromy filtrations by \cite[Corollary 2.12, 2.14]{Sai03}.

\begin{proposition}\label{vanis}
    For the weight spectral sequence of $\cSh_{1,1}(K_\gothp^1)$, it satisfies $H^{1}_{\acute{e}t}(Y^{(1)}_{\overline{\FF}_{p}},k_\lambda)=0$ in the $E_{1}$-page and after localizaing at $\mathfrak{m},$ the $E_{2}$-page can be expressed as the following diagram:
    \[\begin{tikzcd}[
    column sep=small, 
    row sep=scriptsize,
    cells={font=\footnotesize} 
]
	{E^{-2,4}_{2,\mathfrak{m}}} && 0 && 0 && 0 && 0 \\
	0 && 0 && 0 && 0 && 0 \\
	0 && 0 && {E^{0,2}_{2,\mathfrak{m}}} && 0 && 0 \\
	0 && 0 && 0 && 0 && 0 \\
	0 && 0 && 0 && 0 && {E^{2,0}_{2,\mathfrak{m}}}
\end{tikzcd}.\] Moreover, we have $E_{2,m}^{-2,4}={\rm Ch}^{1}({\Sh}^{\rm ss}_{1,1},1,k_\lambda)_{\mathfrak{m}}.$
\end{proposition}
\begin{proof}
     Since ${\rm H}^i_{\acute{e}t}(Y_{\overline{\FF}_p}^{(1)},\cO_\lambda)_\mathfrak{m}$ is torsion-free
    for any integer $i.$ We get the spectral sequence degenerates at $E_2$-page.

    The vanishing of ${\rm H}^1_{\acute{e}t}(Y^{(1)}_{\overline{\FF}_p},k_\lambda)$ is obvious as $Y^{(1)}_{\overline{\FF}_p}$ can be expressed as union of $\PP^1$-bundles over Shimura sets. This shows $E^{-1,3}_{2,\mathfrak{m}}$ and $E^{1,1}_{2,\mathfrak{m}}$ are zero. The calculation of $E_{2,\mathfrak{m}}^{-2,4}$ is the same as Proposition~\ref{gr_2 calculation}.
    Here we finish the proof. 
\end{proof}

Thus we get a filtration of ${\rm H}^{2}(\overline{Sh}_{1,1}(K_{\gothp}^{1}),k_\lambda)_{m}.$ By Lemma~\ref{fil}, we get a surjective map
$$E_{2,m}^{-2,4}(2)\twoheadrightarrow \coker(E_{2,m}^{0,2}(2)\xrightarrow{\alpha} {\rm H}^{2}_{\acute{e}t}(\overline{\Sh}_{1,1},k_\lambda(2))_{\mathfrak{m}}),$$ which we denote by $\Phi.$
Now we give the proof of Theorem~\ref{main theorem n=2}.
\begin{proof}
By the same reason as in the proof of Theorem~\ref{main theorem n=3}, we get $\coker\alpha\simeq {\rm H}^{1}(\mathbb{F}_{p^{2}}, {\rm H}^{2}({\overline{\Sh}_{1,1}}, k_\lambda(2))_{\mathfrak{m}}).$ To get the Theorem~\ref{main theorem n=2} from the surjectivity of $\Phi$ is also the same of Theorem~\ref{main theorem n=3} only to note that $gr_1R\Psi k_\lambda$ on $\overline{\widetilde\cSh}_{1,1}(K_\gothp^1)^{\rm ord}$ is trivial.
\end{proof}

\newpage
\bibliographystyle{alpha}
\bibliography{main}
\end{document}